\documentclass{article}

\usepackage[a4paper, left=30mm, right=20mm, top=20mm, bottom=20mm]{geometry}

\usepackage{graphicx}

\usepackage[utf8]{inputenc}
\usepackage[T2A,T3,T1]{fontenc}

\DeclareSymbolFont{tipa}{T3}{cmr}{m}{n}
\DeclareMathAccent{\invbreve}{\mathalpha}{tipa}{16}

\usepackage{amsmath, amsfonts, amssymb, bm}

\usepackage{yfonts}
\usepackage{mathrsfs}
\usepackage{pifont}
\usepackage{ifsym}
\usepackage{bbm}

\DeclareMathAlphabet\mathbfcal{OMS}{cmsy}{b}{n}
\usepackage{MnSymbol}
\usepackage{xcolor}
\usepackage[colorlinks=true]{hyperref}
\usepackage{caption}
\usepackage{subcaption}
\usepackage{tabularx}
\usepackage{epigraph} 
\usepackage{algorithm}
\usepackage{todonotes}

\definecolor{darkblue}{rgb}{0.0, 0.0, 0.55}
\definecolor{forestgreen}{rgb}{0.13, 0.55, 0.13}
\definecolor{darkteal}{rgb}{0.0, 0.24, 0.18}

\hypersetup{
  linkcolor=darkblue,
  urlcolor=darkteal,
  citecolor=forestgreen
}

\definecolor{w}{RGB}{231, 160, 76}
\definecolor{l}{RGB}{37, 84, 138}

\newcommand{\wbox}{\textcolor{w}{\blacksquare}}
\newcommand{\lbox}{\textcolor{l}{\blacksquare}}

\newcommand{\wb}{\square}
\newcommand{\bb}{\blacksquare}

\usepackage{tikz-cd}
\usetikzlibrary{decorations.pathreplacing,fit,shapes.geometric}

\usepackage{savesym} 

\savesymbol{Cross}
\usepackage{bbding}
\restoresymbol{bb}{Cross}

\usepackage{environ}

\usepackage{amsthm}
\usepackage{etoolbox} 

\theoremstyle{definition}

\newtheorem{definition}{Definition}[section]
\AtBeginEnvironment{definition}{%
  \pushQED{\qed}%
}
\AtEndEnvironment{definition}{\popQED}

\newtheorem{theorem}{Theorem}[section]
\AtBeginEnvironment{theorem}{%
  \pushQED{\qed}%
}
\AtEndEnvironment{theorem}{\popQED}

\newtheorem{lemma}{Lemma}[section]
\AtBeginEnvironment{lemma}{%
  \pushQED{\qed}%
}
\AtEndEnvironment{lemma}{\popQED}

\newtheorem{conjecture}{Conjecture}[section]
\AtBeginEnvironment{conjecture}{%
  \pushQED{\qed}%
}
\AtEndEnvironment{conjecture}{\popQED}

\newtheorem*{remark}{Remark}
\AtBeginEnvironment{remark}{%
  \pushQED{\qed}%
}
\AtEndEnvironment{remark}{\popQED}

\AtBeginEnvironment{proofsketch}{%
  \pushQED{\qed}%
}
\AtEndEnvironment{lemma}{\popQED}

\newcounter{gametable}

\newenvironment{gametable}[1][htb]
  {\refstepcounter{gametable}
   \begin{table}[#1]%
  }
  {\end{table}}

\makeatletter
    \@ifdefinable{\PI}{\def\PI/{\mbox{Player 1}}}
    \@ifdefinable{\PII}{\def\PII/{\mbox{Player 2}}}
\makeatother

\newcommand{\G}[1]{$\textswab{Game}(#1)$}
\newcommand{\BG}[1]{$\textswab{BGame}(#1)$}
\newcommand{\SG}[1]{$\textswab{SGame}(#1)$}
\newcommand{\CG}[1]{$\textswab{CGame}(#1)$}

\newcommand{\InfG}[1]{$\overline{\textswab{Game}}(#1)$}
\newcommand{\InfBG}[1]{$\overline{\textswab{BGame}}(#1)$}

\usepackage{newunicodechar}

\makeatletter
\DeclareRobustCommand{\Hbar}{%
\text{
  \hmode@bgroup
  \vphantom{$H$}%
  \sbox\z@{$H$}%
  \ooalign{%
    $H$\cr
    \hidewidth
    \kern 0.1em 
    \vrule
      height \dimexpr 0.7\ht\z@+0.1ex\relax
      depth  -0.7\ht\z@
      width  0.8\wd\z@
    \hidewidth\cr
  }%
  \egroup
}
}
\makeatother

\newcommand{\E}{\mathbb{E}}

\DeclareMathOperator\supp{supp}

\usepackage{rotating}

\usepackage{CJKutf8}

\usepackage{halloweenmath}

\usepackage{marvosym}

\usepackage{dsfont}

\DeclareRobustCommand{\loopedsquare}{\text{\raisebox{-.035em}{\includegraphics[height=.6em]{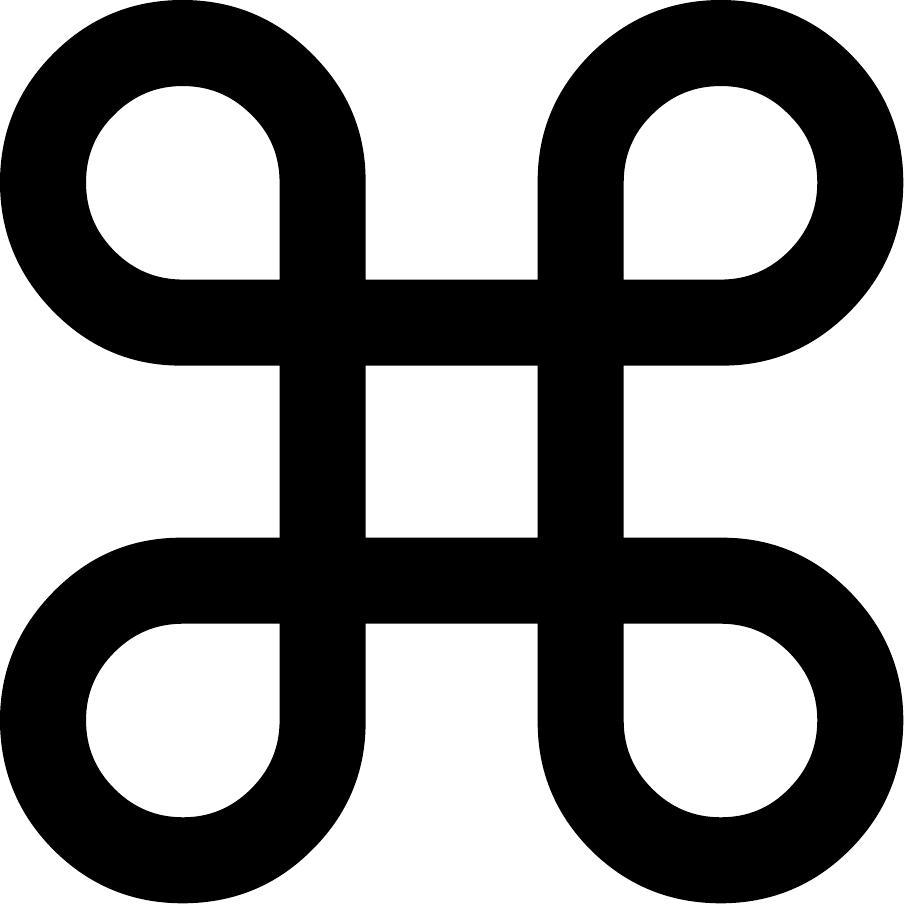}}}}

\newcommand{\sampi}{\ensuremath{\includegraphics[width=0.6em]{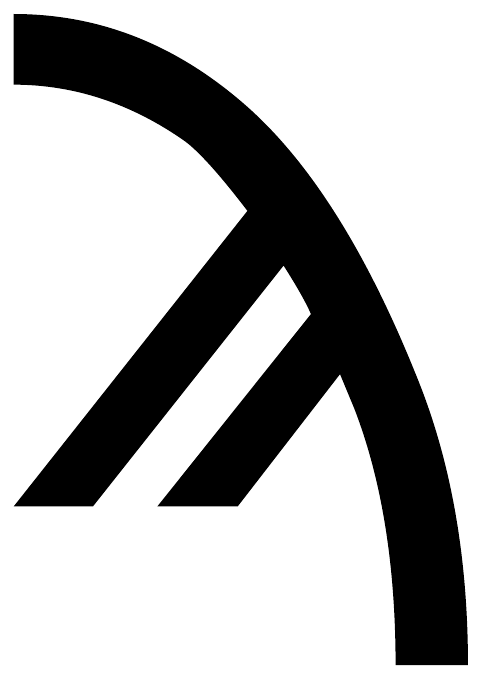}}}

\newcommand\sh[1]{\ensuremath{\mathop{\text{\normalfont\fontencoding{T2A}\selectfont ш}}#1}}

\newcommand{\hexacube}{
    \begin{tikzpicture}[scale=0.15, baseline=-0.5ex]
        \pgfmathsetmacro{\hdist}{sqrt(3)/2}

        \draw (0,1) -- (\hdist,0.5) -- (\hdist,-0.5) -- (0,-1) -- (-\hdist,-0.5) -- (-\hdist,0.5) -- cycle;

        \draw (0,0) -- (0,-1);
        \draw (0,0) -- (\hdist,0.5);
        \draw (0,0) -- (-\hdist,0.5);
    \end{tikzpicture}
}

\title{Statistical Games \\ \large Playful approach to statistics}
\author{József Konczer  \\
        \href{mailto:konczer.j@gmail.com}{konczer.j@gmail.com},
        \href{https://konczer.github.io/}{konczer.github.io}
        }
\date{February 2024}

\begin{document}

\maketitle

\section{Brief introduction}

\epigraph{\textit{It might appeal to some.}}
{Abraham Wald \cite{book:Savage}}

The first and main part of this work has a mathematical character.
It explores and analyses a few simple two-player non-cooperative games (for the explicit definition, see Def.~\ref{def:FisherGame}, or Def.~\ref{def:BayesianGame}), in which many concepts of probability theory and statistics naturally emerge.
In these games -- termed Fisher and Bayesian games -- an adversarial player can choose from a set of possible scenarios, while the other player can collect data, based on which, she has to make a guess or bet on the scenarios.
Besides the mathematical exploration of games, in which concepts from Frequentist and Bayesian statistics can be identified in equilibrium, in Section~\ref{sec:StatisticalGames}, a generalized betting game, termed ``Statistical game'' is introduced (for the definition, see Def.~\ref{def:StatisticalGame}).
Unification of Bayesian and Fisher games is possible by interpreting these as general Statistical games, differing only in the agent's relative risk aversion.
From a mathematical point of view, the emergent structures and nontrivial limit behaviour of such games, along with their equilibrium solutions, seem worthy of detailed examination.
This work can be viewed as the beginning of the mathematical investigation and exploration of statistical games.

Later, in Section~\ref{sec:PhilosophiPart}, these games are proposed to be models or analogues for statistical inference itself.
This suggestion is more philosophical -- and in some sense more radical. The proposition aims to ground statistical and probabilistic concepts with non-cooperative games, instead of devices of chance or subjective degree of belief.
This can be viewed as a different framework for the \emph{interpretation} of probability and related statistical procedures. 

A secondary contribution of this work
comes from a non-exhaustive but broad review of the diverse scientific literature, spread both in time and context.
The main philosophical and some technical concepts promoted in this work have been present in the literature in a fragmented form, often referred to as minimax regret criterion \cite{paper:Milnor,book:LuceRaiffa}. (The ideas of Wald \cite{book:Wald}, Savage \cite{book:Savage}, Good \cite{paper:GoodWeather}, Kelly \cite{Kelly}, Kashyap \cite{Kashyap1974,Kashyap1971} already contained the fundamental concepts from which the framework could have been constructed.)
The scope of the topic -- decision making, statistics, and probabilistic inference -- is enormous and highly interdisciplinary.
An incomplete list of related fields includes: Economics \cite{book:JuliaUncertainty}, Philosophy \cite{sep:InterpretationOfProbability,book:ScientificExplanationBraithwaite}, Statistics \cite{book:CoxStatistics}, Computer science \cite{book:InformationTheory,Kelly}, Mathematics \cite{book:Klenke}, Physics \cite{book:PathriaStatisticalMechanics,book:ErrorAnalysis}, Biology \cite{book:EvolutionCognition}, Machine learning \cite{book:ProbabilisticMachineLearning,book:Bishop,book:RL, book:NorvigRussell} et cetera\footnote{Finance \cite{book:ValueAtRisk,book:FinancialRiskManagerHandbook}, Control theory \cite{book:Control}, Operations Research \cite{book:OperationsResearch,book:IntroductionOperationsResearch}\dots}.
The collection of related ideas -- which do not necessarily refer to each other -- and presenting these concepts in a unified, coherent framework will hopefully inspire further research and stimulate interdisciplinary collaboration.

Hopefully, the presented simple but clear toy models can serve as a solid foundation for future research and development, and -- together with the listed future directions in Section~\ref{sec:FutureWork} -- build a compelling case for a more general and coherent framework for decision making in the face of uncertainty.

\newpage

\tableofcontents

\newpage

\section{Binary decisions}

Also known as forced binary choice \cite{book:StatisticalInference}, binary classification problem \cite{book:DecisionTheory} or two-category classification \cite{book:PatternClassification}. It can be viewed as a hypothesis testing problem \cite{book:StatisticalInference,book:CoxStatistics}, in which one alternative must be selected.

\subsection{Description of the game}

\begin{definition}[Fisher game]
\label{def:FisherGame}
There are two players, \PI/ and \PII/.
\PII/ needs to choose between scenario A or B first and then produce a binary sequence of length $M$ containing precisely $K_A$ or $K_B$ number of $1$-s. (Without losing generality, we will assume $K_A \le K_B$.)
Following this, \PI/ (not knowing the actions of \PII/) can sample $N$ number of bits, and after observing their value, she guesses scenario A or B.

If \PI/ guessed the scenario correctly, she wins the game (\textcolor{w}{$\blacksquare$}) and loses otherwise (\textcolor{l}{$\blacksquare$}). 
The above-defined Fisher game will be denoted as 
\G{N, K_A, K_B, M}.

\end{definition}

A tabular description of a general Fisher game is shown by Game Table~\ref{game:GeneralFisherGame}.

\begin{gametable}[H]
\captionsetup{justification=centering}
\caption{\label{game:GeneralFisherGame} General description of \\ \G{N, K_A, K_B, M}}
\centering
\begin{tabularx}{0.73\textwidth}{ X | X }
\hline
 &  \\
\multicolumn{1}{c|}{\PI/} & \multicolumn{1}{c}{\PII/} \\
 &
\begin{itemize}
        \item Chooses scenario A or B,
        \begin{itemize}
            \item then chooses a binary sequence available for the chosen scenario.
        \end{itemize}
\end{itemize}
\\
\begin{itemize}
        \item Chooses $N$ indices for sampling,
        \begin{itemize}
            \item based on the bits in the chosen sample, guesses scenario A or B.
        \end{itemize}
    \end{itemize}
& \\
& \\
\multicolumn{2}{c}{
    \begin{minipage}{0.5\linewidth}
        \centering
        If \PI/ guessed the scenario correctly, she wins the game (\textcolor{w}{$\blacksquare$}) and loses otherwise (\textcolor{l}{$\blacksquare$}).
    \end{minipage}
} \\
\end{tabularx}
\end{gametable}

\subsubsection{Trivial cases}
\label{sec:FisherTrivialCases}

\paragraph{Blind guessing:}
A maximally trivial version of this guessing game occurs when $M=0$, i.e. there are no sequences to investigate. In this case, \PI/ needs to guess blindly.
Formally, we will call this game \G{N=0,K_A=0,K_B=0,M=0}.
The utility matrix of this game looks the following \footnote{a utility matrix for \PI/ is a matrix, in which rows represent the actions of \PI/ (in this case guessing A, guessing B), and columns represent the actions of \PII/ (in this case choosing A, choosing B). Therefore \PI/ can pick a row, while \PII/ selects a column. The matrix element in the intersection represents the utility for \PI/ of the consequence of the selected actions.}:

\begin{equation*}
u_1=
\begin{bmatrix}
\wbox & \lbox \\
\lbox & \wbox
\end{bmatrix}
\end{equation*}

This is equivalent to the well-known game called Matching Pennies \cite{book:EssentialGameTheory} \footnote{alternatively it can be viewed as a ``Guess The Hand'', ``Which hand is the coin in?'' or ``Hand Game'' \cite{book:CulinIndianGames,book:PrehistoricGames} guessing game}. In the context of Game Theory \cite{book:EssentialGameTheory,book:GameTheory,review:NeumannMorgensternGameThoery,book:GameTheoryOriginal} the game has one unique Nash equilibrium: the players are guessing/choosing A or B with an equal $50\%$ chance, and both players have a $50\%$ chance to win or lose.

Any further games, where $M > 0$, while $N=0$ are degenerate \cite{book:AlgorithmicGameTheory} cases of the blind guessing game. \footnote{Another family of games leading to blind guessing are games in which $K_A=K_B$. In this case, \PI/ gains no information about A or B by sampling from the provided sequence.}

\paragraph{Sure winning:}
A further trivial case is \G{N=1,K_A=0,K_B=1,M=1}. For this situation, if \PII/ chooses A, then she needs to provide the sequence $(0)$ (or $(\wb)$), while if she chooses B, she must provide sequence $(1)$ (or $(\bb)$).

\PI/ can sample one bit, which can be chosen only one way and which completely reveals the choice of \PII/. \PI/ has an absolute winning strategy denoted by the set of rules: $\{(\wb) \to A, (\bb) \to B \}$ (if she sees $\wb$, chooses A, if she sees $\bb$, chooses B). This strategy guarantees a $100\%$ success rate for \PI/ regardless of \PII/'s strategy.

In all games in which $N>0$, $K_A = 0$, $K_B = M$, $M>0$,  \PI/ has essentially the same trivial absolute winning strategy.

Another trivial sure winning strategy is available when $N=M$, i.e. \PI/ can sample all the bits. \footnote{In general, we will call a game ``Sure winning'' if there is no possible observation after sampling, which could be realized both by scenario A and B. Using the notation in Section~\ref{sec:ActionSpaceStructure} this formally means that $\mathbb{K}_{AB} = \mathbb{K}_A \cap \mathbb{K}_B = \emptyset$.}

\subsubsection{Smallest nontrivial case}

The smallest and simplest nontrivial game is \G{N=1,K_A=0,K_B=1,M=2}, which can be seen from Table~\ref{tab:M<=2}. \footnote{The other smallest nontrivial game is \G{N=1,K_A=1,K_B=2,M=2}. The two games are related by the transformation $\wb \leftrightarrow \bb$, and $A \leftrightarrow B$; therefore, results from one generalizes for both.}

\PII/ can either choose A, in which case only one sequence is possible:

\[
\{(\wb, \wb)\}
\]
Alternatively, she can choose scenario B, by which two sequences can be generated:

\[
\{(\wb, \bb), (\bb, \wb)\}
\]
These are $1+2=3$ distinct sequences in total.

\PI/ can sample from $2$ indices. The observed bit can take $2$ values ($\wb$ or $\bb$). Based on this information, she can guess one of the $2$ scenarios, A or B. This results in a total of $2 \times 2 ^ 2 = 8$ possible choices.

\subsubsection{Exact solution of the smallest nontrivial case}

\paragraph{Action sets:}
In this small finite case, it is feasible to enumerate all possible actions of the players, construct the utility matrix, and find all equilibrium strategies of the players.

The action set of \PII/ is denoted by $\mathcal{A}_2$.
For clarity, the notation will be kept redundant. As an example, $(B,(\wb,\bb))$ stands for the action of \PII/, in which she has chosen scenario B and picked $(\wb,\bb)$ from the allowed sequences in scenario B. Enumeration of all possible choices of \PII/ results:

\begin{equation}
    \label{eq:SimplestFisher_A2}
    \mathcal{A}_2 = \{ (A, (\wb,\wb)), (B, (\wb,\bb)) ,(B,(\bb,\wb)) \}
\end{equation}

The whole action set of \PI/ is more complicated.
To clarify the notation, let us take an example: \PI/ selects the second bit to investigate, and if it is $\wb$, then guesses scenario A and if the bit is $\bb$, then guesses B. This will be denoted as $(2,\{(\wb) \to A, (\bb) \to B\})$. Using this notation, the full action set looks the following:

\begin{equation}
\begin{split}
\mathcal{A}_1 = \{ 
& (1,\{(\wb) \to A, (\bb) \to A\}), (1,\{(\wb) \to A, (\bb) \to B\}), \\
& (1,\{(\wb) \to B, (\bb) \to A\}), (1,\{(\wb) \to B, (\bb) \to B\}), \\
& (2,\{(\wb) \to A, (\bb) \to A\}), (2,\{(\wb) \to A, (\bb) \to B\}), \\
& (2,\{(\wb) \to B, (\bb) \to A\}), (2,\{(\wb) \to B, (\bb) \to B\})
\}
\end{split}
\end{equation}

\vfill

\paragraph{Utility matrix:}
It is tedious but straightforward to construct the utility matrix of this zero-sum game. 
The utility matrix for \PI/ looks the following:

\setcounter{MaxMatrixCols}{12}
\begin{equation}
u_1=
\begin{bmatrix}
\wbox & \lbox & \lbox \\
\wbox & \lbox & \wbox \\
\lbox & \wbox & \lbox \\
\lbox & \wbox & \wbox \\
\wbox & \lbox & \lbox \\
\wbox & \wbox & \lbox \\
\lbox & \lbox & \wbox \\
\lbox & \wbox & \wbox
\end{bmatrix}
\end{equation}

\paragraph{Steps toward solution:}
Either by eliminating dominated strategies \cite{book:EssentialGameTheory, book:GameTheory}, or using common sense, one can exclude strategies in which \PI/ would guess A after seeing a $\bb$. This can reduce the action set of \PI/:

\begin{equation}
\begin{split}
\mathcal{A}_1' = \{ 
& (1,\{(\wb) \to A, (\bb) \to B\}), (1,\{(\wb) \to B, (\bb) \to B\}), \\
& (2,\{(\wb) \to A, (\bb) \to B\}), (2,\{(\wb) \to B, (\bb) \to B\})
\}
\end{split}
\end{equation}
resulting in a reduced utility matrix:

\setcounter{MaxMatrixCols}{12}
\begin{equation}
u_1'=
\begin{bmatrix}
\wbox & \lbox & \wbox \\
\lbox & \wbox & \wbox \\
\wbox & \wbox & \lbox \\
\lbox & \wbox & \wbox
\end{bmatrix}
\end{equation}

This finite zero-sum game can be solved completely \footnote{even using a standard computational software, for example \href{https://cgi.csc.liv.ac.uk/~rahul/bimatrix_solver/}{online bimatrix solver}, \href{https://nashpy.readthedocs.io/en/stable/}{Nashpy}, \href{https://doc.sagemath.org/html/en/reference/game_theory/index.html}{Sage} or \href{https://gambitproject.readthedocs.io/en/latest/intro.html}{Gambit} etc\href{https://www.wolfram.com/broadcast/video.php?v=3526}{\dots}}:

First, we might observe that the $2.$ and $4.$ rows in the reduced utility matrix are identical. Let us assume that we can reach an equilibrium by keeping only one from rows $2$ and $4$. (This assumption needs to be checked in the end.)
If we delete the $4.$ row and swap the $1.$ row with the $2.$ row, then we get the following utility matrix:

\begin{equation}
u_1''=
\begin{bmatrix}
\lbox & \wbox & \wbox \\
\wbox & \lbox & \wbox \\
\wbox & \wbox & \lbox
\end{bmatrix}
\end{equation}

We get the same utility matrix if we delete the $2.$ row and then permute the new rows $\{1 \to 2, 2 \to 3, 3 \to 1 \}$.

\paragraph{Matching Game:}

A game with this utility matrix can be reinterpreted as a (reverse) ``Matching Game'':

Imagine that both \PI/ and \PII/ can choose from $3$ different items (I, II, III). \PI/ loses ($\lbox$) if they choose the same item (and wins ($\wbox$) otherwise). \footnote{The game is also similar to a \href{https://www.britannica.com/art/cups-and-balls-trick}{Cups and Balls trick} (also known as Shell Game) from the tricksters perspective, if her moves were disregarded by the chooser.}

This is a symmetric, zero-sum game, which is easier to understand intuitively.
The game has one single Nash equilibrium, in which all players are randomizing uniformly between all possible items. This behaviour is described by the following mixed strategy profile, which is essentially a pair of probability distributions for all possible actions.

\begin{equation}
    \underline{\sigma}''^*_1 =(1/3,1/3,1/3) \quad  \underline{\sigma}''^*_2 =(1/3,1/3,1/3)
\end{equation}

To prove that this is indeed a Nash equilibrium, one needs to check if \PI/ or \PII/ can expect more utility by unilaterally deviating from her strategy. If \PII/ chooses items I, II, III uniformly, all choices of \PI/ lead to the same $2/3$ expected chance of winning. This shows that there is no item, the choice of which could enhance \PI/'s winning chance. Similarly, if \PI/ chooses items I, II, and III uniformly, then \PII/ has a $1/3$ chance of winning with all items. This proves that this strategy profile is a Nash equilibrium.

There is another, more dynamic way to approach this game:
If \PI/ would use any mixed strategy profile $\underline{\sigma}''_1$, then \PII/ could learn her distribution and always choose the item which is chosen by \PI/ most frequently.
Formally, the minimal chance of winning for \PI/ when she uses a $\underline{\sigma}''_1$ profile is:

\begin{equation}
    v_1^{\mathrm{worst}}(\underline{\sigma}''_1) = 1- \max(\underline{\sigma}''_1)
\end{equation}

\PI/ can maximize her worst-case winning rate if she minimizes the maximum of her mixed strategy's distribution. For $3$ items, the minimum of the probabilities maximum is $1/3$, when all items are chosen uniformly.
A similar line of reasoning shows that \PII/ can maximize her minimal winning rate by uniform randomization.

This is why equilibrium strategies in zero-sum games are often called ``minimax'' strategies.

\paragraph{Switching to the original game:}

We can translate this equilibrium strategy profile back to the original game and actions. To check if this is a Nash equilibrium in the complete action set, we need to show that \PI/ can not get better results by playing the deleted row: 

To see this, one needs to recall that the deleted row was always identical to a row that remained in the reduced game (namely the row: choosing I), so the expected winning rate is the same $2/3$ for it. This shows that \PI/ can not gain more by including the excluded action in the mixture.

\paragraph{Complete solution of the smallest nontrivial game:}

It might have been surprising from the original description that this game has not only one unique Nash equilibrium but a continuous set of possible equilibrium strategies. \footnote{This is equivalent to the statement that this game is degenerate \cite{book:AlgorithmicGameTheory}.} The spanning Nash equilibria for the reduced action set are:

\begin{equation}
    \begin{split}
        {\underline{\sigma}'^*_1}^{(1)}=(1/3,1/3,1/3,0), & \quad {\underline{\sigma}'^*_2}^{(1)}=(1/3,1/3,1/3), \\
        {\underline{\sigma}'^*_1}^{(2)}=(1/3,0,1/3,1/3), & \quad {\underline{\sigma}'^*_2}^{(2)}=(1/3,1/3,1/3)
    \end{split}
\end{equation}

All equilibria can be characterized as a one-parameter family, with a continuous parameter $\lambda \in [0,1]$:

\begin{equation}
    \label{eq:SimplestFisherComplete}
    \underline{\sigma}'^*_1(\lambda)=(1/3,\lambda/3,1/3,(1-\lambda)/3), \quad \underline{\sigma}'^*_2(\lambda)=(1/3,1/3,1/3), 
\end{equation}

\paragraph{Symmetric solution:}

There is only one symmetrical solution when \PI/ picks the first and second bit equally likely. This solution corresponds to $\lambda=1/2$:

\begin{equation}
    \underline{\sigma}'^{*S}_1=(1/3,1/6,1/3,1/6), \quad 
    \underline{\sigma}'^{*S}_2=(1/3,1/3,1/3)
\end{equation}

One can reinterpret the symmetrical solution in a more familiar procedural language, described in Game Table~\ref{game:1_2_1_2}.

\begin{gametable}[H]
\captionsetup{justification=centering}
\caption{\label{game:1_2_1_2} Symmetric equilibrium strategy for \\ \G{N=1, K_A=0, K_B=1, M=2}}
\centering
\begin{tabularx}{0.73\textwidth}{ X | X }

\hline
 &  \\
\multicolumn{1}{c|}{\PI/} & \multicolumn{1}{c}{\PII/} \\
 &
\begin{itemize}
        \item Choose scenario A with probability $1/3$, and scenario B with probability $2/3$
        \begin{itemize}
            \item Choose uniformly from all different allowed sequences.
        \end{itemize}
\end{itemize}
\\
\begin{itemize}
        \item Sample randomly from all possible indices uniformly
        \begin{itemize}
            \item in case the sampled bit is $\wb$:
            \begin{itemize}
                \item guess A with probability $2/3$
                \item and B with probability $1/3$ 
            \end{itemize}
            \item in case the sampled bit is $\bb$:
            \begin{itemize}
                \item guess B
            \end{itemize}
        \end{itemize}
    \end{itemize}
& \\
\end{tabularx}
\end{gametable}

The optimal policy is visualized in figure \ref{fig:PolicyPlot1012}.

\begin{figure}[H]
    \centering
    \includegraphics[scale=0.7]{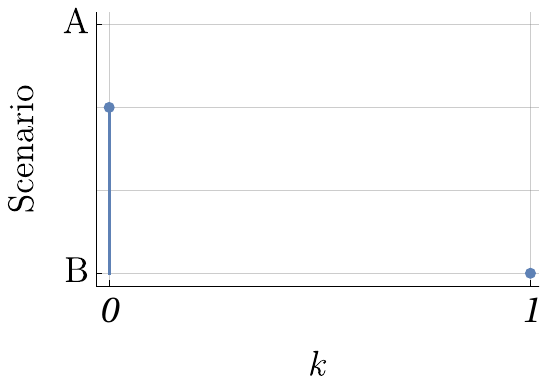}
    \caption{Visualization of the optimal guessing policy for \PI/ in the game: \G{N=1,K_A=0,K_B=1,M=2}. Where $k$ is the number of $\bb$-s in the sample, and the height corresponds to the chance of choosing A from the possible scenarios.}
    \label{fig:PolicyPlot1012}
\end{figure}

\subsection{Table of smallest games}

Even at this stage, we can enumerate all possible games where the number of bits is less or equal to 2. The results are collected in Table~\ref{tab:M<=2}.

\begin{table}[H]
    \centering
    \begin{tabular}{|c|c|c|c|c|c|c|}
        \hline
        $N$ & $K_A$ & $K_B$ & $M$ & Trivial? & Type & Winning rate \\
        \hline
        1 & 0 & 0 & 1 & $\checkmark$ & Blind guessing & $50\%$ \\
        1 & 1 & 0 & 1 & $\checkmark$ & Sure winning & $100\%$ \\
        1 & 1 & 1 & 1 & $\checkmark$ & Blind guessing & $50\%$ \\
        \hline
        1 & 0 & 0 & 2 & $\checkmark$ & Blind guessing & $50\%$ \\
        1 & 0 & 1 & 2 & \ding{55} & - & $66.\overline{6}\%$ \\
        1 & 0 & 2 & 2 & $\checkmark$ & Sure winning & $100\%$ \\
        1 & 1 & 1 & 2 & $\checkmark$ & Blind guessing & $50\%$ \\
        1 & 1 & 2 & 2 & \ding{55} & - & $66.\overline{6}\%$ \\
        1 & 2 & 2 & 2 & $\checkmark$ & Blind guessing & $50\%$ \\
        \hline
        2 & 0 & 0 & 2 & $\checkmark$ & Blind guessing & $50\%$ \\
        2 & 0 & 1 & 2 & $\checkmark$ & Sure winning & $100\%$ \\
        2 & 0 & 2 & 2 & $\checkmark$ & Sure winning & $100\%$ \\
        2 & 1 & 1 & 2 & $\checkmark$ & Blind guessing & $50\%$ \\
        2 & 1 & 2 & 2 & $\checkmark$ & Sure winning & $100\%$ \\
        2 & 2 & 2 & 2 & $\checkmark$ & Blind guessing & $50\%$ \\
        \hline
    \end{tabular}
    \caption{\centering Enumeration of all games, \G{N,K_A,K_B,M}, satisfying: $0<M \le 2$, $0\le K_A \le K_B \le M$, $0 < N \le M$.}
    \label{tab:M<=2}
\end{table}

\subsection{General structure of the action sets}
\label{sec:ActionSpaceStructure}

\paragraph{Most general construction:}
The action set of \PII/ contains all the possible permutations of the sequences allowed in scenario A or B, denoted by $\mathcal{K}_A$ and $\mathcal{K}_B$:

\begin{equation}
    \mathcal{K}_A = \{\underline{b} \in \{\wb,\bb\}^M \ | \ \#_\bb \ \underline{b} = K_A \}
\end{equation}

\begin{equation}
    \mathcal{K}_B = \{\underline{b} \in \{\wb,\bb\}^M \ | \ \#_\bb \ \underline{b} = K_B \}
\end{equation}

Where the notation $\#_y X$ stands for the number of members in $X$ equal to $y$.

The possible actions of \PII/ is the union of all allowed sequences:

\begin{equation}
    \mathcal{A}_2 = \mathcal{K}_A \cup \mathcal{K}_B
\end{equation}
(To reproduce the redundant notation used in the previous section, one can additionally label the sequences with A or B: $\mathcal{A}_2 = \left ( \{A\} \times \mathcal{K}_A \right ) \cup \left ( \{B\} \times \mathcal{K}_B \right )$.)

The action set of \PI/ consists of two parts: firstly, choosing a concrete sampling and, subsequently, selecting a policy.

\begin{equation}
    \mathcal{A}_1 = \mathcal{S} \times \mathcal{P}
\end{equation}

\begin{equation}
    \label{eq:SamplingSet}
    \mathcal{S} = \{ S \subset \{1,\dots,M \} \ | \  |S| = N \}
\end{equation}

\begin{equation}
    \mathcal{P} = \{ \phi : \{\wb,\bb\}^N \mapsto \{A,B\} \}
\end{equation}

\paragraph{Action set sizes:}

To navigate better in the action spaces, calculating the size of these sets might be illuminating. This can be done by simple combinatorics.
In the general case \G{N,K_A,K_B,M} \PII/ has the following number of possible actions:

\begin{equation}
    |\mathcal{A}_2| = |\mathcal{K}_A| + |\mathcal{K}_B|=  \binom{M}{K_A} + \binom{M}{K_B}
\end{equation}

If we take into account all possible guesses for all possible sample sequences, we get the following double exponential expression for the action set size of \PI/:

\begin{equation}
    |\mathcal{A}_1| = |\mathcal{S}| \times |\mathcal{P}| = \binom{M}{N} \ 2^{2^N} 
\end{equation}

\subsubsection{Policy restrictions}

For a general \G{N,K_A,K_B,M} the number of $\bb$-s, $k$ might take values only from a restricted set $\mathbb{K}$.
If \PII/ chooses scenario A, then in the sample, maximally $K_A$ number of $\bb$-s, and maximally $M-K_A$ number of $\wb$-s can appear. If $N$ is greater then the maximal number of $\bb$-s or $\wb$-s, then $k$ can take values only from a restricted set:

\begin{equation}
    \mathbb{K}_A = \{ \max\{0,N-(M-K_A)\}, \dots, \min\{K_A,N\} \}
\end{equation}

If \PII/ chooses scenario B, then the potentially restricted set looks the following:

\begin{equation}
    \mathbb{K}_B = \{ \max \{0,N-(M-K_B)\}, \dots, \min \{K_B,N\} \}
\end{equation}

\begin{figure}[H]
    \centering
    \includegraphics{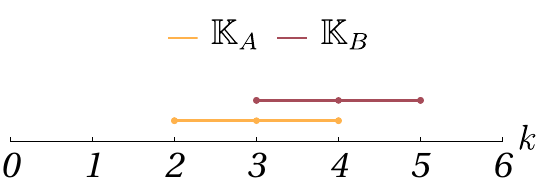}
    \caption{Illustration of $\mathbb{K}_A$ and $\mathbb{K}_B$ for \G{N=6,K_A=4,K_B=5,M=8}}
    \label{fig:KAB}
\end{figure}

\PI/ does not know in advance if scenario A or B has been chosen by \PII/. Therefore, she needs to prepare strategies for any possible restricted outcome:

\begin{equation}
    \mathbb{K} = \mathbb{K}_A \cup \mathbb{K}_B
\end{equation}

This can slightly restrict the policy set of \PI/ \footnote{because \PI/ does not need to consider policies for sampled sequences, which can not come out under any scenario}:

\begin{equation}
        |\mathcal{P}'| = \left (  2^{\sum_{k \in \mathbb{K}} \binom{N}{k}} \right ) \le |\mathcal{P}| = 2^{2^N}
\end{equation}

\paragraph{Permutation invariant policies:}

A meaningful restriction of the guessing game could be that \PI/ needs to make her guess only by knowing the total number of $\bb$-s in the sampled subsequence, regardless of the exact order of bits in the sample. In this case, the number of actions for \PI/ is greatly reduced:

\begin{equation}
    |\mathcal{P}^R| = 2^{N+1}
\end{equation}

\begin{equation}
    |\mathcal{P'}^R| = 2^{|\mathbb{K}|}
\end{equation}

However, there might be strategies played by \PII/ for which the best response of \PI/ is not part of the $\mathcal{P}^R$ set.
In the following parts, we will focus on equilibrium strategies in which this restriction will be justified.

\paragraph{Obviously dominated policies:}

We can restrict further the set of policies by removing obviously wrong choices.
A sane player will not choose scenario B, seeing a sample which could only come out under scenario A (and vice versa).
Formally: any policy, which gives B for any sample, in which the number of $\bb$-s $k \in \mathbb{K}_A \setminus \mathbb{K}_B$ is dominated by policies, which always gives A in such case.
To reduce the possible policies, one can keep only those which can not be dominated in this obvious way:

\begin{equation}
    \begin{split}
    \mathcal{P}'' =  
    \{ 
    \phi : \{\wb,\bb\}^N \mapsto \{A,B\} \ | 
    &
    \ \forall (\underline{d} \in \{\wb,\bb\}^N, \#_\bb \underline{d} \in \mathbb{K_A} \setminus \mathbb{K}_B) \  \phi(\underline{d})=A \\
    & 
    \ \forall (\underline{d} \in \{\wb,\bb\}^N, \#_\bb \underline{d} \in \mathbb{K_B} \setminus \mathbb{K}_A) \  \phi(\underline{d})=B
    \}
    \end{split}
\end{equation}

This condition keeps the possible policies unrestricted only on the set:

\begin{equation}
    \mathbb{K}_{AB} = \mathbb{K}_A \cap \mathbb{K}_B
\end{equation}

\begin{equation}
        |\mathcal{P}''| = 2^{\sum_{k \in \mathbb{K}_{AB}} \binom{N}{k}} 
\end{equation}

This, combined with permutation invariance, gives:

\begin{equation}
        |\mathcal{P}''^R| = 2^{|\mathbb{K}_{AB}|} 
\end{equation}

\subsection{Symmetric solution of the general case}

Without losing generality, we will assume that $K_A < K_B$ in this section.

\paragraph{The Ansatz:}

An Ansatz\footnote{an educated guess, or an assumption about the form of the solution} can be formulated for the general case, described in Game Table~\ref{game:Ansatz1}.

\begin{gametable}[H]

\captionsetup{justification=centering}
\caption{\label{game:Ansatz1} {\bf Ansatz} for the general \G{N, K_A, K_B, M} case, \\ having 3 free variables: $P$, $k^\bullet$ and $\nu$:}

\centering
\begin{tabularx}{0.73\textwidth}{ X | X }

\hline
 &  \\
\multicolumn{1}{c|}{\PI/} & \multicolumn{1}{c}{\PII/} \\
 &
\begin{itemize}
        \item choose scenario A with probability $\pi(\mathrm{A})=P$, or B with probability $\pi(\mathrm{B})=1-P$.
        \begin{itemize}
            \item Choose uniformly from all different allowed sequences.
        \end{itemize}
    \end{itemize}
\\
 \begin{itemize}
        \item sample $N$ bits randomly and uniformly from all available $M$ bits,
        \begin{itemize}
            \item in case the number of $\bb$-s $k < k^\bullet$ guess A,
            \item in case $k=k^\bullet$:
            \begin{itemize}
                \item guess A with probability $\mu(\hat{\phi}) = \nu$,
                \item guess B with probability $\mu(\check{\phi}) = 1-\nu$,
            \end{itemize}
            \item in case the number of $\bb$-s $k > k^\bullet$ guess B.
        \end{itemize}
    \end{itemize}
& \\
\end{tabularx}
\end{gametable}

\paragraph{Fixing the parameters:}

For the sake of derivation, we denote by $p_k(A)$ the probability that in scenario A, $N$ randomly selected bits contain precisely $k$ number of $\bb$-s. (While $p_k(B)$ denotes the same concept for scenario B).

To be specific, these probabilities follow a hypergeometric distribution \cite{book:IntroToProbability, book:StatisticalInference, book:Renyi1970}:

\begin{equation}
    p_k(A) = \frac{\binom{K_A}{k} \binom{M-K_A}{N-k}}{\binom{M}{N}}, \quad
    p_k(B) = \frac{\binom{K_B}{k} \binom{M-K_B}{N-k}}{\binom{M}{N}}
\end{equation}

\begin{figure}[H]
    \centering
    \includegraphics[scale=0.8]{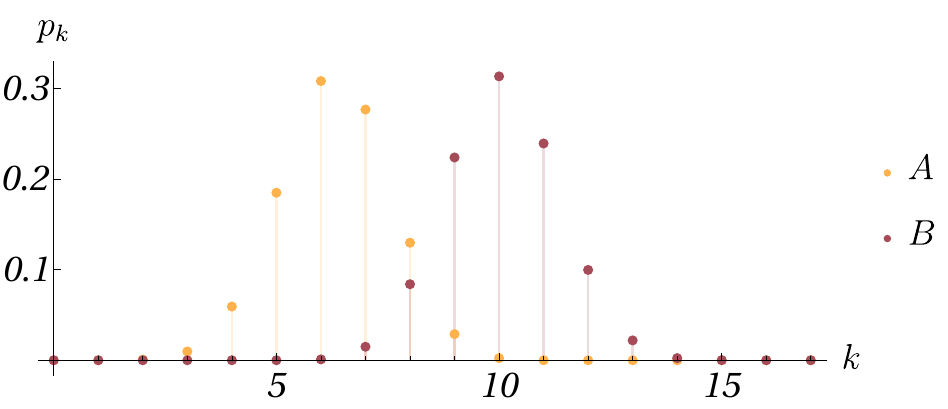}
    \caption{Illustration of $p_k(A)$ and $p_k(B)$ for \G{N=17,K_A=10,K_B=16,M=27}.}
    \label{fig:pkAB}
\end{figure}

In formulating the parameters, we will not use the specific form of these probabilities, only a generic observation, vaguely stating: ``as $k$ becomes larger, scenario B becomes more probable than scenario A''. The property we will use can be formulated in the following way:

\begin{equation}
    \forall \ \ell, k \in \mathbb{K}_{AB}, \ \ell > k \implies \frac{p_\ell(B)}{p_\ell(A)} > \frac{p_k(B)}{p_k(A)}
\end{equation}

To maximize her expected utility, \PI/ should always choose:

\begin{equation}
    \begin{split}
        & A,\quad \mathrm{If} \quad P \ p_k(A) > (1-P) \ p_k(B) \\
        & B,\quad \mathrm{If} \quad P \ p_k(A) < (1-P) \ p_k(B)
    \end{split}
\end{equation}

This leaves only one possibility for the case in which $k=k^\bullet$ \footnote{in the general case, when $\nu \ne 0$}, which is equality:

\begin{equation}
    \label{eq:Pk*}
    P \  p_{k^\bullet}(A) = (1-P) \ p_{k^\bullet}(B)
\end{equation}

Equation \eqref{eq:Pk*} can determine $P$ as a function of $k^\bullet$:

\begin{equation}
    P(k^\bullet) = \frac{p_{k^\bullet}(B)}{p_{k^\bullet}(A)+p_{k^\bullet}(B)}
\end{equation}

\PII/ is mixing between A and B; therefore, she needs to be indifferent about choosing between them:

\begin{equation}
    \sum_{k<k^\bullet} p_k(A) + \nu \ p_{k^\bullet}(A) = \sum_{k>k^\bullet} p_k(B) + (1-\nu) \ p_{k^\bullet}(B)
\end{equation}

By rearranging the terms, we get an expression for $\nu$:

\begin{equation}
    \label{eq:nuk*}
    \begin{split}
        \mu(\hat{\phi}) &= \nu = \frac{\sum_{k \ge k^\bullet} p_k(B) - \sum_{k < k^\bullet} p_k(A)}{p_{k^\bullet}(A)+p_{k^\bullet}(B)} \\
        \mu(\check{\phi}) &= 1 - \nu = \frac{\sum_{k \leq k^\bullet} p_k(A) - \sum_{k > k^\bullet} p_k(B)}{p_{k^\bullet}(A)+p_{k^\bullet}(B)}
    \end{split}
\end{equation}

(It is important to note that $\nu$ represents a probability. We will use the convention that $\nu \in [0,1)$)

Collecting all the variables, we can express the probability of guessing correctly for \PI/ \footnote{setting the utility of winning to $1$ and losing to $0$, this is equivalent to the expected utility or ``value'' of the game} as a function of $k^\bullet$:

\begin{equation}
    v(k^\bullet) =  \frac{p_{k^\bullet}(A) \ \left (\sum_{k \ge k^\bullet} p_k(B) \right ) + p_{k^\bullet}(B) \ \left (\sum_{k < k^\bullet} p_k(A) \right )}{p_{k^\bullet}(A)+p_{k^\bullet}(B)}
\end{equation}

\PI/ can maximize this quantity by appropriately choosing $k^\bullet$. 
An alternative approach to find $k^\bullet$ is to satisfy the inequalities: $0 \le \nu < 1$, $0<1-\nu\le1$ together with equations \eqref{eq:nuk*}:

\begin{equation}
    \sum_{k \geq k^\bullet} p_k(B) \ge \sum_{k < k^\bullet} p_k(A) 
\end{equation}

and

\begin{equation}
    \sum_{k > k^\bullet} p_k(B) < \sum_{k \leq k^\bullet} p_k(A)
\end{equation}

\subsubsection{Main theorem on Fisher games}

\begin{theorem}[Symmetrical equilibrium]
\label{thm:Symmetric}
\G{N,K_A,K_B,M} has a symmetrical Nash equilibrium, in which:

\begin{itemize}
    \item \PII/ chooses scenario A or B with probability $P^*$ and $1-P^*$;
    \begin{itemize}
        \item then picks an allowed sequence with equal probability (from $\mathcal{K}_A$ or $\mathcal{K}_B$).
    \end{itemize}

    \item \PI/ first samples $N$ bits uniformly from the provided sequence. Based on $k$ -- the number of $\bb$-s -- she performs the following action:
    \begin{itemize}
        \item if $k<k^*$ she guesses A
        \item if $k=k^*$ then 
        \begin{itemize}
            \item she guesses A with probability $\nu^*$ or B with probability $1-\nu^*$
        \end{itemize} 
        \item if $k>k^*$ she guesses B
    \end{itemize}
\end{itemize}

The parameters $(k^*, \nu^*, P^*)$ can be determined from the parameters of the game $(N, K_A, K_B, M)$:

\begin{equation}
    \label{thm:SymEqHypergeom}
    p_k(A) = \frac{\binom{K_A}{k} \binom{M-K_A}{N-k}}{\binom{M}{N}}, \quad
    p_k(B) = \frac{\binom{K_B}{k} \binom{M-K_B}{N-k}}{\binom{M}{N}}
\end{equation}

\begin{equation}
    \label{thm:Fisher_Ps}
    P^* = \frac{p_{k^*}(B)}{p_{k^*}(A)+p_{k^*}(B)}
\end{equation}

\begin{equation}
    \label{thm:SymEqNu}
    \nu^* = \frac{\sum_{k \ge k^*} p_k(B) - \sum_{k < k^*} p_k(A)}{p_{k^*}(A)+p_{k^*}(B)}
\end{equation}

Finally, $k^*$ is the smallest integer, for which the sum of probabilities becomes greater than $1$:

\begin{equation}
    \label{thm:Fisher_ks}
    \sum_{k \le k^*} p_k(A)+p_k(B) > 1, \quad \mathrm{while} \quad \sum_{k<k^*} p_k(A)+p_k(B) \le 1
\end{equation}

\end{theorem}

For the proof, see Appendix \ref{proof:Symmetric}.

\begin{remark}
    The equilibrium quantity $k^*$ can be interpreted as the median \cite{book:DeGrootProbabilityAndStatistics,book:IntroToProbability} of a mixture of random variables $\kappa_A, \kappa_B$:

    \begin{equation}
        k^* = \mathbbm{m}
        \left [ \frac{1}{2} \kappa_A+
        \frac{1}{2}\kappa_B
        \right ]
    \end{equation}

    where $\kappa_A$ and $\kappa_B$ are characterized by $p_k(A)$ and $p_k(B)$. And $\mathbbm{m}[.]$ represents the (strong) median \cite{paper:BinomialMedianMode}:

    \begin{equation}
        \mathbbm{m}[\kappa] = c \iff
        \Pr(\kappa \le c) \ge \frac{1}{2}
        \ \wedge \  
        \Pr(\kappa \ge c) \ge \frac{1}{2}
    \end{equation}

    (The only difference is that such definition would imply the $\nu \in (0,1]$ choice instead of $\nu \in [0,1)$.)
    
\end{remark}

\begin{lemma}[Intermediacy of the median]
\label{lemma:IntermediacyMedian}
    For any random variables with unique median $\mathbbm{m}[\xi] \le \mathbbm{m}[\eta]$ and any $\lambda \in [0,1]$ mixing parameter (if $\lambda \cdot \xi + (1-\lambda) \cdot \eta$ also has a unique median):

    \begin{equation}
        \mathbbm{m}[\xi] 
        \le
        \mathbbm{m}[\lambda \cdot \xi + (1-\lambda) \cdot \eta]
        \le
        \mathbbm{m}[\eta] 
    \end{equation}
    
\end{lemma}

The proof can be found in Appendix \ref{Appendix:FisherGameEquilibrium}, Section~\ref{sec:IntermediacyMedian}.

\begin{theorem}
\label{thm:FiniteKMedianBounds}

    The intermediacy property of the median implies the following bounds for $k^*$:

    \begin{equation}
        \mathbbm{m}[\kappa_A]
        \le
        k^*
        \le
        \mathbbm{m}[\kappa_B]
    \end{equation}

    where $\kappa_A$ and $\kappa_B$ represent random variables with hypergeometric distribution:

    \begin{equation}
        \kappa_A \sim \mathrm{Hypergeom}(N,K_A,M), \quad
        \kappa_B \sim \mathrm{Hypergeom}(N,K_B,M)
    \end{equation}

    \begin{equation}
    \Pr(\kappa_A = k) =
    p_k(A) = \frac{\binom{K_A}{k} \binom{M-K_A}{N-k}}{\binom{M}{N}}, \quad
    \Pr(\kappa_B = k) =
    p_k(B) = \frac{\binom{K_B}{k} \binom{M-K_B}{N-k}}{\binom{M}{N}}
\end{equation}

\end{theorem}

\begin{theorem}[Symmetric equilibrium is Optimal]

    There is no equilibrium which yields a greater winning rate for \PI/ than the Symmetric equilibrium.

\end{theorem}

\begin{proof}
    This is a special case of a more general theorem:

    \begin{quote}
    {\bf Minimax theorem (von Neumann, 1928)}
    In any finite, two-player, zero-sum game, in any Nash equilibrium each player receives a payoff that is equal to both his maxmin value and his minmax value.

    \hfill --- Theorem 3.1.4 in \cite{book:EssentialGameTheory}.
    \end{quote}

\end{proof}

\begin{remark}
In the general case, \G{N,K_A,K_B,M} is highly degenerate. Meaning that its solution is very far from unique and, in fact, has a continuous set of Nash equilibria. (Similar to the complete solution of the simplest case in \eqref{eq:SimplestFisherComplete}.)

These solutions are spanned by strategies which are not symmetric in their sampling or the choice of binary sequences.
Any continuous mixture of these spanning solutions is itself a Nash equilibrium, meaning that the set of solutions forms a high dimensional simplex.

\end{remark}

\begin{definition}
A two-player game is called \emph{nondegenerate} if no mixed strategy
of support size $k$ \footnote{a mixed strategy can be viewed as a discrete probability distribution on the possible actions, and so its support size is the support of this distribution, i.e. the set of actions which are played with nonzero probability} has more than $k$ pure best responses \footnote{an action belongs to the set of ``pure best responses'' -- responding to the other player's strategy --, if no other action could yield higher expected utility (or expected winning chance in the context of Fisher games)}. \cite{book:AlgorithmicGameTheory}

\end{definition}

\paragraph{Actions appearing in equilibrium:}

In the symmetric solution, the support of \PII/'s strategy includes all possible sequences:

\begin{equation}
    \supp(\sigma_2) = \mathcal{A}_2 = \mathcal{K}_A \cup \mathcal{K}_B
\end{equation}

\begin{equation}
    |\supp(\sigma_2)| = |\mathcal{K}_A| + |\mathcal{K}_B| = \binom{M}{K_A} + \binom{M}{K_B}
\end{equation}

For \PI/, the symmetric equilibrium strategy contains all sampling choices twice if $\nu > 0$. First, with a policy, in which she guesses A if $k=k^*$ (denoted by $\hat{\phi}$), and second in which she guesses B if $k=k^*$ (denoted by $\check{\phi}$). If $\nu = 0$, then only the second policy remains part of the support:

\begin{equation}
        \supp(\sigma_1)=
        \begin{cases}
        \mathcal{S} \times \{\hat{\phi}, \check{\phi} \} & \text{if } \nu > 0 \\
        \mathcal{S} \times \{\check{\phi} \} & \text{if } \nu = 0
    \end{cases}
\end{equation}

\begin{equation}
    |\supp(\sigma_1)| = \binom{M}{N} \times (1 + \lceil \nu \rceil)
\end{equation}

All actions with nonzero probability in the equilibrium strategy profile are, by definition, best response regarding the other player's strategy. In the general case

\begin{equation}
    \label{eq:DegCondition}
    |\supp(\sigma_1)| = (1 + \lceil \nu \rceil) \ \binom{M}{N} \ne \binom{M}{K_A} + \binom{M}{K_B} = |\supp(\sigma_2)| 
\end{equation}

This shows that in general, \G{N,K_A,K_B,M} is degenerate, i.e. the symmetric Nash equilibrium is not a unique equilibrium strategy profile. Uniqueness could be possible only if \eqref{eq:DegCondition} becomes an equality.

\subsection{Examples and Visualizations}

\paragraph{Equilibrium strategy plot:}

To show the structure and equilibrium parameters of general \G{N,K_A,K_B,M}-s, we introduce the ``strategy plot''. Figure \ref{fig:StrategyPlot_1012} shows the equilibrium solution of the smallest nontrivial example, \G{N=1,K_A=0,K_B=1,M=2}.

\begin{figure}[H]
    \centering
    \includegraphics[width=6 cm]{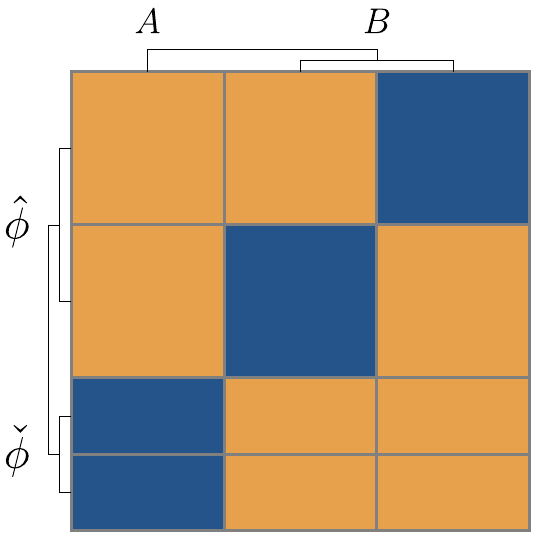}
    \caption{Strategy plot for the symmetric equilibrium of \G{N=1,K_A=0,K_B=1,M=2}.}
    \label{fig:StrategyPlot_1012}
\end{figure}

\begin{definition}[Equilibrium strategy plot]
Equilibrium strategy plot is a square-shaped collection of tiles which are divided into two parts, both horizontally and vertically:

\begin{itemize}
    \item Horizontally, it is first divided into two parts, in the ratio $P^*$ : $1-P^*$, representing the equilibrium probability of \PII/ choosing A or B. These two bigger domains are subdivided into equal parts, representing all possible sequences (in lexicographic ordering \cite{book:OrderedSets}), which can be realized in scenarios A or B.
    \item Vertically it is first divided into two parts, in the ratio $\nu^*$ : $1-\nu^*$, representing the equilibrium probability of \PI/ choosing the policy $\hat{\phi}$ or $\check{\phi}$. These bigger two domains are subdivided into equal parts, representing all possible $N$ long sampling choices from $M$ long sequences (in lexicographic ordering \cite{book:OrderedSets}).
\end{itemize}

By this construction, all pairs of actions appearing in the equilibrium solution are represented proportionally to the probability they appear in the equilibrium strategy profile.

Every tile, representing a specific pair of actions, is coloured according to the winning ($\wbox$) or losing ($\lbox$) of \PI/ in that case.

\end{definition}

For further, slightly larger examples, see figures \ref{fig:StrategyPlot_G_2_2_4_7}, \ref{fig:StrategyPlot_G_3_2_4_8}, which starts to reveal a more general structure and intriguing patterns.

\begin{figure}[H]
    \centering
    \begin{subfigure}[b]{0.45\textwidth}
        \includegraphics[width=\textwidth]{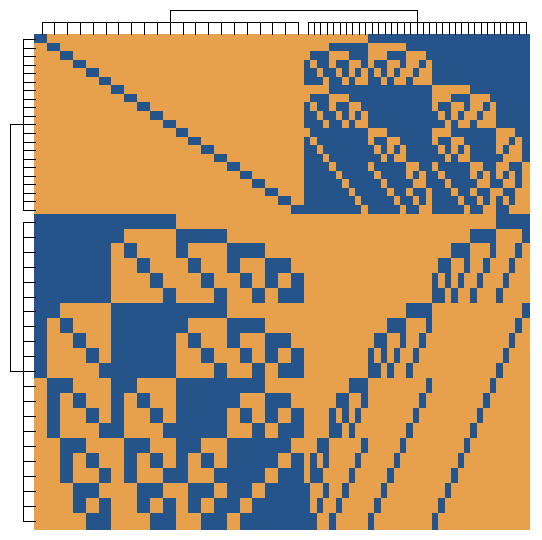}
        \caption{\G{N=2,K_A=2,K_B=4,M=7}}
        \label{fig:StrategyPlot_G_2_2_4_7}
    \end{subfigure}
    \hspace{0.05\textwidth} 
    \begin{subfigure}[b]{0.45\textwidth}
        \includegraphics[width=\textwidth]{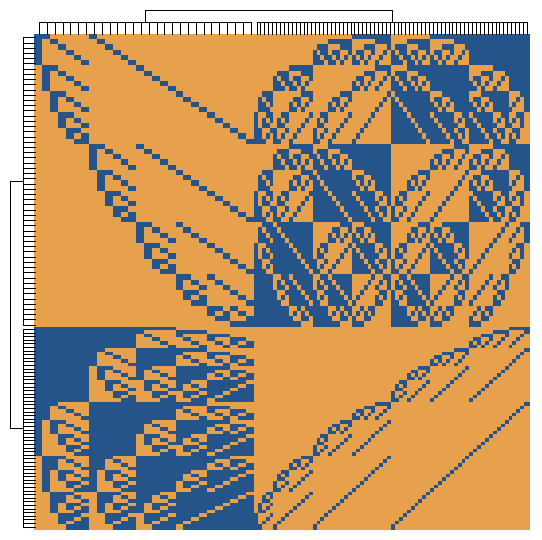}
        \caption{\G{N=3,K_A=2,K_B=4,M=8}}
        \label{fig:StrategyPlot_G_3_2_4_8}
    \end{subfigure}
    \caption{Strategy plots for symmetric equilibria}
    \label{fig:StrategyPlots}
\end{figure}

Mainly for aesthetic reasons, a relatively large, highly symmetric equilibrium strategy has been selected in figure \ref{fig:StrategyPlot_G_4_4_6_10}, showing an intricate fractal-like pattern.

\begin{figure}[H]
    \centering
    \includegraphics[width=12 cm]{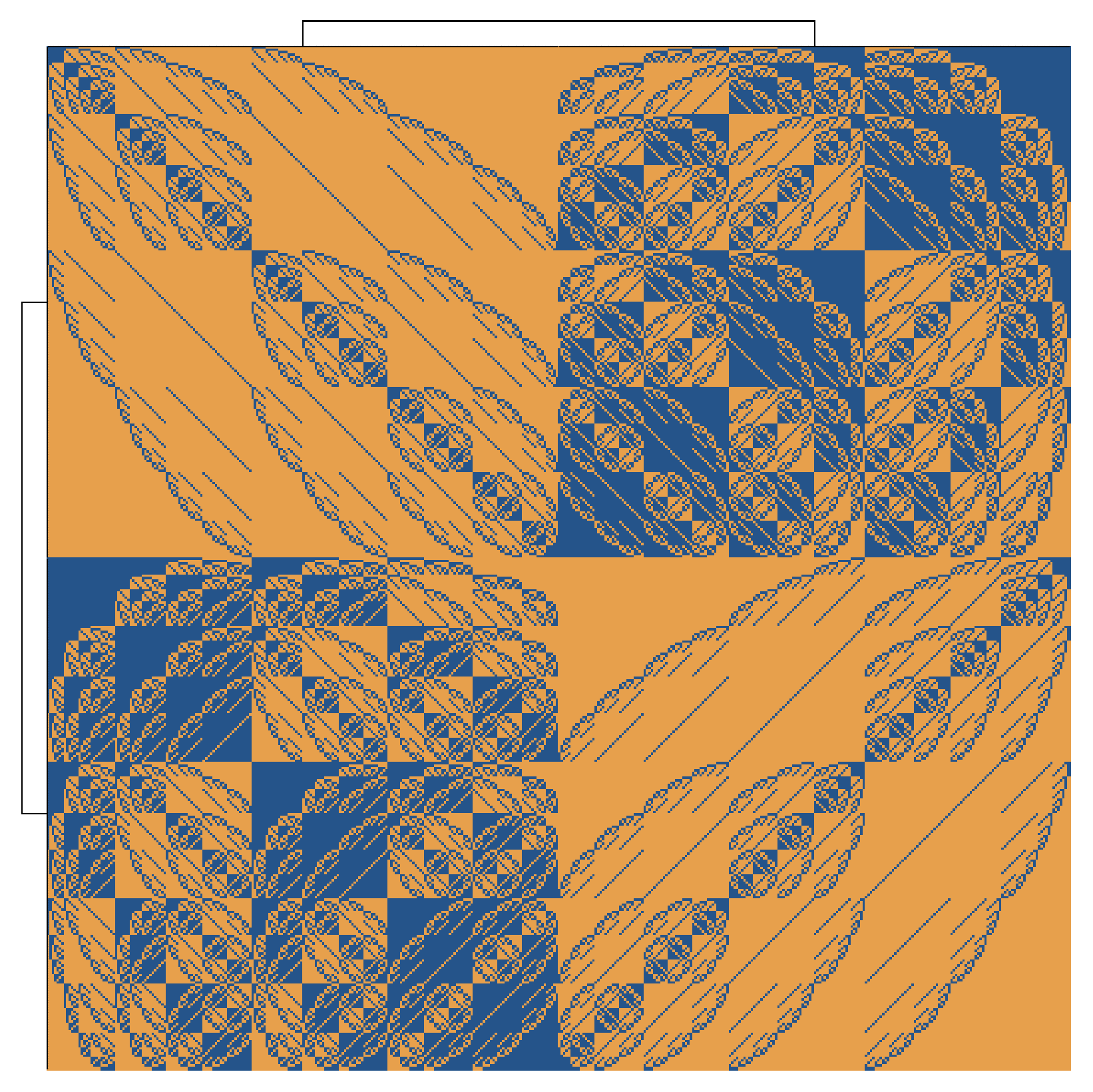}
    \caption{Strategy plot for the symmetric equilibrium of \G{N=4,K_A=4,K_B=6,M=10}.}
    \label{fig:StrategyPlot_G_4_4_6_10}
\end{figure}

\paragraph{Generalization of variables:}

For aesthetic reasons, we will allow games with $K_A \ge K_B$, with the following convention \footnote{$P^*_o$, $k^*_o$, $\nu^*_o$ and $v^*_o$ stands for the original quantities described in the previous sections (where we assumed that $K_A<K_B$)}:

\begin{equation}
\label{eq:PKAKBextended}
    P^*(K_A,K_B) =
        \begin{cases}
        \pi(A) = P^*_o(K_A,K_B) & \text{if } K_A < K_B \\
        1/2                     & \text{if } K_A = K_B \\
        \pi(B) = P^*_o(K_B,K_A) & \text{if } K_A > K_B
    \end{cases}
\end{equation}

\begin{equation}
    k^*(K_A,K_B) =
        \begin{cases}
        k^*_o(K_A,K_B) & \text{if } K_A < K_B \\
        -              & \text{if } K_A = K_B \\
        k^*_o(K_B,K_A) & \text{if } K_A > K_B
    \end{cases}
\end{equation}

\begin{equation}
    \nu^*(K_A,K_B) =
        \begin{cases}
        \nu^*_o(K_A,K_B)  & \text{if } K_A < K_B \\
        -                 & \text{if } K_A = K_B \\
        \nu^*_o(K_B,K_A)  & \text{if } K_A > K_B
    \end{cases}
\end{equation}

For the winning rate (or value):

\begin{equation}
    v^*(K_A,K_B) =
        \begin{cases}
        v^*_o(K_A,K_B)  & \text{if } K_A < K_B \\
        1/2           & \text{if } K_A = K_B \\
        v^*_o(K_B,K_A)  & \text{if } K_A > K_B
    \end{cases}
\end{equation}

And an introduced new variable $s^*$ combining $k^*$ and $\nu^*$ can be generalized automaticly:

\begin{equation}
    s^* = \frac{k^* + \nu^*}{N + 1}
\end{equation}

\paragraph{Equilibrium parameter plots:}

The first two figures \ref{fig:Game4__10_Pknu}, \ref{fig:Game4__10_vs} contain all equilibrium parameters of games with $N=4$ and $M=10$ \footnote{if a quantity has no well-defined value for a given pair of parameters, then we indicate this with a dark red region on the plot.}:

\begin{figure}[H]
    \centering
    \begin{subfigure}[b]{0.3\textwidth}
        \includegraphics[width=\textwidth]{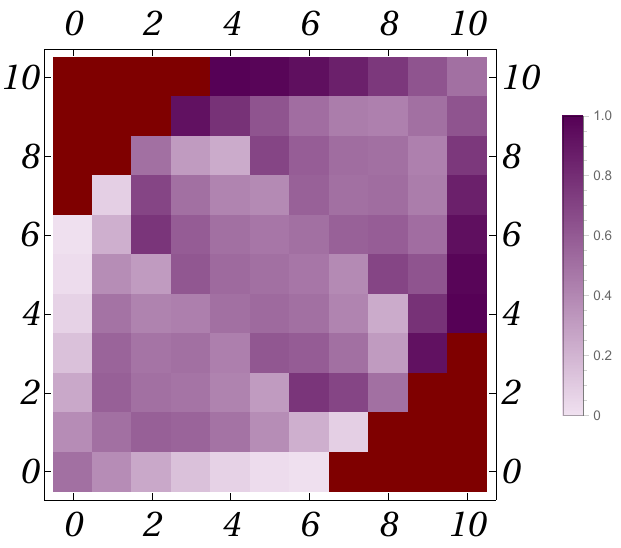}
        \caption{$P^*(K_A,K_B)$}
        \label{fig:sub1}
    \end{subfigure}
    \hfill 
    \begin{subfigure}[b]{0.3\textwidth}
        \includegraphics[width=\textwidth]{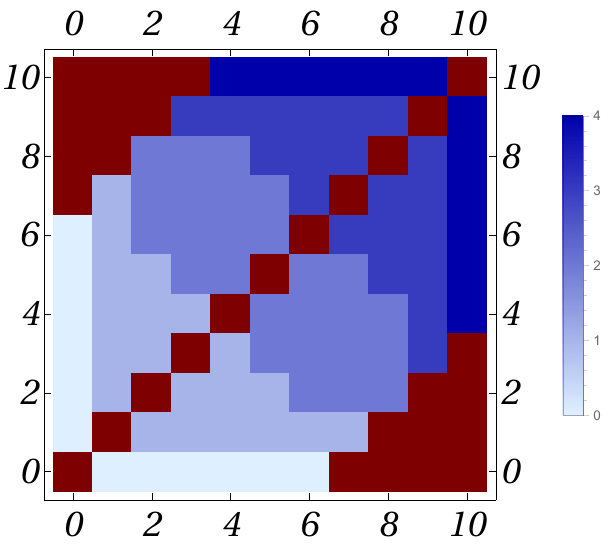}
        \caption{$k^*(K_A,K_B)$}
        \label{fig:sub2}
    \end{subfigure}
    \hfill 
    \begin{subfigure}[b]{0.3\textwidth}
        \includegraphics[width=\textwidth]{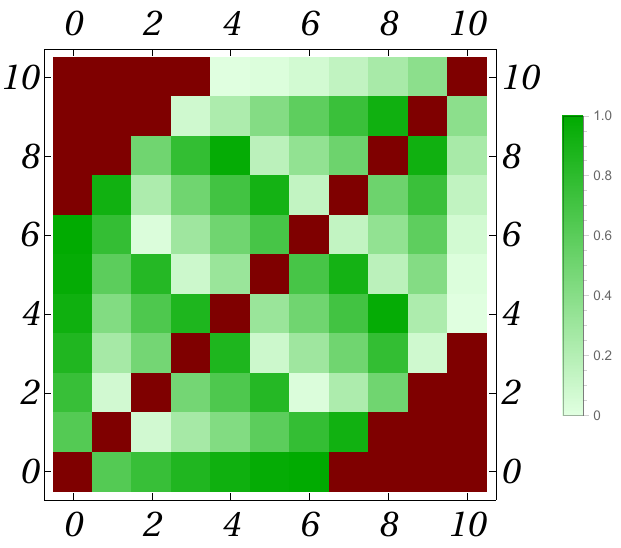}
        \caption{$\nu^*(K_A,K_B)$}
        \label{fig:sub3}
    \end{subfigure}
    
    \caption{\G{N=4,K_A,K_B,M=10}}
    \label{fig:Game4__10_Pknu}
\end{figure}

\begin{figure}[H]
    \centering
    \begin{subfigure}[b]{0.3\textwidth}
        \includegraphics[width=\textwidth]{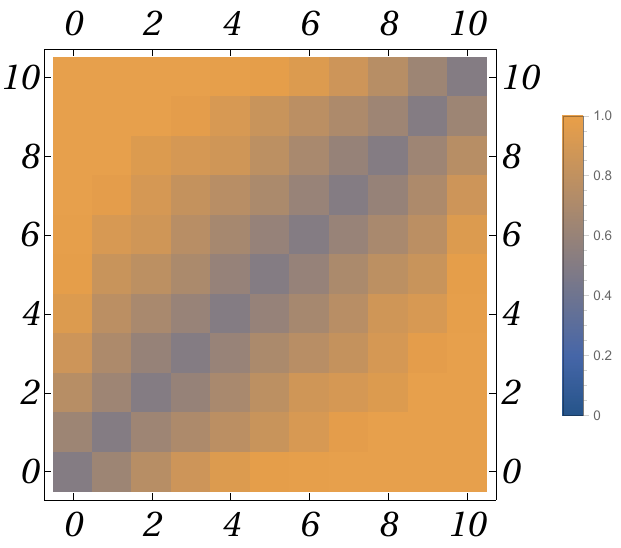}
        \caption{$v^*(K_A,K_B)$}
        \label{fig:sub11}
    \end{subfigure}
    \hspace{0.05\textwidth} 
    \begin{subfigure}[b]{0.3\textwidth}
        \includegraphics[width=\textwidth]{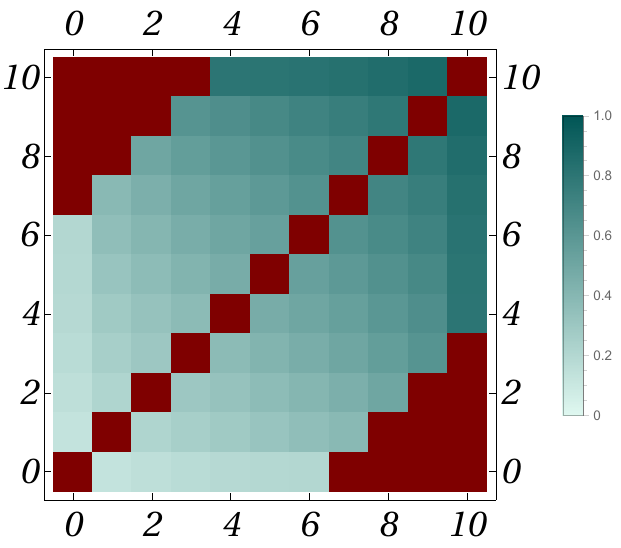}
        \caption{$s^*(K_A,K_B)$}
        
    \end{subfigure}
    
    \caption{\G{N=4,K_A,K_B,M=10}}
    \label{fig:Game4__10_vs}
\end{figure}

The next two figures \ref{fig:Game4__100_Pknu}, \ref{fig:Game4__100_vs} contains all equilibrium parameters of games with $N=4$ and $M=100$:

\begin{figure}[H]
    \centering
    \begin{subfigure}[b]{0.3\textwidth}
        \includegraphics[width=\textwidth]{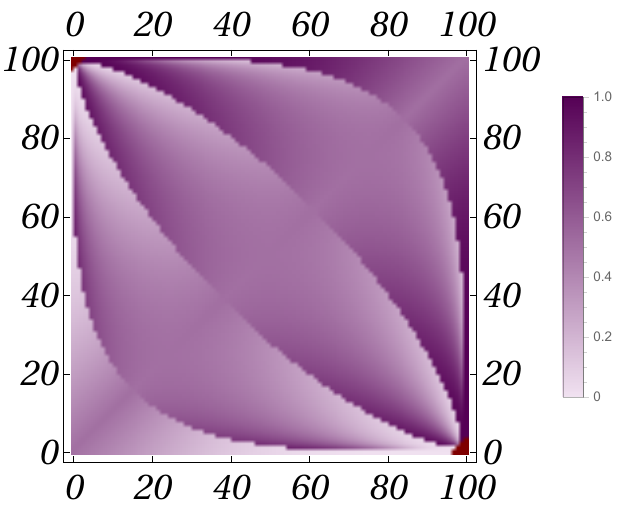}
        \caption{$P^*(K_A,K_B)$}
        \label{fig:Game4__100_P}
    \end{subfigure}
    \hfill 
    \begin{subfigure}[b]{0.3\textwidth}
        \includegraphics[width=\textwidth]{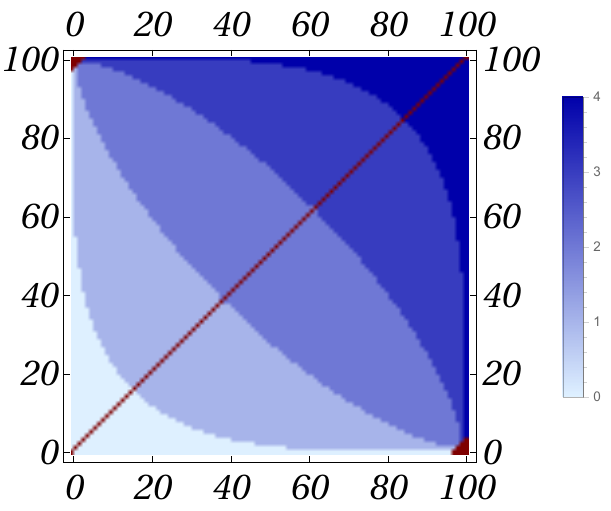}
        \caption{$k^*(K_A,K_B)$}
        \label{fig:Game4__100_k}
    \end{subfigure}
    \hfill 
    \begin{subfigure}[b]{0.3\textwidth}
        \includegraphics[width=\textwidth]{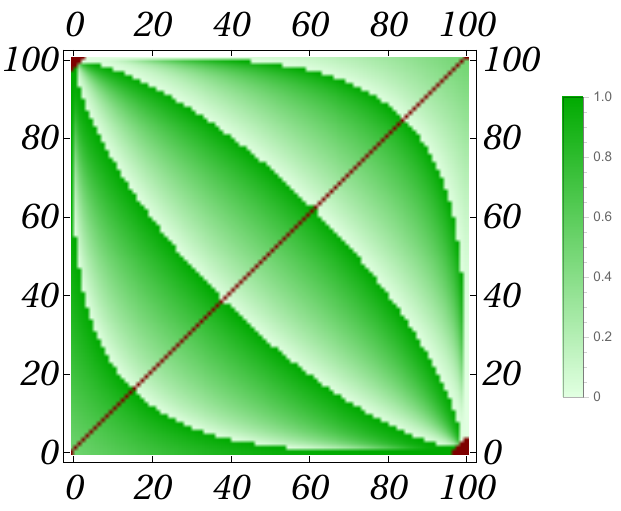}
        \caption{$\nu^*(K_A,K_B)$}
        \label{fig:Game4__100_nu}
    \end{subfigure}
    
    \caption{\G{N=4,K_A,K_B,M=100}}
    \label{fig:Game4__100_Pknu}
\end{figure}

\begin{figure}[H]
    \centering
    \begin{subfigure}[b]{0.3\textwidth}
        \includegraphics[width=\textwidth]{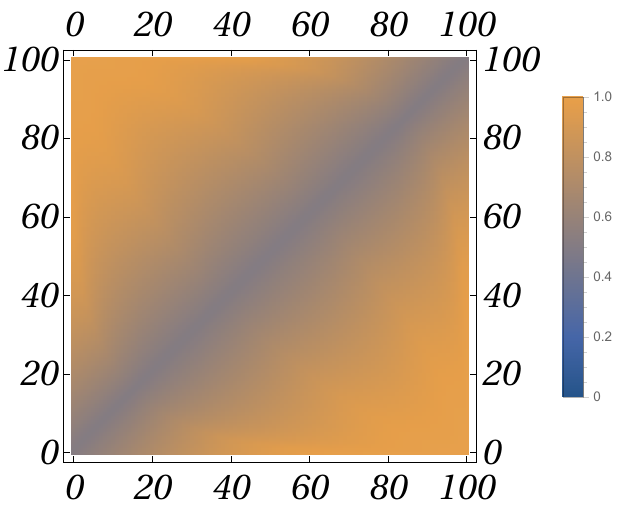}
        \caption{$v^*(K_A,K_B)$}
        \label{fig:Game4__100_v}
    \end{subfigure}
    \hspace{0.05\textwidth} 
    \begin{subfigure}[b]{0.3\textwidth}
        \includegraphics[width=\textwidth]{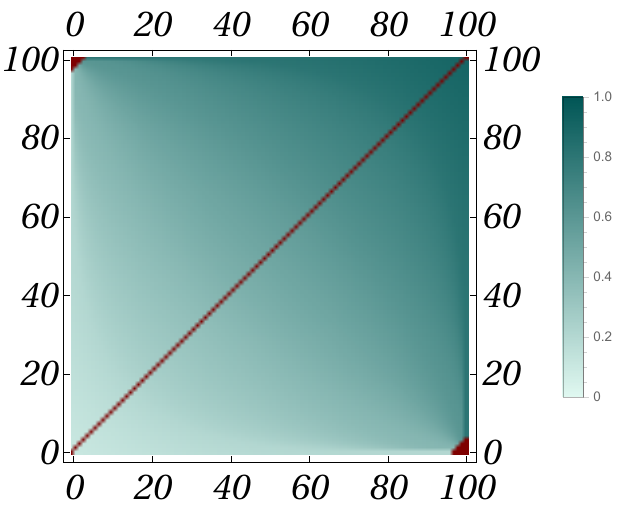}
        \caption{$s^*(K_A,K_B)$}
        \label{fig:Game4__100_s}
    \end{subfigure}
    \caption{\G{N=4,K_A,K_B,M=100}}
    \label{fig:Game4__100_vs}
\end{figure}

The last two figures \ref{fig:Game15__100_Pknu}, \ref{fig:Game15__100_vs} contain all equilibrium parameters of games with $N=15$ and $M=100$:

\begin{figure}[H]
    \centering
    \begin{subfigure}[b]{0.3\textwidth}
        \includegraphics[width=\textwidth]{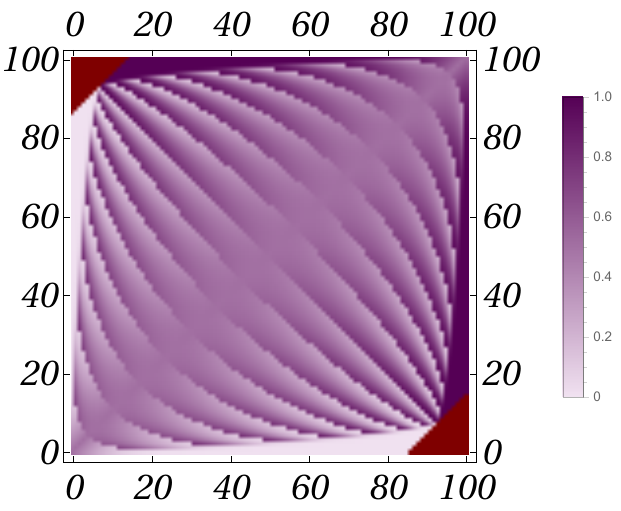}
        \caption{$P^*(K_A,K_B)$}
        \label{fig:sub15}
    \end{subfigure}
    \hfill 
    \begin{subfigure}[b]{0.3\textwidth}
        \includegraphics[width=\textwidth]{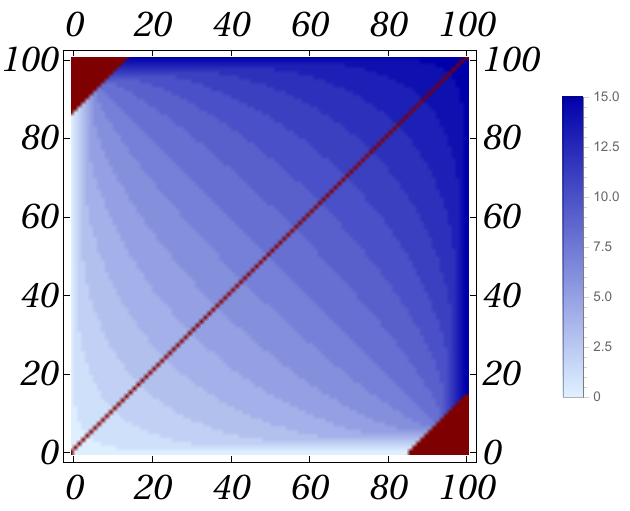}
        \caption{$k^*(K_A,K_B)$}
        \label{fig:sub25}
    \end{subfigure}
    \hfill 
    \begin{subfigure}[b]{0.3\textwidth}
        \includegraphics[width=\textwidth]{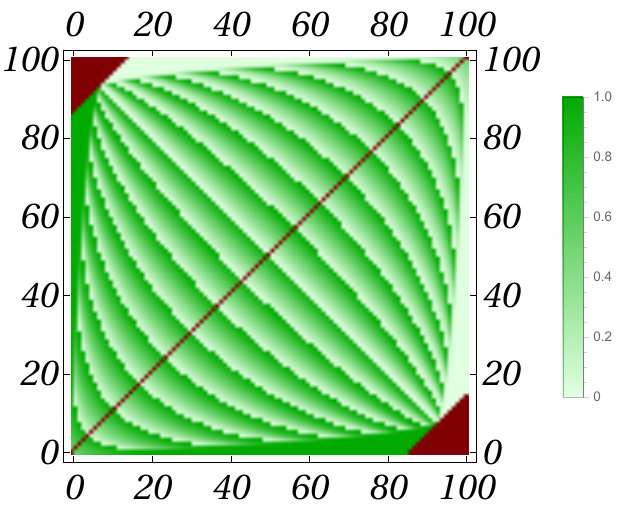}
        \caption{$\nu^*(K_A,K_B)$}
        \label{fig:sub35}
    \end{subfigure}
    
    \caption{\G{N=15,K_A,K_B,M=100}}
    \label{fig:Game15__100_Pknu}
\end{figure}

\begin{figure}[H]
    \centering
    \begin{subfigure}[b]{0.3\textwidth}
        \includegraphics[width=\textwidth]{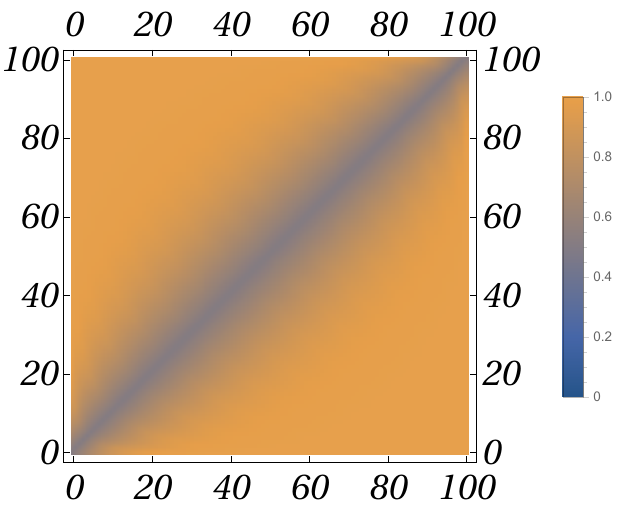}
        \caption{$v^*(K_A,K_B)$}
        \label{fig:sub16}
    \end{subfigure}
    \hspace{0.05\textwidth} 
    \begin{subfigure}[b]{0.3\textwidth}
        \includegraphics[width=\textwidth]{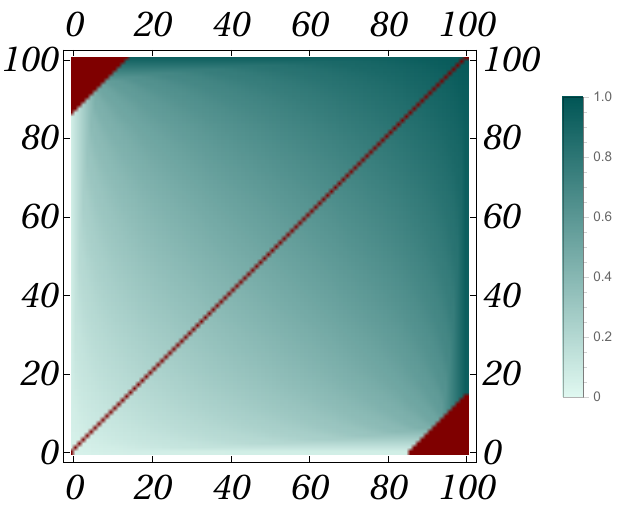}
        \caption{$s^*(K_A,K_B)$}
        \label{fig:sub26}
    \end{subfigure}
    
    \caption{\G{N=15,K_A,K_B,M=100}}
    \label{fig:Game15__100_vs}
\end{figure}

\subsection{Interpretation of the equilibrium strategy}

\paragraph{Sufficient statistics:}

It is worth noting that from the potentially huge amount of available policies, for the symmetric equilibrium, only the total number of $\bb$-s, $k$ matters for choosing a scenario.
\PI/ can ignore the actual order of bits in the sampled sequence $\underline{d}$ to form an optimal strategy. This quantity (the total number of $\bb$-s) can be interpreted as a \emph{sufficient statistics} \cite{book:StatisticalInference,book:Savage} for this decision problem.

\paragraph{Type I and Type II errors:}

The concepts of Type I (``false positive'', the rejection of the null hypothesis when it is true) and Type II (``false negative'', the failure to reject the null hypothesis that is false) errors were central to Neyman and Pearson \cite{paper:NeymanPearson1933b} and are still important concepts in Hypothesis testing in ``Frequentist'' statistics \cite{book:StatisticalInference,book:HandbookParametricNonparametric}.

The equilibrium policy in a Fisher game can be viewed as a statistical test for the hypothesis, stating that A is the correct scenario.
This test's special feature, however, is that for a particular value ($k=k^*$), it gives a probabilistic answer (guess/accept A with probability $\nu^*$ and reject A with probability $1-\nu^*$).
In this context, we can make the following identification:

\begin{equation}
    \text{Type I error:} \quad \alpha_\mathrm{Error} = 1-v_A
\end{equation}

\begin{equation}
    \text{Type II error:} \quad \beta_\mathrm{Error} = 1-v_B
\end{equation}

where $v_A$ and $v_B$, defined in equations \eqref{eq:vAnu} and \eqref{eq:vBnu}, are the winning ratios if \PII/ chose scenario A or B respectively.

In equilibrium, the winning ratios are equal; therefore, the winning rate, i.e. the probability of making an error, is independent of the ``a priory probability'' of A being the true scenario $\pi(A)=P$. By introducing the probabilities of making Type I and Type II errors:

\begin{equation}
    P_\mathrm{I} = \pi(A) \alpha_\mathrm{Error}
\end{equation}

\begin{equation}
    P_\mathrm{II} = \pi(B) \beta_\mathrm{Error}
\end{equation}

and the probability of an error of any kind:

\begin{equation}
    P_\mathrm{Error} = P_\mathrm{I} + P_\mathrm{II}
\end{equation}

We can conclude that the equilibrium strategy can be viewed as a ``statistical test'', which is ``independent of the probability law a priori'', and fulfils Definition A in \cite{paper:NeymanPearson1933b}: 

    \begin{quote}
    We may now discuss what meaning could be given to the words:
    ``a test independent of the probability law a priori.''

    {\it
    Definition A. The phrase might be defined as implying a choice
    of critical region $w$ in such a way that the probability $P_\mathrm{Error}$ of
    making an error in testing $H_0$ had a value independent of the
    probabilities $\pi(A)$, $\pi(B)$.
    }

    \hfill --- Definition A in \cite{paper:NeymanPearson1933b} \footnote{the notation has been slightly modified to harmonize with the notation used in this work}.
    \end{quote}

\paragraph{Randomized sampling:}

For \PI/, taking the sample randomly with a uniform distribution is a vital part of her equilibrium strategy.
This has also been an essential concept in ``Frequentist'' statistics, Fisher being its prime proponent \cite{book:FisherDesignOfExperiments,paper:FishersDevil}.
In the game theoretic framework, such randomized acts appear naturally as part of mixed equilibrium strategies. 

\paragraph{Emergence of probability distributions:}

The emergence of $p_k(A)$ and $p_k(B)$ distributions cannot be interpreted by purely combinatorial arguments (as the Classical interpretation of probability might suggest \cite{sep:InterpretationOfProbability,book:Laplace}).
These distributions result from the mixed equilibrium strategies of both \PI/ and \PII/.
In this framework, \PI/ is not merely a passive observer of random events but, by randomized sampling, partially contributes to the behaviour and specific distribution of the sample.

\paragraph{Randomized policies:}

In the game theoretic framework for general parameters \PI/ is not only sampling randomly, but there is a critical value of $\bb$-s, $k=k^*$ when her action is probabilistic: chooses A with probability $\nu^*$ and B with probability $1-\nu^*$.
Such probabilistic choice or rejection is absent from both ``Frequentist'' and Bayesian hypothesis testing frameworks.
However, randomized policies are not uncommon in the real world. Evidence for randomization is supported from many sides:
    \begin{itemize}
        \item From a human behavioural point of view, experts seem to realize mixed strategies by ``uncorrelated'' complex patterns \cite{MinimaxExperiment01,MinimaxExperiment02}.
        \item From an animal behavioural point of view, wasps seem to realize a mixed strategy \cite{WaspRandomArticle} (for broader context, see for example \cite{EvolutionGameBookJMSmith}), and primates can learn to play Matching Pennies \cite{PrimateRandomArticle}.
        \item Even in collective decision-making, ritualistic practices such as divination could play a role in randomizing the tribe's actions \cite{DivinationArticle}.
    \end{itemize}  

\paragraph{Interpretation of mixed states:}

In the definition of Fisher games \ref{def:FisherGame}, we did not include any stochastic or probabilistic parameter. 
Therefore, when we talk about ``probability'' in the game theoretical context, then we can use a purely game theoretical interpretation:

\begin{quote}
    {\it
    Mixed strategy can alternatively be viewed as the belief held by all other players concerning a player's actions. A mixed strategy equilibrium is then an $n$-tuple of common knowledge expectations, which has the property that all the actions to which a strictly positive probability is assigned are optimal, given the beliefs. A player's behavior may be perceived by all the other players as the outcome of a random device even though this is not the case.
    }
    
    \hfill --- Ariel Rubinstein in \cite{paper:Rubinstein}.
    \end{quote}

This means that inside a Fisher game, \PI/ does not need to adopt a probability concept based on frequencies \cite{book:VonMises}, degree of belief\footnote{or even as qua bases of action \cite{essay:Ramsey}} \cite{book:Jaynes} or the principle of insufficient reason \cite{book:Laplace}.
``Probabilities'' in this context appear only in the reasoning process of players, who act \emph{as if} certain distributions would be justified, but these concepts do not have to have any grounding outside the game.
This might be called a pre-Bayesian or game theoretical interpretation of probability.

``\dots we do not approach the problem as Bayesians, saying that agent $i$’s
beliefs can be described by a probability distribution; instead, we use a 'pre-Bayesian' model in which $i$ does not know such a distribution and indeed has no beliefs about it.'' \cite{book:EssentialGameTheory}.

\section{Bayesian betting}

The concept is intimately related to Bayesian inference \cite{book:Jaynes,book:Bernardo}, Bayesian hypothesis testing \cite{book:Jaynes} and Bayesian decision theory \cite{book:BayesianDataAnalysis}. Some authors use the analogy of a horse race \cite{book:InformationTheory} to introduce the concept.

\subsection{General concepts}

\paragraph{Double or nothing:}

Now, we enter a different casino, where instead of betting simply on scenario A or B, we can place some portion of our capital on each alternative. The portion we put on the winning scenario (chosen by \PII/) will be doubled, while the portion placed on the other scenario will be lost.

\subsubsection{The statistically trivial case}

\label{subsection:BettingGame}

Remarkably, if a gambler can repeatedly bet a portion of her capital to outcome A or B (which occur with probability $P$ and $1-P$ respectively), then in the long run, splitting the bet can be ``better'' compared to placing all the money to the more likely outcome \cite{Kelly}.

\paragraph{Growth factor:}

To demonstrate that a balanced splitting strategy might be ``better'', we can consider the following simple example: 

Let us investigate how the expectation of the ratio of capitals for two gamblers looks after $n$ rounds if they use different splitting strategies for their capital. They place $p'_1$ and $p'_2$ portion of their capital to A, and $1-p'_1$, $1-p'_2$ to B.

The ratio can be represented by a random variable $\rho$:

\begin{equation}
    \rho = \frac{(2 p'_1)^\kappa (2 (1-p'_1))^{n-\kappa}}{(2 p'_2)^\kappa (2 (1-p'_2))^{n-\kappa}}
\end{equation}

Where $\kappa$ is a random variable representing the number of A-s in the randomly generated $n$ long sequence of A-s and B-s. If the occurrence of scenarios A and B are independent in separate rounds, and their distribution is identical $\pi(A) = P$, $\pi(B) = 1-P$, then $\kappa$ can be characterized by the following Binomial distribution:

\begin{equation}
    \kappa \sim \mathrm{Binom}(n,P)
\end{equation}

\begin{equation}
    \Pr(\kappa = k) = \binom{n}{k} P^k (1-P)^{n-k} 
\end{equation}

At this point, we introduce the difference of the two gamblers' ''growth factors``:

\begin{equation}
    \Delta G =\frac{1}{n} \mathbb{E} [\log(\rho)]
\end{equation}

By recalling that for a Binomial random variable $\E[\kappa/n] = P$, this difference can be expressed by the parameters $p'_1$, $p'_2$ and $P$:

\begin{equation}
    \begin{split}
        \Delta G =&  \left (  P \log(2 \ p'_1) + (1-P) \log(2(1-p'_1))  \right ) \\
        &- \left (  P \log(2 \ p'_2) + (1-P) \log(2(1-p'_2))  \right )
    \end{split}
\end{equation}

$\Delta G(p'_1,p'_2)$ is visualized in figure \ref{fig:LogRatio}, with $P=2/3$. The graph clearly shows a saddle point at $p'^*_1=p'^*_2=P=2/3$.

\begin{theorem}
    In a repeated double or nothing gamble (where the right scenarios are identically and independently distributed with probability $P$), two gamblers' growth rate difference has a saddle point at:

    \begin{equation}
        p'^*_1 = p'^*_2 = P
    \end{equation}

    This means that for any gambler deviating from this splitting ratio, the expectation of her relative capital will be lower than her optimally playing counterpart's capital. 
    
\end{theorem}

\begin{proof}
    This can be proved by directly solving the set of equations:
    
    \begin{equation}
        \frac{\partial}{\partial p'_1} \Delta G(p'_1,p'_2) \Bigr|_{p'_1=p'^*_1,p'_2=p'^*_2} = 0, \quad
        \frac{\partial}{\partial p'_2} \Delta G(p'_1,p'_2) \Bigr|_{p'_1=p'^*_1,p'_2=p'^*_2} = 0
    \end{equation}

    \begin{equation}
        \frac{P}{p'^*_1} - \frac{1-P}{1-p'^*_1} = 0,
        \quad
        -\frac{P}{p'^*_2} + \frac{1-P}{1-p'^*_2} = 0
    \end{equation}

    and directly calculating the Hessian of $\Delta G(p'_1,p'_2)$:

    \begin{equation}
        \underline{\underline{H}}^* = 
        \frac{1}{P(1-P)}
        \begin{bmatrix}
        -1 & 0 \\
        0 & 1
        \end{bmatrix}
    \end{equation}

    The signature of the Hessian verifies that we indeed found a saddle point.
    
\end{proof}

\begin{remark}[Connection with Entropy]
Remarkably, the growth factor in equilibrium is closely related to the Shannon entropy ($H$) \cite{paper:ShannonOriginal, book:InformationTheory} of the random variable $\Pi$, which is characterized by a probability distribution $\pi$ on $\Theta = \{A,B\}$: $\pi(A)=P$, $\pi(B)=1-P$.

\begin{equation}
    G^* = P \log(2 \ P) + (1-P) \log(2(1-P))
\end{equation}

\begin{equation}
    G^* = \log(2) + P \log(P) + (1-P) \log(1-P)
\end{equation}

\begin{equation}
    G^*= \log(2) - H(\Pi)
\end{equation}

\end{remark}

\begin{figure}[H]
    \centering
    \includegraphics[width=12 cm]{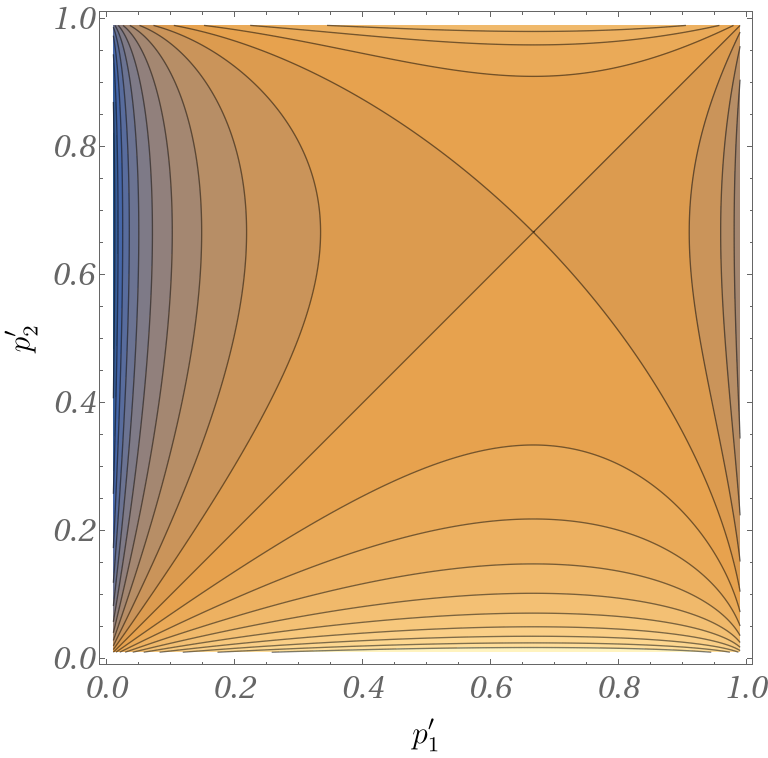}
    \caption{Contour plot of the expected log ratio of two gamblers' capitals (or growth factor difference) $\Delta G(p'_1,p'_2)$ for $P=2/3$. (Colouring closer to blue represents negative values, while colouring closer to gold represents positive values.)}
    \label{fig:LogRatio}
\end{figure}

\subsubsection{Adopting a utility function}

In appendix \ref{appendix:UtilityFunctions}, I argue that adopting a utility function, which associates the logarithm of the gained capital to the utility of the outcome:

\begin{equation}
    u(c) = \log(c)
\end{equation}

is a ``natural choice'' in many aspects. This is because this utility function is \emph{aligned} with a wide variety of multiplicative multiplayer contests.
Alignment in this context means that the directly calculated optimal policies in the multiplicative multiplayer games are the same \emph{as if} agents were trying to maximise their expected logarithmic utility (For details, see Section~\ref{sec:ContestView}). 

However, I want to point out that adopting a utility function is ultimately a normative choice of an agent, and different choices may result in different optimal policies.

For instance, a whole one-parameter family of utility functions fulfil a weaker consistency requirement \footnote{also known as isoelastic, isocurvature, power utility \cite{book:EconomicsDictionary} or constant relative risk aversion (CRRA) utility \cite{book:Arrow, paper:Pratt}}:

\begin{equation}
\label{eq:ugamma}
    u_\gamma(c) = \frac{c^{1-\gamma}-1}{1-\gamma}
\end{equation}

Where $\gamma > 0$ is the so-called Arrow-Pratt measure of relative risk aversion \cite{book:MathematicsForEconomists, book:EconomicsDictionary}. 

Assuming that the agent is maximising her expected utility \cite{plato:ExpectedUtility}, different utility functions result in different optimal splitting ratios:

\begin{equation}
    p'^*_\gamma(P) = \frac{P^{1/\gamma}}{P^{1/\gamma}+(1-P)^{1/\gamma}}
\end{equation}

In the following sections -- until stated otherwise -- we will assume that \PI/ is adopting the ``most natural'' logarithmic utility function.

\subsection{Description of the general Bayesian game}
\label{def:BayesianGame}

\begin{definition}[Bayesian game \footnote{not to be confused with the Bayesian Game introduced by Harsányi \cite{paper:Harsanyi_I,paper:Harsanyi_II,paper:Harsanyi_III},\cite{paper:ZamirBayesianGames}. An alternative naming might be pre-Bayesian game \cite{book:EssentialGameTheory}.}]

There are two players, \PI/ and \PII/.
\PII/ needs to choose between scenario A or B, then produce a binary sequence of length $M$ containing precisely $K_A$ or $K_B$ number of $1$-s. (Without losing generality, we will assume $K_A \le K_B$.)
Following this, \PI/ (not knowing the actions of \PII/) can sample $N$ number of bits. After observing their values, she determines what portion of her capital $p'$ she places on scenario A (while the other $1-p'$ portion is placed on scenario B).

The portion \PI/ places on the scenario, chosen by \PII/, will be doubled, while the other part of her capital will be lost.
For this specific game, we will assume that \PI/ has a logarithmic utility function and that \PI/ and \PII/ are playing a zero-sum game\footnote{zero-sum in utilities (not in capital)}.
The above-defined Bayesian game will be denoted as 
\BG{N, K_A, K_B, M}.

\end{definition}

A tabular description of the general Bayesian game is shown by Game Table~\ref{game:GeneralBayesianGame}.

\begin{gametable}[H]
\captionsetup{justification=centering}
\caption{\label{game:GeneralBayesianGame} General description of \\ \BG{N, K_A, K_B, M}}

\centering
\begin{tabularx}{0.73\textwidth}{ X | X }

\hline
 &  \\
\multicolumn{1}{c|}{\PI/} & \multicolumn{1}{c}{\PII/} \\
 &
\begin{itemize}
        \item Chooses scenario A or B,
        \begin{itemize}
            \item then chooses a binary sequence available for the chosen scenario.
        \end{itemize}
\end{itemize}
\\
\begin{itemize}
        \item Chooses $N$ indices for sampling,
        \begin{itemize}
            \item based on the bits in the chosen sample determines a continuous parameter $p' \in [0,1]$,
            \begin{itemize}
                \item and places her capital's $p'$ portion to scenario A and $1-p'$ portion to scenario B.
            \end{itemize}
        \end{itemize}
    \end{itemize}
& \\
& \\
\multicolumn{2}{c}{
    \begin{minipage}{0.75\linewidth}
        \centering
        The portion \PI/ places on the scenario chosen by \PII/ will be doubled, while the other part of her capital will be lost.
        
        \PI/ has a logarithmic utility function, \PI/ and \PII/ are playing a zero-sum game.
    \end{minipage}
} \\
\end{tabularx}
\end{gametable}

\begin{remark}
    \PI/ can choose from a continuous set of actions. This means that, strictly speaking, this situation cannot be solved using the framework of finite games \cite{book:GameTheory}.

    Unfortunately, because of the unbounded logarithmic utility function (\PI/ can potentially realise $-\infty$ losses), the problem is not even a so-called Continuous Game \cite{paper:ContinuousEquilibrium}.

    However, against all these difficulties, by direct calculation, a unique equilibrium solution can be found for the above-defined Bayesian games.
    
\end{remark}

\subsection{Simplest case}

For Bayesian games, analysing the statistically trivial case might be useful, which can be considered the analogy of Blind guessing described in Section~\ref{sec:FisherTrivialCases}.

\paragraph{Utility function:}

In the game \BG{N=0,K_A=0,K_B=0,M=0} \PI/ needs to determine how she splits her capital without using any meaningful information about the scenarios.

The utility function (which is, in this logarithmic case, equivalent to the growth rate) looks the following:

\begin{equation}
U_1(p')=
\begin{bmatrix}
\log(2 \ p') & \log(2 \ (1-p'))
\end{bmatrix}
\end{equation}

This continuous-discrete utility function is visualised in figure \ref{fig:U_0000}. Since the utility function is unbounded from below (\PI/ can lose all her capital, in which case her utility goes to $-\infty$), we used a negative cutoff. Utilities, lower than the cutoff are not present which is indicated by grey colouring.

\begin{figure}[H]
    \centering
    \includegraphics[width=8 cm]{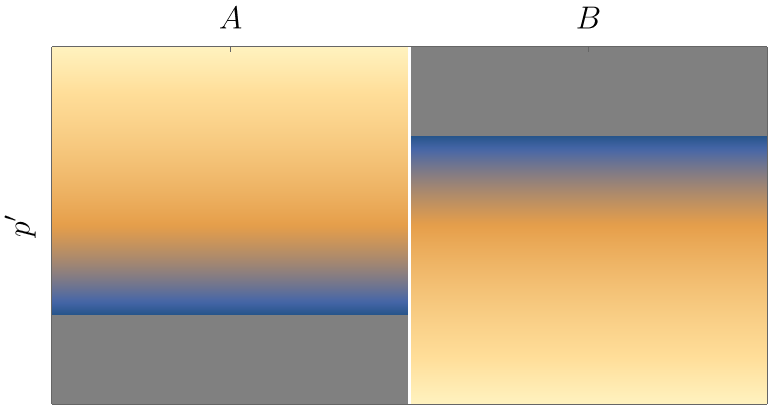}
    \caption{Utility function of \BG{N=0,K_A=0,K_B=0,M=0}.}
    \label{fig:U_0000}
\end{figure}

\paragraph{Equilibrium solution:}

To find equilibrium, the following expected utility (or growth rate) has to be maximised by \PI/ (controlling $p'$) and minimised by \PII/ (controlling P):

\begin{equation}
    G(P,p') = P \log(2 \ p') + (1-P) \log (2 (1-p'))
\end{equation}

Maximisation respect to $p'$:

\begin{equation}
    P \frac{1}{p'^*} - (1-P) \frac{1}{1-p'^*} = 0
\end{equation}

\begin{equation}
    p'^* = P
\end{equation}

Minimising respect to $P$:

\begin{equation}
    G(P) = \log(2) + P \log(P) + (1-P) \log (1-P)
\end{equation}

\begin{equation}
    G(P) = \log(2) - H(P)
\end{equation}

which is minimal when the entropy $H(P)$ is maximal, i.e. $P^*=1/2$.

Summarising the result:

\begin{equation}
    P^* = 1/2, \quad p'^* = 1/2, \quad G^* = 0
\end{equation}

\subsection{Simplest statistically nontrivial case}
\label{sec:SimplestNontrivialBayesian}

We could tabulate the smallest and simplest Bayesian games and compose a table very similar to \ref{tab:M<=2}.
It is straightforward to find the analogues of Blind guessing and Sure winning games for Bayesian games. 

After identifying the trivial cases, we can conclude that the simplest statistically nontrivial Bayesian game is:
\BG{N=1,K_A=0,K_B=1,M=2}.

\paragraph{Action sets:}

The possible actions of \PII/, are the same as in \eqref{eq:SimplestFisher_A2}:

\begin{equation}
    \mathcal{A}_2 = \{ (A, (\wb,\wb)), (B, (\wb,\bb)) ,(B,(\bb,\wb)) \}
\end{equation}

\PI/ can sample one bit from indices $i \in \{1,2\}$, (formally $\mathcal{S}=\{\{1\},\{2\}\}$)
and choose 4 different splitting ratios for scenario A \footnote{which are the splitting ratios for scenario B subtracted from 1, because $p'_{(S, \underline{d} \to A)} + p'_{(S, \underline{d} \to B)}=1$}:

\begin{equation}
    \begin{split}
    p'_{(i,\wb)} &= p'_{(\{i\},(\wb) \to A)} = 1- p'_{(\{i\},(\wb) \to B)} \\
    p'_{(i,\bb)} &= p'_{(\{i\},(\bb) \to A)} = 1- p'_{(\{i\},(\bb) \to B)}
    \end{split}
\end{equation}

\paragraph{Utility function:}

\PI/ has both discrete and continuous choices (sampling is a discrete choice, and choosing a splitting ratio is a continuous one), while \PII/ has only discrete choices.

The utility for all possible outcomes can be captured by this hybrid utility function matrix \footnote{where \PI/ is choosing a row and picks a set of $\{p'\}$-s, while \PII/ is choosing a column}:

\begin{equation}
    U_1(\{p'\}) = 
    \begin{bmatrix}
    \log(2 p'_{(1,\wb)}) & \log(2(1-p'_{(1,\wb)})) & \log(2(1-p'_{(1,\bb)})) \\
    \log(2 p'_{(2,\wb)}) & \log(2(1-p'_{(2,\bb)})) & \log(2(1-p'_{(2,\wb)})) 
    \end{bmatrix}
\end{equation}

\paragraph{Parametrization of strategies:}

We will use the following parametrization for \PII/'s mixed strategy in her discrete choices ($P \in [0,1], r \in [0,1]$) \footnote{formally the mixed strategy of choosing from allowed sequences after choosing a scenario is: $\underline{\varpi}(A) = (1)$, $\underline{\varpi}(B) = (r,1-r)$}:

\begin{equation}
    \underline{\sigma}_2 = (P, (1-P) r, (1-P)(1-r))
\end{equation}

While the mixed choice of sampling of \PI/ will be parametrized by $q \in [0,1]$:

\begin{equation}
    \underline{\eta}_1 = (q,1-q)
\end{equation}

\paragraph{Growth rate difference:}

Using this notation, we can formulate the following expression for the growth rate difference relative to the sure winning case \footnote{measuring the difference between the absolute growth rate and the growth rate of a player who can always win (i.e. double her capital). $\Delta G = G - \log(2)$}:

\begin{equation}
\label{eq:Simplest_DeltaG}
    \begin{split}
        \Delta G = P & \left ( 
        q \log(p'_{(1,\wb)}) + (1-q) \log(p'_{(2,\wb)})
        \right ) + \\
        (1-P) r & \left (
        q \log(1-p'_{(1,\wb)}) + (1-q) \log(1-p'_{(2,\bb)})
        \right) + \\
        (1-P) (1-r) & \left (
        q \log(1-p'_{(1,\bb)}) + (1-q) \log(1-p'_{(2,\wb)}
        \right )
    \end{split}
\end{equation}

\paragraph{Equilibrium solution:}

First, we can observe that \PI/ can set
\begin{equation}
    p'_{(1,\bb)} = p'_{(2,\bb)} = p'^*_1 = 0,
\end{equation}
to maximize her gain. This follows from common sense because \PI/ can safely push all her capital to scenario B if she ever sees a $\bb$ in her sample (because no allowed sequence from scenario A contains $\bb$).

For any $(P, r)$ parameters we can find the optimal $p'^*_{(1,\wb)}$ and $p'^*_{(2,\wb)}$ by which \PI/ can maximize her gain \footnote{this can be calculated by taking the derivative of $\Delta G$ with respect to $p'^*_{(1,\wb)}$ and $p'^*_{(2,\wb)}$; and then finding the values, when the derivatives are equal to zero}:

\begin{equation}
    p'^*_{(1,\wb)} = \frac{P}{P+(1-P) r}, \quad  p'^*_{(2,\wb)} = \frac{P}{(1-r)+ P r}
\end{equation}

If \PI/ is using a mixed strategy, then the expected growth rate from sampling the first or second index must be equal:

\begin{equation}
    P \log(p'^*_{(1,\wb)}) + (1-P) r \log(1-p'^*_{(1,\wb)}) =
    P \log(p'^*_{(2,\wb)}) + (1-P) (1-r) \log(1-p'^*_{(2,\wb)})
\end{equation}

This equation can be solved for $r$. For all $P$, this gives a simple result:

\begin{equation}
    r^*(P) = 1/2
\end{equation}

This means that in equilibrium, \PII/ chooses the two available sequences in scenario B uniformly.

Having a $r \in (0,1)$ value for $r$ means that the expression \eqref{eq:Simplest_DeltaG}'s derivative respect to $r$ has to be 0:

\begin{equation}
    q \log(1-p'^*_{(1,\wb)}) = (1-q) \log(1-p'^*_{(2,\wb)})
\end{equation}

$r^*=1/2$, in which case $p'^*_{(1,\wb)} = p'^*_{(2,\wb)}$, so the equilibrium value for $q$ has to be:

\begin{equation}
    q^* = 1/2
\end{equation}

Combining these results, we get the following expression for $\Delta G(P)$, which has to be minimized by \PII/ by choosing an optimal $P^*$:

\begin{equation}
    \Delta G(P) = P \log \left ( \frac{P}{1/2+P/2} \right ) +
    \frac{1}{2} (1-P) \log \left ( 1 - \frac{P}{1/2+P/2} \right )
\end{equation}

After taking the derivative respect to $P$, and simplification, we get the following equation for $P^*$:

\begin{equation}
    \log \left ( \frac{2 P^*}{1 + P^*} \right ) +
    \frac{1}{2} \log \left ( \frac{1 - P^*}{1 + P^*} \right ) = 0
\end{equation}

\begin{equation}
    \frac{1}{2} \log \left ( \frac{4 {P^*}^2}{1 - {P^*}^2} \right ) = 0
\end{equation}

\begin{equation}
     \frac{4 {P^*}^2}{1 - {P^*}^2}  = 1
\end{equation}

\begin{equation}
     P^*  = \frac{1}{\sqrt{5}}
\end{equation}

Substituting back to the expressions for $p'^*_{(1,\wb)}$ and $p'^*_{(2,\wb)}$ we get:

\begin{equation}
     p'^*_{(1,\wb)} = p'^*_{(2,\wb)} = p'^*_0 = \frac{1}{2} \left ( \sqrt{5} -1 \right )
\end{equation}

Which is remarkably the reciprocal of the celebrated golden ratio $\phi$ (or $\tau$) \cite{book:Fibinacci, book:MathConstants, book:ConwayBookOfNumbers}.

\paragraph{Summary:}

Figure \ref{fig:GainPlot1012} shows $\Delta G(P)/\log(2)$, which can be interpreted as a doubling factor difference relative to Sure winning. ($\Delta G(P)/\log(2)=0$ means that \PI/ can double her capital in every round, while $\Delta G(P)/\log(2)=-1$ means that her capital will stagnate and not have an expected growth.)

\begin{figure}[H]
    \centering
    \includegraphics[scale=0.7]{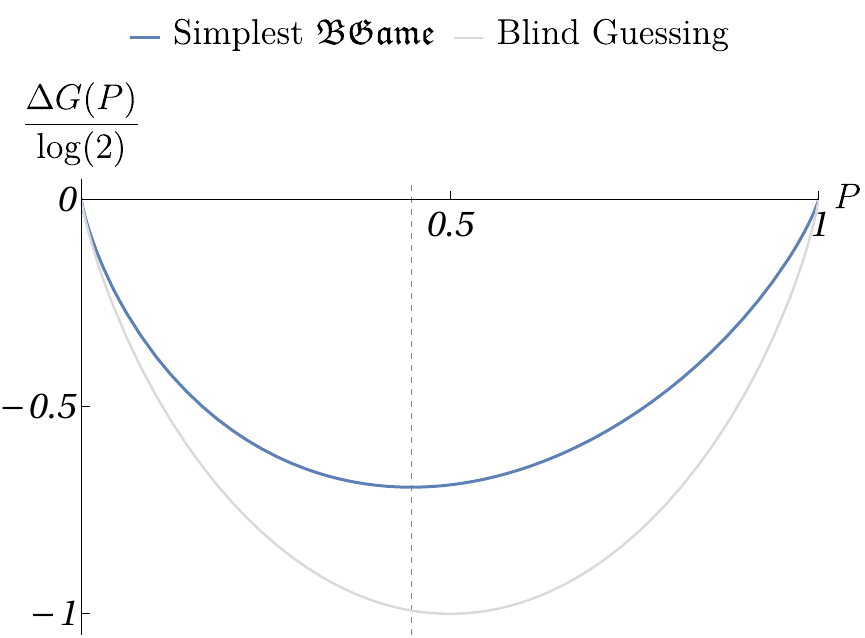}
    \caption{Visualization of the doubling factor difference (relative to a Sure winning game) $\Delta G(P)/\log(2)$ for \PI/ in the game: \BG{N=1,K_A=0,K_B=1,M=2} and a Blind guessing game, as a function of \PII/'s strategy of choosing A. The dashed grid line is placed to the equilibrium value of $P^*=1/\sqrt{5}$.}
    \label{fig:GainPlot1012}
\end{figure}

The equilibrium quantities are:

\begin{equation}
    P^* = \frac{1}{\sqrt{5}}, \quad p'^*_0 = \frac{1}{2} \left ( \sqrt{5} -1 \right ), \quad p'^*_1 = 0 
\end{equation}

Optimal gain and doubling factor:

\begin{equation}
    G^* \approx 0.2119, \quad \frac{G^*}{\log(2)} \approx 0.3058
\end{equation}

The game table of the optimal play can be found in Game Table~\ref{game:B_1_2_1_2}.

\begin{gametable}[H]
\captionsetup{justification=centering}
\caption{\label{game:B_1_2_1_2} Equilibrium strategy for \\ \BG{N=1, K_A=0, K_B=1, M=2}}

\centering
\begin{tabularx}{0.73\textwidth}{ X | X }

\hline
 &  \\
\multicolumn{1}{c|}{\PI/} & \multicolumn{1}{c}{\PII/} \\
 &
\begin{itemize}
        \item Choose scenario A with probability $P^* = 1/\sqrt{5} \approx 0.447$, and scenario B with probability $1-P^* = 1-1/\sqrt{5} \approx 0.553$
        \begin{itemize}
            \item Choose uniformly from all different allowed sequences.
        \end{itemize}
\end{itemize}
\\
\begin{itemize}
        \item Sample randomly from all possible indices uniformly
        \begin{itemize}
            \item in case the sampled bit is $\wb$:
            \begin{itemize}
                \item place 
                
                $p'^*_0 = \frac{1}{2} \left ( \sqrt{5} - 1 \right ) \approx 0.618$ portion of the capital to A 
                \item place 
                
                $1-p'^*_0 = \frac{1}{2} \left (3 - \sqrt{5} \right ) \approx 0.382$ portion of the capital to B
            \end{itemize}
            \item in case the sampled bit is $\bb$:
            \begin{itemize}
                \item place all capital to B
            \end{itemize}
        \end{itemize}
    \end{itemize}
& \\
\end{tabularx}
\end{gametable}

\begin{remark}
    We found an equilibrium solution assuming that \PI/ will use a pure policy, i.e. she will not mix different $\{p'\}$ splitting ratios.

    A mixed policy would be suboptimal relative to an appropriately chosen pure policy because of the concavity of both expressions:

    \begin{equation}
        \log(p'), \quad \mathrm{and} \quad \log(1-p')
    \end{equation}

\end{remark}

\subsection{General structure of the action sets}

\paragraph{Continuous Policy set:}

Most of the general statements about action sets remain valid from \ref{sec:ActionSpaceStructure}. The most important difference is that in Bayesian games, policies are mappings from sequences to continuous splitting ratios in $[0,1]$ interval.

\paragraph{Excluding all-in policies:}

If the sample could come out both from A and B, then we can exclude ``all-in'' strategies, where \PI/ would bet all her capital only to A or B. This is because, by such ``reckless'' behaviour, she could potentially lose all her capital, i.e. realize a $-\infty$ loss.

Considering these restrictions, we can define the following general continuous policy sets:

\begin{equation}
    \mathcal{P}''_c \cong \{ \phi : \{\wb,\bb\}^{|\mathbb{K}_{AB}|} \mapsto (0,1) \} \cong (0,1)^{2^{|\mathbb{K}_{AB}|}}
\end{equation}

\begin{equation}
    \mathcal{P}''^R_c \cong \{ \phi : \mathbb{K}_{AB} \mapsto (0,1) \} \cong (0,1)^{|\mathbb{K}_{AB}|}
\end{equation}

\begin{remark}
These policy sets are open sets; therefore, not compact.
This prevents us from automatically using theorems for Continuous Games \cite{paper:ContinuousEquilibrium} or Separable Games \cite{paper:SeparableGames}, for which the action sets have to be compact.

\end{remark}

\paragraph{Action sets:}

The action sets are the following for Bayesian games:

\begin{equation}
    \mathcal{A}_1 = \mathcal{S} \times \mathcal{P}_c
\end{equation}

Where the sampling set $\mathcal{S}$ is identical to \eqref{eq:SamplingSet}.

\begin{equation}
    \mathcal{A}_2 = \mathcal{K}_A \cup \mathcal{K}_B
\end{equation}

Together with the logarithmic utility function, these conditions result in a Non-Compact Continuous Game.

\subsection{Equilibrium solution of the general Bayesian game}
\label{sec:EquilibriumSolutionGeneralBayesianGame}

\paragraph{The Ansatz:}

An Ansatz\footnote{an educated guess, or an assumption about the form of the solution} can be formulated for the general case, described in Game Table~\ref{game:BayesianGameAnsatz}.

\begin{gametable}[H]
\captionsetup{justification=centering}
\caption{\label{game:BayesianGameAnsatz} {\bf Ansatz} for \\ \BG{N, K_A, K_B, M} \\ having a set of free variables: $P$, $\{p'_k\}_{k \in \mathbb{K} } $}

\centering
\begin{tabularx}{0.73\textwidth}{ X | X }

\hline
 &  \\
\multicolumn{1}{c|}{\PI/} & \multicolumn{1}{c}{\PII/} \\
 &
\begin{itemize}
        \item choose scenario A with probability $\pi(\mathrm{A})=P$, or B with probability $\pi(\mathrm{B})=1-P$.
        \begin{itemize}
            \item Choose uniformly from all different allowed sequences.
        \end{itemize}
\end{itemize}
\\
\begin{itemize}
        \item sample $N$ bits randomly and uniformly from all available $M$ bits,
        \begin{itemize}
            \item based on the number of $\bb$-s $k$, determine a continuous parameter $p'_k \in [0,1]$,
            \begin{itemize}
                \item and place the capitals $p'_k$ portion to scenario A and $1-p'_k$ portion to scenario B.
            \end{itemize}
        \end{itemize}
    \end{itemize}
& \\
\end{tabularx}
\end{gametable}

\paragraph{Fixing the splitting ratios:}

Assuming this Ansatz, we can introduce the probabilities of seeing $k$ $\bb$-s in the sample, given the chosen scenario by \PII/ has been A or B:

\begin{equation}
    p_k(A) = \frac{\binom{K_A}{k} \binom{M-K_A}{N-k}}{\binom{M}{N}}, \quad
    p_k(B) = \frac{\binom{K_B}{k} \binom{M-K_B}{N-k}}{\binom{M}{N}}
\end{equation}

Using these quantities, the expected logarithmic utility (or growth rate) is the following:

\begin{equation}
    G = \pi(A) \sum_k p_k(A) \log(2 \ p'_k) + \pi(B)  \sum_k p_k(B) \log(2 \ (1 - p'_k))
\end{equation}

We can take the derivative respect to $p'_\ell$, to find the maximum respect to that splitting ratio:

\begin{equation}
    \frac{\partial}{\partial p'_\ell} G \Bigr|_{p'_\ell=p'^*_\ell} = \pi(A) p_\ell(A) \frac{1}{p'^*_\ell} - \pi(B) p_\ell(B) \frac{1}{1-p'^*_\ell} = 0
\end{equation}

\begin{equation}
    p'^*_\ell = \frac{\pi(A) p_\ell(A)}{\pi(A) p_\ell(A)+\pi(B) p_\ell(B)}
\end{equation}

\begin{remark}
    The result for the optimal splitting ratio formally follows Bayes' rule \cite{sep:BayesTheorem}, however, the interpretation of each individual component differs from their conventional meanings in Bayes' rule.
    
\end{remark}

\paragraph{Finding $P^*$:}
\label{par:FindingPs}

After substituting the optimal splitting ratios $p'^*_k$, we have the following equation for the growth rate difference \footnote{relative to the growth rate of an agent who always wins (i.e. doubles in every round). $\Delta G = G - \log(2)$}:

\begin{equation}
    \begin{split}
        \Delta G(P)=&P \sum_k p_k(A) \log \left ( \frac{P \ p_k(A)}{P \ p_k(A)+(1-P) p_k(B)} \right ) + \\
                    &(1-P) \sum_k p_k(B) \log \left ( \frac{(1-P) p_k(B)}{P \ p_k(A)+(1-P) p_k(B)} \right )
    \end{split}
\end{equation}

Which can be reorganized:

\begin{equation}
    \label{eq:DelataGP}
    \begin{split}
    \Delta G(P) =  & P \log(P) + (1-P) \log(1-P) + \\
                & P \sum_k p_k(A) \log(p_k(A)) + (1-P) \sum_k p_k(B) \log(p_k(B)) - \\
                & \sum_k (P \ p_k(A) + (1-P) p_k(B)) \log(P \ p_k(A) + (1-P) p_k(B))
    \end{split}
\end{equation}

\begin{remark}[Connection with Conditional Entropy]

Introducing the random variables: $\Pi$, representing the randomly chosen scenarios, and $X$ representing the data, i.e. the number of $\bb$-s in the sample, the growth rate difference can be expressed as:

\begin{equation}
    \Delta G = - H(\Pi) - H(X|\Pi) + H(X)  
\end{equation}

\begin{equation}
    \Delta G = - H(\Pi|X)
\end{equation}

i.e. the conditional Shannon entropy \cite{book:InformationTheory} of the scenarios, given the observed data.
(where $H(X)=H(\underline{p})=-\sum_k p_k \log(p_k)$ stands for the Shannon entropy of a discrete random variable $X$ characterized by a probability distribution $\underline{p}$.)

\end{remark}

In general, the equilibrium value $P^*$, where $\Delta G(P)$ takes its minimum, can not be expressed in closed form using elementary functions \cite{paper:ElementaryFunctions}.

However to show that a unique $P^*\in(0,1)$ exists, one can observe that: $\Delta G(0) = \Delta G(1) = 0$, and directly calculate the second derivative of $\Delta G(P)$ respect to $P$:

\begin{equation}
\label{eq:DeltaGPrime(P)}
    \begin{split}
        \Delta G'(P) =& \log(P) - \log(1-P) - \Delta H + \\
        & - \sum_k (p_k(A)-p_k(B)) \log(P \ p_k(A) + (1-P) p_k(B))
    \end{split}
\end{equation}

where:

\begin{equation}
    \Delta H = H_A - H_B
\end{equation}

\begin{equation}
    H_A = - \sum_k p_k(A) \log(p_k(A)), \quad H_B = - \sum_k p_k(B) \log(p_k(B))
\end{equation}

and

\begin{equation}
    \Delta G''(P) = \frac{1}{P} + \frac{1}{1-P} - \sum_k \frac{(p_k(A)-p_k(B))^2}{P p_k(A) + (1-P) p_k(B)}
\end{equation}

\begin{equation}
    \Delta G''(P) = \sum_k  \left ( \frac{p_k(A)}{P} + \frac{p_k(B)}{1-P} - \frac{(p_k(A)-p_k(B))^2}{P p_k(A) + (1-P) p_k(B)} \right )
\end{equation}

\begin{equation}
    \Delta G''(P) = \sum_k  \frac{p_k(A) p_k(B)}{P(1-P)(P p_k(A) + (1-P) p_k(B))}
\end{equation}

The second derivative is positive for all $P \in (0,1)$, meaning that this is a convex optimization problem \cite{book:ConvexOpt}, which has one unique solution $P^* \in (0,1)$.

$P^*$ can be found numerically, for instance, using the Bisection method \cite{book:NumericalRecipes} on $\Delta G'(P)$, resulting in a simple and robust controlled approximation.

\subsubsection{Main theorem on Bayesian games}

\begin{theorem}[Bayesian equilibrium]
\label{thm:Bayesian}
\BG{N,K_A,K_B,M} has a unique Nash equilibrium, in which:

\begin{itemize}
    \item \PII/ chooses scenario A or B with probability $P^*$ and $1-P^*$;
    \begin{itemize}
        \item then picks an allowed sequence with equal probability (from $\mathcal{K}_A$ or $\mathcal{K}_B$).
    \end{itemize}

    \item \PI/ first samples uniformly $N$ bits from the provided sequence. Based on $k$ -- the number of $\bb$-s -- she determines $p'^*_k \in [0,1]$, and bets in the following way:
    \begin{itemize}
        \item places her capitals $p'^*_k$ portion to A
        \item places her capitals $1-p'^*_k$ portion to B
    \end{itemize}
\end{itemize}

The parameters $(P^*, \{p'^*_k\})$ can be determined from the parameters of the game $(N, K_A, K_B, M)$:

\begin{equation}
    \label{thm:BayesEqHypergeom}
    p_k(A) = \frac{\binom{K_A}{k} \binom{M-K_A}{N-k}}{\binom{M}{N}}, \quad
    p_k(B) = \frac{\binom{K_B}{k} \binom{M-K_B}{N-k}}{\binom{M}{N}}
\end{equation}

\begin{equation}
    p'^*_k = \frac{P^* \ p_k(A)}{P^* \ p_k(A) + (1-P^*) \ p_k(B)}
\end{equation}

while $P^*$ is the unique minimum of the growth rate difference:

\begin{equation}
    \begin{split}
        \Delta G(P)=&P \ \sum_{k \in \mathbb{K}_A} p_k(A) \log \left ( \frac{P \ p_k(A)}{P \ p_k(A)+(1-P) \ p_k(B)} \right ) + \\
                    &(1-P) \ \sum_{k \in \mathbb{K}_B} p_k(B) \log \left ( \frac{(1-P) \ p_k(B)}{P \  p_k(A)+(1-P) \ p_k(B)} \right )
    \end{split}
\end{equation}

\end{theorem}

For the proof, see Appendix \ref{proof:Bayesian}.

\subsection{Numerical approximations of $P^*$}
\label{sec:NumericalMethods}

In this section, we list a few numerical algorithms by which the value of $P^*$ can be approximated.

\subsubsection{Controlled approximations}

\paragraph{Bisection method:}

Directly using the Bisection method \cite{book:NumericalRecipes} for $\Delta G'(P)$ \eqref{eq:DeltaGPrime(P)}.

\paragraph{Interval iteration:}

Applying $F(.)$ (from equation \eqref{eq:Fdef}) iteratively starting from the whole $\mathbb{R}$ number line:

\begin{equation}
    \mathbb{F}_{n+1} = F(\mathbb{F}_n), \quad \mathbb{F}_0 = \mathbb{R}
\end{equation}

Works well if $\mathbb{K}_A=\mathbb{K}_B$, because in this case:

\begin{equation}
\label{eq:F1DKLbounds}
    \mathbb{F}_1 = [-D_{KL}(\underline{p}(A)||\underline{p}(B),D_{KL}(\underline{p}(B)||\underline{p}(A)]
\end{equation}

every iteration provides a shrinking exact interval for $\vartheta^*$, because $F(.)$ is a contraction on the log-odds space.

\begin{remark}
    If $\mathbb{K}_A \ne \mathbb{K}_B$, then a modified variable and contraction can be introduced:

    \begin{equation}
        Z_A = \sum_{k \in \mathbb{K}_{AB}} p_k(A), \quad Z_B = \sum_{k \in \mathbb{K}_{AB}} p_k(B)
    \end{equation}

    \begin{equation}
        \chi = \sqrt{\frac{Z_A}{Z_B}} \log(P) - \sqrt{\frac{Z_B}{Z_A}} \log(1-P)
    \end{equation}

    For which a modified iteration can be defined:

    \begin{equation}
        \invbreve{F}(\chi) = \frac{1}{\sqrt{Z_A Z_B}} 
        \left (
        \Delta \invbreve{H} +
        \sum_{k \in \mathbb{K}_{AB}} (p_k(A)-p_k(B)) \log(P(\chi) p_k(A) + (1-P(\chi) p_k(B)))
        \right )
    \end{equation}

    \begin{equation}
        \Delta \invbreve{H} = \invbreve{H}_A - \invbreve{H}_B, \quad
        \invbreve{H}_\theta = - \sum_{k\in \mathbb{K}_{AB}} p_k(\theta) \log(p_k(\theta))
    \end{equation}

    \begin{equation}
        \chi^* = \invbreve{F}(\chi^*)
    \end{equation}

    \begin{equation}
        \invbreve{F}(\mathbb{R}) = 
        [-\invbreve{D}_{KL}(\underline{p}(A)||\underline{p}(B)),
        \invbreve{D}_{KL}(\underline{p}(B)||\underline{p}(A))]
    \end{equation}

    \begin{equation}
        \invbreve{D}_{KL}(\underline{p}(A)||\underline{p}(B)) = 
        \frac{
        \sum_{k\in \mathbb{K}_{AB}} p_k(A)
        \log \left ( 
        \frac{p_k(A)}{p_k(B)}
        \right )
        }{\sqrt{\sum_{k \in \mathbb{K}_{AB}} p_k(A)}
        \sqrt{\sum_{k \in \mathbb{K}_{AB}} p_k(B)}}
    \end{equation}

\end{remark}

\begin{conjecture}
\label{conj:Restricted}

    $\invbreve{F}(\chi)$ is a contraction:

    \begin{equation}
        |\invbreve{F}'(\chi)| \le \invbreve{q}, \quad \invbreve{q}<1
    \end{equation}
    
\end{conjecture}

See the proof attempt in appendix \ref{appendix:Restricted}.

\paragraph{Fixed point iteration with error estimation:}

Standard iteration of $F(.)$ starting from zero:

\begin{equation}
    \vartheta_{n+1} = F(\vartheta_n), \quad \vartheta_0 = 0
\end{equation}

and estimating the error \cite{book:BanachFixedPoint}:

\begin{equation}
    |\vartheta^* - \vartheta_n| \le \frac{|\vartheta_n-\vartheta_{n-1}|}{1-q}
\end{equation}

where equation \eqref{eq:qdef} gives an upper bound for $q$.

\paragraph{Controlled Newton–Raphson method:}

Newton–Raphson method \cite{book:NumericalRecipes} for $\Delta(\vartheta) = \vartheta - F(\vartheta)$, for which the error can be estimated by knowing that 

\begin{equation}
    |F''(\theta)| \le \frac{1}{2} q = Q
\end{equation}

The iteration has the following explicit form:

\begin{equation}
    \vartheta_{n+1} = \frac{F(\vartheta_n) - \vartheta_n F'(\vartheta_n)}{1-F'(\vartheta_n)}
\end{equation}

If we reach a point where $2 Q |\vartheta_n-F(\vartheta_n)| \le \left ( 1- F'(\vartheta_n)\right )^2$, then the following bounds hold for the exact fixed point $\vartheta^*$:

\begin{equation}
    \vartheta^* - \vartheta_n \ge \frac{1-F'(\vartheta_n)}{Q} \left ( \sqrt{1- 2 Q \frac{|\vartheta_n-F(\vartheta_n)|}{\left ( 1- F'(\vartheta_n)\right )^2}} 
    -1
    \right ) 
\end{equation}

\begin{equation}
    \vartheta^* - \vartheta_n \le \frac{1-F'(\vartheta_n)}{Q} \left ( \sqrt{1+ 2 Q \frac{|\vartheta_n-F(\vartheta_n)|}{\left ( 1- F'(\vartheta_n)\right )^2}} 
    -1
    \right ) 
\end{equation}

\subsubsection{Monte Carlo Methods}

\paragraph{MCMC:} Markov chain Monte Carlo (MCMC) \cite{book:MCMC} can be combined with Blahut–Arimoto algorithm \cite{paper:MCMCBlahutArimoto,arxiv:MCMCBlahutArimoto}, which with some modifications might yield a general and scalable stochastic approximation for the equilibrium prior.

\subsection{Examples and Visualization}

In this section, we visualize the equilibrium values for various Bayesian games based on Theorem~\ref{thm:Bayesian}.
The numerical results are provided by 
the controlled numerical methods listed in Section~\ref{sec:NumericalMethods}.

\begin{figure}[H]
    \centering
    \includegraphics[scale=0.8]{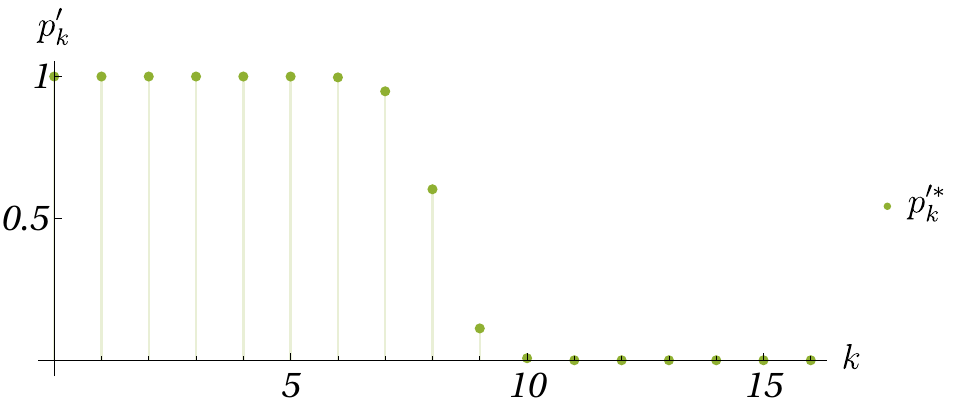}
    \caption{Illustration of $p'^*_k$ for \BG{N=17,K_A=10,K_B=16,M=27}, in which case $P^*\approx 0.4953$. (The game has the same parameters as the illustration in figure \ref{fig:pkAB}.)}
    \label{fig:ppkAB}
\end{figure}

\paragraph{Prior and splitting ratio plots:}

For aesthetic reasons, we will allow Bayesian games with $K_A \ge K_B$, with the following convention \footnote{$P^*_o$, $p'^*_{o,k}$ and $G^*_o = \Delta G^*_o + \log(2)$ stands for the original quantities described in the previous sections (where we assumed that $K_A<K_B$)}:

\begin{equation}
\label{eq:BPKAKBextended}
    P^*(K_A,K_B) =
        \begin{cases}
        \pi(A) = P^*_o(K_A,K_B) & \text{if } K_A < K_B \\
        1/2                     & \text{if } K_A = K_B \\
        \pi(B) = P^*_o(K_B,K_A) & \text{if } K_A > K_B
    \end{cases}
\end{equation}

\begin{equation}
    p'^*_k(K_A,K_B) =
        \begin{cases}
        p'^*_{o,k}(K_A,K_B) & \text{if } K_A < K_B \\
        1/2              & \text{if } K_A = K_B \\
        p'^*_{o,k}(K_B,K_A) & \text{if } K_A > K_B
    \end{cases}
\end{equation}

\begin{equation}
    G^*(K_A,K_B) =
        \begin{cases}
        G^*_o(K_A,K_B)  & \text{if } K_A < K_B \\
        0                 & \text{if } K_A = K_B \\
        G^*_o(K_B,K_A)  & \text{if } K_A > K_B
    \end{cases}
\end{equation}

\begin{figure}[H]
    \centering
    \begin{subfigure}[b]{0.3\textwidth}
        \includegraphics[width=\textwidth]{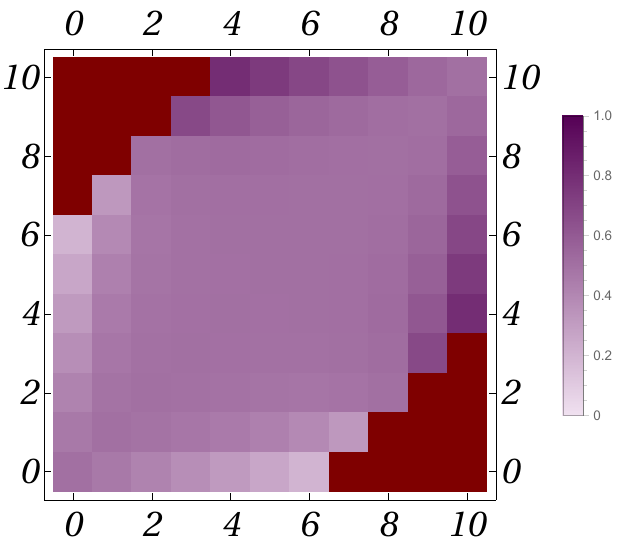}
        \caption{$P^*(K_A,K_B)$}
        \label{fig:BGame4_10_P}
    \end{subfigure}
    \hfill 
    \begin{subfigure}[b]{0.3\textwidth}
        \includegraphics[width=\textwidth]{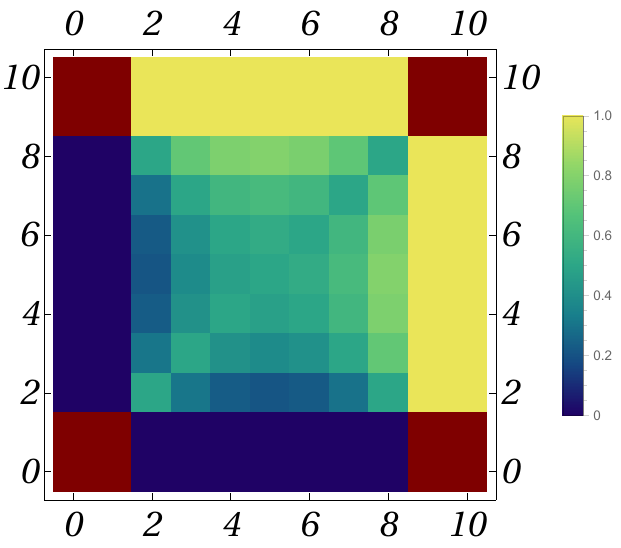}
        \caption{$p'^*_{k=2}(K_A,K_B)$}
        \label{fig:BGame4__10_ppk_2}
    \end{subfigure}
    \hfill 
    \begin{subfigure}[b]{0.3\textwidth}
        \includegraphics[width=\textwidth]{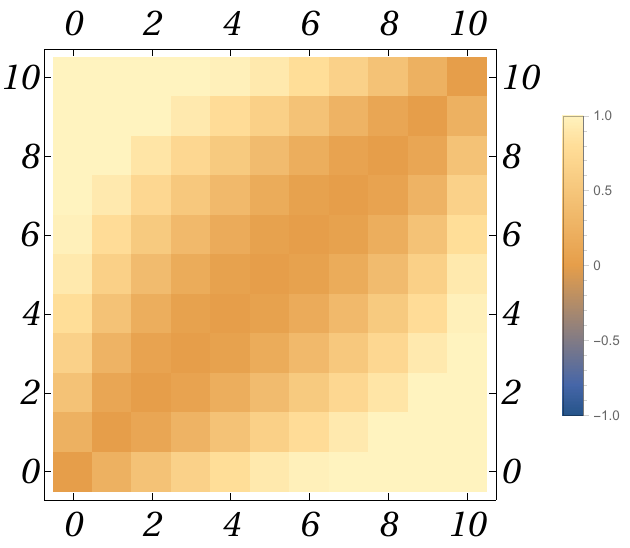}
        \caption{$G^*(K_A,K_B)/\log(2)$}
        \label{fig:BGame4__10_G}
    \end{subfigure}
    
    \caption{\BG{N=4,K_A,K_B,M=10}}
    \label{fig:BGame4__10_PppkG}
\end{figure}

\begin{figure}[H]
    \centering
    \begin{subfigure}[b]{0.185\textwidth}
        \includegraphics[width=\textwidth]{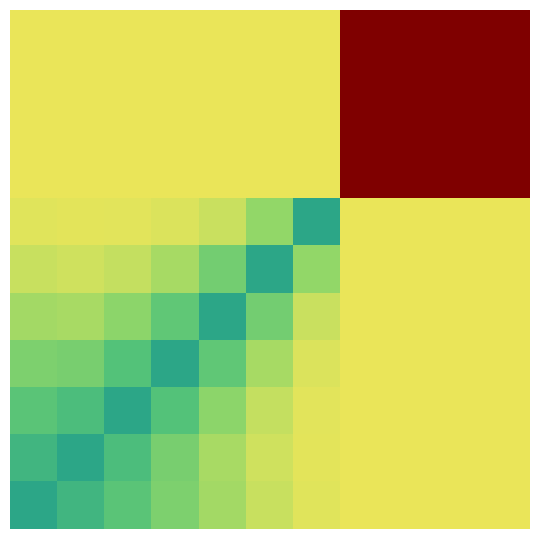}
        \caption{$k=0$}
        \label{fig:ppkBG_4_0}
    \end{subfigure}
    \hspace{0.00\textwidth} 
    \begin{subfigure}[b]{0.185\textwidth}
        \includegraphics[width=\textwidth]{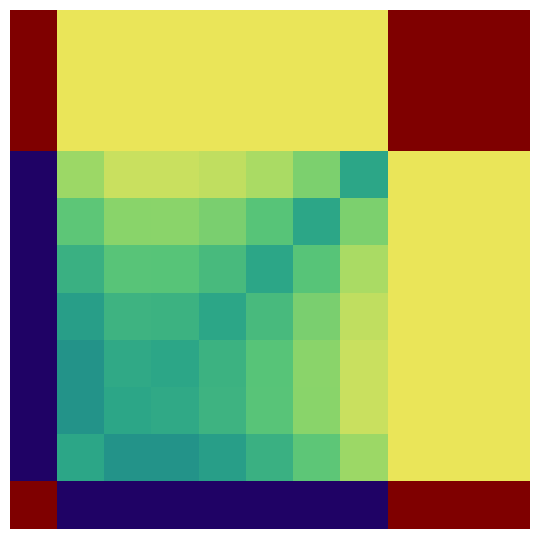}
        \caption{$k=1$}
        \label{fig:ppkBG_4_1}
    \end{subfigure}
    \hspace{0.00\textwidth} 
    \begin{subfigure}[b]{0.185\textwidth}
        \includegraphics[width=\textwidth]{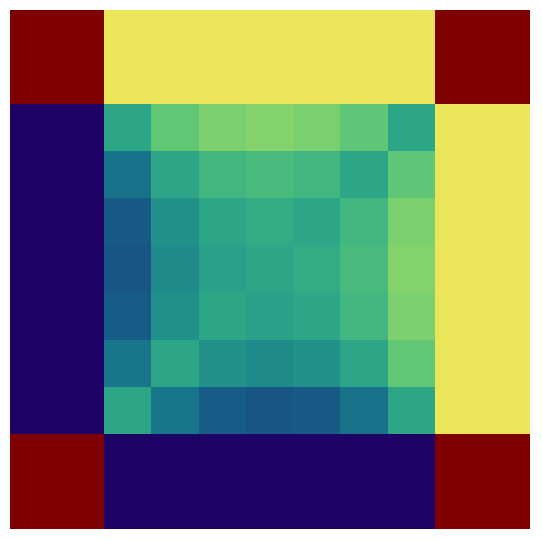}
        \caption{$k=2$}
        \label{fig:ppkBG_4_2}
    \end{subfigure}
    \hspace{0.00\textwidth} 
        \begin{subfigure}[b]{0.185\textwidth}
        \includegraphics[width=\textwidth]{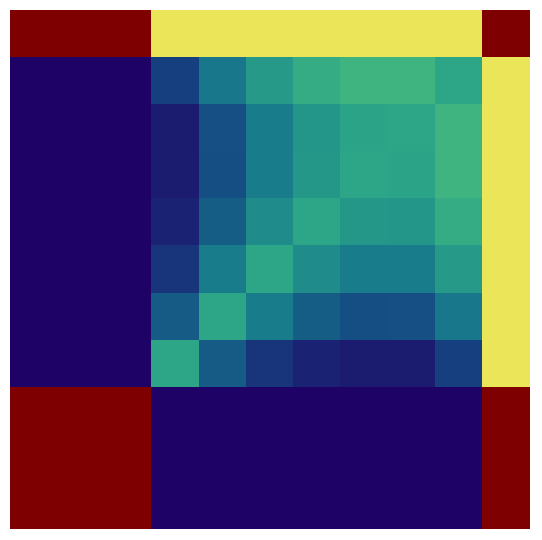}
        \caption{$k=3$}
        \label{fig:ppkBG_4_3}
    \end{subfigure}
    \hspace{0.00\textwidth} 
        \begin{subfigure}[b]{0.185\textwidth}
        \includegraphics[width=\textwidth]{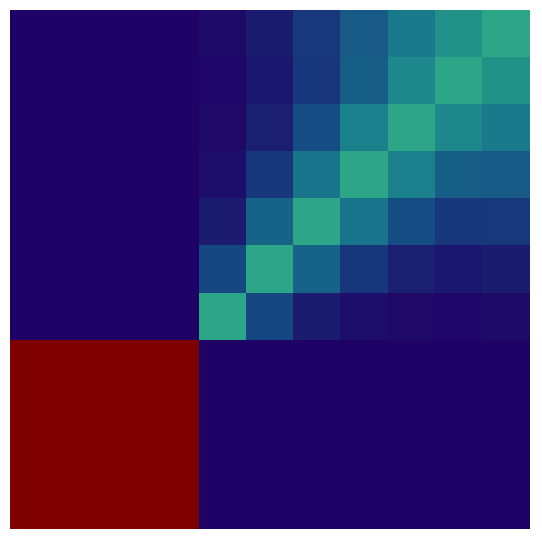}
        \caption{$k=4$}
        \label{fig:ppkBG_4_4}
    \end{subfigure}

    \caption{$p'^*_k(K_A,K_B)$ for \BG{4,K_A,K_B}.
    (Axes and the colour scale have the same meaning as in figure \ref{fig:BGame4__10_ppk_2}.)
    }
    \label{fig:ppk_InfBG_4_01234}
\end{figure}

\begin{remark}
    We do not associate prior values for Bayesian games, where \PI/ can always guess the scenario correctly (while splitting ratios and growth factors can be well-defined for these Sure winning games).
    For some Bayesian games, $p'^*_k$ can be undefined if there can be no $k$ number of $\bb$-s in the sampled bits. If a quantity has no well-defined value for a given pair of parameters, then we indicated this with a dark red region on the plots.

\end{remark}

\begin{figure}[H]
    \centering
    \begin{subfigure}[b]{0.3\textwidth}
        \includegraphics[width=\textwidth]{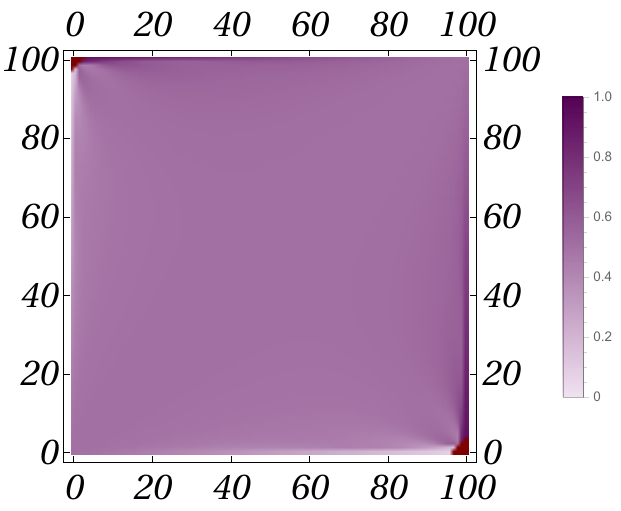}
        \caption{$P^*(K_A,K_B)$}
        \label{fig:BGame4__100_P}
    \end{subfigure}
    \hfill 
    \begin{subfigure}[b]{0.3\textwidth}
        \includegraphics[width=\textwidth]{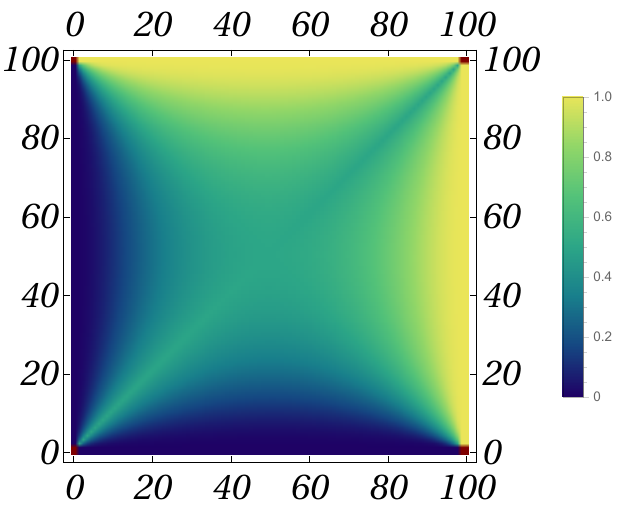}
        \caption{$p'^*_{k=2}(K_A,K_B)$}
        \label{fig:BGame4__100_ppk_2}
    \end{subfigure}
    \hfill 
    \begin{subfigure}[b]{0.3\textwidth}
        \includegraphics[width=\textwidth]{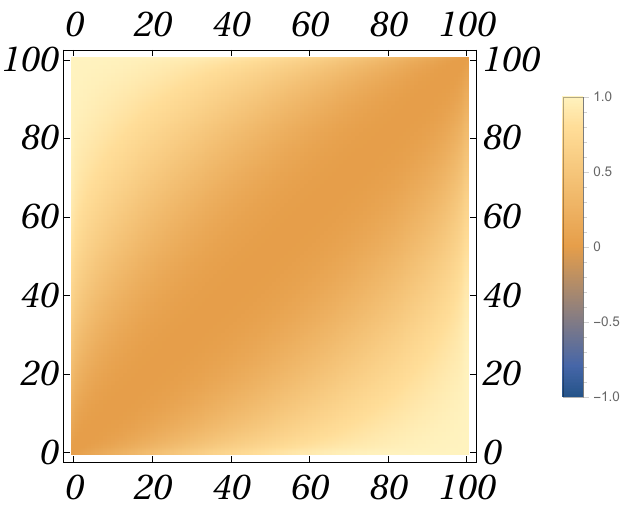}
        \caption{$G^*(K_A,K_B)/\log(2)$}
        \label{fig:BGame4__100_G}
    \end{subfigure}
    
    \caption{\BG{N=4,K_A,K_B,M=100}}
    \label{fig:BGame4__100_PppkG}
\end{figure}

\begin{figure}[H]
    \centering
    \begin{subfigure}[b]{0.24\textwidth}
        \includegraphics[width=\textwidth]{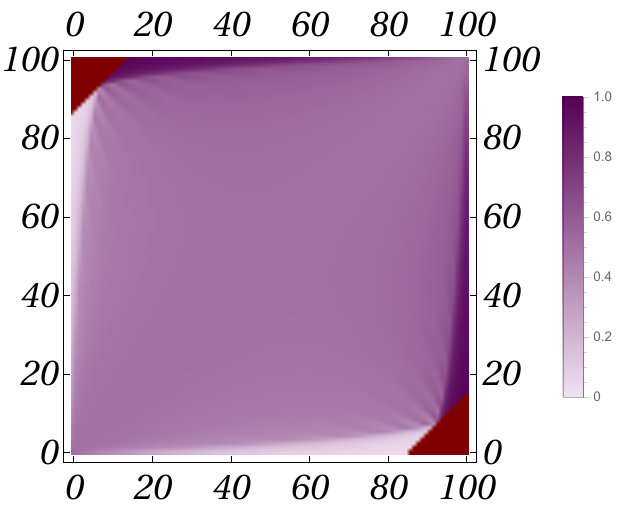}
        \caption{$P^*(K_A,K_B)$}
        \label{fig:BGame15__100_P}
    \end{subfigure}
    \hfill 
    \begin{subfigure}[b]{0.24\textwidth}
        \includegraphics[width=\textwidth]{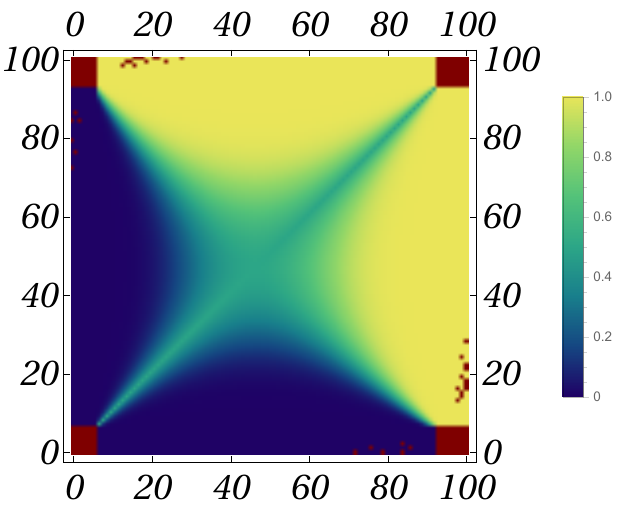}
        \caption{$p'^*_{k=7}(K_A,K_B)$}
        \label{fig:BGame15__100_ppk_7}
    \end{subfigure}
    \hfill 
    \begin{subfigure}[b]{0.24\textwidth}
        \includegraphics[width=\textwidth]{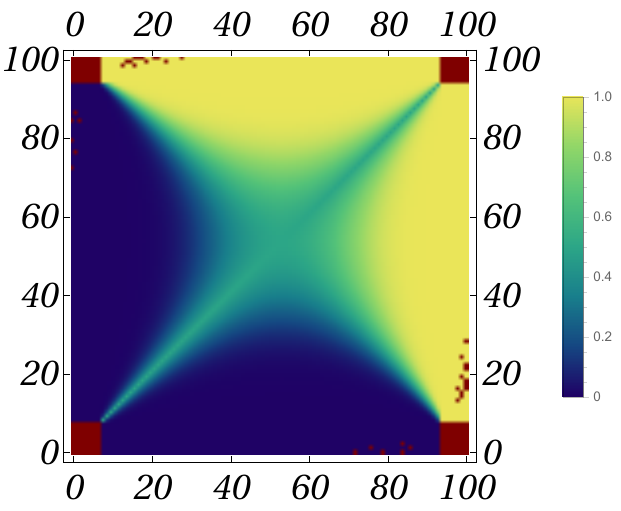}
        \caption{$p'^*_{k=8}(K_A,K_B)$}
        \label{fig:BGame15__100_ppk_8}
    \end{subfigure}
    \hfill 
    \begin{subfigure}[b]{0.24\textwidth}
        \includegraphics[width=\textwidth]{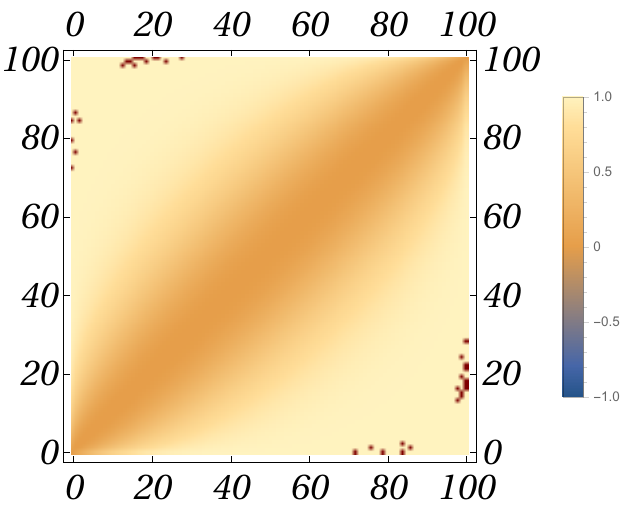}
        \caption{$G^*(K_A,K_B)/\log(2)$}
        \label{fig:BGame15__100_G}
    \end{subfigure}
    
    \caption[]{\BG{N=15,K_A,K_B,M=100}\footnotemark}
    \label{fig:BGame15__100_PppkG}
\end{figure}

\footnotetext{The freckles seen on \ref{fig:BGame15__100_ppk_7}, \ref{fig:BGame15__100_ppk_8} and \ref{fig:BGame15__100_G} are caused by numerical artifact, signalling the limits of the performed numerical calculations.}

Mainly for aesthetic reasons, we show the complete set of splitting ratios for Bayesian games, where 6-long bits can be chosen, \BG{N,K_A,K_B,M=6}.
This ``policy pyramid'' can be seen in figure \ref{fig:BayesianStrategyPyramid_6}.

\begin{figure}[H]
    \centering
    \includegraphics[width=10 cm]{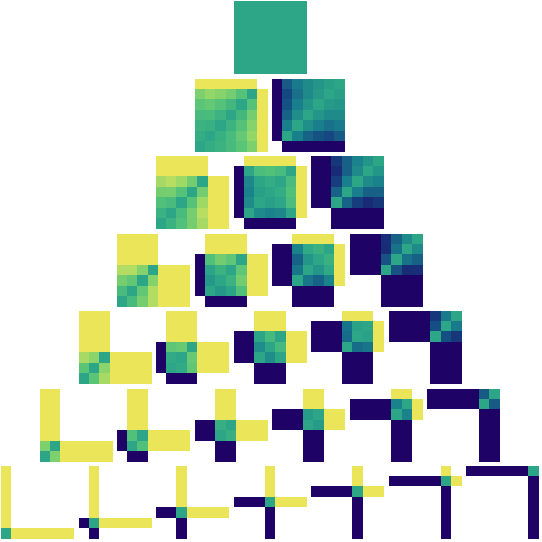}
    \caption{
    ``Policy pyramid'' for Bayesian games \BG{N,K_A,K_B,M=6}.
    On the top $p'^*_{k=0}(K_A,K_B)$ values are presented for \BG{N=0,K_A,K_B,M=6} games.
    In each row \BG{N=0,K_A,K_B,M=6}, \BG{N=1,K_A,K_B,M=6}, \dots, \BG{N=6,K_A,K_B,M=6} games are listed,
    while in each row $p'^*_{k=0}(K_A,K_B)$, $p'^*_{k=1}(K_A,K_B)$, \dots, $p'^*_{k=N}(K_A,K_B)$ policy plots are shown.
    (Axes and the colour scale have the same meaning as in figures \ref{fig:BGame4__10_ppk_2}, \ref{fig:BGame4__100_ppk_2}, \ref{fig:BGame15__100_ppk_7}, \ref{fig:BGame15__100_ppk_8}. The only difference is that missing values are represented by white colouring instead of dark red.)
    }
    \label{fig:BayesianStrategyPyramid_6}
\end{figure}

\subsection{Interpretation of the results}

The construction and the results are in some ways, similar to the Reference prior \cite{paper:ReferencePrior,book:Bernardo}. However, we will see that in the $N\to\infty$ limit, $P^*_N$ seems to converge to a different limit prior than the Reference prior.
    This is because of the slight difference in the objective of the optimization problem.
    \begin{itemize}
        \item The Reference prior is related to the mutual information $I(\Pi;X)$,
        \item while the equilibrium prior is related to conditional entropy $H(\Pi|X)$. 
    \end{itemize}

\begin{remark}
    The ``Prior'' depends on the game, $P^*_N \ne P^*_{N+1}$ even if other parameters of the game ($K_A,K_B,M$) are kept the same.
    
\end{remark}

\begin{remark}
    Exchangeability (central to de Finetti \cite{book:deFinetti}) is emerging by the uniform randomization of sequences and sampling, both by \PI/ and \PII/.
    
\end{remark}

\section{Limiting cases}

\subsection{Binomial games}

\subsubsection{Binomial Fisher games}

There is an important limit case, when $M \to \infty$, while $K_A/M \to x_A \in (0,1)$ and $K_B/M \to x_B \in (0,1)$ \footnote{$x_A, x_B$ are similar fractions as mole fractions in \href{https://goldbook.iupac.org/terms/view/A00296}{chemistry} \cite{book:ChemistryGoldBook}}.

This sequence of games will be denoted by:

\begin{equation}
    \lim_{\begin{smallmatrix} M \to \infty & \\ K_A/M \to x_A \\ K_B/M \to x_B \end{smallmatrix}} \textswab{Game}(N, K_A, K_B, M) = \overline{\textswab{Game}}(N, x_A, x_B)
\end{equation}

This limit is hopefully intuitive; however, because of the importance of the construction, it might be useful to specify the meaning of this triple limit.

In technical terms, \InfG{N, x_A, x_B} is the equivalence class of ``convergent'' infinite sequences of the form \footnote{the construction is in some way parallel to the construction of real numbers as a collection of equivalent Cauchy sequences of rational numbers $\mathbb{R} = \overline{\mathbb{Q}}$ \cite{book:TerenceTaoAnalysis}}:

\begin{equation}
    \textswab{G}_i = \textswab{Game}(N, K_{A,i}, K_{B,i}, M_i)
\end{equation}

\begin{equation}
    \begin{split}
        \{\textswab{G}_i\}_{i=1}^\infty \in \overline{\textswab{Game}}(N, x_A, x_B) \iff & \lim_{i\to \infty} M_i = \infty \ \wedge \\ 
        & \lim_{i\to \infty} K_{A,i}/M_i = x_A \ \wedge \\
        & \lim_{i\to \infty} K_{B,i}/M_i = x_B
    \end{split}
\end{equation}

One particular choice of this kind of sequence might be:

\begin{equation}
    \{ \textswab{Game}(N, [ x_A \ i ], [ x_B \ i ] ,i) \}_{i=1}^\infty \in  \overline{\textswab{Game}}(N, x_A, x_B)
\end{equation}

Or if both parameters are rational: $x_A = a_A/b_A$, $x_B = a_B/b_B$, $a_\theta, b_\theta \in \mathbb{N}$

\begin{equation}
    \{ \textswab{Game}(N, i \ (a_A \ b_B), i \ (a_B \ b_A) ,  i \ (b_A \ b_B)) \}_{i=1}^\infty \in  \overline{\textswab{Game}}(N, x_A, x_B)
\end{equation}

\begin{remark}
    The concept of simultaneous approximation of two real numbers ($x_A,x_B$) with rationals having the same denominator appears in the Littlewood Conjecture \cite{paper:LittlewoodConjecture} 
    
\end{remark}

Intuitively, one can expect that in this sequence, some aspects of the symmetric solution will ``diverge'', while other aspects will ``converge'' to a specific strategy.

Concretely, for \PII/, the number of sequences for a given scenario undergo a combinatorial explosion, while her mixing between scenarios might converge to a definite distribution.

For \PI/, the sampling part of the action ``diverges'' (undergoes a combinatorial explosion), while her policy will ``converge'' to a definite policy. \footnote{To define a proper concept of convergence, we need to define topology or a metric on the space of policies. Without providing a general definition, in our case we will combine $(k^*,\nu^*)$ to one single variable $s^* = (k^* + \nu^*)/(N+1) \in [0,1]$, and associate the standard topology and metric to this interval.}

\begin{theorem}[Symmetrical equilibrium of Binomial Fisher games]
\label{thm:BinomialSymmetric}
\InfG{N,x_A,x_B}, where $x_A, x_B \in (0,1)$, $N \in \mathbb{N}$, has converging policy for \PI/:

\begin{itemize}
    \item \PI/ samples $N$ bits uniformly from the provided (infinite) sequence. Based on $k$ -- the number of $\bb$-s -- she performs the following action:
    \begin{itemize}
        \item if $k<k^*$ she guesses A
        \item if $k=k^*$ then 
        \begin{itemize}
            \item she guesses A with probability $\nu^*$ or B with probability $1-\nu^*$
        \end{itemize} 
        \item if $k>k^*$ she guesses B
    \end{itemize}
\end{itemize}

\PII/ might have a convergent policy:

\begin{itemize}
    \item \PII/'s probability of choosing scenario A or B converges to $P^*$ and $1-P^*$ if $(x_A,x_B) \notin \mathbb{S}_N$, and has an accumulation point in the interval $[\underline{P}^*,\overline{P}^*]$ if $(x_A,x_B) \in \mathbb{S}_N$;
    \begin{itemize}
        \item after choosing a scenario, she picks a sequence uniformly from the allowed sequences for scenario A or B.
    \end{itemize}
\end{itemize}

The parameters $(k^*, P^*, \nu^*)$ or $(k^*, \underline{P}^*, \overline{P}^*, \nu^*)$ can be determined from the parameters of the game $(N, x_A, x_B)$:

\begin{equation}
    \label{eq:BinomialDistribution}
    p_k(A) = \binom{N}{k} x_A^k (1-x_A)^{N-k}, \quad
    p_k(B) = \binom{N}{k} x_B^k (1-x_B)^{N-k}
\end{equation}

\begin{equation}
    \label{thm:SymBinEqNu}
    \nu^* = \frac{\sum_{k \ge k^*} p_k(B) - \sum_{k < k^*} p_k(A)}{p_{k^*}(A)+p_{k^*}(B)}
\end{equation}

$k^*$ is the smallest integer, for which the sum of probabilities becomes greater than $1$:

\begin{equation}
    \label{thm:SymBinEqk}
    \sum_{k \le k^*} p_k(A)+p_k(B) > 1, \quad \mathrm{while} \quad \sum_{k<k^*} p_k(A)+p_k(B) \le 1
\end{equation}

The definition of $\mathbb{S}_N$ is the following:

\begin{equation}
    \label{eq:defOfS_N}
    \mathbb{S}_N = \{(x_A,x_B) \ | \ \nu^*(x_A,x_B) = 0 \}
\end{equation}

\begin{equation}
    \label{eq:SymBinEqP}
    P^* = \frac{p_{k^*}(B)}{p_{k^*}(A)+p_{k^*}(B)}, \quad \text{if } (x_A,x_B) \notin \mathbb{S}_N
\end{equation}

And if $(x_A,x_B) \in \mathbb{S}_N$

\begin{equation}
    \label{eq:SymBinEqPud}
    \underline{P}^* = \frac{p_{k^*-1}(B)}{p_{k^*-1}(A)+p_{k^*-1}(B)}, \quad
    \overline{P}^* = \frac{p_{k^*}(B)}{p_{k^*}(A)+p_{k^*}(B)}
\end{equation}

\end{theorem}

See Appendix \ref{appendix:BinomialGames} for the build-up of the proof, or see the specific proof in \ref{proof:BinomialFisher}.

\begin{remark}
    The function $P^*(x_A,x_B)$ has discontinuities or ``scars'' on $\mathbb{S}_N$. At these points, the function can jump, but any value between the lower and upper bounds $P^* \in [\underline{P}^*(x_A,x_B), \overline{P}^*(x_A,x_B)]$ could be chosen by \PII/. This technically means that for all such $P^*$, the maximal expected winning rate for \PI/ converges to the same $v^*$ value.
    
    Remarkably, this does not impact the convergence of \PI/'s strategy, expressed with $s^*(x_A,x_B)$, which is a continuous function in the limit.

\begin{figure}[H]
    \centering
    \begin{subfigure}[b]{0.4\textwidth}
        \includegraphics[width=\textwidth]{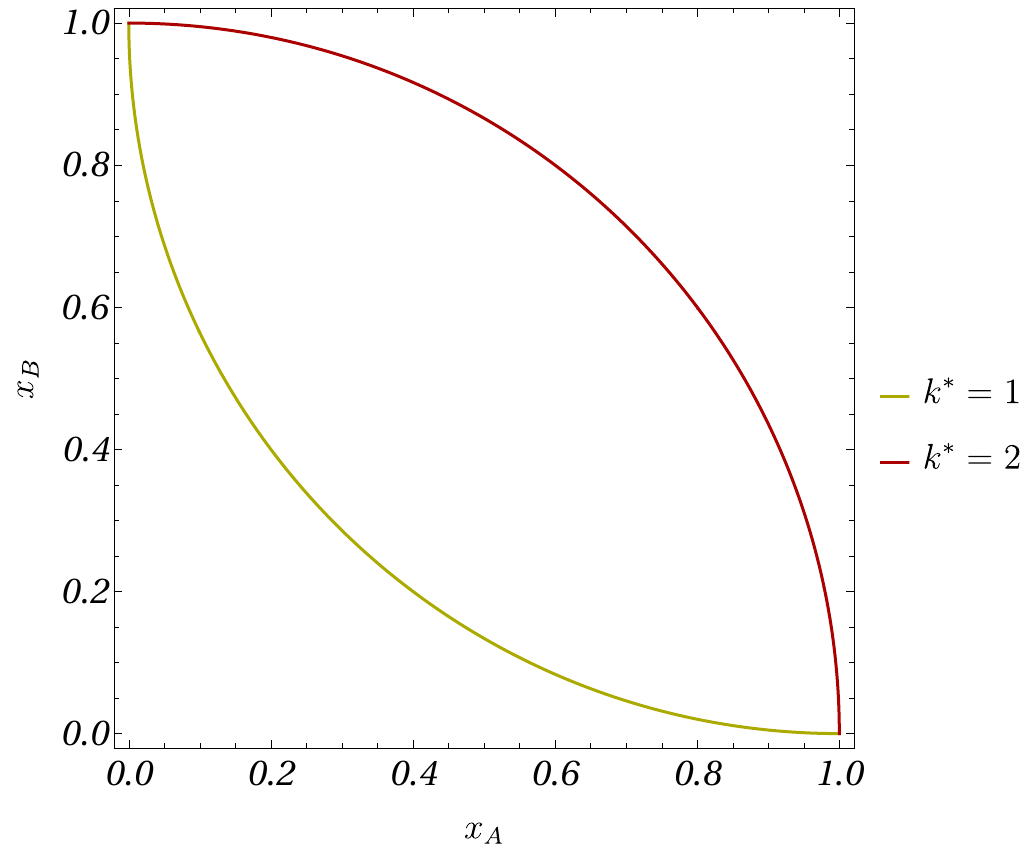}
        \caption{$N=2$}
        \label{fig:Scars_2}
    \end{subfigure}
    \hspace{0.05\textwidth} 
    \begin{subfigure}[b]{0.4\textwidth}
        \includegraphics[width=\textwidth]{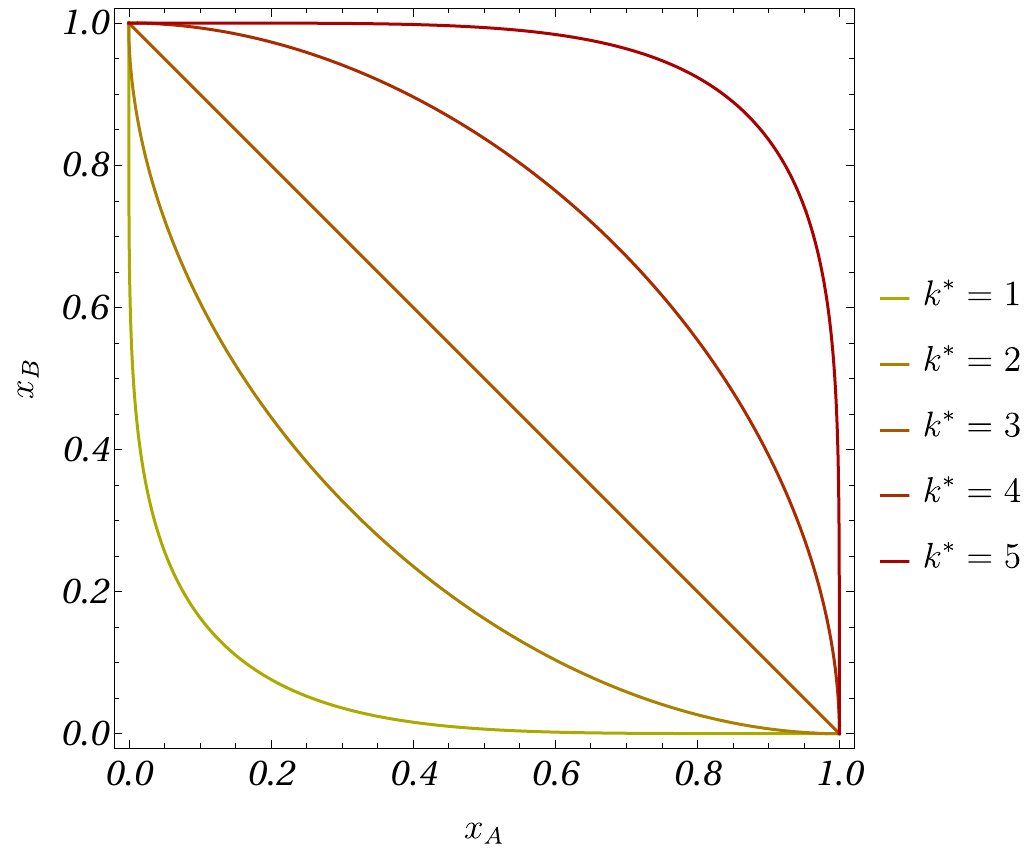}
        \caption{$N=5$}
        \label{fig:Scars_5}
    \end{subfigure}
    
    \caption{Scars, $\mathbb{S}_N \subset [0,1] \times [0,1]$ for $N=2$ and $N=5$.}
    \label{fig:Scars}
\end{figure}

\end{remark}

\begin{remark}
    The equilibrium quantity $k^*$ can be interpreted as the median \cite{book:DeGrootProbabilityAndStatistics,book:IntroToProbability} of a mixture of random variables $\kappa_A, \kappa_B$:

    \begin{equation}
        k^* = \mathbbm{m}
        \left [ \frac{1}{2} \kappa_A+
        \frac{1}{2}\kappa_B
        \right ]
    \end{equation}

    where $\kappa_A$ and $\kappa_B$ are characterized by $p_k(A)$ and $p_k(B)$.
    (The only difference is that such definition would imply the $\nu \in (0,1]$ choice instead of $\nu \in [0,1)$.)
    
\end{remark}

The intermediacy property of the median implies the following bounds for $k^*$:

\begin{theorem}
\label{thm:BinomialKMedianBounds}

    \begin{equation}
        \lfloor x_A N \rfloor
        \le
        \mathbbm{m}[\kappa_A]
        \le
        k^*
        \le
        \mathbbm{m}[\kappa_B]
        \le
        \lceil x_B N \rceil
    \end{equation}

    where $\kappa_A$ and $\kappa_B$ represent random variables with Binomial distribution:

    \begin{equation}
        \kappa_A \sim \mathrm{Binom}(N,x_A), \quad
        \kappa_B \sim \mathrm{Binom}(N,x_B)
    \end{equation}

\end{theorem}

\begin{proof}
    The statement follows from lemma \ref{lemma:IntermediacyMedian}, and bounds obtained for the median of the Binomial distribution in \cite{paper:BinomialMedianMode}:

    \begin{equation}
    \lfloor x N \rfloor 
    \le 
    \mathbbm{m}[\kappa] 
    \le
    \lceil x N \rceil, 
    \quad
    \kappa \sim \mathrm{Binom}(N,x)
    \end{equation}
    
\end{proof}

\subsubsection{Binomial Bayesian games}

One can define Binomial Bayesian games \InfBG{N,x_A,x_B} with the same triple limit,
and state an analogous theorem:

\begin{theorem}[Binomial Bayesian equilibrium]
\label{thm:BayesianBinomial}
\InfBG{N,x_A,x_B}, where $x_A, x_B \in (0,1)$, $N \in \mathbb{N}$, has converging policies:

\begin{itemize}
    \item \PII/ chooses scenario A or B with probability $P^*$ and $1-P^*$;
    \begin{itemize}
        \item then picks a sequence uniformly from the allowed sequences for scenario A or B.
    \end{itemize}

    \item \PI/ samples $N$ bits uniformly from the provided (infinite) sequence. Based on $k$ -- the number of $\bb$-s -- she determines $p'^*_k \in (0,1)$, and bets in the following way:
    \begin{itemize}
        \item places her capitals $p'^*_k$ portion to A
        \item places her capitals $1-p'^*_k$ portion to B
    \end{itemize}
\end{itemize}

The parameters $(P^*, \{p'^*_k\})$ can be determined from the parameters of the game $(N, x_A, x_B)$:

\begin{equation}
    p_k(A) = \binom{N}{k} x_A^k (1-x_A)^{N-k}, \quad
    p_k(B) = \binom{N}{k} x_B^k (1-x_B)^{N-k}
\end{equation}

\begin{equation}
    \label{eq:BinomialBayesianpp}
    p'^*_k = \frac{P^* \ p_k(A)}{P^* \ p_k(A) + (1-P^*) \ p_k(B)}
\end{equation}

while $P^*$ is the unique minimum of the growth rate difference:

\begin{equation}
    \label{eq:BinomialBayesianP}
    \begin{split}
        \Delta G(P)=&P \ \sum_{k} p_k(A) \log \left ( \frac{P \ p_k(A)}{P \ p_k(A)+(1-P) \ p_k(B)} \right ) + \\
                    &(1-P) \ \sum_{k} p_k(B) \log \left ( \frac{(1-P) \ p_k(B)}{P \  p_k(A)+(1-P) \ p_k(B)} \right )
    \end{split}
\end{equation}

\end{theorem}

See Section~\ref{section:BinomialBayesianProof} in Appendix \ref{appendix:BinomialGames} for the build-up of the proof, or see the specific proof at \ref{proof:LimitBayesian}.

\begin{remark}
    All equilibrium quantities are continuous as a function of $x_A,x_B$ for all $N \in \mathbb{N}$,  $0<x_A \le x_B<1$.
    
\end{remark}

\subsection{Examples and Visualization}

\subsubsection{Binomial Fisher games}

\paragraph{Prior plots:} Examples show limiting prior plots $P^*(x_A,x_B)$ in fig. \ref{fig:P_InfG_1_2} and \ref{fig:P_InfG_10_15}.

\begin{figure}[H]
    \centering
    \begin{subfigure}[b]{0.4\textwidth}
        \includegraphics[width=\textwidth]{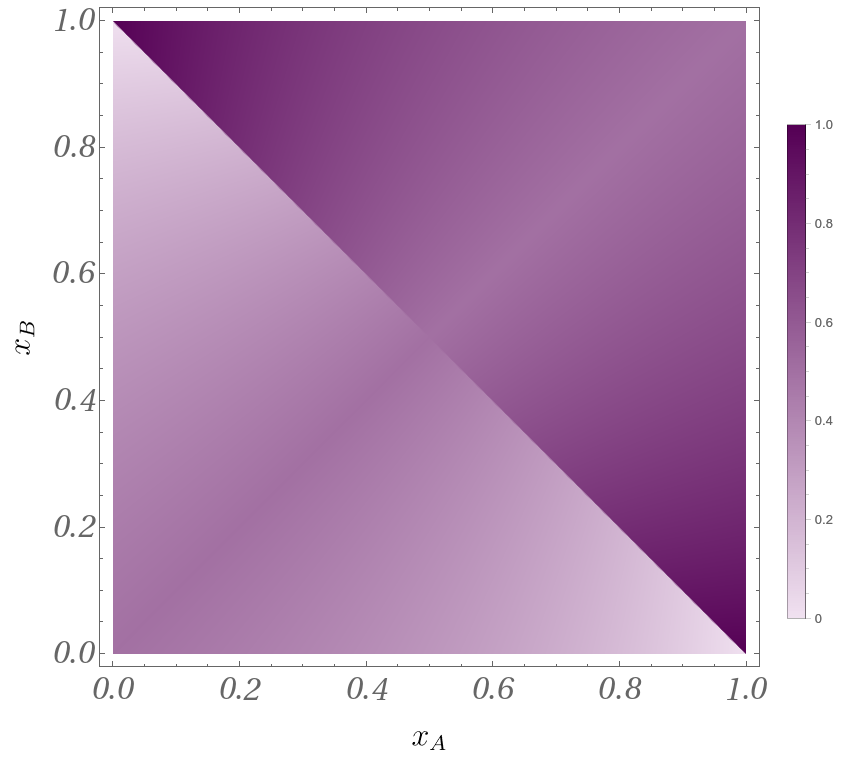}
        \caption{\InfG{1,x_A,x_B}}
        \label{fig:PG_1}
    \end{subfigure}
    \hspace{0.05\textwidth} 
    \begin{subfigure}[b]{0.4\textwidth}
        \includegraphics[width=\textwidth]{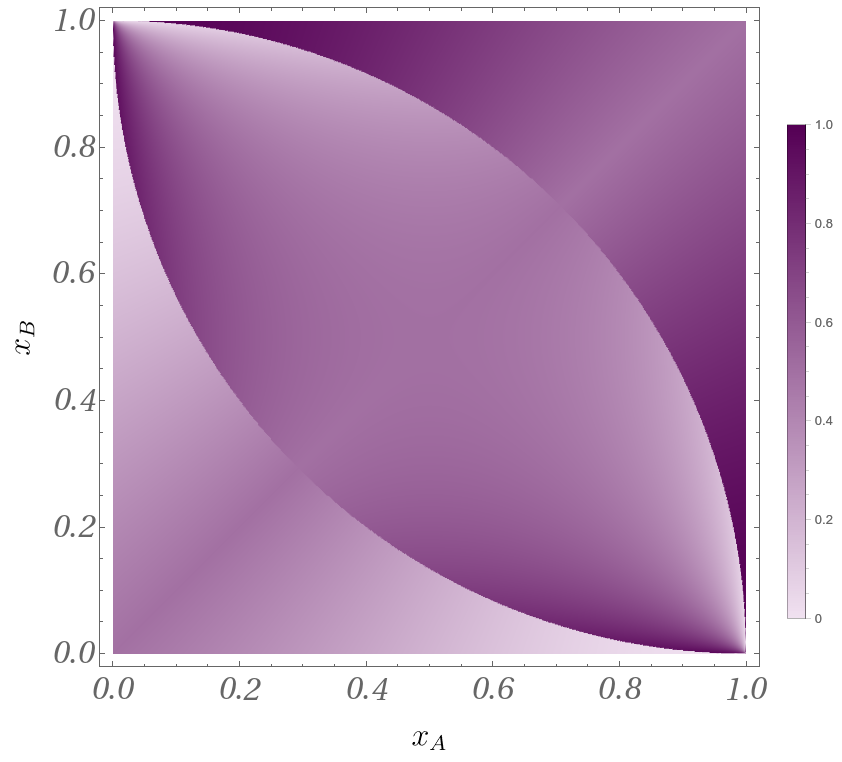}
        \caption{\InfG{2,x_A,x_B}}
        \label{fig:PG_2}
    \end{subfigure}
    
    \caption{$P^*(x_A,x_B)$ for Binomial Fisher games.}
    \label{fig:P_InfG_1_2}
\end{figure}

\begin{figure}[H]
    \centering
    \begin{subfigure}[b]{0.4\textwidth}
        \includegraphics[width=\textwidth]{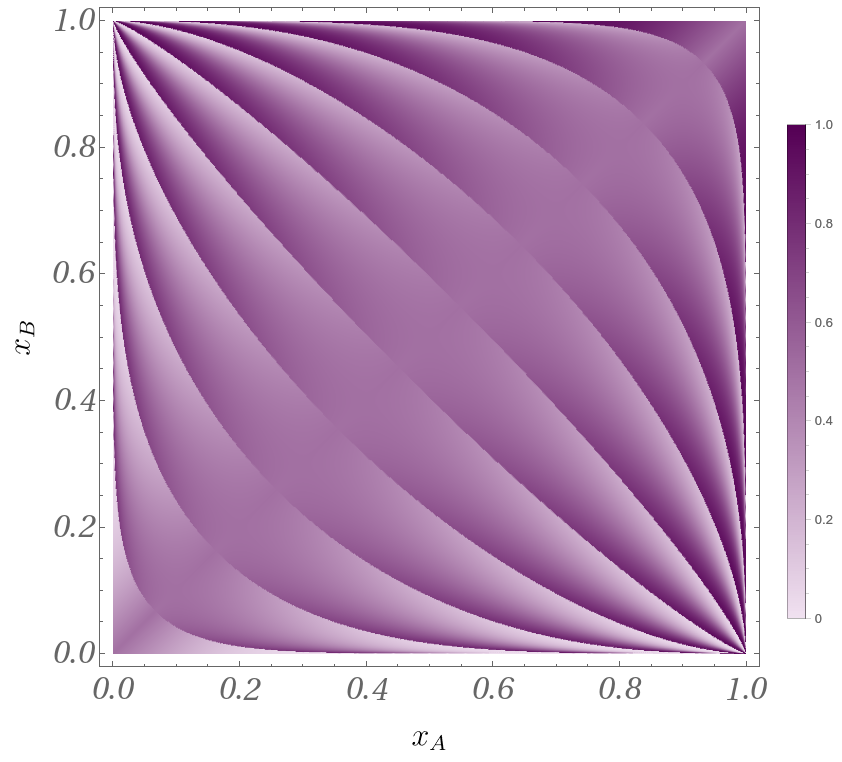}
        \caption{\InfG{10,x_A,x_B}}
        \label{fig:PG_10}
    \end{subfigure}
    \hspace{0.05\textwidth} 
    \begin{subfigure}[b]{0.4\textwidth}
        \includegraphics[width=\textwidth]{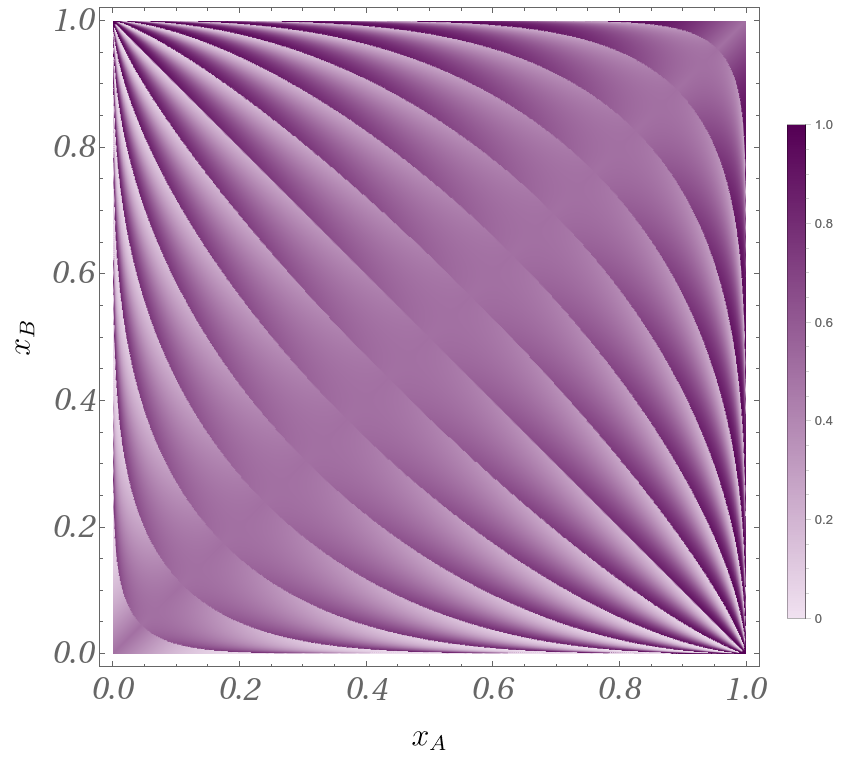}
        \caption{\InfG{15,x_A,x_B}}
        \label{fig:PG_15}
    \end{subfigure}
    
    \caption{$P^*(x_A,x_B)$ for Binomial Fisher games.}
    \label{fig:P_InfG_10_15}
\end{figure}

\paragraph{Policy plots:} Examples show continuous limiting policy plots $s^*(x_A,x_B)$ in fig. \ref{fig:s_InfG_1_2} and \ref{fig:s_InfG_10_15}.

\begin{figure}[H]
    \centering
    \begin{subfigure}[b]{0.4\textwidth}
        \includegraphics[width=\textwidth]{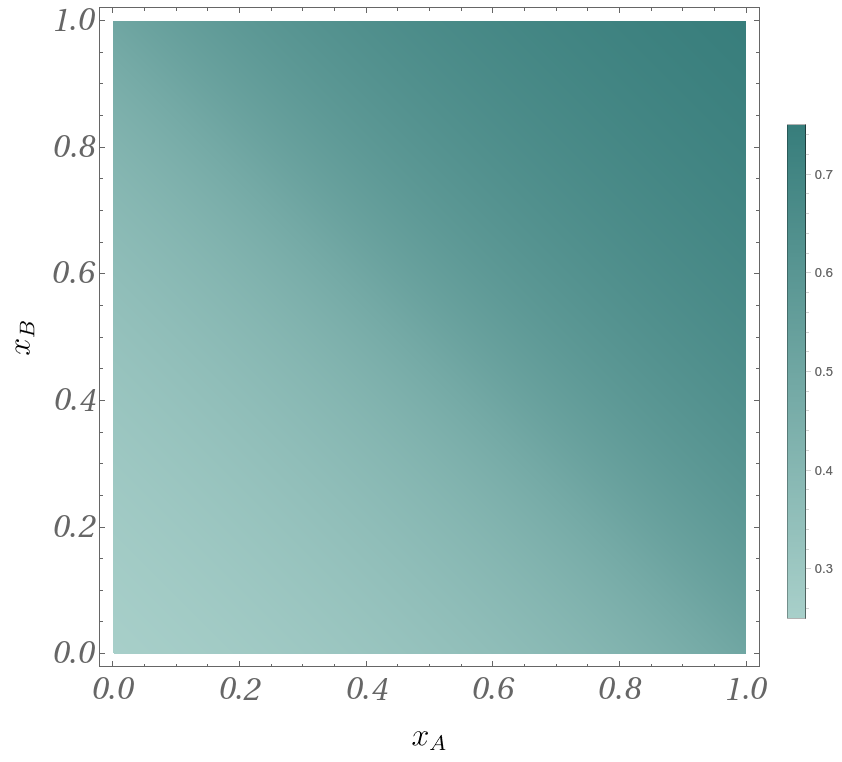}
        \caption{\InfG{1,x_A,x_B}}
        \label{fig:sG_1}
    \end{subfigure}
    \hspace{0.05\textwidth} 
    \begin{subfigure}[b]{0.4\textwidth}
        \includegraphics[width=\textwidth]{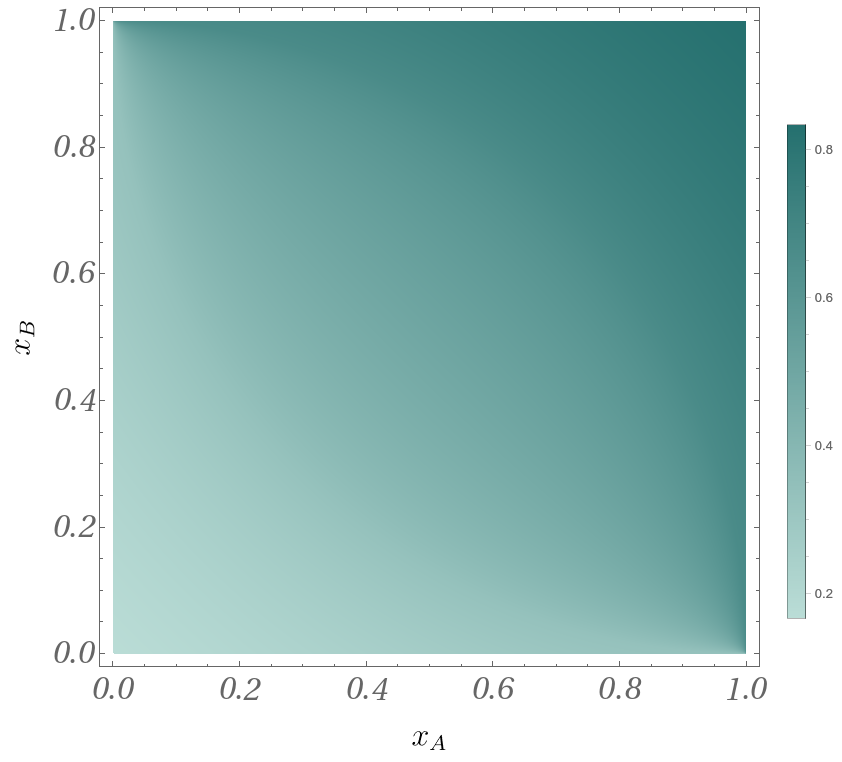}
        \caption{\InfG{2,x_A,x_B}}
        \label{fig:sG_2}
    \end{subfigure}
    
    \caption{$s^*(x_A,x_B)$ for Binomial Fisher games.}
    \label{fig:s_InfG_1_2}
\end{figure}

\begin{figure}[H]
    \centering
    \begin{subfigure}[b]{0.4\textwidth}
        \includegraphics[width=\textwidth]{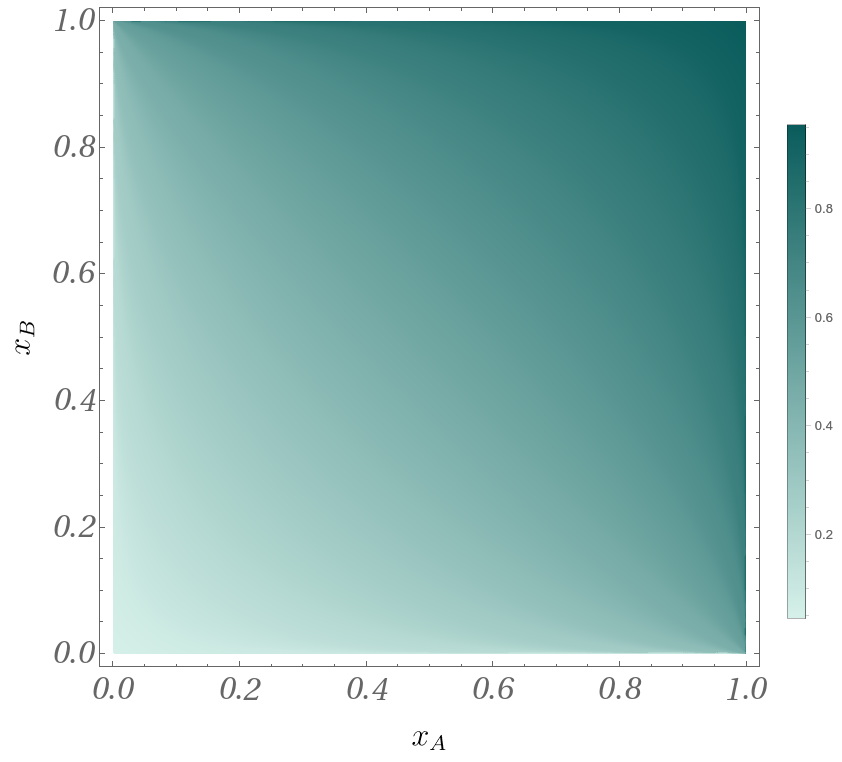}
        \caption{\InfG{10,x_A,x_B}}
        \label{fig:sG_10}
    \end{subfigure}
    \hspace{0.05\textwidth} 
    \begin{subfigure}[b]{0.4\textwidth}
        \includegraphics[width=\textwidth]{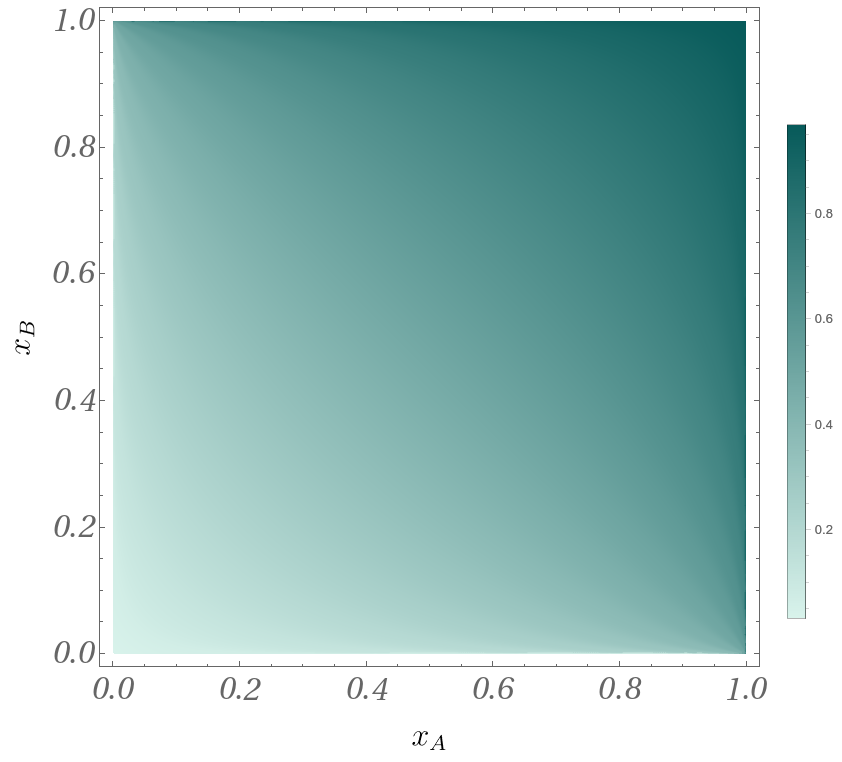}
        \caption{\InfG{15,x_A,x_B}}
        \label{fig:sG_15}
    \end{subfigure}
    
    \caption{$s^*(x_A,x_B)$ for Binomial Fisher games.}
    \label{fig:s_InfG_10_15}
\end{figure}

\subsubsection{Binomial Bayesian games}

\paragraph{Prior plots:} Examples show continuous limiting prior plots $P^*(x_A,x_B)$ in fig. \ref{fig:P_InfBG_1_2} and \ref{fig:P_InfBG_10_15}.

\begin{figure}[H]
    \centering
    \begin{subfigure}[b]{0.4\textwidth}
        \includegraphics[width=\textwidth]{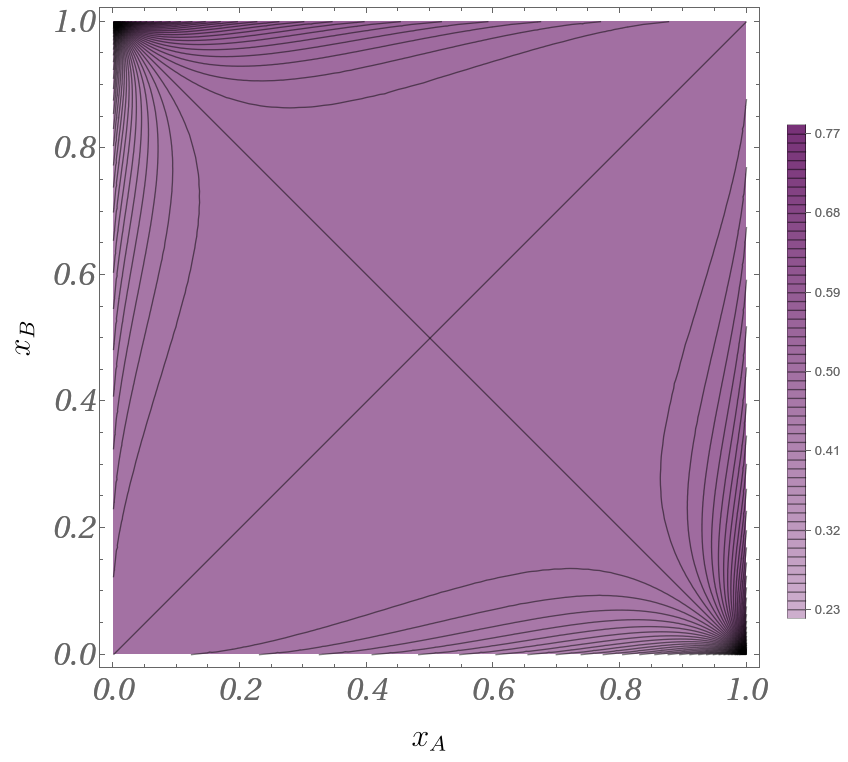}
        \caption{\InfBG{1,x_A,x_B}}
        \label{fig:PBG_1}
    \end{subfigure}
    \hspace{0.05\textwidth} 
    \begin{subfigure}[b]{0.4\textwidth}
        \includegraphics[width=\textwidth]{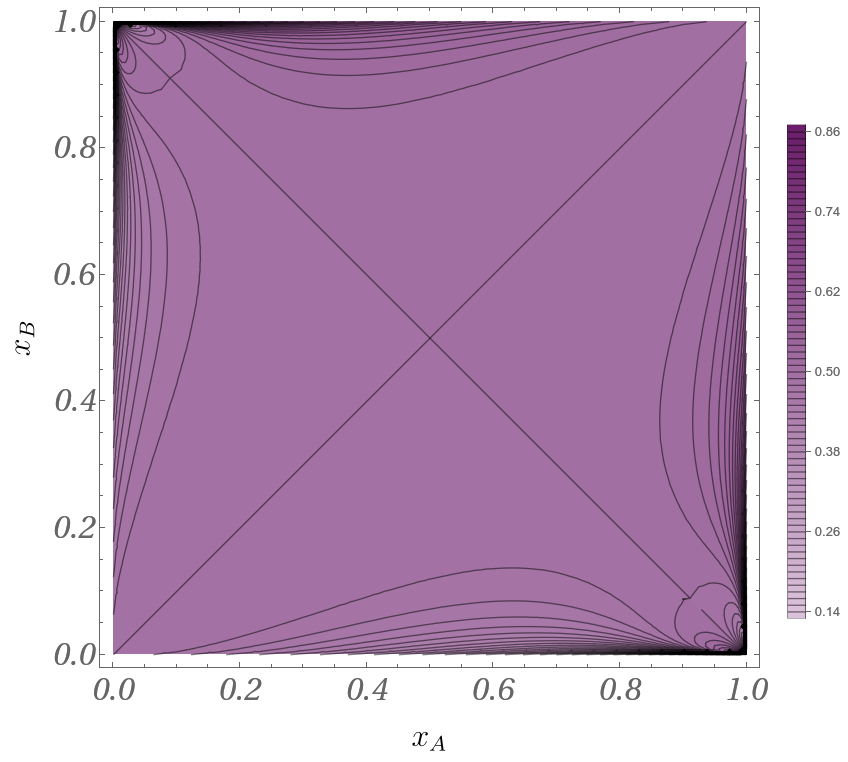}
        \caption{\InfBG{2,x_A,x_B}}
        \label{fig:PBG_2}
    \end{subfigure}
    
    \caption{$P^*(x_A,x_B)$ for Binomial Bayesian games. Contour lines show $1\%$ difference.}
    \label{fig:P_InfBG_1_2}
\end{figure}

\begin{figure}[H]
    \centering
    \begin{subfigure}[b]{0.4\textwidth}
        \includegraphics[width=\textwidth]{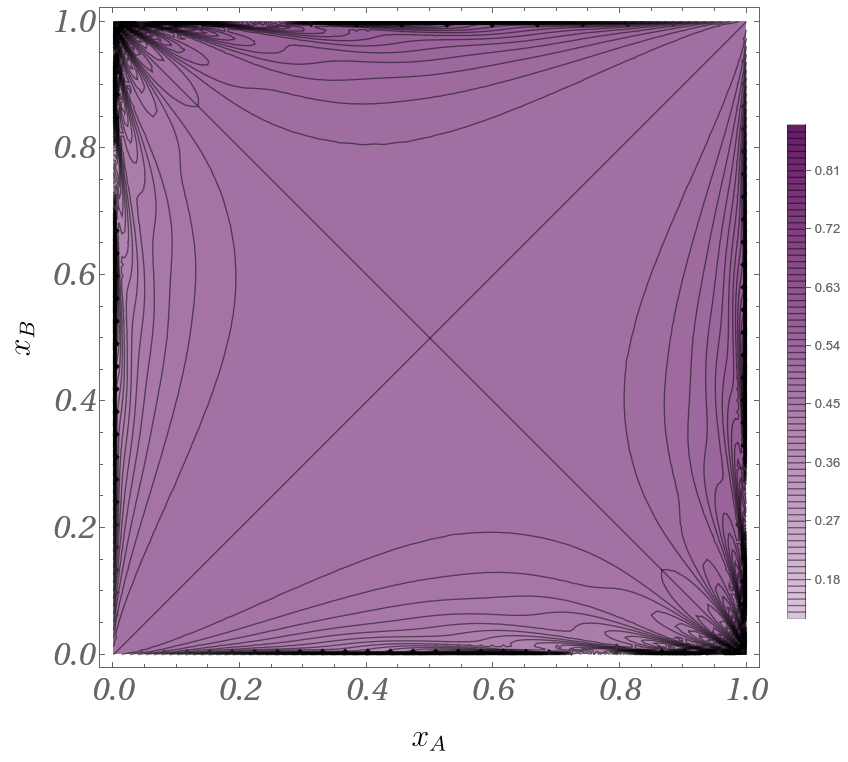}
        \caption{\InfBG{10,x_A,x_B}}
        \label{fig:PBG_10}
    \end{subfigure}
    \hspace{0.05\textwidth} 
    \begin{subfigure}[b]{0.4\textwidth}
        \includegraphics[width=\textwidth]{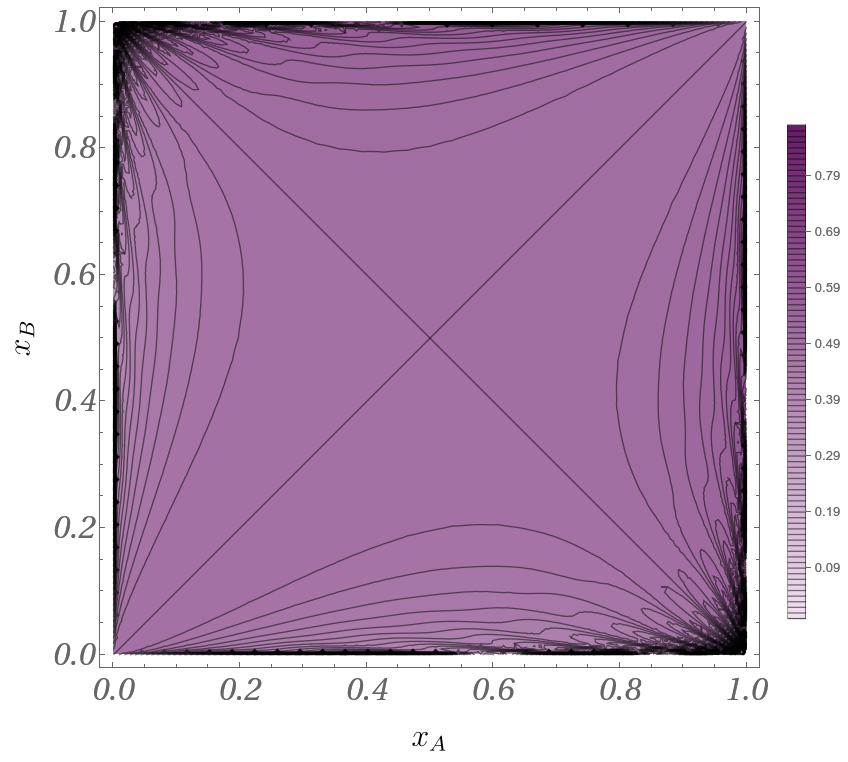}
        \caption{\InfBG{15,x_A,x_B}}
        \label{fig:PBG_15}
    \end{subfigure}
    
    \caption{$P^*(x_A,x_B)$ for Binomial Bayesian games. Contour lines show $1\%$ difference.}
    \label{fig:P_InfBG_10_15}
\end{figure}

\paragraph{Policy plots:} Examples show continuous limiting policy plots i.e. splitting ratios $p'^*_k(x_A,x_B)$ in fig. \ref{fig:ppk_InfBG_1_01}, \ref{fig:ppk_InfBG_2_012} and \ref{fig:ppk_InfBG_6_0123456}.

\begin{figure}[H]
    \centering
    \begin{subfigure}[b]{0.4\textwidth}
        \includegraphics[width=\textwidth]{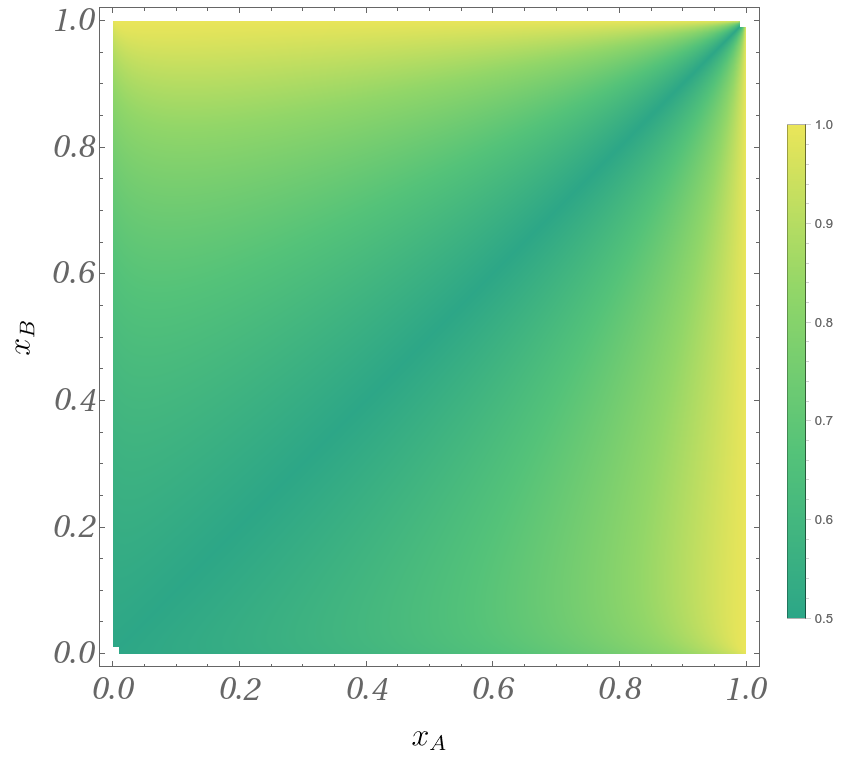}
        \caption{$k=0$}
        \label{fig:ppkBG_1_0}
    \end{subfigure}
    \hspace{0.05\textwidth} 
    \begin{subfigure}[b]{0.4\textwidth}
        \includegraphics[width=\textwidth]{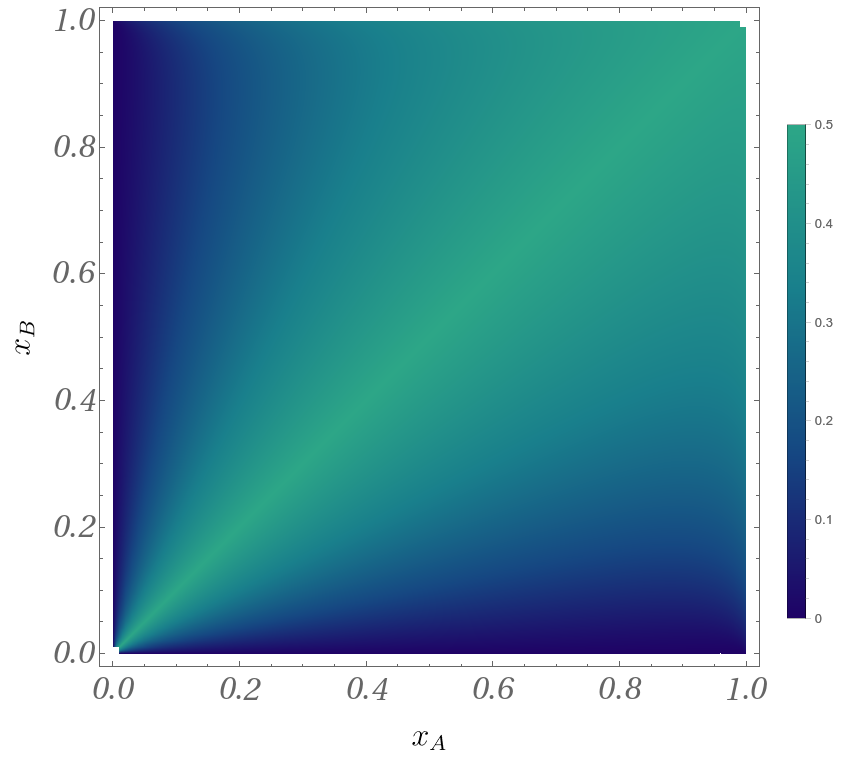}
        \caption{$k=1$}
        \label{fig:ppkBG_1_1}
    \end{subfigure}
    
    \caption{$p'^*_k(x_A,x_B)$ for \InfBG{1,x_A,x_B}.}
    \label{fig:ppk_InfBG_1_01}
\end{figure}

\begin{figure}[H]
    \centering
    \begin{subfigure}[b]{0.3\textwidth}
        \includegraphics[width=\textwidth]{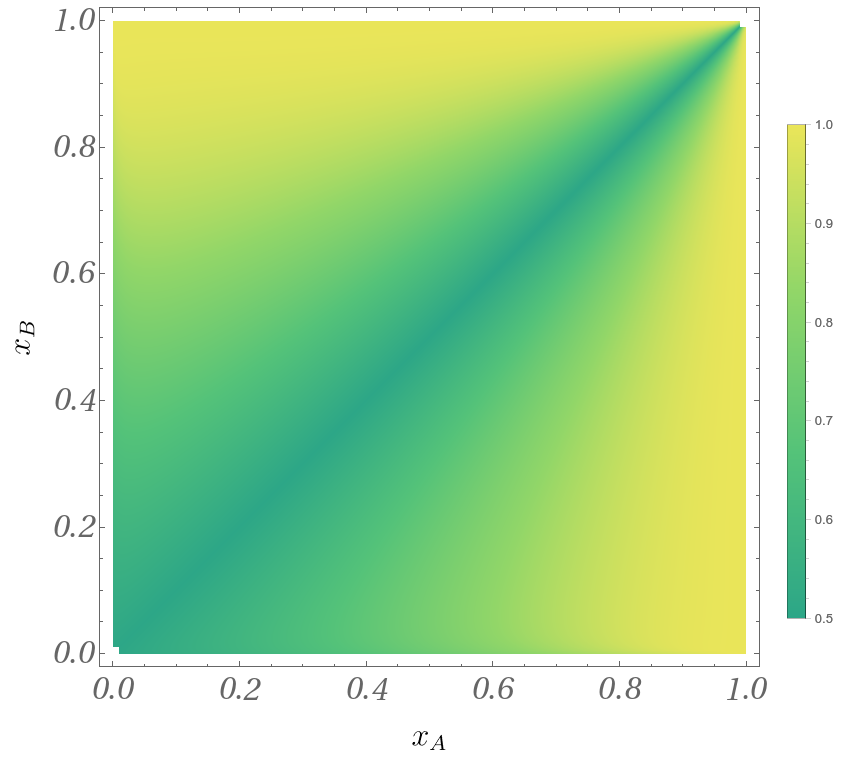}
        \caption{$k=0$}
        \label{fig:ppkBG_2_0}
    \end{subfigure}
    \hspace{0.01\textwidth} 
    \begin{subfigure}[b]{0.3\textwidth}
        \includegraphics[width=\textwidth]{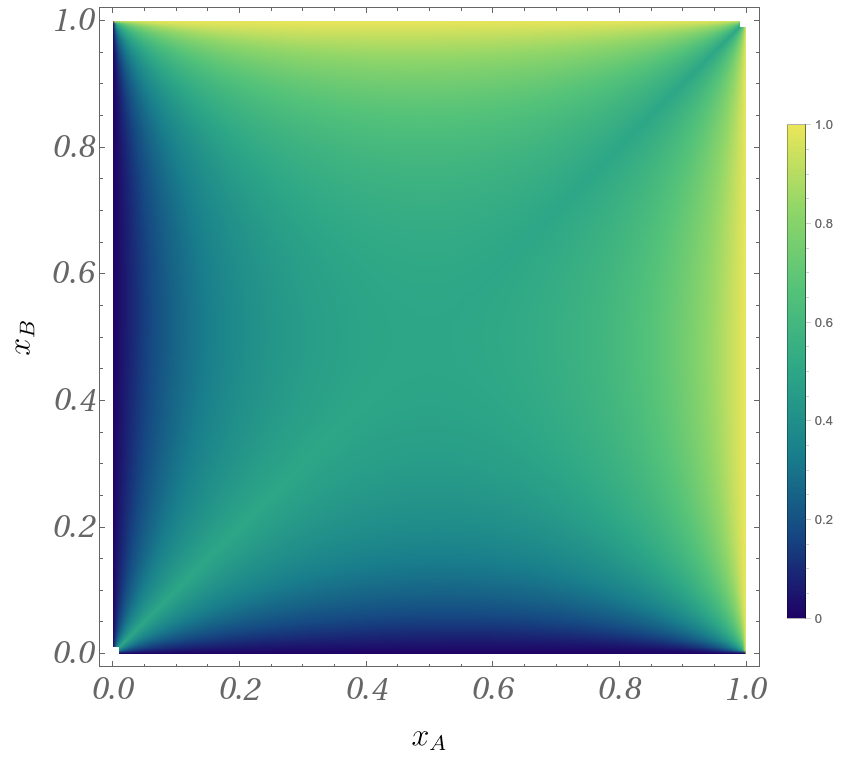}
        \caption{$k=1$}
        \label{fig:ppkBG_2_1}
    \end{subfigure}
    \hspace{0.01\textwidth} 
    \begin{subfigure}[b]{0.3\textwidth}
        \includegraphics[width=\textwidth]{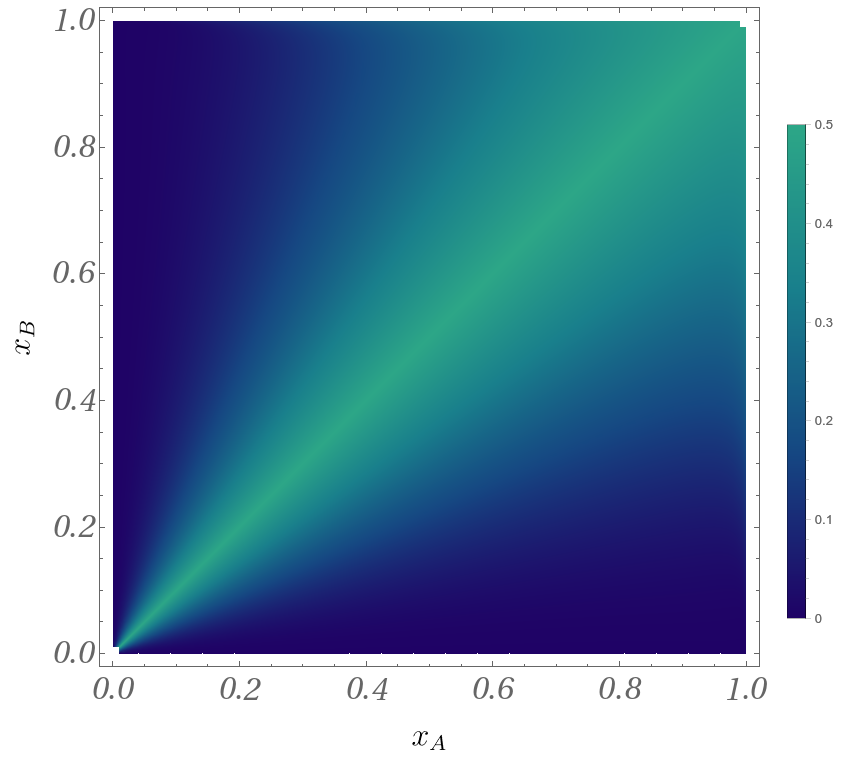}
        \caption{$k=2$}
        \label{fig:ppkBG_2_2}
    \end{subfigure}
    
    \caption{$p'^*_k(x_A,x_B)$ for \InfBG{2,x_A,x_B}.}
    \label{fig:ppk_InfBG_2_012}
\end{figure}

\begin{figure}[H]
    \centering
    \begin{subfigure}[b]{0.130\textwidth}
        \includegraphics[width=\textwidth]{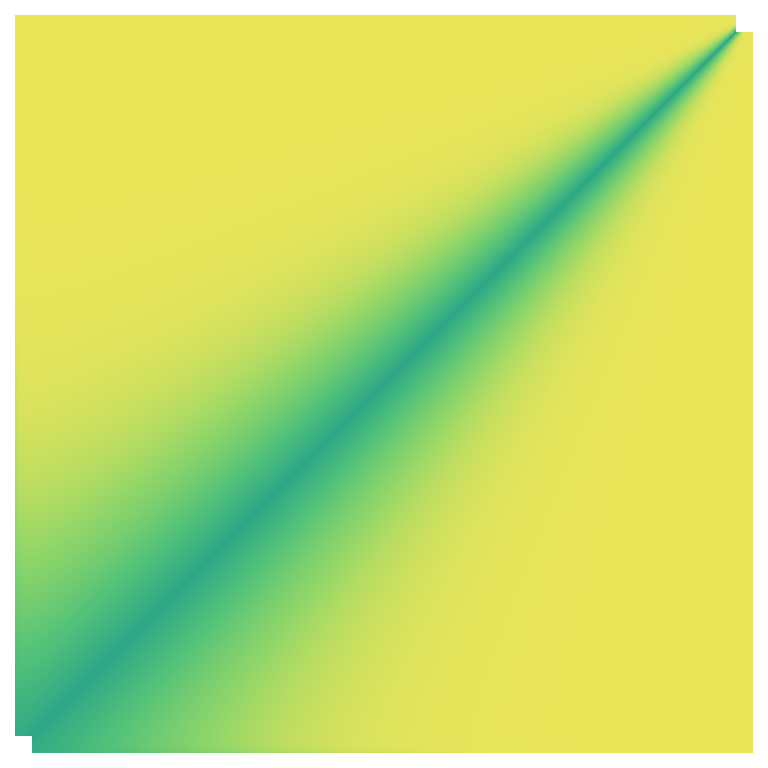}
        \caption{$k=0$}
        \label{fig:ppkBG_6_0}
    \end{subfigure}
    \hspace{0.00\textwidth} 
    \begin{subfigure}[b]{0.130\textwidth}
        \includegraphics[width=\textwidth]{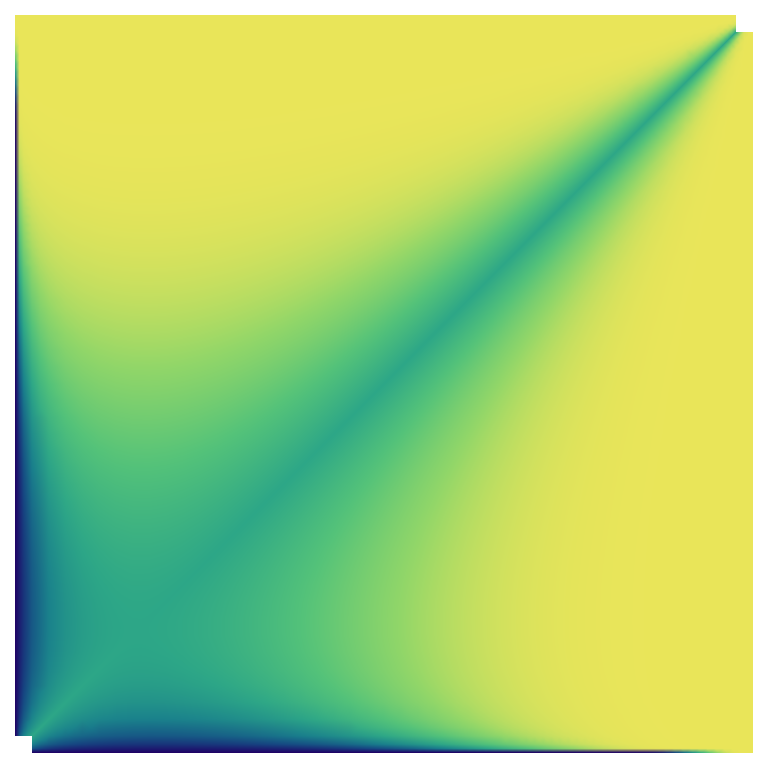}
        \caption{$k=1$}
        \label{fig:ppkBG_6_1}
    \end{subfigure}
    \hspace{0.00\textwidth} 
    \begin{subfigure}[b]{0.130\textwidth}
        \includegraphics[width=\textwidth]{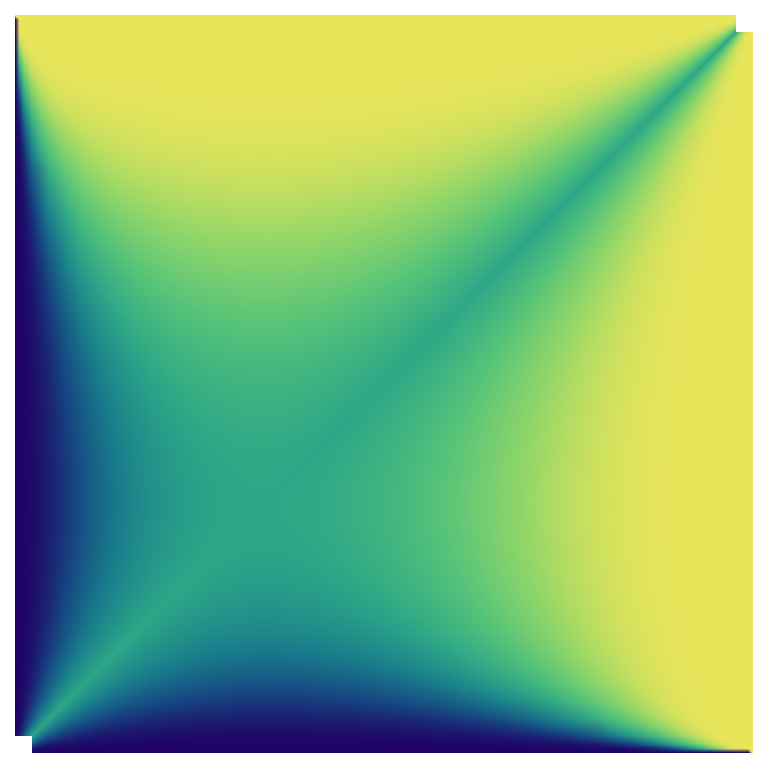}
        \caption{$k=2$}
        \label{fig:ppkBG_6_2}
    \end{subfigure}
    \hspace{0.00\textwidth} 
        \begin{subfigure}[b]{0.130\textwidth}
        \includegraphics[width=\textwidth]{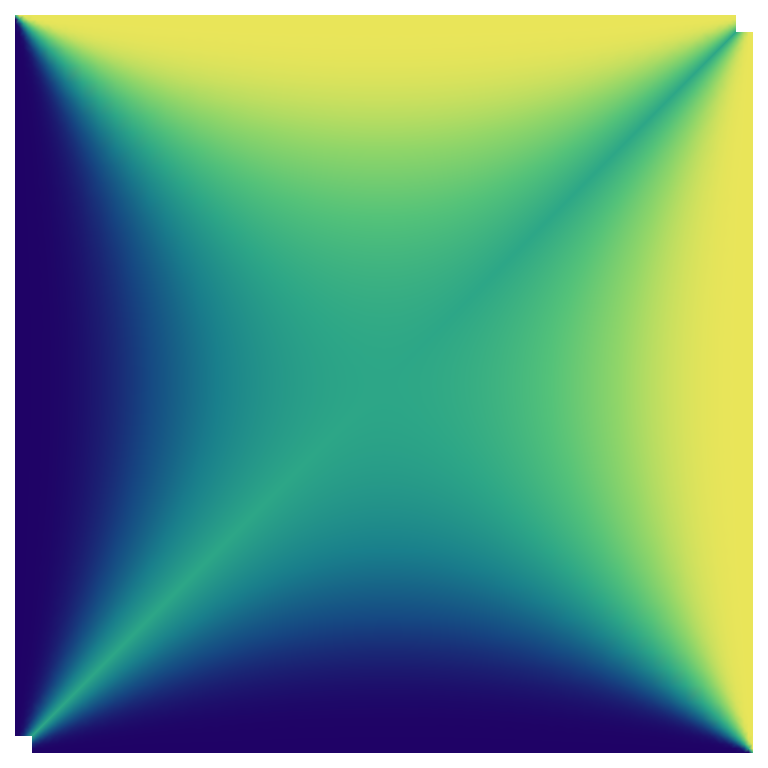}
        \caption{$k=3$}
        \label{fig:ppkBG_6_3}
    \end{subfigure}
    \hspace{0.00\textwidth} 
        \begin{subfigure}[b]{0.130\textwidth}
        \includegraphics[width=\textwidth]{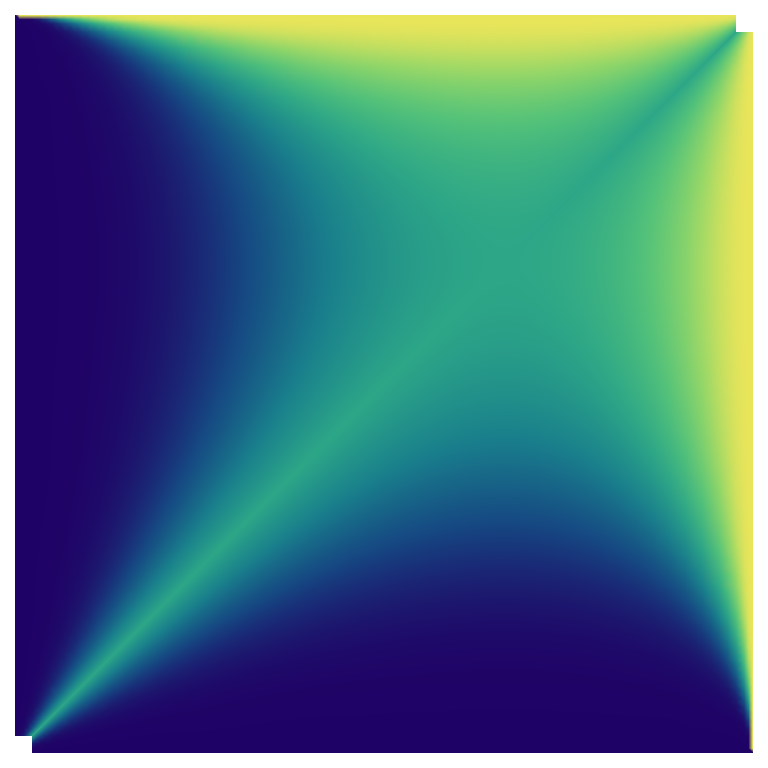}
        \caption{$k=4$}
        \label{fig:ppkBG_6_4}
    \end{subfigure}
    \hspace{0.00\textwidth} 
        \begin{subfigure}[b]{0.130\textwidth}
        \includegraphics[width=\textwidth]{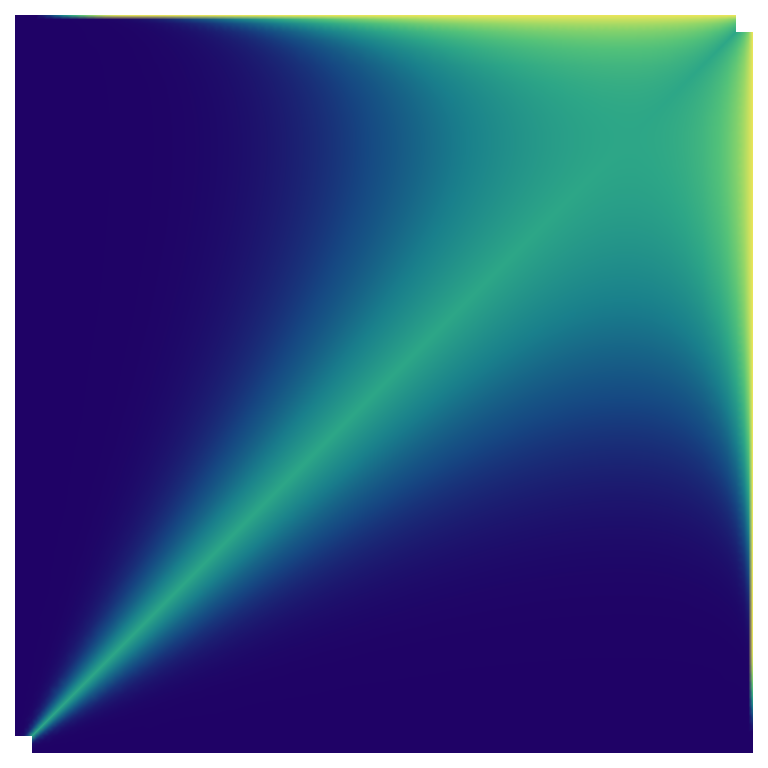}
        \caption{$k=5$}
        \label{fig:ppkBG_6_5}
    \end{subfigure}
    \hspace{0.00\textwidth} 
        \begin{subfigure}[b]{0.130\textwidth}
        \includegraphics[width=\textwidth]{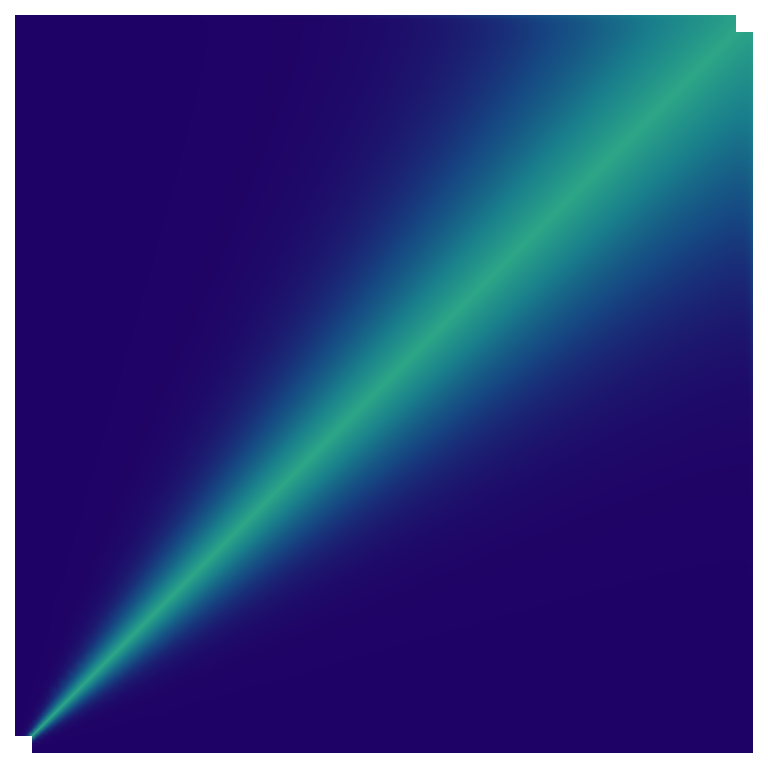}
        \caption{$k=6$}
        \label{fig:ppkBG_6_6}
    \end{subfigure}
    \hspace{0.00\textwidth} 
    
    \caption{$p'^*_k(x_A,x_B)$ for \InfBG{6,x_A,x_B}.
    (Axes and the color scale have the same meaning as in figures \ref{fig:ppk_InfBG_1_01} and \ref{fig:ppk_InfBG_2_012}.)
    }
    \label{fig:ppk_InfBG_6_0123456}
\end{figure}

\subsubsection{Interpretation}

\paragraph{Emergent probabilities:}

Here, I want to emphasize that the distributions appearing in the construction $p_k(A)$ and $p_k(B)$ do not represent independently introduced ``stochastic'' variables. 
The imagined distribution of the number of $\bb$-s in the sample emerges from the joint mixing (or randomization) of \PI/ and \PII/.

In the limit $M \to \infty$, sampling and sequence choosing strategies do not converge to a proper distribution, while the imagined distributions of $k$ do converge. 
This is why the emergent binomial distributions $p_k(A)$ and $p_k(B)$ remain a useful part of the description of the limit case but are fundamentally only a consequence of the mixed strategies of both players.

``The importance of randomization in applied statistics can scarcely
be exaggerated. From the personalistic viewpoint it is one of the most
important ways to bring groups of people into virtual unanimity; from
the objectivistic viewpoint it not only makes possible great reductions
in maximum loss, but it is seen as an invention by which the theory of
probability is brought to bear on situations to which probability on
first (objectivistic) sight would seem irrelevant.''
\cite{book:Savage}

\subsection{Limiting policies for $N \to \infty$}

\subsubsection{Fisher policy limit}

\paragraph{Notation:}

\begin{equation}
    x_0^*(x_A,x_B) = \frac{\log \left ( \frac{1-x_A}{1-x_B} \right )}{\log \left ( \frac{(1-x_A) x_B}{(1-x_B) x_A} \right )}
\end{equation}

\begin{theorem}[Binomial Fisher limiting policy]
\label{thm:BinomialFisherLimitingPolicy}
    For any $0 < x_A < x_B < 1$ the equilibrium policy of \PI/ in a Binomial Fisher game converges to:
    \begin{equation}
        \lim_{N \to \infty} s^*_N(x_A,x_B) = 
        s^*_\loopedsquare(x_A,x_B) =  x_0^*(x_A,x_B)
    \end{equation}
    meaning that for fixed $x_A,x_B$:
    \begin{equation}
        \lim_{N \to \infty} \frac{k^*_N + \nu^*_N}{N} =
        \lim_{N \to \infty} \frac{k^*_N}{N} =
        s^*_\loopedsquare =
        x_0^*
    \end{equation}
\end{theorem}

The proof will be presented in Appendix \ref{appendix:Limiting}.

\paragraph{Tabulated values:}

See tabulated explicit numerical values in Table~\ref{tab:FisherPolicyLimit}.

\begin{table}[H]
\centering
\begin{tabular}{|c|c|c|c|c|c|c|c|c|}
\hline
$x_A \backslash x_B$ & $20\%$ & $30\%$ & $40\%$ & $50\%$ & $60\%$ & $70\%$ & $80\%$ & $90\%$ \\
\hline
$10\%$ & 0.1452 & 0.1862 & 0.2263 & 0.2675 & 0.3116 & 0.3608 & 0.4197 & 0.5000 \\
$20\%$ &  & 0.2477 & 0.2933 & 0.3390 & 0.3869 & 0.4391 & 0.5000 & 0.5803 \\
$30\%$ &  &  & 0.3489 & 0.3971 & 0.4467 & 0.5000 & 0.5609 & 0.6392 \\
$40\%$ &  &  &  & 0.4497 & 0.5000 & 0.5533 & 0.6131 & 0.6884 \\
$50\%$ &  &  &  &  & 0.5503 & 0.6029 & 0.6610 & 0.7325 \\
$60\%$ &  &  &  &  &  & 0.6511 & 0.7067 & 0.7737 \\
$70\%$ &  &  &  &  &  &  & 0.7523 & 0.8138 \\
$80\%$ &  &  &  &  &  &  &  & 0.8548 \\
\hline
\end{tabular}
\caption{Binomial Fisher policy limit $s^*_\loopedsquare(x_A,x_B)$ up to 4 digits.}
\label{tab:FisherPolicyLimit}
\end{table}

\paragraph{Visualization:}

See the visualization of $s^*_\loopedsquare(x_A,x_B)$ in figure \ref{fig:FisherPolicyLimit}.

\begin{figure}[H]
    \centering
    \includegraphics[width=12 cm]{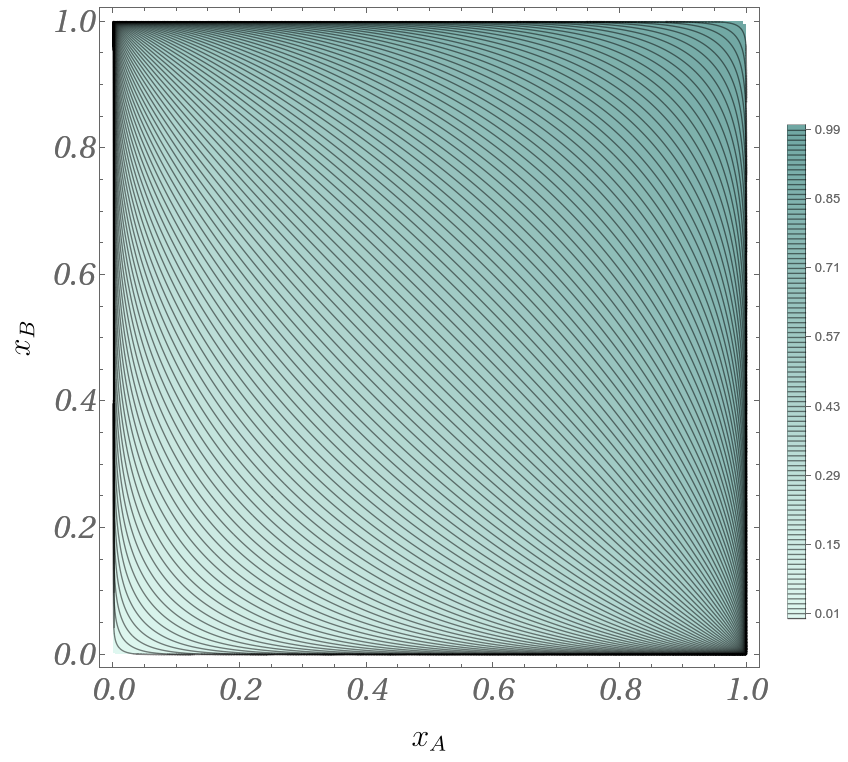}
    \caption{Binomial Fisher policy limit $s^*_\loopedsquare(x_A,x_B)$. Contour lines show $1\%$ difference.}
    \label{fig:FisherPolicyLimit}
\end{figure}

\begin{remark}
    Because of Theorem~\ref{thm:FiniteKMedianBounds}, we can easily get upper and lower bounds for $s^*_\loopedsquare$:

    \begin{equation}
        \lim_{N \to \infty} \frac{\lfloor N x_A \rfloor}{N} 
        =
        x_A \le s^*_\loopedsquare(x_A,x_B) \le x_B
        =
        \lim_{N \to \infty} \frac{\lceil N x_B \rceil}{N}
\end{equation}

\end{remark}

\subsubsection{Binomial Fisher limiting prior bounds}

\begin{conjecture}[Binomial Fisher limiting prior bounds]
\label{conj:FisherLimitingPriorBounds}
    The prior of Binomial Fisher games does not converge as $N \to \infty$, but it has finite upper and lower bounds in the limit:

    \begin{equation}
        \limsup\limits_{N \to \infty} P^*_N(x_A,x_B) = \overline{P}^*_\loopedsquare(x_A,x_B)
    \end{equation}

    \begin{equation}
        \liminf\limits_{N \to \infty} P^*_N(x_A,x_B) = \underline{P}^*_\loopedsquare(x_A,x_B)
    \end{equation}

    The conjectured explicit expression for the bounds can be found in equation \ref{deriv:liminfPN} and \ref{deriv:limsupPN}.
    
\end{conjecture}

A derivation supporting this conjecture is presented in Appendix \ref{appendix:Limiting}.

\subsubsection{Binomial Bayesian limiting prior}

\paragraph{Notation:}

A phase factor can be introduced to characterize the asymptotic behaviour of equilibrium quantities:

\begin{equation}
    \varphi = 2 \pi N x_0^* \mod 2 \pi
\end{equation}

\begin{conjecture}[Binomial Bayesian limiting prior]
\label{conj:BayesianLimitPrior}

For a general \InfBG{N,x_A,x_B} the prior $P^*_N(x_A,x_B)$ converges to a periodic function $P^{*,\varphi}_\loopedsquare$, depending on the phase factor $\varphi = \varphi_N(x_A,x_B)$ as $N \to \infty$:

\begin{equation}
    P^*_N(x_A,x_B) = P^{*,\varphi_N(x_A,x_B)}_\loopedsquare(x_A,x_B) + \mathcal{O}(1/N)
\end{equation}

The implicit expression for the conjectured asymptotic log-odds $\vartheta^{*,\varphi}_\loopedsquare$ can be found in equations \eqref{eq:AsymptoticBayesianLogOdds01}, \eqref{eq:AsymptoticBayesianLogOdds02}, \eqref{eq:AsymptoticBayesianLogOdds03}.

\end{conjecture}

A derivation to support this conjecture is presented in Appendix \ref{appendix:Limiting}.

\begin{remark}[Binomial Bayesian limiting prior approximation]
\label{remark:BayesianLimitPrior}

There is a finite 
$P^\approx_\loopedsquare$ for \InfBG{N,x_A,x_B}, which is a ``good approximation'' of $P^*_N(x_A,x_B)$ as $N \to \infty$:

\begin{equation}
    P^\approx_\loopedsquare(x_A,x_B) = \frac{ \log \left ( \frac{(1-x_0^*) x_B}{(1-x_B)x_0^*}  \right )}{\log \left ( \frac{(1-x_A) x_B}{(1-x_B) x_A} \right )}, \quad
    x_0^*(x_A,x_B) = \frac{\log \left ( \frac{1-x_A}{1-x_B} \right )}{\log \left ( \frac{(1-x_A) x_B}{(1-x_B) x_A} \right )}
\end{equation}

\end{remark}

A derivation which motivates this approximation can be found in Appendix \ref{appendix:Limiting}.

\paragraph{Tabulated values:}
See tabulated explicit numerical values of the Binomial Bayesian limiting prior approximation in Table~\ref{tab:BayesianPriorApprox}.

\begin{table}[H]
\centering
\begin{tabular}{|c|c|c|c|c|c|c|c|c|}
\hline
$x_A \backslash x_B$ & $20\%$ & $30\%$ & $40\%$ & $50\%$ & $60\%$ & $70\%$ & $80\%$ & $90\%$ \\
\hline
$10\%$ & 0.4761 & 0.4651 & 0.4598 & 0.4584 & 0.4604 & 0.4661 & 0.4772 & 0.5000 \\
$20\%$ &  & 0.4887 & 0.4832 & 0.4816 & 0.4833 & 0.4889 & 0.5000 & 0.5228 \\
$30\%$ &  &  & 0.4944 & 0.4928 & 0.4945 & 0.5000 & 0.5111 & 0.5339 \\
$40\%$ &  &  &  & 0.4983 & 0.5000 & 0.5055 & 0.5167 & 0.5396 \\
$50\%$ &  &  &  &  & 0.5017 & 0.5072 & 0.5184 & 0.5416 \\
$60\%$ &  &  &  &  &  & 0.5056 & 0.5168 & 0.5402 \\
$70\%$ &  &  &  &  &  &  & 0.5113 & 0.5349 \\
$80\%$ &  &  &  &  &  &  &  & 0.5239 \\
\hline
\end{tabular}
\caption{Binomial Bayesian limit prior approximation $P^\approx_\loopedsquare(x_A,x_B)$ up to 4 digits.}
\label{tab:BayesianPriorApprox}
\end{table}

\paragraph{Visualization:}
See the visualization of $P^\approx_\loopedsquare(x_A,x_B)$ in figure \ref{fig:BayesianPriorApprox}.

\begin{figure}[H]
    \centering
    \includegraphics[width=12 cm]{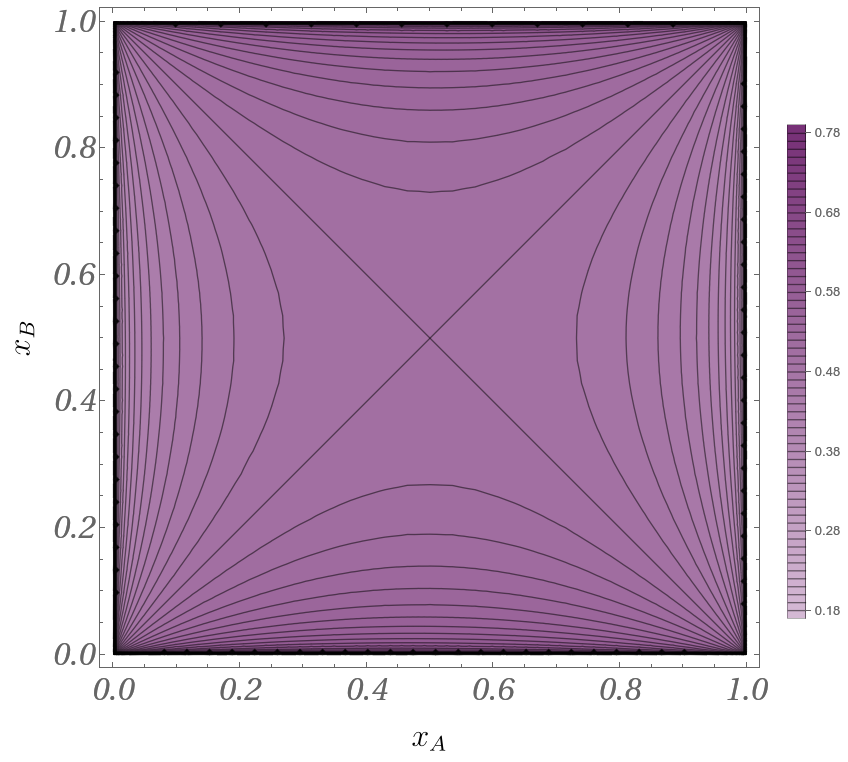}
    \caption{Binomial Bayesian limiting prior approximation $P^\approx_\loopedsquare(x_A,x_B)$. Contour lines show $1\%$ difference.}
    \label{fig:BayesianPriorApprox}
\end{figure}

\paragraph{Numerical evidence:}

\begin{figure}[H]
    \centering
    \includegraphics[width=12 cm]{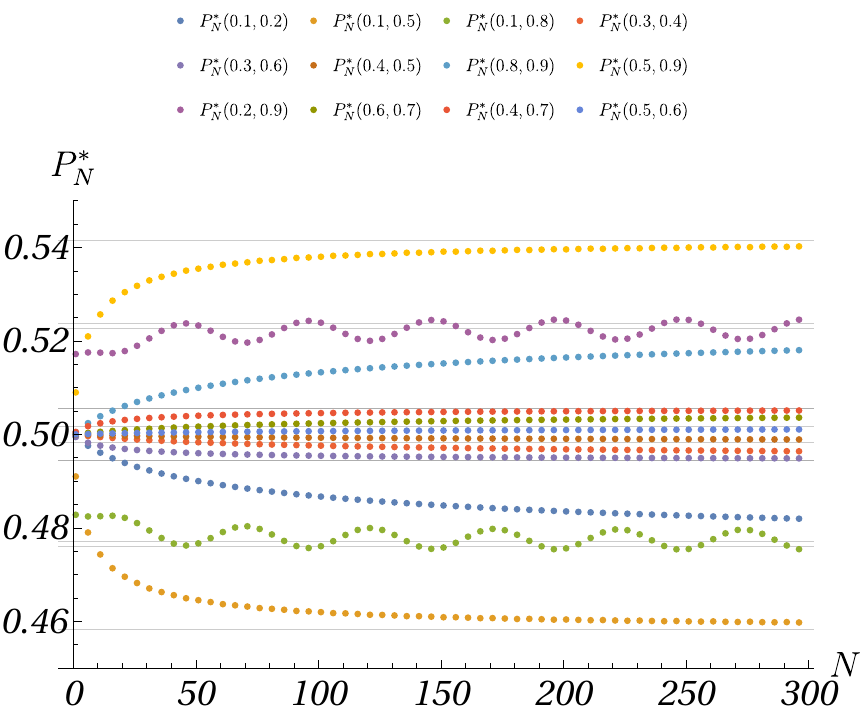}
    \caption{Numerically calculated $P^*_N(x_A,x_B)$ values, and gridlines at the limit prior approximations $P^\approx_\loopedsquare(x_A,x_B)$.}
    \label{fig:PriorNumerics}
\end{figure}

The behaviour of $P^*_N(0.1,0.8)$ and $P^*_N(0.2,0.9)$ on figure \ref{fig:PriorNumerics} shows a visible oscillatory pattern.
\footnote{The sampling is set to $\Delta N=5$, which generates a Stroboscopic effect \cite{paper:Fans}, resulting an apparent wave with relatively long wave length in $N$.}
The plotted values are the results of controlled approximations. Therefore, the oscillation is not caused by numerical errors.

All other examples seemingly converge slowly but steadily to values, which can be approximated well with the Binomial Bayesian limiting prior approximation $P^{\approx}_\loopedsquare(x_A,x_B)$.
Qualitatively, the behaviour is consistent with the conjectured periodic Binomial Bayesian limiting prior asymptotics $P^{*,\varphi}_\loopedsquare(x_A,x_B)$.

\paragraph{First-order asymptotics:}

\begin{conjecture}
\label{conj:AsymptoticExpansionSubleadingTerm}
    The Binomial Bayesian prior has the following asymptotic expansion for any $0 < x_A < x_B < 1$:
    
    \begin{equation}
        \vartheta^*_N(x_A,x_B) = \vartheta^{*,\varphi}_\loopedsquare(x_A,x_B) + \frac{1}{N} \sampi^\varphi(x_A,x_B) +
        \mathcal{O} \left ( \frac{1}{N^2} \right )
    \end{equation}

    where $\sampi^\varphi(x_A,x_B)$\footnote{$\sampi$, sampi ``like pi'' is an archaic letter of the Greek alphabet \cite{book:ArchaicGreece}. Originally, it might stand for a sibilant sound, probably [ss] or [ts]. It also represents $900$ among the Milesian numerals \cite{book:GreekMathematics,book:ConwayBookOfNumbers}. The symbol is part of UTF-8 \href{https://www.unicode.org/charts/beta/nameslist/n_0370.html}{character encoding} \cite{book:Unicode} \href{https://www.fileformat.info/info/unicode/char/03e1/index.htm}{U+03E1}, and can be accessed in mathematical software such as {\it Wolfram Mathematic} \cite{tool:WolframSampi} (the form used in this text is generated by {\it Mathematica}, using the \href{https://www.ctan.org/tex-archive/fonts/wasy}{wasy10} Font Family)} can be defined explicitly by knowing $\vartheta^{*,\varphi}_\loopedsquare$.
    The explicit definition can be found in equation \eqref{eq:sampiPhiGammaResult}.
    
\end{conjecture}

\begin{figure}[H]
    \centering
    \includegraphics[width=12 cm]{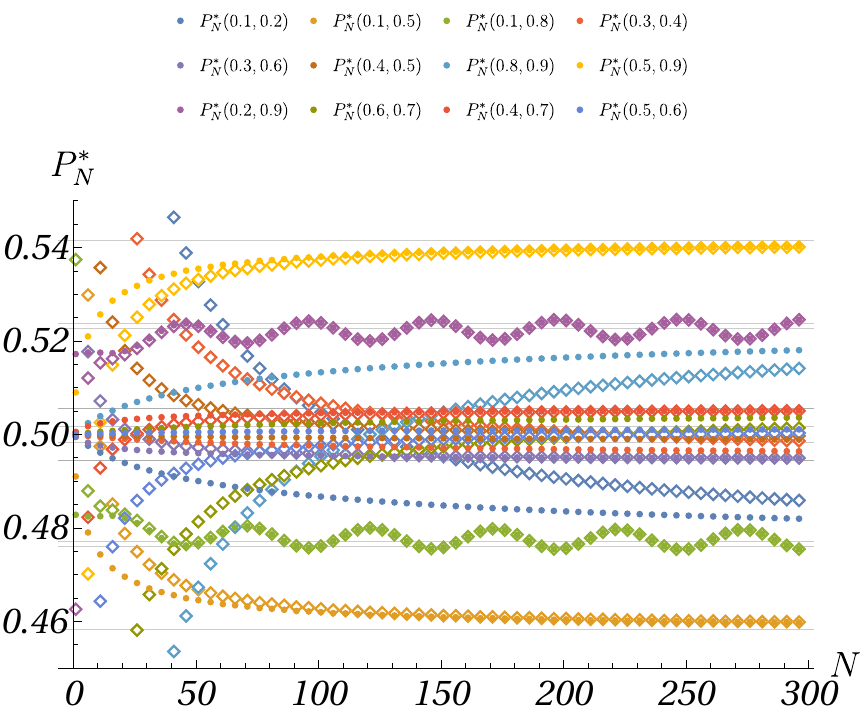}
    \caption{Numerically calculated $P^*_N(x_A,x_B)$ values marked by coloured dots ($\bullet$), and the values of first-order asymptotics $P^{*,(1)}_N(x_A,x_B)$ marked by coloured diamonds ($\diamond$).  
    (Gridlines are placed at zeroth-order limiting prior approximations $P^\approx_\loopedsquare(x_A,x_B)$.)
    }
    \label{fig:PriorNumericsAsymptotics}
\end{figure}

The derivation supporting the conjecture is presented in Appendix \ref{appendix:Limiting}.

\begin{equation}
    P^{*,(1)}_N(x_A,x_B) =
    \sigma \left (
    \vartheta^{*,\varphi_N(x_A,x_B)}_\loopedsquare(x_A,x_B) + \frac{1}{N} \sampi^{\varphi_N(x_A,x_B)}(x_A,x_B)
    \right )
\end{equation}

where $\sigma(.)$ stands for the sigmoid function $\sigma(x) = 1/(1+e^{-x})$.

\paragraph{First-order approximation:}

\begin{remark}
\label{remark:AsymptoticExpansionSubleadingTermApproximation}

The Binomial Bayesian prior can be approximated by the following asymptotic expansion for $0 < x_A < x_B < 1$:
    
    \begin{equation}
        \vartheta^*_N(x_A,x_B) \approx \vartheta^\approx_\loopedsquare(x_A,x_B) + \frac{1}{N} \sampi(x_A,x_B) +
        \mathcal{O} \left ( \frac{1}{N^2} \right )
    \end{equation}

    where $\sampi(x_A,x_B)$ can be expressed with $x_A$ and $x_B$ using elementary functions.
    The explicit expression can be found in equation \eqref{eq:sampiResult01} and \eqref{eq:sampiResult02}.

\end{remark}

\begin{figure}[H]
    \centering
    \includegraphics[width=12 cm]{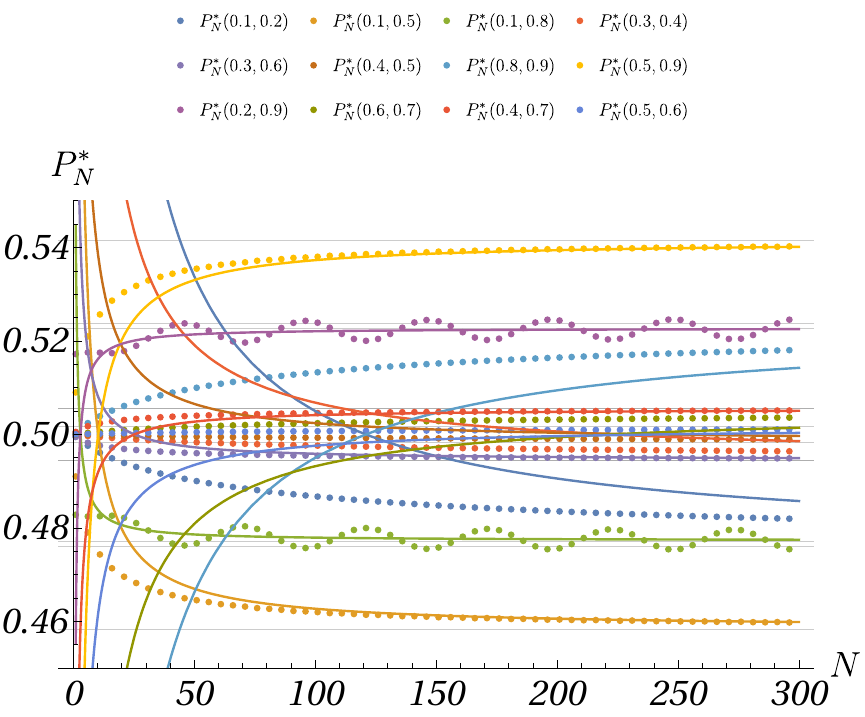}
    \caption{Numerically calculated $P^*_N(x_A,x_B)$ values marked by coloured dots ($\bullet$), and the values of first-order approximation $P^{\approx,(1)}_N(x_A,x_B)$ represented by continuous coloured lines (\textbf{---}).  
    (Gridlines are placed at zeroth-order limiting prior approximations $P^\approx_\loopedsquare(x_A,x_B)$.)
    }
    \label{fig:PriorNumericsApproximation}
\end{figure}

The derivation supporting the approximation is presented in Appendix \ref{appendix:Limiting}.

\begin{equation}
    P^{\approx,(1)}_N(x_A,x_B) =
    \sigma \left (
    \vartheta^\approx_\loopedsquare(x_A,x_B) + \frac{1}{N} \sampi(x_A,x_B)
    \right )
\end{equation}

\section{Unification through relative risk aversion}
\label{sec:StatisticalGames}

\subsection{Description of the general Statistical game}

\begin{definition}[Statistical game]
\label{def:StatisticalGame}

There are two players, \PI/ and \PII/.
\PII/ needs to choose between scenario A or B first and then produce a binary sequence of length $M$ containing precisely $K_A$ or $K_B$ number of $1$-s. (Without losing generality, we will assume $K_A \le K_B$.)
Following this, \PI/ (not knowing the actions of \PII/) can sample $N$ number of bits. After observing their value, she determines what portion of her capital $p'$ she places on scenario A (while the other $1-p'$ portion is placed on scenario B).

The portion \PI/ places on the scenario, chosen by \PII/, will be doubled, while the other part of her capital will be lost.
For this game, we will assume that \PI/ has an isoelastic utility function:

\begin{equation}
    u_\gamma(c) = \frac{c^{1-\gamma}-1}{1-\gamma}
\end{equation}
with relative risk aversion parameter $\gamma > 0, \gamma \ne 1$.
Furthermore, we will assume that \PI/ and \PII/ are playing a zero-sum game\footnote{zero-sum in utilities (not in capital)}.

The above defined Statistical Game will be denoted as 
\SG{N, K_A, K_B, M, \gamma}.

\end{definition}

\begin{remark}
    The crucial difference between the definition of Bayesian games \ref{def:BayesianGame} and Statistical games \ref{def:StatisticalGame} is the utility function for \PI/.
    It is logarithmic for Bayesian games \BG{N,K_A,K_B,M} and isoelastic with relative risk aversion parameter $\gamma$ for Statistical games \SG{N,K_A,K_B,M,\gamma}.
    
\end{remark}

\subsection{Simplest nontrivial example}

The simplest statistically nontrivial general Statistical game is \SG{N=1,K_A=0,K_B=1,M=2,\gamma}.
To find its equilibrium solution, we can introduce very similar notation and go through analogous steps to Section~\ref{sec:SimplestNontrivialBayesian}.

\paragraph{Expected utility:}
The expected utility for \PI/ looks the following in the notation borrowed from Section~\ref{sec:SimplestNontrivialBayesian}:

\begin{equation}
\label{eq:Simplest_EU}
    \begin{split}
        U_\gamma = P & \left ( 
        q \ u_\gamma(p'_{(1,\wb)}) + (1-q) u_\gamma(p'_{(2,\wb)})
        \right ) + \\
        (1-P) r & \left (
        q \ u_\gamma(1-p'_{(1,\wb)}) + (1-q) u_\gamma(1-p'_{(2,\bb)})
        \right) + \\
        (1-P) (1-r) & \left (
        q \ u_\gamma(1-p'_{(1,\bb)}) + (1-q) u_\gamma(1-p'_{(2,\wb)}
        \right )
    \end{split}
\end{equation}

\paragraph{Equilibrium parameters:}
Following the steps analogous to the derivation in Section~\ref{sec:SimplestNontrivialBayesian}, we get:

\begin{equation}
    p'_{(1,\bb)} = p'_{(2,\bb)} = p'^*_1 = 0,
\end{equation}

\begin{equation}
\label{eq:SimplestEUsplitting}
    p'^*_{(1,\wb)} = \frac{P^{1/\gamma}}{P^{1/\gamma}+((1-P) r)^{1/\gamma}}, \quad  p'^*_{(2,\wb)} = \frac{P^{1/\gamma}}{P^{1/\gamma}+((1-P)(1-r))^{1/\gamma}}
\end{equation}

\begin{equation}
    r^* = 1/2, \quad q^* = 1/2
\end{equation}

\paragraph{Finding $P^*_\gamma$:}
Substituting back to equation \eqref{eq:Simplest_EU}, we get the following expression for the expected utility as a function of $P$:

\begin{equation}
    U_\gamma(P) =
    \frac{1}{1-\gamma}
    \left (
    P
    \left (
    \frac{P^{(1-\gamma)/\gamma}}{(P^{1/\gamma}+((1-P) /2)^{1/\gamma})^{1-\gamma} }
    - 1
    \right ) +
    \frac{(1-P)}{2}
    \left (
    \frac{((1-P) /2)^{1/\gamma})}{(P^{1/\gamma}+((1-P) /2)^{1/\gamma})^{1-\gamma} }
    - 1
    \right )
    \right )
\end{equation}

\begin{figure}[H]
    \centering
    \includegraphics[width=12 cm]{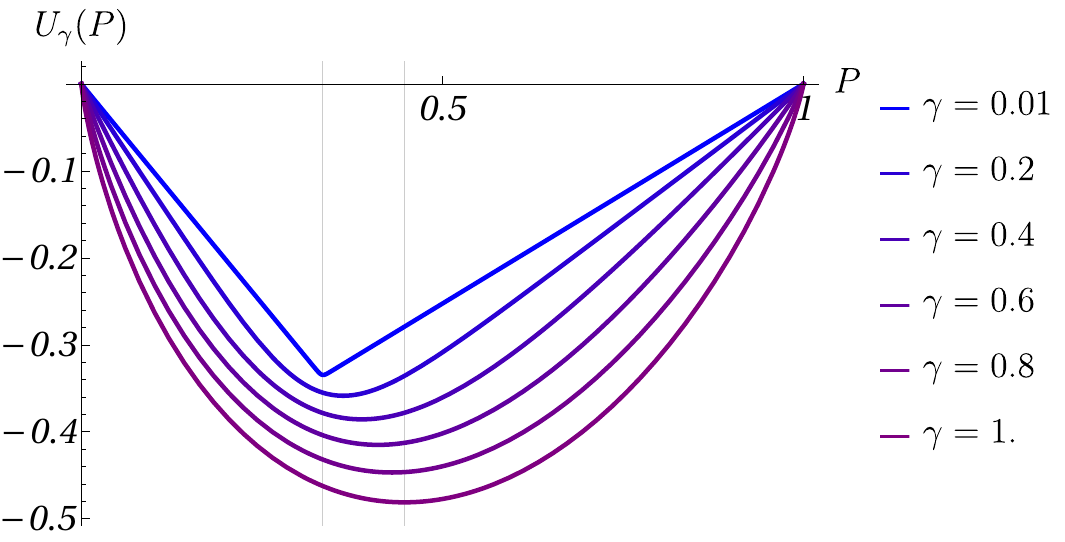}
    \caption{Expected utility for \SG{N=1, K_A=0, K_B=1, M=2} as a function of $P$ for different $\gamma$ (relative risk aversion) values.
    (Vertical gridlines are placed at $1/3$ and $1/\sqrt{5}$ values.)}
    \label{fig:EUP_1012_01}
\end{figure}

$P^*_\gamma$ is defined as the value, where $U_\gamma(P)$ takes its minimum.
We can take the derivative with respect to $P$, and find the root of this expression. For general $\gamma \ne 1$ parameters this gives an explicit, but transcendental equation for $P^*_\gamma$:

\begin{equation}
\label{eq:SimplestEUPs}
    \left (
    Q (2 P)^{1/\gamma} - P Q^{1/\gamma}
    \right )
    \left (
    P^{1/\gamma} + (Q/2)^{1/\gamma}
    \right )^\gamma 
    =
    \frac{P Q}{2}
    \left (
    (2 P)^{1/\gamma} + Q^{1/\gamma}
    \right ), \quad Q=1-P, \quad P=P^*_\gamma
\end{equation}

In general, equation \eqref{eq:SimplestEUPs} has no closed-form solution, but it can be effectively solved numerically.
Numerically calculated values are plotted in figures \ref{fig:EURT_01}, \ref{fig:EURT_03} and \ref{fig:EURT_010} in Section~\ref{sec:EntropiesSimplestPriors}.

\begin{remark}
    There are a few special relative risk aversion values $\gamma$, for which equation \eqref{eq:SimplestEUPs} has a closed form solution:

    \begin{equation}
        P^*_{\gamma=1/2} = 2/5, \quad
        P^*_{\gamma=2} = 1/2, \quad
        P^*_{\gamma=3} = (3 + \sqrt[4]{12} - \sqrt{3})/6
    \end{equation}

\end{remark}

\begin{remark}
    Based on numerical evidence and explicit calculation, we can determine three important limit cases for $P^*_\gamma$:

    \begin{equation}
        \lim_{\gamma \to 0} P^*_\gamma = \frac{1}{3}, \quad
        \lim_{\gamma \to 1} P^*_\gamma = \frac{1}{\sqrt{5}}, \quad
        \lim_{\gamma \to \infty} P^*_\gamma = \frac{1}{2}
    \end{equation}

\end{remark}

\paragraph{Equilibrium splitting ratios:}

Based on the numerically calculated $P^*_\gamma$, and equation \eqref{eq:SimplestEUsplitting}, we can numerically determine the equilibrium splitting ratios $p'^*_{\gamma,k}$.
The results of numerical calculations are shown in figure \ref{fig:pp(gamma)_1012}.
    
\begin{figure}[H]
    \centering
    \includegraphics[width=10 cm]{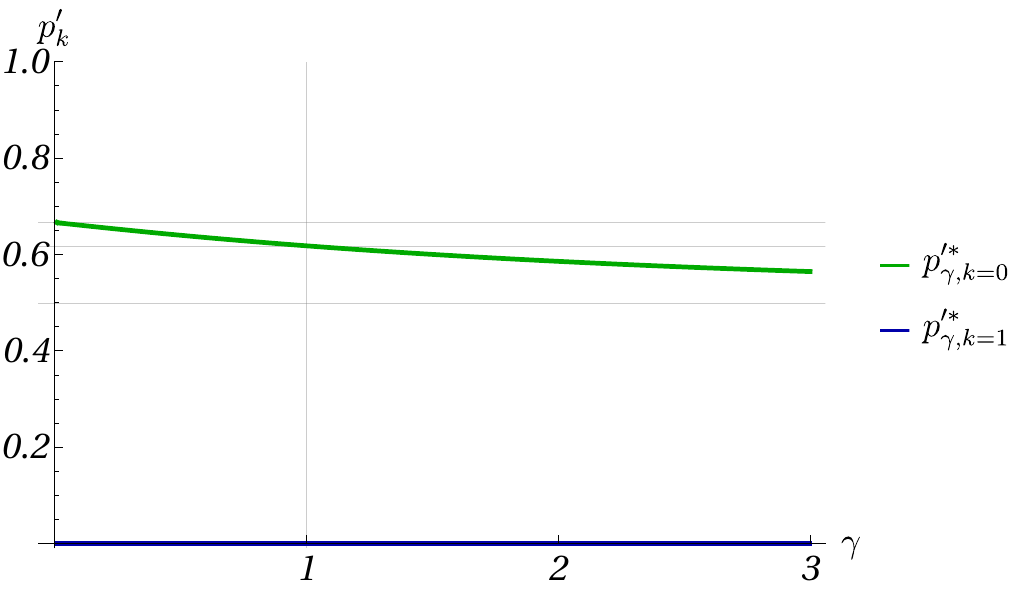}
    \caption{Splitting strategies, $p'^*_{\gamma,k}$ for \SG{N=1, K_A=0, K_B=1, M=2} as a function of $\gamma$ (relative risk aversion parameter). Horizontal gridlines are placed at $\{2/3 , (\sqrt{5}-1)/2, 1/2\}$.}
    \label{fig:pp(gamma)_1012}
\end{figure}

\subsection{General Statistical games}

\subsubsection{Main theorem on Statistical games}

\begin{theorem}[Isoelastic equilibrium]
\label{thm:StatisticalGameEquilibrium}
\SG{N,K_A,K_B,M,\gamma} has a unique Nash equilibrium, in which:

\begin{itemize}
    \item \PII/ chooses scenario A or B with probability $P^*_\gamma$ and $1-P^*_\gamma$;
    \begin{itemize}
        \item then picks an allowed sequence with equal probability (from $\mathcal{K}_A$ or $\mathcal{K}_B$).
    \end{itemize}

    \item \PI/ first samples uniformly $N$ bits from the provided sequence. Based on $k$ -- the number of $\bb$-s -- she determines $p'^*_{\gamma,k} \in [0,1]$, and bets in the following way:
    \begin{itemize}
        \item places her capitals $p'^*_{\gamma,k}$ portion to A
        \item places her capitals $1-p'^*_{\gamma,k}$ portion to B
    \end{itemize}
\end{itemize}

The parameters $(P^*_\gamma, \{p'^*_{\gamma,k}\})$ can be determined from the parameters of the game $(N, K_A, K_B, M, \gamma)$:

\begin{equation}
    \label{thm:StatisticalEqHypergeom}
    p_k(A) = \frac{\binom{K_A}{k} \binom{M-K_A}{N-k}}{\binom{M}{N}}, \quad
    p_k(B) = \frac{\binom{K_B}{k} \binom{M-K_B}{N-k}}{\binom{M}{N}}
\end{equation}

\begin{equation}
\label{eq:ppkPgamma}
    p'_k(P) = \frac{\left ( P \ p_k(A) \right )^{1/\gamma}}{\left ( P \ p_k(A) \right )^{1/\gamma} + \left ( (1-P) \ p_k(B) \right )^{1/\gamma}}
\end{equation}

\begin{equation}
    p'^*_k = p'_k(P^*_\gamma)
\end{equation}

while $P^*_\gamma$ is the unique minimum of the expected utility:

\begin{equation}
    \label{eq:SGameExpectedUtility}
    U_\gamma(P) = P \ \left ( \sum_k p_k(A) u_\gamma(p'_k(P)) \right ) 
    + (1-P) \left ( \sum_k p_k(B) u_\gamma(1-p'_k(P)) \right )
\end{equation}

where the utility function for a given relative risk aversion $\gamma>0, \gamma \ne 1$ is:

\begin{equation}
    u_\gamma(c) = \frac{c^{1-\gamma}-1}{1-\gamma}
\end{equation}

\end{theorem}

For the proof, see Appendix \ref{appendix:Unifiction}.

\subsection{Examples and Visualization}

\begin{figure}[H]
    \centering
    \includegraphics[width=12 cm]{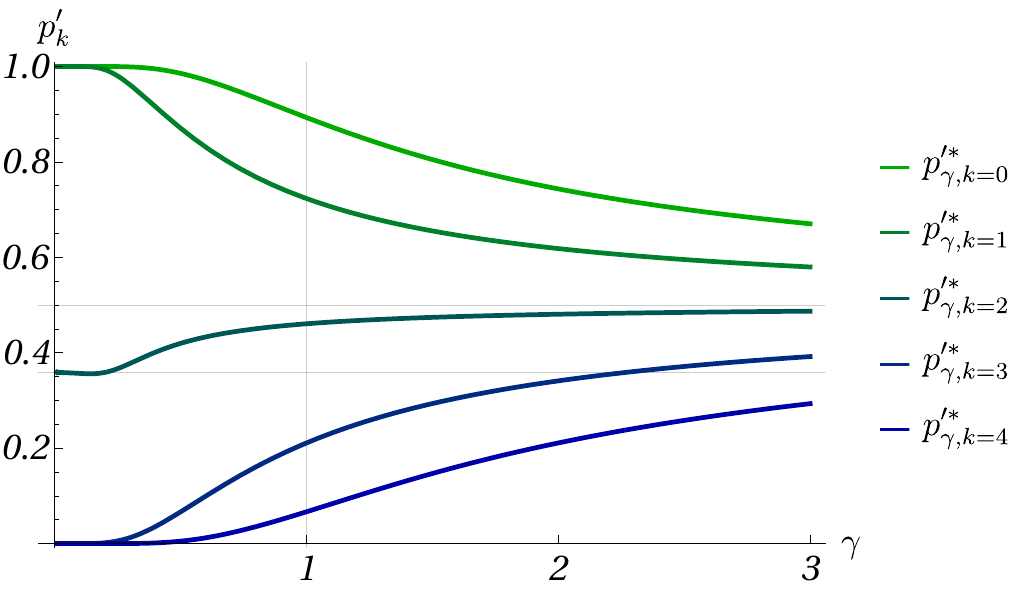}
    \caption{Splitting strategies, $p'^*_{\gamma,k}$ for \SG{N=4, K_A=5, K_B=8, M=14} as a function of $\gamma$ (relative risk aversion parameter). Horizontal gridlines are placed at $\{1/2 , 14/39 \approx 0.359 \}$.}
    \label{fig:pp(gamma)_45814}
\end{figure}

\subsection{Unification of Fisher and Statistical games}

\begin{theorem}
    The equilibrium strategies of a Statistical game \SG{N,K_A,K_B,M,\gamma} in the $\gamma \to 0$ limit can be mapped to the symmetric equilibrium strategies of a Fisher game \G{N,K_A,K_B,M} \footnote{if $\nu^* \ne 0$} with the following identification:
    \begin{equation}
        \lim_{\gamma \to 0} p'^*_{\gamma,k^*} = \nu^*
    \end{equation}
        
    \begin{equation}
        \lim_{\gamma \to 0} P^*_\gamma = P^*_0
    \end{equation}

    Where $(P^*_\gamma,\{p'^*_{\gamma,k}\})$ are the equilibrium parameters of 
    \SG{N,K_A,K_B,M,\gamma},
    while $(k^*,\nu^*,P^*_0)$ are the equilibrium parameters of 
    \G{N,K_A,K_B,M}.
    
\end{theorem}

For the proof, see Appendix \ref{appendix:Unifiction}.

\subsection{Unification by generalized entropies}

\subsubsection{Expected Utility-based entropy}

\paragraph{Generalized entropy for probability vectors:}

\begin{equation}
\label{eq:HEU}
    \Hbar_\gamma(\underline{p}) = H^{\mathrm{EU}}_\gamma(\underline{p}) = 
    \frac{1- \left( \sum_k p_k^{1/\gamma} \right )^\gamma}{1-\gamma} = 
    \frac{1-||\underline{p}||_{1/\gamma}}{1-\gamma}
\end{equation}

\paragraph{Generalized conditional entropy:}

\begin{equation}
    H^{\mathrm{EU}}_\gamma[Y|X] = \mathbb{E}_{x \sim X} H^{\mathrm{EU}}_\gamma[Y|X = x]
\end{equation}

\paragraph{Generalized entropy and Expected Utility:}

The expected utility in eq. \eqref{eq:SGameExpectedUtility} can be expressed by the generalized entropy in the following way:

\begin{equation}
    U_\gamma(P) = -H^{\mathrm{EU}}_\gamma[\Pi|X]
\end{equation}

\subsubsection{Connection with other generalized entropies}

\begin{itemize}
    \item Connection with Rényi \cite{paper:RenyiOriginal,book:RenyiAczel} and Tsallis \cite{paper:TsallisOriginal,book:TsallisGellMann,book:Tsallis} entropy
    \begin{itemize}
        \item The order parameter $\alpha$ and deformation parameter $q$, can be identified with $1/\gamma$
    \end{itemize}
    \item In the limit $\gamma \to 0$, all 3 entropies give the same equilibrium strategy, meaning all could unify Bayesian and ``Frequentist'' statistics.
    \item However, the proposed entropy \eqref{eq:HEU} can be derived from expected utility theory with an isoelastic utility function. Its parameter has an intuitive interpretation, namely $\gamma$ representing the relative risk aversion of the deciding agent.
    \begin{itemize}
        \item Brave or reckless as $\gamma \to 0$, balanced when $\gamma \to 1$, and cautious, shy or anxious as $\gamma \to \infty$
    \end{itemize}
\end{itemize}

\paragraph{Rényi entropy:} \cite{paper:RenyiOriginal}

\begin{equation}
    H_\alpha(\underline{p}) = \frac{1}{1-\alpha} \log \left ( \sum_k p_k^\alpha \right )
\end{equation}

\paragraph{Conditional Rényi entropy:} \cite{paper:ConditionalRenyi}

\begin{equation}
    H_\alpha(Y|X) = \mathbb{E}_{x \sim X} H_\alpha(Y|X=x)
\end{equation}

\paragraph{Tsallis entropy:} \cite{paper:TsallisOriginal}

\begin{equation}
    S_q(\underline{p}) = \frac{\sum_k p_k^q - 1}{1-q}
\end{equation}

\paragraph{Conditional Tsallis entropy:} \cite{book:TsallisGellMann,paper:ConditionalTsallis}

\begin{equation}
    S_q(Y|X) = \frac{S_q(X,Y) - S_q(X)}{1+ (1-q) S_q(X)}
\end{equation}

\paragraph{Connecting with $\gamma$ parameter:}

\begin{equation}
    H^\mathrm{R}_\gamma(\underline{p}) = H_{\alpha = 1/\gamma}(\underline{p}), \quad
    H^\mathrm{T}_\gamma(\underline{p}) = S_{q=1/\gamma}(\underline{p})
\end{equation}

\begin{equation}
    H^\mathrm{R}_\gamma(Y|X) = H_{\alpha = 1/\gamma}(Y|X), \quad
    H^\mathrm{T}_\gamma(Y|X) = S_{q = 1/\gamma}(Y|X)
\end{equation}

\subsubsection{Simplest nontrivial example}
\label{sec:EntropiesSimplestPriors}

\begin{figure}[H]
    \centering
    \includegraphics[width=10 cm]{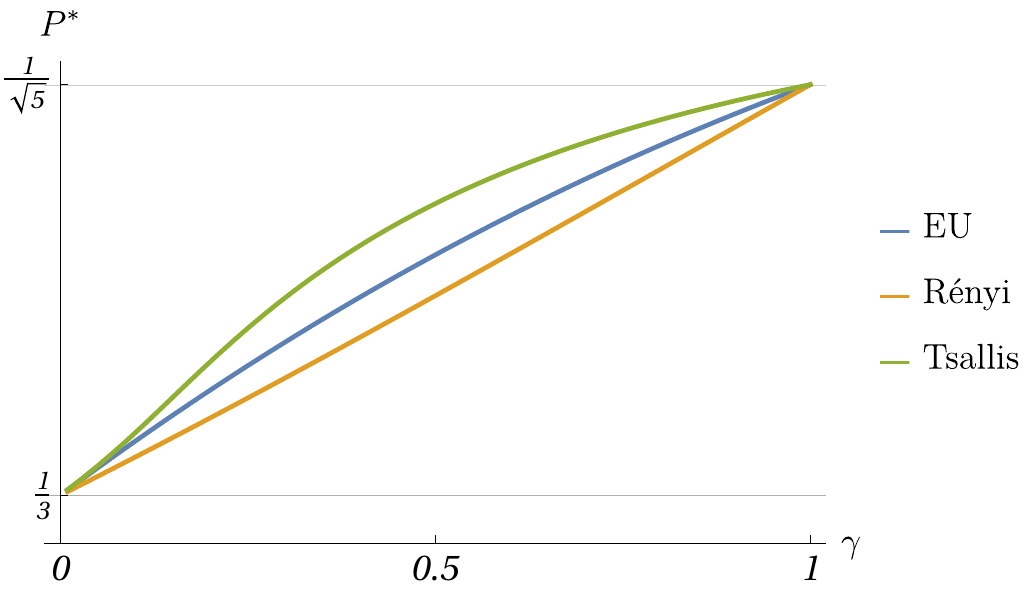}
    \caption{$P^*$ for a continuous betting game \CG{N=1,K_A=0,K_B=1,M=2} calculated by maximizing Expected Utility ($H^\text{EU}_\gamma(\Pi|X)$), Rényi entropy ($H^\text{R}_\gamma(\Pi|X)$) and Tsallis entropy ($H^\text{T}_\gamma(\Pi|X)$).}
    \label{fig:EURT_01}
\end{figure}

\begin{figure}[H]
    \centering
    \begin{subfigure}[b]{0.45\textwidth}
        \includegraphics[width=\textwidth]{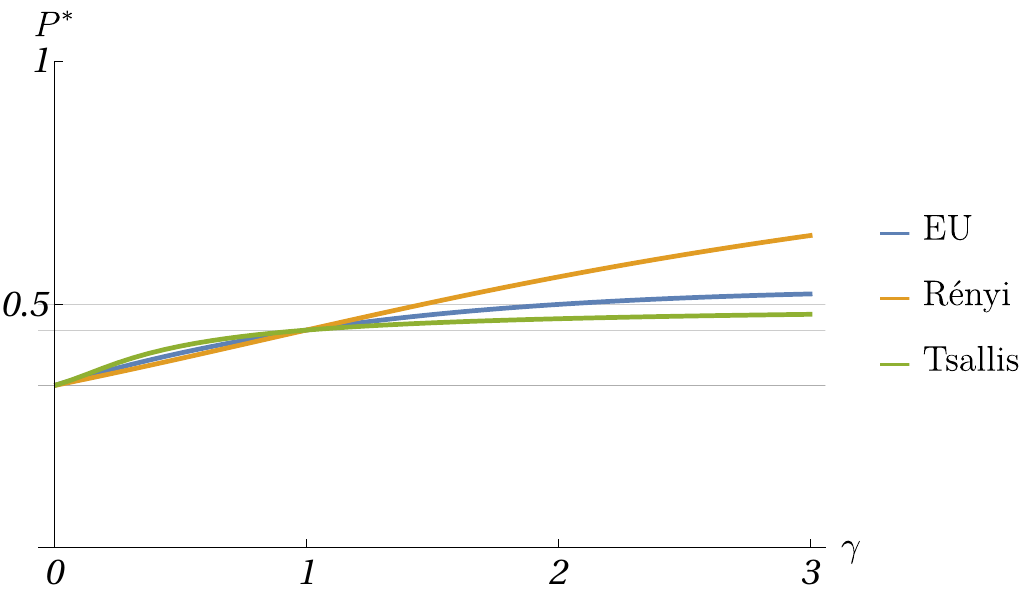}
        \caption{$\gamma \in (0,3]$}
        \label{fig:EURT_03}
    \end{subfigure}
    \hspace{0.01\textwidth} 
    \begin{subfigure}[b]{0.45\textwidth}
        \includegraphics[width=\textwidth]{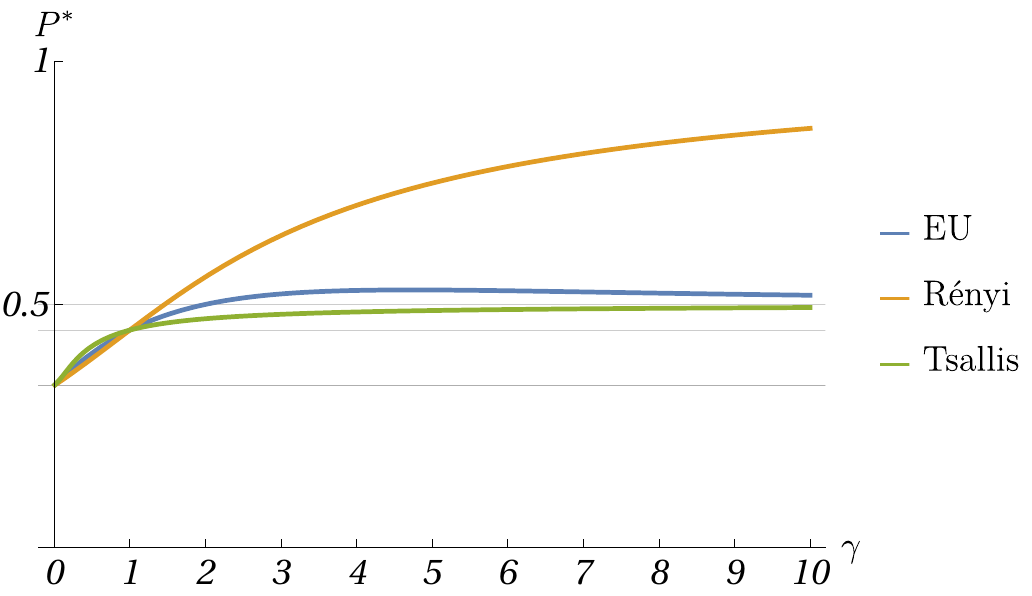}
        \caption{$\gamma \in (0,10]$}
        \label{fig:EURT_010}
    \end{subfigure}
    
    \caption{$P^*$ for a continuous betting game \CG{N=1,K_A=0,K_B=1,M=2} calculated by maximizing Expected Utility ($H^\text{EU}_\gamma(\Pi|X)$), Rényi entropy ($H^\text{R}_\gamma(\Pi|X)$) and Tsallis entropy ($H^\text{T}_\gamma(\Pi|X)$).}
    \label{fig:EURT_shy}
\end{figure}

\subsection{Generalized quantities}

\paragraph{Generalized divergence:}

\begin{equation}
    D_\mathrm{EU}^\gamma(\underline{p}||\underline{q}) =
    \frac{\gamma}{1-\gamma}
    \log
    \left (
    \sum_k q_k 
    \left (
    \frac{p_k}{q_k}
    \right )^{1/\gamma}
    \right )
    =
    \frac{\gamma}{1-\gamma}
    \log
    \left (
    \sum_k p_k 
    \left (
    \frac{p_k}{q_k}
    \right )^{1/\gamma-1}
    \right )
\end{equation}

Using this divergence, we can make an analogous statement to equation \eqref{eq:F1DKLbounds}, for $\Phi(.)$ (defined in \eqref{eq:PhiDef1}, \eqref{eq:PhiDef2}):

\begin{equation}
    \mathbb{G}_{n+1} = \Phi(\mathbb{G}_n),
    \quad \mathbb{G}_0 = \mathbb{R}
\end{equation}

\begin{equation}
    \mathbb{G}_1 = [
    -D_\mathrm{EU}^\gamma(\underline{p}(A)||\underline{p}(B)),
    D_\mathrm{EU}^\gamma(\underline{p}(B)||\underline{p}(A))
    ]
\end{equation}

Direct calculation shows that:

\begin{equation}
    \lim_{\gamma \to 1}
    D_\mathrm{EU}^\gamma(\underline{p}||\underline{q}) =
    D_{KL}(\underline{p}||\underline{q})
\end{equation}

\begin{remark}
    This expected utility-based divergence is related to the Rényi divergence \cite{arxiv:RenyiAndKLDivergence,paper:RenyiOriginal}

    \begin{equation}
        D_\mathrm{EU}^\gamma(\underline{p}||\underline{q}) =
        D_{\alpha=1/\gamma}(\underline{p}||\underline{q})
    \end{equation}

\end{remark}

\section{Game theoretical framework for statistics}
\label{sec:PhilosophiPart}

\subsection{Motivation}

``To us, probability is the very guide of life.'' \cite{book:Butler}
However, the exact interpretation of the concept is still debated among statisticians, philosophers, economists, and various other practitioners who wish to live by its guidance.
Probability theory and statistics are tremendously successful disciplines, providing a plethora of useful tools both for theoretical research and real-world problems.
Nevertheless, for many students and professionals, some techniques or the interpretation of results might cause unease because of the lack of a unified, coherent structure.

This work attempts to take a step back and look at the discipline with a fresh eyes.
It aims to identify the tasks in which probabilistic and statistical concepts might arise and approach these problems with a broader scope.
Hopefully, this more general approach will provide a philosophically simpler and more coherent framework in which most of the successful techniques and concepts of statistics can be derived and interpreted.
The present work collects the first and simplest building blocks and begins the construction of such a framework. It can hopefully provide tools in the form of concepts, formulas, theorems, computable algorithms, controlled approximations and executable code to empower those who genuinely care about making decisions in the face of uncertainty.

(For a more detailed and personal narrative about the questions which motivated this work, see Appendix \ref{Appendix:PersonalMotivation}.)

\subsection{Scope}

\subsubsection{Scope of the problem}

The present work attempts to discuss a more general problem than statistics itself, which is decision-making in the face of uncertainty.
The pressure of acting without being able to collect all necessary information is a fundamental part of any organism. Therefore, mechanisms and heuristics to collect and process data and form  actionable strategies accordingly are more ancient than abstract thinking and reasoning. \cite{book:EvolutionCognition}

Historically, many important problems and concepts of such decision problems came from economics, which can be viewed as the science and art of making choices under scarcity \footnote{and information can be a very real resource which is usually not fully available}.
In the first half of the 20\textsuperscript{th} century, various kinds of uncertainties were discussed in foundational works by Knight \cite{book:Knight} and Keynes \cite{book:Keynes}.
The concept of risk (where we can in some way associate probabilities to possibilities), is central in modern economics. However, the concept of Knightian uncertainty or ambiguity (in which case only the set of possibilities is known, but we have no further information which might help to associate probabilities to them) seems to fell out of favour \cite{book:JuliaUncertainty}.

It might be possible that there is no universal theory for such fundamentally uncertain situations, and only isolated heuristics can be identified, adopted by various agents.
However, I will argue that a surprisingly general framework can be built by adopting a small set of assumptions.

We have seen in previous sections that constructing a toy model in which decisions under uncertainty can be introduced and analysed does not require the external introduction of probabilities and stochastic variables. All these concepts can be born out from the setup of the dilemma (modelled by a game with deterministic rules) and are interpretable as instrumental intermediate concepts emerging in the construction of equilibrium strategies.

\subsubsection{Scope of this work}

For the sake of precise and rigorous mathematical statements, the scope of this present work has been limited to the simplest statistical problems, in which the fundamental concept of uncertainty and the proposed decision-making process can be demonstrated.
Thus the framework presented in the previous sections is far from complete.
There are various directions on how this work could be naturally extended. For an incomplete list of future directions, see Section~\ref{sec:FutureWork}.

The work builds mainly on the results of mathematical explorations. It can be viewed as a pursuit of coherence and philosophical simplicity rather than a new theory which explains and predicts experimental observations. \footnote{Although there are real-world phenomena in the realm of experimental economics, such as Ellsberg paradox \cite{paper:Ellsberg} related to Ambiguity aversion \cite{book:EconomicsDictionary,paper:GilboaSchmeidler}, which are hard to interpret by applying mainstream interpretations of probability, but could be a natural phenomenon in a game theoretic framework.}

A secondary aim of this work was an attempt to collect related concepts from diverse disciplines that might be understood with a refined definition of uncertainty.

Hopefully, this handful of ideas can serve as the foundation for broader joint work, which can restart the serious investigation of the framework and its applications.
From the point of view of developed methods applicable to real-life statistical problems, the current work is a humble achievement; however, hopefully, it can serve as a solid foundation demonstrating the consistency and applicability of the presented game theoretic framework.

\subsubsection{Problems outside of scope}

To avoid the false impression that the presented refined concept of uncertainty could encompass all unknowns, which real-life agents need to face, I explicitly spell out some limitations of the framework.

A major assumption that has been made is that the agent knows all possibilities of an unknown parameter of her environment. This is never a finite set in the real world when unpredictable, unexpected events can alter the environment and make previously imagined categorisation impossible.
\footnote{This could be relaxed in a way that we only assume that future consequences will be describable by arbitrary long strings. A future generalisation possibility is addressed in Section~\ref{sec:Solomonoff}; however, if rewards could take arbitrary values, then the framework seems to break down and fail to converge.}

Another general assumption is that agents know all the utilities for all possible future consequences. This is not realistic, both for complex or numerous consequences.
In particular, this means, for example that we assume that the agent's utilities do not change as new information has been gathered, actions have been made, or simply time elapsed.
There are numerous cases where these assumptions do not hold.

\subsection{Proposed framework}

The proposed framework is fundamentally built on Game Theory \cite{book:EssentialGameTheory,book:GameTheory,review:NeumannMorgensternGameThoery,book:GameTheoryOriginal}.
Not because of any ``rationality'' assumption but on the contrary, because this looks like the most natural framework to model the behaviour of adaptable agents, whose primary goal is to become and stay successful in their environment. \footnote{As in evolutionary game theory \cite{book:EvolutionaryGames,book:DarwinianDynamics}, where we do not assume ``rationally'' contemplated strategic behaviour from the agents, but a long evolutionary history can produce action patterns \emph{as if} it would come from a game theoretic reasoning.}
In this way, this is not primarily a prescriptive framework demanding some logical or other consistency but a partially descriptive theory that tries to draw general conclusions about agents and their behaviour, which successfully achieved their goals.

\paragraph{Experimenter and Nature:}
The ``Experimenter'' (in previous sections called \PI/) can face a dilemma, which can be characterized by possible actions, uncertain possible states (or parameter values) of ``Nature'', and the consequences of possible action-state pairs.
In the game-theoretic framework, the Experimenter is advised to take the following steps to determine her strategy in this situation:
Evaluate her utilities for all consequences, and form a strategy \emph{as if}, the unknown parameters had been chosen by an animated ``Nature'' (in previous sections \PII/ played this role), whose goal is to maximize the Experimenters expected regret \footnote{for a given state of ``Nature'', and action of Experimenter, the regret of the Experimenter is the difference between the maximal utility which could have been achieved -- by the best action for that state -- and the actually achieved utility}.
The game theoretic framework suggests the strategy to the Experimenter, which is equivalent to the equilibrium strategy in the previously constructed imaginary game. \cite{paper:Milnor,book:Wald,book:Savage}

\subsubsection{Normative attitude}

The framework does not aim to be prescriptive but suggestive in a way that it is slightly more ambitious than a descriptive theory and dares to suggest a default choice among possible policies.

\subsubsection{Alternative frameworks}

Acting in realistic environments is inherently an art and not a task which could be carried out by following a formal prescription.
Expected utility theory \cite{plato:ExpectedUtility} is not the only normative framework \cite{sep:NormativeTheories}. This is important to point out to avoid the possibility of a dogmatic solidification of the theory.
    \begin{itemize}
        \item Heuristics in the context of bounded ``rationality'' \cite{sep:BoundedRationality}, also know as behavioral economics \cite{paper:Tversky,book:Tversky}
        \item Imprecise probabilities \cite{paper:GilboaSchmeidler,book:GilboaDecision,sep:ImpreciseProbabilities}
        \item Heuristics in mathematics \cite{paper:PolyaHeuristics,book:GoodThinking}
        \item Algorithmic probability theory \cite{book:Solomonoff,book:Vitanyi,arxiv:MullerSolomonoff}
        \item Berge equilibrium \cite{book:BergeEquilibrium}
        \item Change of the definition of the agent, i.e. identity constructions
        \item Hierarchical constructions allow different types of opponents, while the type (being an unknown parameter) could be chosen by a ``higher level'' player.
        Similar to Bayesian hierarchical models \cite{book:BayesianDataAnalysis}
        \item Yet unknown or unimaginable construction (for example, a formalization of ``wise'' Optimism)
    \end{itemize}

\subsubsection{History of the approach}

For historical background see a quote from the textbook {\it Theoretical statistics} by D.R. Cox and D.V. Hinkley:

``Many of the basic ideas of statistical decision theory were stated by
Neyman and Pearson (1933b). A systematic theory for the situation
when there is no prior distribution was developed by Wald and set out
in a book (Wald, 1950) completed just before his death; this included
a formulation of common statistical problems in decision theory
terms. Important contributions continuing Wald’s work are by
Girshick and Savage (1951) and Stein (1956 and unpublished
lecture notes). A detailed introduction is given by Ferguson (1967).
This work lead to the view, widely held for a period, that all statistical
problems should be formulated in terms of decision making.
von Neumann and Morgenstern developed an axiomatic approach
to subjective utilities in connexion with their work on game theory
(von Neumann and Morgenstern, 1953, 1st ed. 1935); see DeGroot
(1970), Fishburn (1969) and, for an elementary account, Chernoff
and Moses (1959).
In the late 1950’s and early 1960’s theoretical interest shifted from
the situation without prior distributions to Bayesian decision theory
in which personalistic prior distributions and utilities are central to
the argument; an authoritative account was given by Raiffa and
Schlaifer (1961). There are now a number of excellent elementary
introductions (Aitchison, 1970a; Lindley, 1971b; Raiffa, 1968).
There is also an extensive economic, psychological, sociological and
operational research literature on decision making. A review on the
practical measurement of utilities is given by Hull, Moore and Thomas'' \cite{book:CoxStatistics}

\section{Future work and extension}
\label{sec:FutureWork}

\includegraphics[width=0.15\textwidth]{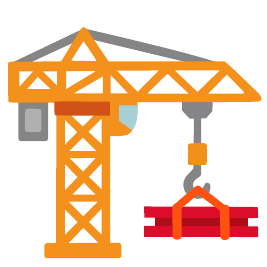}
\includegraphics[width=0.15\textwidth]{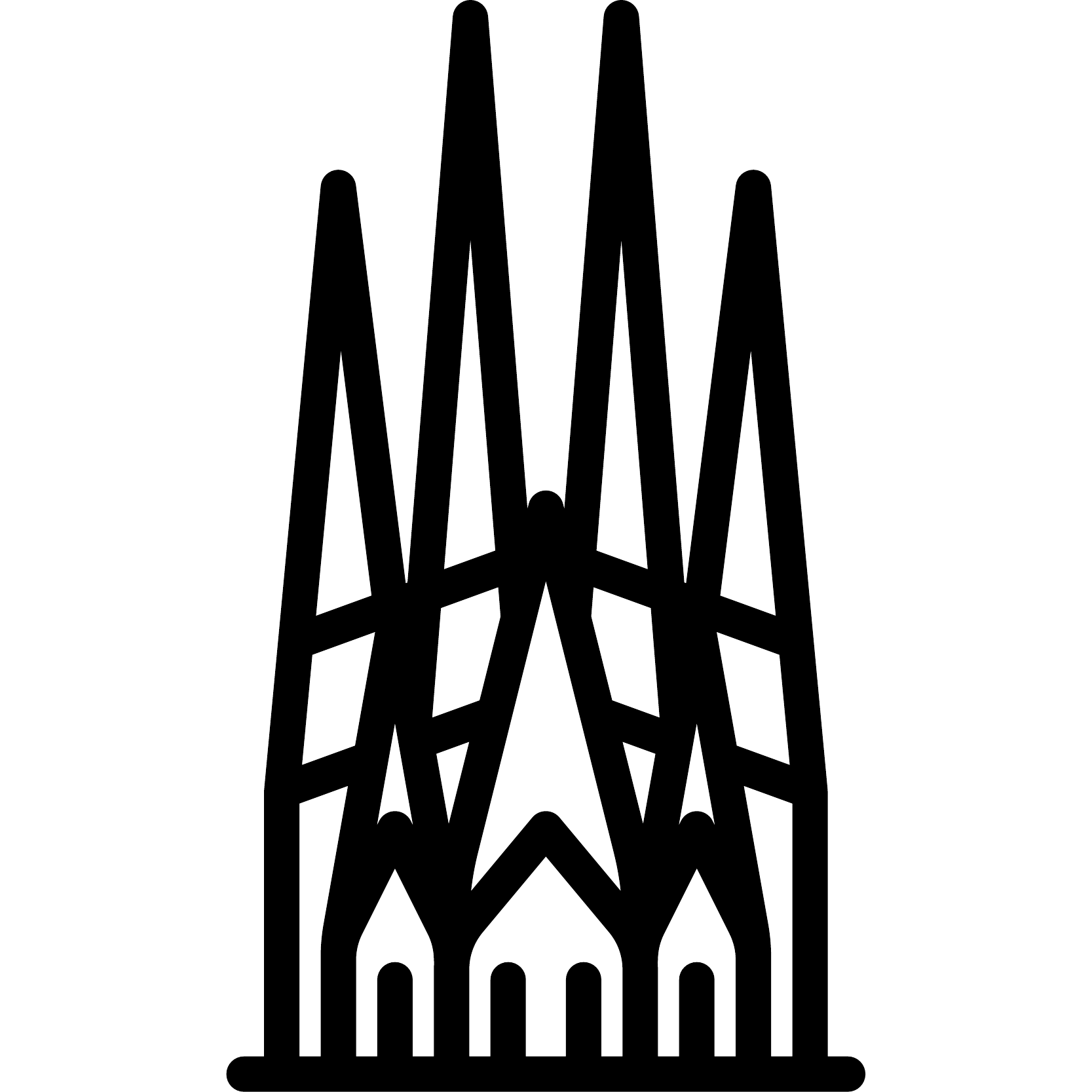}

\subsection{Assumptions about the unknown}

When we face uncertainty, there is an important metaphysical question: what assumptions do we make about the unknown?

In this work, we considered only ``Balanced games'', where the utility of guessing correctly does not depend on the scenario (A or B).
However, in a general dilemma -- where all 4 possible choosing-guessing pairs can have different utilities -- different assumptions can lead to different behaviour. In the following I list four different potential metaphysical constructions, each yielding different prototypical games for such a dilemma:

\begin{itemize}
    \item \begin{CJK*}{UTF8}{bsmi}\bf 鬼\end{CJK*} \footnote{Demon, evil spirit or ghost. For context see \href{https://humanum.arts.cuhk.edu.hk/Lexis/lexi-mf/search.php?word=\%E9\%AC\%BC}{Chinese Character Database}, \href{https://en.wiktionary.org/wiki/\%E9\%AC\%BC}{Wiktionary} or \cite{book:ChineseSymbols}}:
    $U_{1,2} = (U_1, -U_1)$
    Assuming that the unknown is a ``malevolent demon'', playing a zero-sum game with us. \cite{book:Savage,book:CoxStatistics}
    \item \Pisces \ \footnote{Pisces, symbolizing Alchemical final transmutation (Projection) \cite{book:Alchemy}. (The analogy came from my impression that changing an agent's utility function is hard, such as changing \href{https://www.scientificamerican.com/article/fact-or-fiction-lead-can-be-turned-into-gold/}{base metals to gold}, or mental transmutation.)}:
    $U_{1,2} = (-R_1, R_1)$
    This construction assumes that we can somehow change our utility function to regret \cite{book:EssentialGameTheory}, and then the unknown is involved in a zero-sum game not regarding our utility but regarding our regret.
    \item $\mathghost$ \footnote{Ghost or trickster, which is not motivated to harm, but to spook and cause unease. Because, in many cases, the details of an imagined ghost's equilibrium strategy can not be directly observed, the naming has been inspired by the Faddeev-Popov ghosts \cite{book:PeskinSchroeder} in Quantum Field Theory}:
    $U_{1,2} = (U_1, R_1)$
    Assuming that we are playing a non-zero-sum game with the unknown, which is not trying to minimise our utility but to maximise our regret. \cite{book:Savage,book:CoxStatistics}
    \item  \hexacube{} \footnote{symbolising a dice, which were manufactured and used from prehistoric times. 2500 - 1900 BC \href{http://jameelcentre.ashmolean.org/collection/4/6739/6741/11231}{Mohenjo-daro}, 3100 - 2400 BC \href{https://nms.scran.ac.uk/database/record.php?usi=000-100-040-457-C}{Orkney}, 3550 – 2300 BC \href{https://www.parstimes.com/history/oldest_backgammon.html}{Burnt City} \cite{paper:BurntCityDice}}:
    $(U_1, -)$
    This is essentially the Bayesian framework, in which we assume that regardless of our utilities, the unknown can be viewed as a stochastic actor, choosing actions with prior probability $\pi(P)$.
\end{itemize}

I will argue that out of these candidates, the Utility-Regret ($\mathghost$) non-zero-sum game seems to have the most desirable properties.

\paragraph{Further possibilities:}

\begin{itemize}
    \item $\odot$: $U_{1,2} = (U_1'$, $R_1')$ A ``wisely optimistic'' \cite{book:OptimismPessimismPsychology} metaphysical construction, which might be based on some kind of relatedness or sharing between the agents, and the forces influencing the uncertain parameters. 
    \item {\bf ?}: yet unexplored possibilities. 
\end{itemize}

\subsection{Target of the Inference}

It can be shown that the equilibrium of such statistical games depends on the target of our inference, i.e. the unknown quantities we are interested in.
Regarding this aspect of statistical games, I mention 3 main categories:

\begin{itemize}
    \item ``Platonian'' Inference, focusing and betting on parameters
    \begin{itemize}
        \item This current work can be viewed as a purely ``Platonian'' Inference.
        \item In the case of a continuous parameter space, further complications can occur, such as a default measure or the ``density of states'' on the parameter space.
        \item the ``Platonian'' framework seems to be closer to previous works on Objective or Non-informative priors such as the Reference prior \cite{paper:ReferencePrior,book:Bernardo}
    \end{itemize}
    \item ``Aristotelian'' Inference, focusing and betting on future observable data
    \begin{itemize}
        \item For example, the work of R. L. Kashyap \cite{Kashyap1974, Kashyap1971}
        \item In this framework, not only priors and probabilities can be considered imaginary, but models as well. In this context, a model can be viewed as an abstract construction connecting collected data with potential future data.
        \item In some cases, ``Platonian'' Inference can be viewed as the limiting case of an ``Aristotelian'' Inference.
    \end{itemize}
    \item General Inference: betting on a general Target space.
    \begin{itemize}
        \item A natural example might be the case when we can only observe indirect data to our interests.
        \item A further example might be data compression and the choice of coding protocols.
    \end{itemize}
\end{itemize}

\subsection{Correspondence with Bayesians and Frequentists}

The correspondence principle \cite{sep:Correspondence,book:BohrCorrespondence}, or more specifically, the general correspondence principle \cite{sep:Correspondence} requires that a more general scientific framework has to incorporate a more specific previous framework:

``The most important heuristic restriction is the General Correspondence Principle. Roughly speaking, this is the requirement that any acceptable new theory L should account for the success of its predecessor S by ‘degenerating’ into that theory under those conditions under which S has been well confirmed by tests.'' \cite{book:CorrespondencePost}

Decision theory might even provide natural quantitative tools by which theories can be compared. A natural measure of divergence of policies could be the cost for which an agent with different preferences (and thus different optimal policy) is willing to adopt another policy.

\begin{itemize}
    \item ``Delegation premium'' / ``Independence tax''
    \begin{itemize}
        \item How much tax would a gambler be willing to pay for being able to play her own optimal policy instead of adopting another decision-making policy?
    \end{itemize}
    \item Policy-(Believes/Preferences/Priors) acceptance matrix
\end{itemize}

\subsection{Effective approximative numerical methods}

\begin{itemize}
    \item Finding an effective approximative numerical method such as Markov Chain Monte Carlo (MCMC) \cite{book:MCMC}
    \begin{itemize}
        \item MCMC can be combined with Blahut–Arimoto algorithm \cite{paper:MCMCBlahutArimoto,arxiv:MCMCBlahutArimoto}, which, with some modifications, might yield a general and scalable stochastic approximation for the equilibrium prior.
    \end{itemize}
\end{itemize}

\subsection{Embedding into larger game Contests}

In Appendix \ref{appendix:UtilityFunctions}, we introduced a Contest view (Section~\ref{sec:ContestView}), and we were able to argue that in this individualistic and competitive framework, the logarithmic utility function is aligned with evolutionary fitness.

However, when we defined Bayesian games, we independently assumed that it is zero-sum.
This assumption could be relaxed by embedding the individual game into a competitive contest, where, among other gamblers, one insider can not only play but influence the model parameters.
Such an extended individualistic and competitive framework might naturally align with Bayesian games and generalise well for unbalanced games.

\subsection{Generalizing Game Theory with a player representing uncertainty}

If multiple ``rational'' agents face each other and uncertainty simultaneously, then the uncertain parameters could be modelled by a ``regret maximising'' player.

A natural question might be: Whose regret ought to be maximised? A somewhat unusual answer might be that all rational players can project their own personal ghosts, maximising their own personal regrets, resulting in their own personal priors for the uncertain parameters.
(Without mentioning these weakly inconsistent priors) this framework has been suggested in the context of Bayesian games and types \cite{arxiv:GameTheoryRegret,paper:GameTheoryRegret}. However, the concept has the potential to be further generalised.

A simple toy example could be a battle situation, where two opposing generals have to plan their strategies, while their positions might be weather-dependent in different way. Suppose the generals have absolutely no information about the uncertain weather. In that case, both may prepare more for different weather conditions, even if they know they will face the same reality on the day of battle.

\subsection{Using the concept in Reinforcement Learning}

A toy model in Reinforcement Learning \cite{book:RL}, which could be analysed in depth using this framework, might be a ``one-armed bandit'' type problem \cite{book:BanditBook}:

\begin{itemize}
    \item The bandit has one stochastic arm, which can give $\{-1,1\}$ reward.
    \begin{itemize}
        \item if it is a type A bandit, we gain $1$ with $p_A = 20\%$ chance and $-1$ with $1-p_A=80\%$ chance,
        \item if it is a type B bandit, we gain $1$ with $p_B=80\%$ chance and $-1$ with $1-p_B=20\%$ chance.
    \end{itemize}
    \item In every round, we can decide if we pull the stochastic arm or refuse to play (in which case we gain a $0$ reward).
    \item We have no prior knowledge about the type of the bandit $\theta \in \Theta = \{A, B\}$, but we can assume a game theoretical setup to find optimal policies for such a gamble.
\end{itemize}

\subsection{Foundations of Statistical Physics}

Statistical physics is a tremendously successful branch of physics, even if its foundations are debatable (and sometimes are debated).
I would mention two main frameworks in which this branch of physics is usually introduced:

\begin{itemize}
    \item Counting microstates and introducing ensembles
    \begin{itemize}
        \item Boltzmann's principle and counting of microstates \cite{book:Greiner,book:PathriaStatisticalMechanics,book:StatisticalPhysicsBerkeley,book:Sommerfeld}
        \item Gibbs entropy and H-Theorem \cite{book:WeinbergModernPhysics}
        \item Boltzmann’s law \cite{book:Feynman,book:FeynmanStatisticalMechanics}
    \end{itemize}
    \item Ergodicity in dynamical systems \cite{book:Krylov, book:ErgodicTheory, book:Landau}
\end{itemize}

A game theoretic foundation in a nutshell would look the following:
We define a statistical physics game where \PI/ needs to bet on (a possibly coarse-grained partition of) the state of the system, but she can choose the time (or multiple instances) of observation.
\PII/, on the other hand, can choose an initial condition for the system (possibly with some constraints).

A natural conjecture is that under specific properties of the dynamical system (ergodicity, chaos, topological mixing, etc.) and in a specific limit (the observation time interval goes to infinity), the equilibrium strategy of \PII/ choosing initial states will converge to the invariant measure.
While the equilibrium betting strategy of \PI/ for finding the system in specific domains will be computable by an ensemble average for this invariant measure.

In this framework, we would not need to say that without knowing the initial condition, we somehow deduced some properties of the system after the ``thermalisation'' time \footnote{``although the motion of systems with a very large number of
degrees of freedom obeys the same laws of mechanics as that of systems consisting of a small number of particles, the existence of many degrees of freedom results in laws of a different kind.'' \cite{book:Landau}}. However, we could say that if we can freely pick the time of observation, then because of certain properties of a dynamical system, we can come up with a suitable strategy by which we can bet on the value of some observable quantities.

\subsection{Quantum metrology}

The concept can be generalised to find an optimal set of measurements for a quantum state with unknown parameters.

A natural generalisation of Bayesian games is to use the Quantum relative entropy \cite{book:QuantumComputationAndInformation,book:NeumannQuantum} (between the chosen density matrix by \PII/ and the estimated density matrix by \PI/) to quantify the result of our estimate.
(In spirit this would be very similar to the game suggested by Kashyap for classical bits \cite{Kashyap1971,Kashyap1974}.)

\subsection{Universal Inductive Inference}
\label{sec:Solomonoff}

We can also formulate the framework as a data compression game.
Imagine the following game:
First, \PII/ chooses an arbitrary $M$-long binary sequence $D$.
\PI/ can take an $N$-long sample from the sequence
and propose a coding/compressing protocol for the whole $M$-long sequence based on the gathered data.
To describe the coding/compressing protocol $C_1$, she can use a  Universal Turing Machine, i.e. her code will be represented as a $|C_1|$-long input tape.

After this, the whole $M$-long binary sequence will be compressed by the protocol suggested by \PI/, resulting in a $|C_1(D)|$ long string. To be able to decompress the data, both the protocol and the compressed data have to be sent or saved; therefore, the real length of the encoding is $L_1 = |C_1| + |C_1(D)|$. \footnote{In this way, the framework is related to the Minimum description length principle \cite{book:MinimumDescriptionLength,arxiv:MinimumDescriptionLengthRevisited} as well.}

If \PI/ would know data $D$ in advance, she could${}^*$ find the shortest coding protocol $C^D$ for the data, which simply outputs $D$ for an empty input.
The length of this optimal code is, by definition, the data's Kolmogorov complexity \cite{book:Vitanyi} $K(D)=|C^D|$.

If we assume that \PI/'s utility function is the length of her compression, then her regret is the difference between her protocol's compression length and the optimal compression length, i.e. the Kolmogorov complexity of the data:
$\Delta L = |C_1| + |C_1(D)| - K(D)$.

If we assume such a game between \PI/ and \PII/, in which \PII/ tries to maximise \PI/'s regret, then the equilibrium coding protocols can be linked to probabilities very similar to those which would be suggested by the Solomonoff induction formula \cite{book:Solomonoff,book:UniversalArtificialIntelligence}.

\subsection{Poisson games}

Limiting case, where all $N, K_A, K_B, M \to \infty$, while:

\begin{equation}
    \lim_{i \to \infty} N_i \frac{K_{A,i}}{M_i} = \lambda_A \in \mathbb{R}_+, \quad \lim_{i \to \infty} N_i \frac{K_{B,i}}{M_i} = \lambda_B \in \mathbb{R}_+
\end{equation}

\subsection{Connection with other frameworks}

\begin{itemize}
    \item Classical probability \cite{book:Laplace}
    \item Frequentist \cite{book:VonMises}
    \item Bayesian
    \begin{itemize}
        \item Subjective \cite{book:deFinetti}
        \item Objective \cite{book:Jaynes}
    \end{itemize}
    \item Axiomatic theory of probability \cite{book:Kolmogorov,book:Klenke}
    \begin{itemize}
        \item Probabilistic proof methods \cite{book:ProofsFromTHEBOOK,book:TheProbabilisticMethod}
    \end{itemize}
    \item Decision theory \cite{book:DecisionTheory,book:GilboaDecision,book:IntroductionToDecisionTheory,paper:DecisionTheoryAnOverview}
    \item Other related fields
    \begin{itemize}
        \item Control theory \cite{book:Control}
        \item Algorithmic probability theory \cite{book:Solomonoff,book:Vitanyi,book:UniversalArtificialIntelligence}
        \item Operations Research \cite{book:OperationsResearch,book:IntroductionOperationsResearch}
    \end{itemize}
\end{itemize}

\section{Acknowledgement}

I am deeply grateful to Anita L. Verő, who helped and supported me and was my partner also in writing and discussing the material. It is difficult to overstate her supportive, motivating and creative role in the birth of this publication.

I want to thank all my friends, colleagues and family for their support and for having the patience to listen to my often unrefined thoughts on the topic.

I am especially thankful to Balázs Gimes, who patiently allowed me to articulate my initial, raw and disorganised thoughts on the subject. Through his questions and attentive listening, he helped me transform these initially nebulous concepts into the form that led to this work.
I want to thank György Fehér for initial coaching an support.\footnote{I need to mention László Ujfalusi, with whom we discovered a ``lost Greek letter'' $\sampi$ many years ago \cite{personal:Ujfalusi}, and committed ourselves to use and popularize it for scientific notation. This is the main reason the symbol has been used in \ref{conj:AsymptoticExpansionSubleadingTerm}.}
I want to thank Zoltán Szabó for his help with the publishing process, as well as for his comments and encouragement that contributed to the final appearance of the work.

To my great surprise, in 2016, R. L. Kashyap replied to me about some details of one of his papers published in 1974. At this point, I would like to express my gratitude for his help and kindness.

Discussions with Roger Germundsson at Wolfram Research, helped me reviewing the field of game theory in depth, and learning about high quality mathematical software design. 

I want to express my deep gratitude to all readers of this relatively long and sometimes technical work. I do believe that a substantial part of a theoretical work's value comes from the readers and the ideas emerging from the interaction. If this work can inspire anybody to think deeper about probability, uncertainty, statistics or decision-making by reading, skimming or appreciating mathematical details, that brings me great joy and gratification.

I am thankful to all my teachers who motivated, challenged, helped and supported me by sharing their knowledge and providing opportunities.

Finally, I would like to thank my former supervisor, Zoltán Bajnok, who effectively demonstrated the value and usefulness of theoretical toy models. Finding the smallest and simplest examples, which show nontrivial behaviour, contributed to a great many results and conjectures collected in this work.

\newpage

\appendix

\section{Proof of the symmetric equilibrium}

\label{Appendix:FisherGameEquilibrium}

The proof belongs to Theorem~\ref{thm:Symmetric}.

\begin{proof}
\label{proof:Symmetric}
To prove that this is indeed a Nash equilibrium, we need to show that:

\begin{itemize}
    \item For \PI/:
    \begin{itemize}
        \item \PI/ can expect the same winning rate $v$ for all actions which are present with non-zero probability in her strategy. (The expectation is calculated with respect to \PII/'s mixed strategy.)
        \item For all actions excluded from the equilibrium strategy, the expected utility can not be greater than $v$.
    \end{itemize}
    \item \PII/ can expect the same winning rate for all possible sequences (this is because \PI/'s mixed sequence choosing strategy contains all possible sequences).
\end{itemize}

Formally:

\begin{equation}
    \label{eq:1equal}
    \E_{\alpha_2 \sim \sigma_2} \left [ u_1[a_1,\alpha_2] \right ] = v, \quad 
    \forall a_1 \in \supp(\sigma_1) = \{ a_1 \in \mathcal{A}_1 | \sigma_1(a_1) > 0 \}
\end{equation}

\begin{equation}
    \label{eq:dominated}
    \E_{\alpha_2 \sim \sigma_2} \left [ u_1[a_1,\alpha_2] \right ] \le v, \quad 
    \forall a_1 \in \supp(\sigma_1)^{\complement} = \{ a_1 \in \mathcal{A}_1 | \sigma_1(a_1) = 0 \}
\end{equation}

Since we assume a zero-sum game, the winning rate of \PII/ is proportional to the winning rate of \PI/.

\begin{equation}
    \label{eq:2equal}
    \E_{\alpha_1 \sim \sigma_1} \left [ u_1[\alpha_1,a_2] \right ] = v, \quad 
    \forall a_2 \in \supp (\sigma_2) = \mathcal{A}_2
\end{equation}

It is easy to see that because of the randomized sampling of \PI/, the expected utility for \PII/ depends only on the choice of A or B and does not depend on which sequence has been chosen from $\mathcal{K}_A$ or $\mathcal{K}_B$.

Similarly, because of the mixing of \PII/, \PI/'s utility is independent of the concrete choice of the sampling.

Some useful notation:
\begin{equation}
\label{eq:FisherNotation_p*AB}
    p_{k^*}(A) = p^*_A, \quad p_{k^*}(B) = p^*_B
\end{equation}

\begin{equation}
\label{eq:FisherNotation_SigmaAB}
    \sum_{k < k^*} p_k(A) = \Sigma_A, \quad \sum_{k > k^*} p_k(B) = \Sigma_B
\end{equation}

First, we show that if \PI/ would change her policy not to guess B when $k=\ell > k^*$, then she would decrease her chance of winning.

This is equivalent to the statement:

\begin{equation}
    \label{le:PpA>QpB}
    \ell > k^* \implies P^* \ p_\ell(A) < (1-P^*) \  p_\ell(B)
\end{equation}

\begin{equation}
    \ell > k^* \implies  \frac{p_\ell(A)}{p_\ell(B)} < \frac{1-P^*}{P^*}
\end{equation}

To prove this, we show that if $k, k-1 \in \mathbb{K}_{AB}$:

\begin{equation}
    \label{le:pa/pb}
    \frac{p_{k}(A)}{p_{k}(B)} < \frac{p_{k-1}(A)}{p_{k-1}(B)}
\end{equation}

\begin{equation}
    \frac{p_{k}(A)}{p_{k-1}(A)} \frac{p_{k-1}(B)}{p_{k}(B)} < 1
\end{equation}

By direct simplification of the binomial factors in Theorem~\ref{thm:SymEqHypergeom}, we get:

\begin{equation}
    \frac{(K_A-k+1) (k-K_B+M-N)}{(K_B-k+1) (k-K_A+M-N)} < 1
\end{equation}

If $k \le K_B$ and $N-k < M-K_A$

\begin{equation}
    (K_B - K_A)(M-N+1) > 0
\end{equation}

This is true for all $K_B>K_A$ and $M \ge N$ cases.

From the definition of $P^*$ we see that:

\begin{equation}
    \label{eq:(1-P)/P}
    \frac{1-P^*}{P^*} = \frac{p_{k^*}(A)}{p_{k^*}(B)}
\end{equation}

Equation \eqref{eq:(1-P)/P} together with equation \ref{le:pa/pb} proves \ref{le:PpA>QpB}.

With essentially the same line of reasoning, we can see that:

\begin{equation}
    \label{le:PpA<QpB}
    \ell < k^* \implies P^* \ p_\ell(A) > (1-P^*) \  p_\ell(B)
\end{equation}

This means that if \PI/ would change her policy not to guess A when $k=\ell < k^*$, then she would decrease her chance of winning.
This proves requirement \eqref{eq:dominated}.

After this point, the proof can be split into two cases:

\begin{itemize}
    \item the case, when $\nu^* > 0$,
    \item and the case when $\nu^* = 0$.
\end{itemize}

\paragraph{Mixed policy i.e. $\nu^*>0$:}

\PI/ is mixing between two policies. We will call the policy Up when she guesses A if $k=k^*$, and we will call Down if she guesses B if $k=k^*$.

It is useful to introduce 4 possibly different values:

\begin{align}
    \hat{v}_A   &= \sum_{k \le k^*} p_k(A) = \Sigma_A + p^*_A, & \hat{v}_B   &= \sum_{k > k^*} p_k(B) = \Sigma_B \\
    \check{v}_A &= \sum_{k < k^*} p_k(A) = \Sigma_A,           & \check{v}_B &= \sum_{k \ge k^*} p_k(B) = \Sigma_B + p^*_B
\end{align}

To prove requirement \eqref{eq:1equal} we need to show that the expected winning rates for both Up and Down strategies are equal:

\begin{equation}
    \hat{v} = P^* \ \hat{v}_A + (1-P^*) \ \hat{v}_B
\end{equation}

\begin{equation}
    \check{v} = P^* \ \check{v}_A + (1-P^*) \ \check{v}_B
\end{equation}

Taking their difference results:

\begin{equation}
    \begin{split}
        \hat{v} - \check{v} &= P^* (\hat{v}_A - \check{v}_A) + (1-P^*) (\hat{v}_B - \check{v}_B) \\
                            &= P^* \ p^*_A - (1-P^*) p^*_B \\
                            &= \frac{p^*_B}{p^*_A+p^*_B} \ p^*_A - \frac{p^*_A}{p^*_A+p^*_B} \ p^*_B \\
                            & = 0
    \end{split}  
\end{equation}

The final step is to prove requirement \eqref{eq:2equal}. In order to do this, we need to show that the chance of winning the game for \PI/ is the same if \PII/ chooses A or B.

\begin{equation}
\label{eq:vAnu}
    v_A = \nu^* \ \hat{v}_A + (1-\nu^*) \ \check{v}_A = \check{v}_A + \nu^* (\hat{v}_A - \check{v}_A) = \check{v}_A + \nu^* \ p^*_A
\end{equation}

\begin{equation}
\label{eq:vBnu}
    v_B = \nu^* \ \hat{v}_B + (1-\nu^*) \ \check{v}_B = \check{v}_B + \nu^* (\hat{v}_B - \check{v}_B) = \check{v}_B - \nu^* \ p^*_B
\end{equation}

Recalling expression \ref{thm:SymEqNu} for $\nu^*$, and using the new notation:

\begin{equation}
    \nu^* = \frac{\Sigma_B + p^*_B - \Sigma_A}{p^*_A + p^*_B}
\end{equation}

\begin{equation}
    v_A = \Sigma_A + \frac{\Sigma_B + p^*_B - \Sigma_A}{p^*_A + p^*_B} \ p^*_A
\end{equation}

\begin{equation}
    v_A = \frac{p^*_B \Sigma_A + p^*_A \Sigma_B + p^*_A \ p^*_B}{p^*_A + p^*_B} 
\end{equation}

\begin{equation}
    v_B = \Sigma_B + p^*_B - \frac{\Sigma_B + p^*_B - \Sigma_A}{p^*_A + p^*_B} \ p^*_B
\end{equation}

\begin{equation}
    v_B = \frac{p^*_A \Sigma_B + p^*_B \Sigma_A + p^*_A \ p^*_B}{p^*_A + p^*_B}
\end{equation}

\begin{equation}
    v^* = v_A = v_B
\end{equation}

\paragraph{Pure policy i.e. $\nu^*=0$:}

Based on \ref{thm:SymEqNu}, $\nu^*$ can be $0$, only if there is a $k^*$, for which:

\begin{equation}
    \sum_{k \geq k^*} p_k(B) - \sum_{k < k^*} p_k(A) = 0
\end{equation}

or 

\begin{equation}
    \label{proof:SymEqnu=0}
    \Sigma_B + p^*_B - \Sigma_A = 0
\end{equation}

In this case, $P^*$ is not fixed by an equality but only bounded by an inequality:

\begin{equation}
    \begin{split}
    P^* \ p_{k^*}(A)   &\le (1-P^*) \ p_{k^*}(B) \\
    P^* \ p_{k^*-1}(A) &\ge (1-P^*) \ p_{k^*-1}(B)
    \end{split}
\end{equation}

\begin{equation}
    \frac{p_{k^*-1}(B)}{p_{k^*-1}(A)+p_{k^*-1}(B)} \le P^* \le \frac{p_{k^*}(B)}{p_{k^*}(A)+p_{k^*}(B)}
\end{equation}

The choice in Theorem~\ref{thm:Symmetric} satisfies this inequality. (However, the proof shows that there are other choices which would give an equilibrium in this case.)

\PI/ is not mixing. Therefore, there is no requirement for her indifference about the policies  $\hat{\phi}$ and $\check{\phi}$. When she adopts a strategy with $\nu^* = 0$, she uses only $\check{\phi}$.

Therefore, the only further requirement is that \PII/ has to be indifferent about choosing A or B:

\begin{equation}
    v_A = v_B
\end{equation}

\begin{equation}
    \check{v}_A = \check{v}_B
\end{equation}

\begin{equation}
    \Sigma_A = \Sigma_B + p^*_B
\end{equation}

Which is satisfied because of equation \ref{proof:SymEqnu=0}.

\end{proof}

\subsubsection{Intermediacy lemma}
\label{sec:IntermediacyMedian}

The following proof belongs to the intermediacy lemma of the median \ref{lemma:IntermediacyMedian}:

\begin{proof}

    Let $\xi$ and $\eta$ be real-valued random variables, and $\zeta = \lambda \xi + (1-\lambda) \eta$ be their mixture with $\lambda \in [0,1]$.

    We can define positive and negative $p$–quantile \cite{book:DeGrootProbabilityAndStatistics,book:IntroToProbability} tail sets for real-valued random variables in the following way:

    \begin{equation}
        \mathbb{T}^{-}_p[\xi] = 
        \{ x \in \mathbb{R} \ | \ 
        \Pr(\xi \le x) < p
        \}
    \end{equation}

        \begin{equation}
        \mathbb{T}^{+}_p[\xi] = 
        \{ x \in \mathbb{R} \ | \ 
        \Pr(\xi \ge x) < p
        \}
    \end{equation}

    In general, the definition of the $p$-quantiles and the median (which can be viewed as the $0.5$-quantile) does not guarantee a unique value; therefore, it is better to define them as a set-valued function:

    \begin{equation}
        c \in \mathbb{M}_p[\xi]
        \iff
        \Pr(\xi \le c) \ge p
        \ \wedge \  
        \Pr(\xi \ge c) \ge p
    \end{equation}

    The $p$–quantile set-valued function can be expressed by the tail sets in the following way:

    \begin{equation}
        \mathbb{M}_p[\xi] = 
        \mathbb{R} \setminus
        (\mathbb{T}^{-}_p[\xi] \cup
        \mathbb{T}^{+}_p[\xi])
    \end{equation}

    Because $\zeta$ is a mixture, $\Pr(\zeta \le x) = \lambda \Pr(\xi \le x)+ (1-\lambda) \Pr(\eta \le x)$. It is easy to see that:

    \begin{equation}
        \mathbb{T}^{-}_p[\xi] \cap
        \mathbb{T}^{-}_p[\eta]
        \subset
        \mathbb{T}^{-}_p[\zeta]
    \end{equation}

    and analogously that:

    \begin{equation}
        \mathbb{T}^{+}_p[\xi] \cap
        \mathbb{T}^{+}_p[\eta]
        \subset
        \mathbb{T}^{+}_p[\zeta]
    \end{equation}

    Therefore $c \in \mathbb{M_p}[\zeta]$ can not be in $\mathbb{T}^{-}_p[\xi]$ or $\mathbb{T}^{-}_p[\eta]$ i.e. less than any median point in $\mathbb{M}_p[\xi]$ or $\mathbb{M}_p[\eta]$;
    and $c$ can not be in $\mathbb{T}^{+}_p[\xi]$ or $\mathbb{T}^{+}_p[\eta]$ i.e. greater than any median point in $\mathbb{M}_p[\xi]$ or $\mathbb{M}_p[\eta]$.

    If the $0.5$-quantiles i.e. the medians are unique values and $\mathbbm{m}[\xi] \le \mathbbm{m}[\eta]$ then we have:

    \begin{equation}
        \mathbbm{m}[\xi]
        \le
        \mathbbm{m}[\zeta]
        \le
        \mathbbm{m}[\eta]
    \end{equation}

    This completes the proof.
    
\end{proof}

\section{Proof of the Bayesian equilibrium}

The proof belongs to Theorem~\ref{thm:Bayesian}.

\begin{proof}
\label{proof:Bayesian}

It is straightforward to see that because of \PII/'s mixing strategy, \PI/ has the same expected utility for any specific sampling.
Similarly, because of the randomized sampling of \PI/, the expected utility of \PII/ depends only on the chosen scenario (A or B) and not the specific sequence (chosen from $\mathcal{K}_A$ or $\mathcal{K}_B$).

The expected growth rate difference can be expressed in the following way:

\begin{equation}
    \label{eq:G(P,p')}
    \Delta G(P,\{p'\}) = P \ \sum_{k \in \mathbb{K}_A} p_k(A) \log(p'_k) + (1-P) \ \sum_{k \in \mathbb{K}_B} p_k(B) \log(1-p'_k)
\end{equation}

This expression has to be maximized by a set of splitting ratios described by $\{p'_k\}$, which is surprisingly simple. After taking the partial derivative respect to $p'_\ell$:

\begin{equation}
    P \ p_\ell(A) \frac{1}{p'^*_\ell} - (1-P) \ p_\ell(B) \frac{1}{1-p'^*_\ell} = 0 \quad \forall \ell \in \mathbb{K}_A \cap \mathbb{K}_B
\end{equation}

Which can be satisfied by setting:

\begin{equation}
    p'^*_\ell = \frac{P \ p_\ell(A)}{P \ p_\ell(A)+(1-P) \  p_\ell(B)}
\end{equation}

When $\ell \in \mathbb{K}_A \setminus \mathbb{K}_B$, the derivative is:

\begin{equation}
    P \ p_\ell(A) \frac{1}{p'_\ell}
\end{equation}

which is positive on the whole $p'^*_\ell \in (0,1]$ domain. This implies that in this case, $\Delta G$ is maximized if:

\begin{equation}
    p'^*_\ell = 1, \quad \forall \ell \in \mathbb{K}_A \setminus \mathbb{K}_B
\end{equation}

Using the same reasoning, we get:

\begin{equation}
    p'^*_\ell = 0, \quad \forall \ell \in \mathbb{K}_B \setminus \mathbb{K}_A
\end{equation}

To formally prove that this is indeed a global maximum, we can check the second derivative with respect to $p'_m$:

\begin{equation}
    \begin{split}
    \frac{\partial^2 }{\partial p'_m \partial p'_\ell} \Delta G(P, \{p'\}) \Bigr|_{\{p'\}=\{p'^*\}} &=
    - \delta_{m,\ell} \left ( \frac{P \ p_\ell(A)}{{p'^*_\ell}^2} +
    \frac{(1-P) \ p_\ell(B)}{(1-p'^*_\ell)^2}
    \right ) \\
    &= -  \delta_{m,\ell} \frac{(P \ p_\ell(A)+(1-P) \  p_\ell(B))^3}{P \ (1-P) \ p_\ell(A) \ p_\ell(B)}
    \end{split}
\end{equation}

which gives a strictly negative definite Hessian for all $P \in (0,1)$, $p_\ell(A)>0$, $p_\ell(B)>0$. Ensuring that this is a global maximum.

\paragraph{Finding $P^*$:}

To find the global saddle point of the expression \eqref{eq:G(P,p')}, we can take its partial derivative with respect to $P$.
By introducing the following notation:

\begin{equation}
    \label{eq:Appendix_GA(P)}
    \Delta G_A(P) = \sum_{k \in \mathbb{K}_A} p_k(A) \log \left ( \frac{P \ p_k(A)}{P \ p_k(A)+(1-P) \ p_k(B)} \right )
\end{equation}

\begin{equation}
    \label{eq:Appendix_GB(P)}
    \Delta G_B(P) = \sum_{k \in \mathbb{K}_B} p_k(B) \log \left ( \frac{(1-P) \ p_k(B)}{P \ p_k(A)+(1-P) \ p_k(B)} \right )
\end{equation}

The partial derivative can be expressed as \footnote{we use ``the apparatus of partial derivatives, in which even the notation is ambiguous'' \cite{book:Arnold}, which is hopefully still more useful than confusing}:

\begin{equation}
    \frac{\partial}{\partial P} \Delta G(P, \{p'\}) \Bigr|_{\{p'\}=\{p'^*\}} = \Delta(P) = \Delta G_A(P) - \Delta G_B(P)
\end{equation}

Direct calculation shows that:

\begin{equation}
    \Delta (P) = \frac{d}{d P} \Delta G(P,\{p'^*(P)\}) = \Delta G'(P)
\end{equation}

\begin{equation}
    \Delta (P) = 0 \iff \frac{d}{d P} \Delta G(P,\{p'^*(P)\}) = \Delta G'(P) = 0
\end{equation}

To show the existence and uniqueness of $P^*$, for which $\Delta(P^*) = 0$, it is useful to change variables from $P \in (0,1)$ to log-odds $\vartheta \in \mathbb{R}$, using the following transformation \footnote{also known as ``logit'' \cite{book:StatisticalLearning}. ``Evidence'' \cite{book:Jaynes} as a measure of the Bayes factor is also a related concept.}:

\begin{equation}
    \vartheta = \log \left ( \frac{P}{1-P} \right ), \quad P = \frac{e^{\vartheta/2}}{e^{\vartheta/2}+e^{-\vartheta/2}}
\end{equation}

Using the convention $p_k(A) = 0$ if $k \notin \mathbb{K}_A$, and $p_k(B) = 0$ if $k \notin \mathbb{K}_B$, we can rewrite $\Delta$ as:

\begin{equation}
    \begin{split}
    \Delta(\vartheta) = & \vartheta + \sum_{k \in \mathbb{K}_A} p_k(A) \log(p_k(A)) - 
                        \sum_{k \in \mathbb{K}_B} p_k(B) \log(p_k(B)) - \\
                        & \sum_{k \in \mathbb{K}} (p_k(A) - p_k(B))
                        \log \left ( e^{\vartheta/2} p_k(A) + e^{-\vartheta/2} p_k(B) \right )
    \end{split}
\end{equation}

or by the entropy difference $\Delta H = H_A-H_B$:

\begin{equation}
    \begin{split}
    \Delta(\vartheta) = & \vartheta - \Delta H - 
                         \sum_{k \in \mathbb{K}} (p_k(A) - p_k(B))
                        \log \left ( e^{\vartheta/2} p_k(A) + e^{-\vartheta/2} p_k(B) \right )
    \end{split}
\end{equation}

The root of $\Delta(\vartheta)$ can be redefined as the iterative fixed point of the function $F(\vartheta)$:

\begin{equation}
    \vartheta_{n+1} = F(\vartheta_n)
\end{equation}

\begin{equation}
    \label{eq:Fdef}
    F(\vartheta) = \Delta H + \sum_{k \in \mathbb{K}} (p_k(A) - p_k(B))
                        \log \left ( e^{\vartheta/2} p_k(A) + e^{-\vartheta/2} p_k(B) \right )
\end{equation}

First, we show that $F$ is a contraction:

\begin{equation}
    \frac{d}{d \vartheta} F(\vartheta) = F'(\vartheta) = \frac{1}{2} \sum_{k \in \mathbb{K}} (p_k(A) - p_k(B))
                        \left ( \frac{e^{\vartheta/2} p_k(A) - e^{-\vartheta/2} p_k(B)}{e^{\vartheta/2} p_k(A) + e^{-\vartheta/2} p_k(B)} \right )
\end{equation}

Using Hölder's inequality \cite{book:Holder}, and that $|(x-y)/(x+y)| \le 1$ for all $x,y \ge 0$, $x+y>0$

\begin{equation}
    |F'(\vartheta)| \le \frac{1}{2} \sum_{k \in \mathbb{K}} |p_k(A) - p_k(B)|
\end{equation}

This is the total variation distance \cite{book:MarkovChainsMixingTimes} of probability measures $p_{k}(A)$ and $p_{k}(B)$, which is strictly less then $1$ if there exists an index $k$, for which $p_k(A) > 0 \wedge p_k(B)>0$ i.e. $\mathbb{K}_A \cap \mathbb{K}_B = \mathbb{K}_{AB} \ne \emptyset$.

\begin{equation}
    \label{eq:qdef}
    |F'(\vartheta)| \le \frac{1}{2} \sum_{k \in \mathbb{K}} |p_k(A) - p_k(B)| = q < 1
\end{equation}

This implies that $F:\mathbb{R} \to \mathbb{R}$ is a contraction, with $q<1$:

\begin{equation}
    |F(\vartheta_2) - F(\vartheta_1)| = \left | \int_{\vartheta_1}^{\vartheta_2} F'(\vartheta) d \vartheta \right | \le \int_{\vartheta_1}^{\vartheta_2} |F'(\vartheta)| d \vartheta \le q\ |\vartheta_2 - \vartheta_1|
\end{equation}

The Banach fixed point theorem \cite{book:BanachFixedPoint} guarantees a unique fixed point for this contraction:

\begin{equation}
    \exists! \ \vartheta^* \in \mathbb{R}, \quad F(\vartheta^*) = \vartheta^*
\end{equation}

This completes the proof.

\end{proof}

\section{Proof of the Convergence of Binomial games}
\label{appendix:BinomialGames}

For the convergence of both Binomial Fisher games and Binomial Bayesian games, the following lemma \ref{lemma:BinomHyperApprox} is essential:

\begin{lemma}[Binomial approximation of Hypergeometric distribution]
\label{lemma:BinomHyperApprox}

For any sequences of natural numbers $\{M_i\}_{i=0}^\infty$, $\{K_i\}_{i=0}^\infty$, if $\ M_i \to \infty$ and $K_i/M_i \to x \in (0,1)$, then

\begin{equation}
    \forall k \in \{0,\dots,N \}, \ \lim_{i \to \infty} p_k(N,K_i,M_i) = p_k(N,x)
\end{equation}

where $p_k(N,K,M)$ is the Hypergeometric distribution:

\begin{equation}
    p_k(N,K,M) = \frac{\binom{K}{k} \binom{M-K}{N-k}}{\binom{M}{N}}
\end{equation}

and $p_k(N,x)$ stands for the Binomial distribution:

\begin{equation}
     p_k(N,x) = \binom{N}{k} x^k (1-x)^{N-k}
\end{equation}

\end{lemma}

\begin{proof}

    \begin{equation}
        p_k(N,K,M) = \frac{
        \frac{K(K-1)\dots(K-(k-1))}{k!}
        \frac{(M-K)(M-K-1)\dots(M-K-(k-1))}{(N-k)!}
        }
        {
        \frac{M(M-1)\dots(M-(k-1)) \ (M-k)(M-k-1)\dots(M-(N-1))}{N!}
        }
    \end{equation}

    \begin{equation}
    \begin{split}
        p_k(N,K,M) =& \frac{N!}{k!(N-k)!}
        \left ( \frac{K}{M} \frac{K-1}{M-1} \dots \frac{K-(k-1)}{M-(k-1)} \right ) \\
        & \left ( \frac{M-K}{M-k} \frac{M-K-1}{M-k-1} \dots \frac{M-K-(N-k-1)}{M-(N-1)} \right )
    \end{split}
    \end{equation}

    which simplifies term by term to:

    \begin{equation}
        \lim_{i \to \infty} p_k(N,K_i,M_i) = \frac{N!}{k!(N-k)!} x^k (1-x)^{N-k} 
        = p_k(N,x)
    \end{equation}

\end{proof}

The proof can be found in standard textbooks on probability theory \cite{book:IntroToProbability} or is frequently left to the reader as an exercise \cite{book:Renyi1970}.

\begin{remark}
    Explicit upper and lower bounds for the total variation distance can be obtained between Hypergeometric and Binomial distributions:

    \begin{equation}
         \frac{1}{28} \frac{N-1}{M-1} \le ||\underline{p}(N,K,M) - \underline{p}(N,K/M)||_{\mathrm{TV}} \le \frac{N-1}{M-1}
    \end{equation}
    (see the proof in the following \href{https://stat.ethz.ch/~kuensch/papers/hypergeom.pdf}{derivation}.)
    
\end{remark}

\subsection{Binomial Fisher games}

\paragraph{Notation:}

We introduce a $2 \times (N+1)$ dimensional array of positive real numbers, which will represent general ``Off-shell'' \cite{book:Weinberg} ``probabilities'' or ``weights'' (which does not have to be normalized to $1$):

\begin{equation}
    \underline{\underline{p}} \in 
    \mathbb{W}_2
    =
    \mathbb{R}_{>0}^{2 \times (N+1)}
\end{equation}

We can represent normalized ``On-shell'' probability distributions by the following identification:

\begin{equation}
    p_{A,k} = p_k(A), \quad p_{B,k} = p_{k}(B)
\end{equation}

We can rewrite the requirement for $k^*$ in Theorem~\ref{thm:Fisher_ks} in the following way:

\begin{equation}
    \sum_{k \le k^*} p_k(A)+p_k(B) > 1 \iff \sum_{k \le k^*} p_k(A) > \sum_{k > k^*} p_k(B)
\end{equation}

Introducing an indexed difference function:

\begin{equation}
    \label{eq:delta_k}
    \delta_{i}(\underline{\underline{p}}) =  \sum_{k > i} p_{B,k} - \sum_{k \le i} p_{A,k}
\end{equation}

We can define a function, which gives an integer for a general set of weights $\underline{\underline{p}}$ \footnote{this can be viewed as an ``Off-shell'' \cite{book:Weinberg} generalizations of equilibrium parameters, where $\underline{p}_A$ and $\underline{p}_B$ are not necessarily normalized to 1, but are positive.}

\begin{equation}
    k^\circ_N(\underline{\underline{p}}) = \sum_{k'=0}^N 
    \mathds{1} \left ( \delta_{k'}(\underline{\underline{p}}) > 0 \right )
\end{equation}

Where $\mathds{1}(\varphi)$ is the indicator function, which gives $1$ if the logical formula $\varphi$ is true, and 0 if $\varphi$ is false.

We can promote the expression for $\nu^*$ in Theorem~\ref{thm:SymEqNu} to a general function as well:

\begin{equation}
    \nu^\circ_N(\underline{\underline{p}}) = \frac{\sum_{k\ge k^*} p_{B,k} - \sum_{k < k^*} p_{A,k}}{p_{A,k^*}+p_{B,k^*}}, \quad 
    k^* \leftarrow k^\circ_N(\underline{\underline{p}})
\end{equation}

which we can rewrite to:

\begin{equation}
    \label{eq:nuCirc}
    \nu^\circ_N(\underline{\underline{p}}) = \frac{
    \delta_{k^*}(\underline{\underline{p}}) +
    p_{A,k^*}+p_{B,k^*} }{p_{A,k^*}+p_{B,k^*}}, \quad 
    k^* \leftarrow k^\circ_N(\underline{\underline{p}})
\end{equation}

The generalization of the expression for $P^*$ in Theorem~\ref{thm:Fisher_Ps} is straightforward:

\begin{equation}
    P^\circ_N(\underline{\underline{p}}) = \frac{p_{B,k^*}}{p_{A,k^*}+p_{B,k^*}}, \quad 
    k^* \leftarrow k^\circ_N(\underline{\underline{p}})
\end{equation}

Finally, we can write down the strategy parameter, which describes the equilibrium strategy of \PI/:

\begin{equation}
    s^\circ_N(\underline{\underline{p}}) = \frac{1}{N+1} \left (
    k^\circ_N(\underline{\underline{p}}) + \nu^\circ_N(\underline{\underline{p}})
    \right )
\end{equation}

\begin{lemma}[Smoothness in the bulk]
    \label{lemma:SmoothnessOfFisherPolicies}
    If for some $\underline{\underline{p}} \in \mathbb{W}_2 =\mathbb{R}_{>0}^{2 \times (N+1)}$ there is no $k^* \in \{0,\dots,N\}$ for which $\delta_{k^*}(\underline{\underline{p}}) = 0$, then there is some open neighborhood of $\underline{\underline{p}}$, $\mathcal{U} \subset \mathbb{W}_2$ where the functions:

    \begin{equation}
        k^\circ_N(\underline{\underline{p}}) = k^*, \quad \text{if } \underline{\underline{p}} \in \mathcal{U}
    \end{equation}

  and all equilibrium functions are smooth:

    \begin{equation}
        \nu^\circ_N, P^\circ_N, s^\circ_N \in C^\infty(\mathcal{U}, \mathbb{R})
    \end{equation}
    
\end{lemma}

\begin{proof}
    We can define a strictly positive quantity:

    \begin{equation}
        \Delta \delta = \min \{|\delta_k(\underline{\underline{p}})| k \in \{0,\dots,N\}\}
    \end{equation}

    if we take any different $\underline{\underline{q}} \in \mathbb{R}_+^{2 \times (N+1)}$, where

    \begin{equation}
        |q_{\theta,k} - p_{\theta,k}| < \Delta \delta / (N+1)
    \end{equation}

    then for all $k \in \{0,\dots,N\}$ $\delta_k(\underline{\underline{p}})$ and $\delta_k(\underline{\underline{q}})$ has the same sign, meaning that:

    \begin{equation}
        k^\circ_N(\underline{\underline{p}}) = k^\circ_N(\underline{\underline{q}})
    \end{equation}

    this implies that $k^\circ_N$ is constant on a finite open $p=\infty$ norm ``ball'' (or rather hypercube):

    \begin{equation}
        \mathcal{U} = B_\infty(\underline{\underline{p}},\Delta \delta / (N+1)) = 
        \left \{ 
        \underline{\underline{q}} \in \mathbb{W}_2 | \ 
        ||\underline{\underline{q}} - \underline{\underline{p}}||_\infty < \Delta \delta / (N+1))
        \right \}
    \end{equation}

    $k^\circ_N = k^*$ is constant on $\mathcal{U}$, therefore $\nu^\circ_N, P^\circ_N, s^\circ_N$ are only rational functions of $p_{\theta,k}$ variables, which are all infinitely continuously differentiable, i.e., smooth.

    \begin{equation}
        \nu^\circ_N, P^\circ_N, s^\circ_N \in C^\infty(\mathcal{U}, \mathbb{R})
    \end{equation}
    
\end{proof}

\begin{lemma}[Continuity of policies on the scars]
    \label{lemma:ContinuityOfFisherPolicies}
    If for some $\underline{\underline{p}} \in \mathbb{W}_2$ there is a $k^* \in \{0,\dots,N\}$ for which $\delta_{k^*}(\underline{\underline{p}}) = 0$, then
    $s^\circ_N$ is continuous at $\underline{\underline{p}}$:
    
    \begin{equation}
        s^\circ_N \in C^0(\underline{\underline{p}})
    \end{equation}
    
\end{lemma}

\begin{proof}
    We will show that for any sequence $\left \{ \underline{\underline{p}}^{(i)} \right \}_{i=0}^\infty$, which converges to $\underline{\underline{p}}$, $s^\circ_N(\underline{\underline{p}}^{(i)})$ will converge to $(k^*+1)/(N+1)$:

    \begin{equation}
        \forall \left \{ \underline{\underline{p}}^{(i)} \right \}_{i=0}^\infty, \ \lim_{i\to\infty} \underline{\underline{p}}^{(i)} = \underline{\underline{p}} \implies
        \lim_{i\to\infty} s^\circ_N(\underline{\underline{p}}^{(i)}) = s^\circ_N (\underline{\underline{p}}) = (k^*+1)/(N+1)
    \end{equation}

    To show this, we split a sequence $\left \{ \underline{\underline{p}}^{(i)} \right \}_{i=0}^\infty$ to two subsequences, $\left \{ \underline{\underline{p}}^{> (i)} \right \}_{i=0}^\infty$ and $\left \{ \underline{\underline{p}}^{\le (i)} \right \}_{i=0}^\infty$:

    \begin{equation}
        \underline{\underline{p}}^{> (j)} = 
        \underline{\underline{p}}^{(i_j)}, \ \delta_{k^*}(\underline{\underline{p}}^{(i_j)}) > 0, \ i_{j+1} > i_j
    \end{equation}

        \begin{equation}
        \underline{\underline{p}}^{\le (j)} = 
        \underline{\underline{p}}^{(i_j)}, \ \delta_{k^*}(\underline{\underline{p}}^{(i_j)}) \le 0, \ i_{j+1} > i_j
    \end{equation}

    \paragraph{Greater than zero case:}

    For any $\left \{ \underline{\underline{p}}^{> (i)} \right \}_{i=0}^\infty$ sequence:

    \begin{equation}
        \lim_{j\to\infty} k^\circ_N (\underline{\underline{p}}^{> (j)}) = k^* + 1
    \end{equation}

    and 

    \begin{equation}
        \lim_{j\to\infty} \nu^\circ_N (\underline{\underline{p}}^{> (j)}) = 
        \frac{\delta_{k^*+1}(\underline{\underline{p}}) + p_{1,k^*+1}+p_{2,k^*+1} }
        {p_{1,k^*+1}+p_{2,k^*+1}}
    \end{equation}

    In general:

    \begin{equation}
        \delta_{k+1}(\underline{\underline{p}}) - \delta_k(\underline{\underline{p}}) =
        -(p_{A,(k+1)} + p_{B,(k+1)})
    \end{equation}

    because of $\delta_{k^*}(\underline{\underline{p}})=0$,

    \begin{equation}
        \lim_{j\to\infty} \nu^\circ_N (\underline{\underline{p}}^{> (j)}) = 
        \frac{-(p_{A,(k^*+1)} + p_{B,(k^*+1)}) + p_{A,k^*+1}+p_{B,k^*+1} }
        {p_{A,k^*+1}+p_{B,k^*+1}} = 0
    \end{equation}

    Therefore, the limit for $s^\circ_N$:

    \begin{equation}
        \lim_{j\to\infty} s^\circ_N (\underline{\underline{p}}^{> (j)}) = \left ( \lim_{j\to\infty} k^\circ_N(\underline{\underline{p}}^{> (j)}) + \lim_{j\to\infty} \nu^\circ_N(\underline{\underline{p}}^{> (j)}) \right )/(N+1) =
        \frac{k^*+1}{N+1}
    \end{equation}

    \paragraph{Less or equal to zero case:}

    For any $\left \{ \underline{\underline{p}}^{\le (i)} \right \}_{i=0}^\infty$ sequence:

    \begin{equation}
        \lim_{j\to\infty} k^\circ_N (\underline{\underline{p}}^{\le (j)}) = k^*
    \end{equation}

    and

    \begin{equation}
        \lim_{j\to\infty} \nu^\circ_N (\underline{\underline{p}}^{\le (j)}) = 
        \frac{\delta_{k^*}(\underline{\underline{p}}) + p_{A,k^*}+p_{B,k^*} }
        {p_{A,k^*}+p_{B,k^*}}
    \end{equation}

    because of $\delta_{k^*}(\underline{\underline{p}})=0$,

    \begin{equation}
        \lim_{j\to\infty} \nu^\circ_N (\underline{\underline{p}}^{\le (j)}) = 
        \frac{0 + p_{A,k^*}+p_{B,k^*} }
        {p_{A,k^*}+p_{B,k^*}} =
        1
    \end{equation}

    Therefore, the limit for $s^\circ_N$:

    \begin{equation}
        \lim_{j\to\infty} s^\circ_N (\underline{\underline{p}}^{\le (j)}) = \left ( \lim_{j\to\infty} k^\circ_N(\underline{\underline{p}}^{\le (j)}) + \lim_{j\to\infty} \nu^\circ_N(\underline{\underline{p}}^{\le (j)}) \right )/(N+1) =
        \frac{k^*+1}{N+1}
    \end{equation}

    \paragraph{Summary:}

    While $k^\circ_N$ and $\nu^\circ_N$ have different limits for 
    $\left \{ \underline{\underline{p}}^{> (j)} \right \}_{j=0}^\infty$ and 
    $\left \{ \underline{\underline{p}}^{\le (j)} \right \}_{j=0}^\infty$ sequences, their combination in $s^\circ_N$ converges to the same limit, meaning that the only accumulation point (or limit points \cite{book:Rudin}) of the $\left \{ s^\circ_N(\underline{\underline{p}}^{(i)}) \right \}_{i=0}^\infty$ is $\{(k^*+1)/(N+1)\}$.
    
\end{proof}

\begin{remark}
    \label{remark:PolicyEquivalenceS}
    In a Fisher game, the strategy of \PI/ is fully determined by the pair $\{k^*,\nu^*\}$. In the original construction we allowed only $\nu^* \in [0,1)$, but if we formally allow $\nu^*=1$ values as well, then the strategies: $\{k^*=k^*_1,\nu^*=0\}$ and $\{k^*=k^*_1+1,\nu^*=1\}$ are identical.

    Therefore two strategies for \PI/ $\{k^*_1,\nu^*_1\}$, $\{k^*_2,\nu^*_2\}$ are identical, if the composed $s^*_1=(k^*_1+\nu^*_1)/(N+1)$ and $s^*_2=(k^*_2+\nu^*_2)/(N+1)$ are identical.

    \begin{equation}
        \mu_{\{k^*_1,\nu^*_1\}} = \mu_{\{k^*_2,\nu^*_2\}} \iff
        s^*_1 = s^*_2
    \end{equation}
    
\end{remark}

\begin{lemma}[Finite jump of the prior]
    If for some $\underline{\underline{p}} \in \mathbb{W}_2$ there is a $k^\bullet \in \{0,\dots,N\}$ for which $\delta_{k^\bullet}(\underline{\underline{p}}) = 0$, then
    $P^\circ_N$ has a discontinuity at $\underline{\underline{p}}$, but is bounded by $\underline{P}^\bullet$ and $\overline{P}^\bullet$:
    
    \begin{equation}
    \label{lamma:PLowerUpper}
        \underline{P}^\bullet = \frac{p_{B,k^\bullet}}{p_{A,k^\bullet}+p_{B,k^\bullet}}, \quad
        \overline{P}^\bullet = \frac{p_{B,k^\bullet+1}}{p_{A,k^\bullet+1}+p_{B,k^\bullet+1}}
    \end{equation}
    
\end{lemma}

\begin{proof}
    For the proof, we can use the same constructions of 
     $\left \{ \underline{\underline{p}}^{> (j)} \right \}_{j=0}^\infty$ and
     $\left \{ \underline{\underline{p}}^{\le (j)} \right \}_{j=0}^\infty$ as in the previous proof, and observe that:

     \begin{equation}
         \lim_{j\to\infty} P^\circ_N (\underline{\underline{p}}^{> (j)}) = 
        \frac{p_{B,k^\bullet+1} }
        {p_{A,k^\bullet+1}+p_{B,k^\bullet+1}} = \overline{P}^\bullet
     \end{equation}

     \begin{equation}
         \lim_{j\to\infty} P^\circ_N (\underline{\underline{p}}^{\le (j)}) = 
        \frac{p_{B,k^\bullet} }
        {p_{A,k^\bullet}+p_{B,k^\bullet}} = \underline{P}^\bullet
     \end{equation}

     This shows that only these two accumulation points (or limit points \cite{book:Rudin}) are possible for the Symmetric equilibrium sequence $\left \{ P^\circ_N(\underline{\underline{p}}^{(i)}) \right \}_{i=0}^\infty$, which are $\{\underline{P}^\bullet, \overline{P}^\bullet\}$.
     
\end{proof}

Now we have all the ingredients to prove Theorem~\ref{thm:BinomialSymmetric} about the Symmetrical equilibrium of Binomial Fisher games:

\begin{proof}
\label{proof:BinomialFisher}

For any converging sequence of finite games:

\begin{equation}
    \textswab{G}_i = \textswab{Game}(N, K_{A,i}, K_{B,i}, M_i)
\end{equation}

\begin{equation}
    \begin{split}
        \{\textswab{G}_i\}_{i=1}^\infty \in \overline{\textswab{Game}}(N, x_A, x_B) \iff & \lim_{i\to \infty} M_i = \infty \ \wedge \\ 
        & \lim_{i\to \infty} K_{A,i}/M_i = x_A \ \wedge \\
        & \lim_{i\to \infty} K_{B,i}/M_i = x_B
    \end{split}
\end{equation}

Based on lemma \ref{lemma:BinomHyperApprox}, Hypergeometric probability distributions converge to Binomial distributions:

\begin{equation}
    \begin{split}
    \underline{p}(N,K_{A,i},M_i) \to \underline{p}(N,x_A) \\
    \underline{p}(N,K_{B,i},M_i) \to \underline{p}(N,x_B)
    \end{split}
\end{equation}

\paragraph{Limit policy for \PI/:}

The consequence of lemma \ref{lemma:SmoothnessOfFisherPolicies} and lemma \ref{lemma:ContinuityOfFisherPolicies}, is that:

\begin{equation}
    s^\circ_N \in C^0(\mathbb{W}_2,\mathbb{R})
\end{equation}

Therefore, the expression:

\begin{equation}
    s^*_\infty = \lim_{i\to\infty} s^*(N,K_{A,i},K_{B_i},M_i) = 
    \lim_{i\to\infty} s^\circ_N((\underline{p}(N,K_{A,i},M_i),\underline{p}(N,K_{B,i},M_i)))
\end{equation}

can be simplified to

\begin{equation}
    \lim_{i\to\infty} s^\circ_N((\underline{p}(N,K_{A,i},M_i),\underline{p}(N,K_{B,i},M_i))) = s^\circ_N((\underline{p}(N,x_A),\underline{p}(N,x_B)))
\end{equation}

\begin{equation}
    s^\circ_N((\underline{p}(N,x_A),\underline{p}(N,x_B))) = s^*_\text{Bin}(N,x_A,x_B)
\end{equation}

meaning that:

\begin{equation}
    s^*_\infty = s^*_\text{Bin}(N,x_A,x_B)
\end{equation}

Where $s^*_\text{Bin}(N,x_A,x_B)$ can be composed from $k^*_\text{Bin}(N,x_A,x_B)$ and $\nu^*_\text{Bin}(N,x_A,x_B)$, which can be calculated using the criterion \ref{thm:SymBinEqk} and formula \ref{thm:SymBinEqNu}.

\begin{equation}
    s^*_\text{Bin}(N,x_A,x_B) = (k^*_\text{Bin}(N,x_A,x_B) + \nu^*_\text{Bin}(N,x_A,x_B))/(N+1)
\end{equation}

From $s^*_\infty$ we can define the limit values of $k^*_\infty$ and $\nu^*_\infty$ \footnote{If $k^*_i$ and $\nu^*_i$ fails to converge, then we choose this as the definition from the converging $s^*_i$ sequence, and so we preserve the $\nu^* \in [0,1)$ convention. Because of remark \ref{remark:PolicyEquivalenceS}, this is only a parametrization choice of equivalent policies.}:

\begin{equation}
    k^*_\infty = \lfloor (N+1) \ s^*_\infty \rfloor, \quad
    \nu^*_\infty = (N+1) \ s^*_\infty - \lfloor (N+1) s^*_\infty \rfloor
\end{equation}

resulting that we can identify the limit strategy parameters by the parameters, which can be calculated using the limiting Binomial distribution:

\begin{equation}
    k^*_\infty = k^*_\text{Bin}, \quad \nu^*_\infty = \nu^*_\text{Bin}
\end{equation}

This means that the limit policy of \PI/ is well defined for all $0<x_A<x_B<1$, and can be calculated using \ref{thm:SymBinEqk} and \ref{thm:SymBinEqNu}.

\paragraph{Limit policy for \PII/:}

First we can observe, based on \eqref{eq:delta_k}, \eqref{eq:nuCirc} and \eqref{eq:defOfS_N} that:

\begin{equation}
    (x_A,x_B) \in \mathbb{S}_N \iff
    \exists k\in\{0,\dots,N\}, \ \delta_k((\underline{p}(N,x_A),\underline{p}(N,x_B))) = 0
\end{equation}

If $(x_A,x_B) \notin \mathbb{S}_n$, then because of lemma \ref{lemma:SmoothnessOfFisherPolicies} the smoothness of $P^\circ_N$ guarantees that we can use the Binomial limit distributions to obtain the limit value of $P^*_\infty$:

\begin{equation}
    P^*_\infty = \lim_{i\to\infty} P^*(N,K_{A,i},K_{B_i},M_i) = 
    \lim_{i\to\infty} P^\circ_N((\underline{p}(N,K_{A,i},M_i),\underline{p}(N,K_{B,i},M_i)))
\end{equation}

\begin{equation}
    \lim_{i\to\infty} P^\circ_N((\underline{p}(N,K_{A,i},M_i),\underline{p}(N,K_{B,i},M_i))) = P^\circ_N((\underline{p}(N,x_A),\underline{p}(N,x_B)))
\end{equation}

\begin{equation}
    P^\circ_N((\underline{p}(N,x_A),\underline{p}(N,x_B))) = P^*_\text{Bin}(N,x_A,x_B)
\end{equation}

where $P^*_\text{Bin}(N,x_A,x_B)$ can be calculated by \eqref{eq:SymBinEqP}, meaning that we can identify the limit prior by the prior that can be calculated using the limiting Binomial distribution:

\begin{equation}
    P^*_\infty = P^*_\text{Bin}(N,x_A,x_B)
\end{equation}

If, however, $(x_A,x_B) \in \mathbb{S}_N$, then $\delta_{k^*-1}((\underline{p}(N,x_A),\underline{p}(N,x_B)))=0$ the sequence $\{P^*_i\}_{i=0}^\infty$ can have two accumulation points:

\begin{equation}
    \mathrm{Acc} \left ( \{P^*_i\}_{i=0}^\infty \right) \subset \{\underline{P}^*, \overline{P}^*\}
\end{equation}

By using equation \ref{lamma:PLowerUpper}, while substituting
$k^\bullet = k^* -1$ we get:

\begin{equation}
    \underline{P}^* = \frac{p_{B,k^*-1}}{p_{A,k^*-1}+p_{B,k^*-1}}, \quad
    \overline{P}^* = \frac{p_{B,k^*}}{p_{A,k^*}+p_{B,k^*}}
\end{equation}

However, as $i \to \infty$, the game sequence $\{\mathfrak{G}_i\}_{i=0}^\infty$ ``converges to a degenerate game'', meaning that if \PI/ uses the strategy $\{k^*_\infty, \nu^*_\infty\}$, then the winning rate does not change, if \PII/ chooses any $P^* \in [\underline{P}^*, \overline{P}^*]$.

This completes the proof.
    
\end{proof}

\subsection{Binomial Bayesian games}

\label{section:BinomialBayesianProof}

\paragraph{Notation:}

We can define a function, which gives the growth rate difference for a general set of weights $\underline{\underline{p}} \in \mathbb{W}_2 = \mathbb{R}_{>0}^{2 \times (N+1)}$ \footnote{this can be viewed as the ``Off-shell'' \cite{book:Weinberg} generalizations of the growth rate difference function, where $\underline{p}_A$ and $\underline{p}_B$ are not necessarily normalized to 1, but are positive.}

\begin{equation}
    \begin{split}
        \Delta G^\circ(P,\underline{\underline{p}}) =& P \sum_k p_{A,k} \log \left ( 
        \frac{P \ p_{A,k}}{P p_{A,k} + (1-P) p_{B,k}}  
        \right ) + \\
        & (1-P) \sum_k p_{B,k} \log \left ( 
        \frac{(1-P) p_{B,k}}{P p_{A,k} + (1-P) p_{B,k}}  
        \right )
    \end{split}
\end{equation}

Its first derivative with respect to $P$:

\begin{equation}
    \begin{split}
        \Delta G^\circ{}'(P,\underline{\underline{p}}) =& \log(P) \sum_k p_{A,k} - \log(1-P) \sum_k p_{B,k} - \Delta H(\underline{\underline{p}}) + \\
        & - \sum_k (p_{A,k}-p_{B,k}) \log(P p_{A,k} + (1-P) p_{B,k})
    \end{split}
\end{equation}

where:

\begin{equation}
    \Delta H(\underline{\underline{p}}) = - \sum_k p_{A,k} \log(p_{A,k}) + 
    \sum_k p_{B,k} \log(p_{B,k})
\end{equation}

and the growth rate difference's second derivative with respect to $P$:

\begin{equation}
    \Delta G^\circ{}''(P,\underline{\underline{p}}) = \sum_k  \frac{p_{A,k} \ p_{B,k}}{P(1-P)(P p_{A,k} + (1-P) p_{B,k})}
\end{equation}

The generalized splitting ratio function looks the following:

\begin{equation}
    p'^\circ_k(P,\underline{\underline{p}}) = \frac{P \ p_{A,k}}{P \ p_{A,k}+(1-P) p_{B,k}}
\end{equation}

\begin{definition}
    The generalized $P^\circ(\underline{\underline{p}})$ is defined implicitly:

    \begin{equation}
        \Delta G^\circ{}'(P^\circ(\underline{\underline{p}}),\underline{\underline{p}}) = 0
    \end{equation}
    
\end{definition}

\begin{lemma}[Existence and unicity of the solution]
    The implicitly defined $P^\circ(\underline{\underline{p}})$ has a unique solution on the interval $(0,1)$ for any $\underline{\underline{p}} \in \mathbb{W}_2$.
    
\end{lemma}

\begin{proof}
    The same convexity arguments can be used for a general $\underline{\underline{p}} \in \mathbb{W}_2$ as in Section~\ref{par:FindingPs}
    
\end{proof}

\begin{lemma}[Differentiability of the prior]
    \label{lemma:DiffP}
    $P^\circ(\underline{\underline{p}})$ is continuously differentiable on the whole $\mathbb{R}_+^{2 \times (N+1)}$ domain:

    \begin{equation}
        P^\circ \in C^1(\mathbb{W}_2,(0,1))
    \end{equation}
    
\end{lemma}

\begin{proof}
    The statement is essentially the consequence of the Implicit function theorem \cite{book:ImplicitFunctionTheorem}.

    \begin{equation}
        \frac{\partial}{\partial P} \Delta G^\circ{}'(P,\underline{\underline{p}}) \Bigr|_{P=P^\circ(\underline{\underline{p}})} \ dP + 
        \frac{\partial}{\partial \underline{\underline{p}}} \Delta G^\circ{}'(P,\underline{\underline{p}})\Bigr|_{P=P^\circ(\underline{\underline{p}})} \ d \underline{\underline{p}}= 0
    \end{equation}

    \begin{equation}
        \frac{\partial P^\circ(\underline{\underline{p}})}{\partial \underline{\underline{p}}} = -\frac{1}{\Delta G^\circ{}''(P,\underline{\underline{p}})}
        \frac{\partial}{\partial \underline{\underline{p}}} \Delta G^\circ{}'(P,\underline{\underline{p}})\Bigr|_{P=P^\circ(\underline{\underline{p}})}
    \end{equation}

    or by components:

        \begin{equation}
        \frac{\partial P^\circ(\underline{\underline{p}})}{\partial p_{\theta,\ell}} = -\frac{1}{\Delta G^\circ{}''(P,\underline{\underline{p}})}
        \frac{\partial}{\partial p_{\theta,\ell}} \Delta G^\circ{}'(P,\underline{\underline{p}})\Bigr|_{P=P^\circ(\underline{\underline{p}})}
    \end{equation}

    \begin{equation}
        \begin{split}
            \frac{\partial}{\partial p_{A,\ell}} \Delta G^\circ{}'(P,\underline{\underline{p}}) = & \log(P) - \log(p_{A,\ell}) -1 +\\ 
            & - \log(P \ p_{A,\ell} + (1-P) p_{B,\ell}) - 
            P \frac{p_{A,\ell} -  p_{B,\ell}}{P \ p_{A,\ell} + (1-P) p_{B,\ell}} 
        \end{split}
    \end{equation}

        \begin{equation}
            \begin{split}
                \frac{\partial}{\partial p_{B,\ell}} \Delta G^\circ{}'(P,\underline{\underline{p}}) = &
            -\log(1-P) + \log(p_{B,\ell}) +1 +\\
            &\log(P \ p_{A,\ell} + (1-P) p_{B,\ell}) -(1-P) \frac{p_{A,\ell} -  p_{B,\ell}}{P \ p_{A,\ell} + (1-P) p_{B,\ell}} 
            \end{split}
    \end{equation}

    while $\Delta G^\circ{}''(P,\underline{\underline{p}})$ is strictly positive for all $P \in (0,1), \underline{\underline{p}} \in \mathbb{W}_2$, therefore all partial derivatives of $P^\circ(\underline{\underline{p}})$ exists, and are continuous.

\end{proof}

\begin{lemma}[Differentiability of the splitting ratios]
    \label{lemma:Diffpp}
    $p'^\circ_k(P,\underline{\underline{p}})$ is continuously differentiable on the whole $(0,1) \times \mathbb{W}_2$ domain:

    \begin{equation}
        p'^\circ_k \in C^1((0,1) \times \mathbb{W}_2,(0,1))
    \end{equation}
    
\end{lemma}

\begin{proof}
    Explicit calculation shows that:

    \begin{equation}
        \frac{\partial}{\partial p_{A,\ell}} p'^\circ_k
        (P,\underline{\underline{p}})
        = \delta_{k, \ell} \frac{p_{B,k} \ P (1-P)}{(P \ p_{A,k}+(1-P) \ p_{B,k})^2} 
    \end{equation}
    
    \begin{equation}
        \frac{\partial}{\partial p_{B,\ell}} p'^\circ_k
        (P,\underline{\underline{p}})
        = - \delta_{k, \ell} \frac{p_{A,k} \ P (1-P)}{(P \ p_{A,k}+(1-P) \ p_{B,k})^2} 
    \end{equation}

    \begin{equation}
        \frac{\partial}{\partial P} p'^\circ_k
        (P,\underline{\underline{p}})
        = \frac{p_{A,k} \ p_{B,k}}{(P \ p_{A,k}+(1-P) \ p_{B,k})^2} 
    \end{equation}

    therefore all derivatives exists and are continuous for any $P \in (0,1), \underline{\underline{p}} \in \mathbb{W}_2$.
    
\end{proof}

Now we have all the ingredients to prove Theorem~\ref{thm:BayesianBinomial} about the equilibrium of Binomial Bayesian games:

\begin{proof}
\label{proof:LimitBayesian}

For any converging sequence of finite Bayesian games:

\begin{equation}
    \textswab{BG}_i = \textswab{BGame}(N, K_{A,i}, K_{B,i}, M_i)
\end{equation}

\begin{equation}
    \begin{split}
        \{\textswab{BG}_i\}_{i=1}^\infty \in \overline{\textswab{BGame}}(N, x_A, x_B) \iff & \lim_{i\to \infty} M_i = \infty \ \wedge \\ 
        & \lim_{i\to \infty} K_{A,i}/M_i = x_A \ \wedge \\
        & \lim_{i\to \infty} K_{B,i}/M_i = x_B
    \end{split}
\end{equation}

Based on lemma \ref{lemma:BinomHyperApprox}, Hypergeometric probability distributions converge to Binomial distributions:

\begin{equation}
    \begin{split}
    \underline{p}(N,K_{A,i},M_i) \to \underline{p}(N,x_A) \\
    \underline{p}(N,K_{B,i},M_i) \to \underline{p}(N,x_B)
    \end{split}
\end{equation}

\paragraph{Prior convergence:}

The consequence of lemma \ref{lemma:DiffP} is that:

\begin{equation}
    P^\circ \in C^1(\mathbb{W}_2,(0,1))
\end{equation}

Therefore, the expression:

\begin{equation}
    P^*_\infty = \lim_{i\to\infty} P^*(N,K_{A,i},K_{B_i},M_i) = 
    \lim_{i\to\infty} P^\circ((\underline{p}(N,K_{A,i},M_i),\underline{p}(N,K_{B,i},M_i)))
\end{equation}

can be simplified to

\begin{equation}
    \lim_{i\to\infty} P^\circ((\underline{p}(N,K_{A,i},M_i),\underline{p}(N,K_{B,i},M_i))) = P^\circ((\underline{p}(N,x_A),\underline{p}(N,x_B)))
\end{equation}

\begin{equation}
    P^\circ((\underline{p}(N,x_A),\underline{p}(N,x_B))) = P^*_\text{Bin}(N,x_A,x_B)
\end{equation}

meaning that:

\begin{equation}
    P^*_\infty = P^*_\text{Bin}(N,x_A,x_B)
\end{equation}

where $P^*_\text{Bin}(N,x_A,x_B)$ can be calculated using the implicit formula \eqref{eq:BinomialBayesianP}.

\paragraph{Splitting ratio convergence:}

The same continuity argument based on lemma \ref{lemma:Diffpp} gives that for all $k \in \{0,\dots,N\}$ the splitting ratios converge to:

\begin{equation}
    p'^*_{k,\infty} = p'^*_{k,\mathrm{Bin}}(N,x_A,x_B)
\end{equation}

where $p'^*_{k,\mathrm{Bin}}(N,x_A,x_B)$ can be calculated using the formula \eqref{eq:BinomialBayesianpp}.

This completes the proof.
    
\end{proof}

\section{Derivation of limiting priors}
\label{appendix:Limiting}

Most calculations presented in this section do not yield a solid proof, but hopefully, they can show the path for a more careful derivation, which will lead to exact statements.

\paragraph{Notation:}

For the derivation of both Binomial Fisher and Bayesian limiting priors, the following notation will be useful:

$I(x,y)$ or $I(x,x_\theta)=I_\theta(x)$ will denote the so called rate function \cite{book:LargeDeviation, book:LargeDeviation2, arxiv:LDP} of a Binomial distribution:

\begin{equation}
\label{eq:BinomialRateFunctionI}
    I(x,x_\theta) = I_\theta(x) = x \log \left ( \frac{x}{x_\theta} \right ) +
    (1-x) \log \left ( \frac{1-x}{1-x_\theta} \right )
    \ge 0, \quad x \in [0,1]
\end{equation}

\begin{equation}
\label{eq:BinomialRateFunctionSymmetry}
    I(1-x,x-x_\theta) = I(x,x_\theta)
\end{equation}

Its derivative:

\begin{equation}
\label{eq:BinomialRateFunctionDerivativeI'}
    I'(x,x_\theta) = I'_\theta(x) 
    =
    \log \left (
    \frac{x}{x_\theta} \frac{1-x_\theta}{1-x}
    \right )
    =
    \log \left (
    \frac{x}{1-x}
    \right )
    -
    \log \left (
    \frac{x_\theta}{1-x_\theta}
    \right )
\end{equation}

and its second derivative:

\begin{equation}
\label{eq:BinomialRateFunctionSecondDerivativeI''}
    I''(x,x_\theta) = I''_\theta(x) 
    =
    \frac{1}{x(1-x)}
    =
    \kappa(x)
    > 0, \quad x \in (0,1)
\end{equation}

Another useful quantity will be:

\begin{equation}
\label{eq:beta}
    \beta = \log \left ( \frac{x_B (1-x_A)}{x_A (1-x_B)} \right )
    > 0
\end{equation}

\subsection{Binomial Fisher limiting policy}

\begin{lemma}
For all $0 < x_A < x_B < 1$

    \begin{equation}
        x_A \le 
        \lim_{N \to \infty} \frac{k^*_N(x_A,x_B)}{N}
        \le x_B
    \end{equation}
    
\end{lemma}

\begin{proof}
    The statement follows from Theorem~\ref{thm:BinomialKMedianBounds}:

    \begin{equation}
        x_A
        =
        \lim_{N \to \infty} \frac{\lfloor N x_A \rfloor}{N} 
        \le 
        \lim_{N \to \infty} \frac{k^*_N(x_A,x_B)}{N}
        \le 
        \lim_{N \to \infty} \frac{\lceil N x_B \rceil}{N}
        =
        x_B
\end{equation}
    
\end{proof}

The proof below will belong to Theorem~\ref{thm:BinomialFisherLimitingPolicy}.

\begin{proof}
    {\bf attempt:}
    The proof will rely on Chernoff bound \cite{paper:ChernoffBound} and a lower bound on tail distribution \cite{book:AshInformation} for Binomial random variables:

    \paragraph{Notation:}
    Tail distribution:

    \begin{equation}
        T_N(k^\bullet,x) = \sum_{k=k^\bullet}^N p_k(N,x)
    \end{equation}

    its upper and lower bounds if $0 < x < k^\bullet/N < 1$ \cite{book:AshInformation}:

    \begin{equation}
                \frac{e^{-N \cdot I(k^\bullet/N,x)}}{\sqrt{8 N \cdot k^\bullet/N (1-k^\bullet/N)}} 
        \le
        T_N(k^\bullet,x) 
        \le
        e^{-N \cdot I(k^\bullet/N,x)}
    \end{equation}

    \begin{equation}
        \underline{T}_N(k^\bullet,x) 
        =
        \frac{1}{\sqrt{2 N}} e^{-N \cdot I(k^\bullet/N,x)} 
        \le
        T_N(k^\bullet,x) 
        \le
        e^{-N \cdot I(k^\bullet/N,x)}
        =
        \overline{T}_N(k^\bullet,x)
    \end{equation}

    where $I(.,.)$ is the so called rate function introduced in \eqref{eq:BinomialRateFunctionI}.
    It has a positive second derivative for all $x \in (0,1)$, so the increment of $k^\bullet$ in the $T(k^\bullet,x)$ expression can be estimated in the following way:

    \begin{equation}
        \underline{T}^+_N(k^\bullet,x)
        =
        \frac{e^{-I'(k^\bullet/N,x)}}{\sqrt{2 N}}  e^{-N \cdot I(k^\bullet/N,x)}
        \le
        T_N(k^\bullet+1,x)
        \le
        e^{-N \cdot I(k^\bullet/N,x)}
        =
        \overline{T}^+_N(k^\bullet,x)
    \end{equation}

    \paragraph{Left and right tails:} For Binomial distributions, the left and right tails are connected:

    \begin{equation}
        \sum_{k\le k^\bullet} p_k(N,x) = 
        \sum_{k\ge N - k^\bullet} p_k(N,1-x) =
        T_N(N-k^\bullet,1-x)
    \end{equation}

    because of the symmetry of the rate function \eqref{eq:BinomialRateFunctionSymmetry}, formally the following inequalities hold:

    \begin{equation}
        \underline{T}_N(k^\bullet,x) 
        \le
        T_N(N-k^\bullet,1-x) 
        \le
        \overline{T}_N(k^\bullet,x)
    \end{equation}

    For the increment:

    \begin{equation}
        \underline{T}^+_N(N-k^\bullet,1-x) 
        \le
        T_N(N-(k^\bullet+1),1-x) 
        \le
        \overline{T}^+_N(N-k^\bullet,1-x)
    \end{equation}

    \paragraph{Requirements for $k^*$:}
    Recalling equation \ref{thm:SymBinEqk} and expressing it in terms of Binomial tail distributions, we get:

    \begin{equation}
        \sum_{k \le k^*} p_k(A)+p_k(B) > 1 \implies
        T_N(N-k^*,1-x_B) > T_N(k^*+1,x_A)
    \end{equation}

    \begin{equation}
        \sum_{k<k^*} p_k(A)+p_k(B) \le 1 \implies
        T_N(N-(k^*+1),1-x_B) \le T_N(k^*,x_A)
    \end{equation}

    Recalling that in the limit $x_A \le k^*/N \le x_B$, we can use the tail inequalities for the Binomial distribution, resulting in the following inequalities: 

    \begin{equation}
        \overline{T}_N(N-k^*,1-x_B) 
        >
        \underline{T}^+_N(k^*,x_A)
    \end{equation}

    \begin{equation}
        \underline{T}^+_N(N-k^*,1-x_B)
        \le
        \overline{T}_N(k^*,x_A) 
    \end{equation}

    After taking the logarithm and dividing by $N$, we get:

    \begin{equation}
        -I(k^*/N,x_B) > -I(k^*/N,x_A) + \frac{1}{N} \log \left ( \frac{e^{-I'(k^*/N,x_A)}}{\sqrt{2 N}} \right )
    \end{equation}

    \begin{equation}
        -I(k^*/N,x_B)  + \frac{1}{N} \log \left ( \frac{e^{I'(k^*/N,x_B)}}{\sqrt{2 N}} \right )
        \le
         -I(k^*/N,x_A)
    \end{equation}

    Together, the inequalities give:

    \begin{equation}
        \frac{1}{N} \log \left ( \frac{e^{-I'(k^*/N,x_A)}}{\sqrt{2 N}} \right )
        <
        I(k^*/N,x_A) - I(k^*/N,x_B)
        \le
        - \frac{1}{N} \log \left ( \frac{e^{I'(k^*/N,x_B)}}{\sqrt{2 N}} \right )
    \end{equation}

    And by using that $x_A \le k^*/N \le x_B$ again, we get:

    \begin{equation}
        \frac{1}{N} \log \left ( \frac{e^{-I'(x_B,x_A)}}{\sqrt{2 N}} \right )
        <
        I(k^*/N,x_A) - I(k^*/N,x_B)
        \le
        - \frac{1}{N} \log \left ( \frac{e^{I'(x_A,x_B)}}{\sqrt{2 N}} \right )
    \end{equation}

    \paragraph{Squeeze theorem:}
    This means that in the limit $N \to \infty$, the two rate functions have to be equal:

    \begin{equation}
        \lim_{N \to \infty} I(k^*_N/N,x_A) - I(k^*_N/N,x_B) = 0
    \end{equation}

    Therefore:

    \begin{equation}
        I(x_0^*,x_A) - I(x_0^*,x_B) = 0
    \end{equation}

    Direct calculation shows that this is satisfied by:

    \begin{equation}
    \label{eq:Fisherx0*}
        x_0^* = \frac{\log \left ( \frac{1-x_A}{1-x_B} \right )}{\log \left ( \frac{(1-x_A) x_B}{(1-x_B) x_A} \right )}
    \end{equation}

    This attempts to complete the proof.
    
\end{proof}

\subsubsection{Binomial Fisher limiting prior bounds}

The derivation supports conjecture \ref{conj:FisherLimitingPriorBounds}.

\paragraph{Notation:}
Left and right tail distributions:

\begin{equation}
    T_A(k^\bullet,N) = \sum_{k \ge k^\bullet} p_k(A),
    \quad
    T_B(k^\bullet,N) = \sum_{k \le k^\bullet} p_k(B)
\end{equation}

\paragraph{Derivation:}

Recalling equation \eqref{eq:SymBinEqP} we have (if $\nu^* \ne 0$):

\begin{equation}
    P^*_N \ p_{k^*_N}(A) = (1-P^*_N) \ p_{k^*_N}(B)
\end{equation}

which can be expressed by the left and right tail distributions:

\begin{equation}
\label{eq:P*TATB}
    P^*_N \left ( T_A(k^*_N,N) - T_A(k^*_N+1,N)  \right ) 
    = 
    (1-P^*_N) \left ( T_B(k^*_N,N) - T_B(k^*_N-1,N)  \right ) 
\end{equation}

From the expression for $\nu^*$ in Theorem~\ref{thm:SymBinEqNu} (or more directly from \eqref{eq:vAnu} and \eqref{eq:vBnu}) follows:

\begin{equation}
    1 - T_A(k^*_N) + \nu^* \ p_{k^*_N}(A) = 
    1 - T_B(k^*_N) + (1-\nu^*) \ p_{k^*_N}(B)
\end{equation}

which can be rearranged to:

\begin{equation}
\label{eq:nu*TATB}
    (1-\nu^*) T_A(k^*_N) + \nu^* \ T_A(k^*_N+1,N) =
    \nu^* \ T_B(k^*_N) + (1-\nu^*) T_B(k^*_N-1,N)
\end{equation}

Dividing equation \eqref{eq:P*TATB} with equation \eqref{eq:nu*TATB} results in the following expression for $P^*_N$:
    
\begin{equation}
    \frac{P^*_N}{1-P^*_N} = 
    \frac
    {T_B(k^*_N,N) - T_B(k^*_N-1,N)}
    {T_A(k^*_N,N) - T_A(k^*_N+1,N)}
    \frac
    {(1-\nu^*) T_A(k^*_N) + \nu^* \ T_A(k^*_N+1,N)}
    {\nu^* \ T_B(k^*_N) + (1-\nu^*) T_B(k^*_N-1,N)}
\end{equation}

\begin{lemma}
\label{lemma:TA(k+1)/TA(k)}
    \begin{equation}
        \lim_{N \to \infty} 
        \frac{T_A(k^*_N+1,N)}{T_A(k^*_N,N)} = 
        \lim_{N \to \infty} 
        e^{-I_A'(k^*_N/N)}, \quad
        \lim_{N \to \infty}
        \frac{T_B(k^*_N-1,N)}{T_B(k^*_N,N)} = 
        \lim_{N \to \infty} 
        e^{I_B'(k^*_N/N)}
    \end{equation}
    
\end{lemma}

\begin{proof}
    In \cite{paper:AsymptoticExpansionBinomialTail}, the following asymptotic expansion has been obtained for the tail of Binomial distribution:

    \begin{equation}
        \frac{\sum_{k=n}^{r n} p_n(r n,p)}{p_k(r n,p)} = \frac{p^{-1}-1}{p^{-1}-r} + \mathcal{O}(1/n)
    \end{equation}

    Substituting:

    \begin{equation}
        N = r n, \quad n=k^*=x^*_0 N, \quad p = x_A
    \end{equation}

    we get:

    \begin{equation}
        \frac{T_A(k^*,N)}{p_{k^*}(N,x_A)} =
        \frac{x_A^{-1} - 1}{x_A^{-1} - N/{k^*}} +
        \mathcal{O}(1/N)
    \end{equation}

    Because of

    \begin{equation}
        p_{k^*}(N,x_A) = 
        T_A(k^*,N) - T_A(k^*+1,N)
    \end{equation}

    we get:

    \begin{equation}
        \frac{T_A(k^*,N)}{T_A(k^*,N) - T_A(k^*+1,N)} =
        \frac{k^*/N - x_A k^*/N}{k^*/N - x_A} +
        \mathcal{O}(1/N)
    \end{equation}

    We can express the $T_A(k^*+1,N)/T_A(k^*,N)$ ratio:

    \begin{equation}
        \frac{T_A(k^*_N+1,N)}{T_A(k^*_N,N)} = 
        \frac{x_A}{1-x_A} \frac{1-k^*/N}{k^*/N}
        + \mathcal{O}(1/N)
    \end{equation}

    Recalling the expression for $I'_A(k^*/N)$ in \eqref{eq:BinomialRateFunctionDerivativeI'} we get:

    \begin{equation}
        e^{-I'_A(k^*/N)} =
        \frac{x_A}{1-x_A} \Big / \frac{k^*/N}{1-k^*/N} = 
        \frac{x_A}{1-x_A} \frac{1-k^*/N}{k^*/N}
    \end{equation}

    This proves that:

    \begin{equation}
        \frac{T_A(k^*_N+1,N)}{T_A(k^*_N,N)} = 
        e^{-I'_A(k^*/N)}
        + \mathcal{O}(1/N)
    \end{equation}

    A similar reasoning shows that:

\begin{equation}
    \frac{T_B(k^*_N-1,N)}{T_B(k^*_N,N)} = 
    e^{I_B'(k^*_N/N)}
    + \mathcal{O}(1/N)
\end{equation}
    
\end{proof}

Using lemma \ref{lemma:TA(k+1)/TA(k)} we can simplify the expression for $P^*_N$

\begin{equation}
    \frac{P^*_N}{1-P^*_N} = 
    \frac
    {1 - e^{I_B'(k^*_N/N)}}
    {1 - e^{-I_A'(k^*_N/N)}}
    \frac
    {(1-\nu^*) + \nu^* \ e^{-I_A'(k^*_N/N)}}
    {\nu^* + (1-\nu^*) e^{I_B'(k^*_N/N)}} + o(1)
\end{equation}

(In fact because of the proof of \ref{lemma:TA(k+1)/TA(k)} the $o(1)$ factor could be strengthened to $\mathcal{O}(1/N)$.)
Recalling that $\nu^*$ represents a probability, $\nu^* \in [0,1)$ and introducing the following quantities:

\begin{equation}
    \alpha_A = I'_A(x_0^*)/\beta, \quad
    \alpha_B = I'_B(x_0^*)/\beta
\end{equation}

we get the following bounds for $P^*_N$:

\begin{equation}
    - \alpha_A \beta 
    \lesssim 
    \log \left ( \frac{P^*_N}{1-P^*_N} \right ) - 
    \log \left ( \frac{1-e^{\alpha_B \beta}}
    {1-e^{-\alpha_A \beta}} \right )
    \lesssim
    - \alpha_B \beta
\end{equation}

where we used asymptotic notation $o()$ and $\lesssim$ in accordance with \cite{book:BabaiDiscreteMathematics}.
The results, in a more straightforward notation, look the following:

\begin{equation}
    \liminf_{N \to \infty} 
    \ \log \left ( \frac{P^*_N}{1-P^*_N} \right )
    \ge
    \log \left ( \frac{1-e^{\alpha_B \beta}}
    {1-e^{-\alpha_A \beta}} \right ) 
    -\alpha_A \beta 
\end{equation}

\begin{equation}
    \limsup_{N \to \infty} 
    \ \log \left ( \frac{P^*_N}{1-P^*_N} \right )
    \le
    \log \left ( \frac{1-e^{\alpha_B \beta}}
    {1-e^{-\alpha_A \beta}} \right ) 
    -\alpha_B \beta 
\end{equation}

Conjecture \ref{conj:FisherLimitingPriorBounds} states that these upper and lower bounds are sharp, i.e. the inequalities become equalities:

\begin{equation}
    \label{deriv:liminfPN}
    \liminf_{N \to \infty} 
    \ \log \left ( \frac{P^*_N}{1-P^*_N} \right )
    =
    \log \left ( \frac{1-e^{\alpha_B \beta}}
    {1-e^{-\alpha_A \beta}} \right ) 
    -\alpha_A \beta 
\end{equation}

\begin{equation}
    \label{deriv:limsupPN}
    \limsup_{N \to \infty} 
    \ \log \left ( \frac{P^*_N}{1-P^*_N} \right )
    =
    \log \left ( \frac{1-e^{\alpha_B \beta}}
    {1-e^{-\alpha_A \beta}} \right ) 
    -\alpha_B \beta 
\end{equation}

\subsection{Binomial Bayesian limiting prior approximation}
\label{sec:BinomialBayesianLimitingApproximation}

The derivation supports conjecture \ref{conj:BayesianLimitPrior}, but it starts with an approximation appearing in remark \ref{remark:BayesianLimitPrior}.

\subsubsection{Notation}

Recalling eq. \eqref{eq:BinomialDistribution}:

\begin{equation}
    \label{eq:Appendix_BinomialDistribution}
    p_k(A) = \binom{N}{k} x_A^k (1-x_A)^{N-k}, \quad
    p_k(B) = \binom{N}{k} x_B^k (1-x_B)^{N-k}
\end{equation}

We introduce some notation useful in the continuum ($N \to \infty$) limit:

\begin{equation}
    x = k/N, \quad k = [x N]
\end{equation}

To express $p'^*_k$ with continuum quantities, we introduce a new variable $\tau$:

\begin{equation}
    p'^*_k=\frac{P \ p_k(A)}{P \ p_k(A) + (1-P) p_k(B)}
\end{equation}

\begin{equation}
    \tau = - \log \left ( \frac{p'^*_k}{1-p'^*_k} \right ) 
\end{equation}

\begin{equation}
    \tau = - \log \left ( \frac{P}{1-P} \right ) -  
    \log \left ( \frac{x_A^k (1-x_A)^{N-k}}{x_B^k (1-x_B)^{N-k}} \right )
\end{equation}

introducing a positive ``slope'' quantity \eqref{eq:beta}:

\begin{equation}
    \beta = \log \left ( \frac{x_B (1-x_A)}{x_A (1-x_B)} \right )
\end{equation}

and log-odds $\vartheta$, instead of $P$:

\begin{equation}
\label{eq:varthetaP}
    \vartheta = \log \left ( \frac{P}{1-P} \right ), \quad
    P = \frac{e^\vartheta}{1+e^\vartheta}
\end{equation}

$\tau$ can be expressed as a linear function of $x$:

\begin{equation}
    \tau = - \vartheta + N \beta (x-x_0^*)
\end{equation}

where $x_0^*$ equals to:

\begin{equation}
    x_0^* = \log \left ( \frac{1-x_A}{1-x_B} \right )/\beta =  \frac{\log \left ( \frac{1-x_A}{1-x_B} \right )}{\log \left ( \frac{x_B (1-x_A)}{x_A (1-x_B)} \right )}
\end{equation}

Remarkably, this is the same expression as equation \eqref{eq:Fisherx0*}.
Conversely $x$ can be expressed by $\tau$:

\begin{equation}
    x = x_0^* + \frac{\tau + \vartheta}{N \beta}
\end{equation}

\subsubsection{Expression for the growth factor difference}

In the Binomial case, we can simplify the expression from eq. \eqref{eq:Appendix_GA(P)} and \eqref{eq:Appendix_GB(P)}:

\begin{equation}
    \Delta G_A(P) = \sum_{k} p_k(A) \log \left ( \frac{P \ p_k(A)}{P \ p_k(A)+(1-P) \ p_k(B)} \right )
\end{equation}

\begin{equation}
    \Delta G_B(P) = \sum_{k} p_k(B) \log \left ( \frac{(1-P) p_k(B)}{P \ p_k(A)+(1-P) \ p_k(B)} \right )
\end{equation}

\begin{equation}
    \Delta G_A(P) = \sum_{k} p_k(A) \log \left ( p'^*_k \right )
\end{equation}

\begin{equation}
    \Delta G_B(P) = \sum_{k} p_k(B) \log \left ( 1 - p'^*_k \right )
\end{equation}

\paragraph{Loss function:}

The expressions $\log \left ( p'^*_k \right )$, $\log \left ( 1- p'^*_k \right )$ can be interpreted as a loss function and expressed as functions of $\tau$:

\begin{equation}
    L_A(\tau) = \log \left ( p'^*_k \right ) =
    \log \left ( \frac{e^{-\tau}}{1 + e^{-\tau}} \right ) = 
    - \tau - \log \left ( 1 + e^{-\tau} \right )
\end{equation}

\begin{equation}
    L_B(\tau) = \log \left ( 1 - p'^*_k \right ) =
    \log \left ( \frac{e^{\tau}}{1 + e^{\tau}} \right ) = 
     \tau - \log \left ( 1 + e^{\tau} \right )
\end{equation}

\begin{equation}
    \Delta G_\theta(P) = \sum_{k} p_k(\theta) L_\theta(\tau_k), \quad
    \theta \in \{A,B\}
\end{equation}

\subsubsection{Approximations}
\label{sec:SEC_Approx}

\paragraph{Stirling’s formula:}

In the continuum limit, the discrete probability distributions $p_k(A)$ and $p_k(B)$ can be approximated by a density function, using Stirling's formula \cite{paper:Stirling} \footnote{for notation see also a paper on Large deviation principle \cite{arxiv:LDP}}:

\begin{equation}
    \frac{p_{[N x]}(\theta)}{1/N} \approx f_\theta(x) = \frac{\sqrt{N}}{\sqrt{2 \pi x (1-x)}} e^{- N I(x,x_\theta)}, \quad
    \theta \in \{A,B\}
\end{equation}

where $I(x,x_\theta)=I_\theta(x)$ is the so called rate function introduced in \eqref{eq:BinomialRateFunctionI}.

\begin{equation}
    I(x,x_\theta) = I_\theta(x) = x \log \left ( \frac{x}{x_\theta} \right ) +
    (1-x) \log \left ( \frac{1-x}{1-x_\theta} \right )
\end{equation}

This results in an approximate formula for the growth rate difference:

\begin{equation}
    \Delta G_\theta^\mathrm{S}(\vartheta) = \sum_{k} \frac{1}{N} f_\theta(k/N) L_\theta(\tau_k)
\end{equation}

\paragraph{Euler-Maclaurin formula:}

To transform the summation to an analytically more tractable integration, we can use the Euler-Maclaurin formula \cite{paper:EulerMaclaurinFormula, paper:EulerMaclaurinFormula_Ostrowski}:

\begin{equation}
    \Delta G_\theta^\mathrm{S,E}(\vartheta) = \int_0^1 f_\theta(x) L_\theta(\tau(x)) dx
\end{equation}

\paragraph{Change of variables:}

Finally we change the integration variable $x$ to $\tau$, and substitute the limits of integration $[-\vartheta - N \beta x_0^*, -\vartheta + N \beta (1 - x_0^*)]$ with $(-\infty, \infty)$:

\begin{equation}
    \Delta G_\theta^\mathrm{S,E,C}(\vartheta) = \int_{-\infty}^\infty 
    f_\theta \left ( x_0^* + \frac{\tau + \vartheta}{N \beta} \right ) L_\theta(\tau) \frac{d\tau}{N \beta}
\end{equation}

\paragraph{Approximated gains:}

In the next sections, we will perform calculations based on these approximated expressions:

\begin{equation}
    \Delta G_A^\clubsuit(\vartheta) = \Delta G_A^\mathrm{S,E,C}(\vartheta) = \int_{-\infty}^\infty 
    \frac{1}{N \beta} f_A \left ( x_0^* + \frac{\tau + \vartheta}{N \beta} \right ) L_A(\tau) d\tau
\end{equation}

\begin{equation}
    \Delta G_B^\clubsuit(\vartheta) = \Delta G_B^\mathrm{S,E,C}(\vartheta) = \int_{-\infty}^\infty 
    \frac{1}{N \beta} f_B \left ( x_0^* + \frac{\tau + \vartheta}{N \beta} \right ) L_B(\tau) d\tau
\end{equation}

\subsubsection{Performing the Integral}

\paragraph{Notations:}

Exponent related to the growth rate difference decrease:

\begin{equation}
    \varepsilon(x_A,x_B) = I(x_0^*,x_A) = I(x_0^*,x_B) > 0
\end{equation}

A symmetric expression for $\varepsilon(x_A,x_B)$ can look like:

\begin{equation}
    \varepsilon(x_A,x_B) = x_0^* \log
    \left (
    \frac{x_0^*}{\sqrt{x_A x_B}}
    \right ) +
    (1-x_0^*) \log
    \left (
    \frac{1-x_0^*}{\sqrt{(1-x_A) (1-x_B)}}
    \right )
\end{equation}

Further exponents:

\begin{equation}
    \alpha_A = I'_A(x_0^*)/\beta > 0, \quad
    \alpha_B = I'_B(x_0^*)/\beta < 0
\end{equation}

\paragraph{Leading order expression:}

By taking the first order Taylor expansion of $I_\theta(x)$ around $x_0^*$, $I_\theta(x) = \varepsilon(x_A,x_B) + I'_\theta(x_0^*) (x-x_0^*) + \mathcal{O}((x-x_0^*)^2)$ we get:

\begin{equation}
    \Delta G_\theta^\clubsuit(\vartheta) = 
    \frac{e^{-\alpha_\theta \vartheta}}{\sqrt{2 \pi x_0^* (1-x_0^*)}} \frac{e^{- N \varepsilon(x_A,x_B)}}{\sqrt{N} \beta}
    \left ( C_\theta + \mathcal{O}(1/N) \right )
\end{equation}

Where the constants $C_A$ and $C_B$ are the results of the following integrals:

\begin{equation}
    \label{eq:CA_o}
    C_A = - \int_{-\infty}^\infty e^{-\alpha_A \tau} (\tau + \log \left ( 1 + e^{-\tau} \right )) d \tau
\end{equation}

\begin{equation}
    \label{eq:CB_o}
    C_B = \int_{-\infty}^\infty e^{-\alpha_B \tau} (\tau - \log \left ( 1 + e^{\tau} \right )) d \tau
\end{equation}

Remarkably, these integrals can be expressed in closed form:

\begin{equation}
    \label{eq:CA_res}
    C_A = - \frac{\pi}{\alpha_A} \frac{1}{\sin(\pi \alpha_A)}, \quad \text{if } \alpha_A \in (0,1)
\end{equation}

\begin{equation}
    \label{eq:CB_res}
    C_B = - \frac{\pi}{\alpha_B} \frac{1}{\sin(\pi \alpha_B)}, \quad \text{if } \alpha_B \in (-1,0)
\end{equation}

The results can be obtained by symbolic integration \cite{tool:WolframIntegrate} available in \textit{Mathematica 13.0}, and are formally calculated in Section~\ref{sec:FormalCalculationOfTheIntegral}.

\subsubsection{Obtaining the prior}

In equilibrium, the growth rate difference has to be the same for scenarios A and B.
The equation for the approximated equilibrium quantity $\vartheta^*_\clubsuit$, looks the following:

\begin{equation}
    \Delta G_A^\clubsuit(\vartheta^*_\clubsuit) = \Delta G_B^\clubsuit(\vartheta^*_\clubsuit)
\end{equation}

which simplifies to:

\begin{equation}
    \frac{e^{-\alpha_A \vartheta^*_\clubsuit}}{\alpha_A \sin(\pi \alpha_A)} = 
    \frac{e^{-\alpha_B \vartheta^*_\clubsuit}}{\alpha_B \sin(\pi \alpha_B)} 
\end{equation}

\begin{equation}
    \vartheta^*_\clubsuit = \frac{\log \left (  \frac{\alpha_B \sin(\pi \alpha_B)}{\alpha_A \sin(\pi \alpha_A)}  \right ) }{\alpha_A - \alpha_B}
\end{equation}

\subsubsection{Simplifying the result}

\paragraph{Notation:}

\begin{equation}
    x_0^* = \log \left ( \frac{1-x_A}{1-x_B} \right )/\beta, \quad
    1-x_0^* = \log \left ( \frac{x_B}{x_A} \right )/\beta
\end{equation}

Recalling $I'_\theta(x)$ from \eqref{eq:BinomialRateFunctionDerivativeI'}:

\begin{equation}
    I'_\theta(x) = 
    \log \left ( \frac{x}{1-x} \frac{1-x_\theta}{x_\theta} \right )
\end{equation}

\begin{equation}
    \alpha_\theta = I'_\theta(x) / \beta
\end{equation}

\begin{equation}
    \alpha_\theta = \frac{\log \left ( \frac{x_0^*}{1-x_0^*} \frac{1-x_\theta}{x_\theta} \right )}{\beta} =
    \frac{\log \left ( \frac{x_0^*}{1-x_0^*} \right ) + \log \left ( \frac{1-x_\theta}{x_\theta} \right )}{\beta}
\end{equation}

\paragraph{Identities:}

First, we can observe that:

\begin{equation}
    \alpha_A - \alpha_B = 
    \frac{\log \left ( \frac{1-x_A}{x_A} \right ) - \log \left ( \frac{1-x_B}{x_B} \right )}{\beta} = 1
\end{equation}

Therefore, the expression, containing sinus terms, simplifies to:

\begin{equation}
    \frac{\sin(\pi \alpha_B)}{\sin(\pi \alpha_A)} = 
    \frac{\sin(\pi (\alpha_A - 1))}{\sin(\pi \alpha_A)} = -1
\end{equation}

This simplifies the result:

\begin{equation}
    \vartheta^*_\clubsuit = \log \left (  \frac{-\alpha_B}{\alpha_A}  \right )
    =
    \vartheta^*_\clubsuit = \log \left (  \frac{
    \log \left ( \frac{1-x_0^*}{x_0^*} \frac{x_B}{1-x_B} \right )
    }{
    \log \left ( \frac{x_0^*}{1-x_0^*} \frac{1-x_A}{x_A} \right )
    }  \right )
\end{equation}

\paragraph{Prior probability:}

Recalling equation \eqref{eq:varthetaP}, we can express the approximated prior probability $P^*_\clubsuit$:

\begin{equation}
    P^*_\clubsuit = \frac{e^{\vartheta^*_\clubsuit}}{1+e^{\vartheta^*_\clubsuit}}
    =
    \frac{
    \log \left ( \frac{1-x_0^*}{x_0^*} \frac{x_B}{1-x_B} \right )
    }
    {
    \log \left ( \frac{x_0^*}{1-x_0^*} \frac{1-x_A}{x_A} \right ) +
    \log \left ( \frac{1-x_0^*}{x_0^*} \frac{x_B}{1-x_B} \right )
    }
\end{equation}

or

\begin{equation}
\label{eq:P*_club}
    P^*_\clubsuit = \frac{
    \log \left ( \frac{1-x_0^*}{x_0^*} \frac{x_B}{1-x_B} \right )
    }
    {
    \beta
    } = 
    \frac{
    \log \left ( \frac{1-x_0^*}{x_0^*} \frac{x_B}{1-x_B} \right )
    }
    {
    \log \left ( \frac{x_B}{x_A} \frac{1-x_A}{1-x_B} \right )
    }
\end{equation}

Remark \ref{remark:BayesianLimitPrior} states that for $0<x_A<x_B<1$ values,

\begin{equation}
\label{eq:Papprox_club}
    P^\approx_\loopedsquare(x_A,x_B) = P^*_\clubsuit(x_A,x_B)
\end{equation}

is a ``good approximation'' of $P_N^*$ as $N \to \infty$.

\subsubsection{Formal calculation of the Integrals}
\label{sec:FormalCalculationOfTheIntegral}
In this section we derive the expressions \eqref{eq:CA_res},\eqref{eq:CB_res} from \eqref{eq:CA_o},\eqref{eq:CB_o}:

\begin{lemma}
\label{lemma:1/cosh}
    \begin{equation}
        J=\int_{-\infty}^\infty \frac{e^{a t}}{\cosh(t)} dt = \frac{\pi}{\cos\left ( \frac{\pi}{2} a\right )}, \quad \text{if } a \in (-1,1)
    \end{equation}
\end{lemma}

\begin{proof}
    The complex function $1/\cosh(z): \mathbb{C} \mapsto \mathbb{C}$ is antiperiodic on the imaginary axis, with antiperiod $\pi$

    \begin{figure}[H]
    \centering
    \includegraphics[width=12 cm]{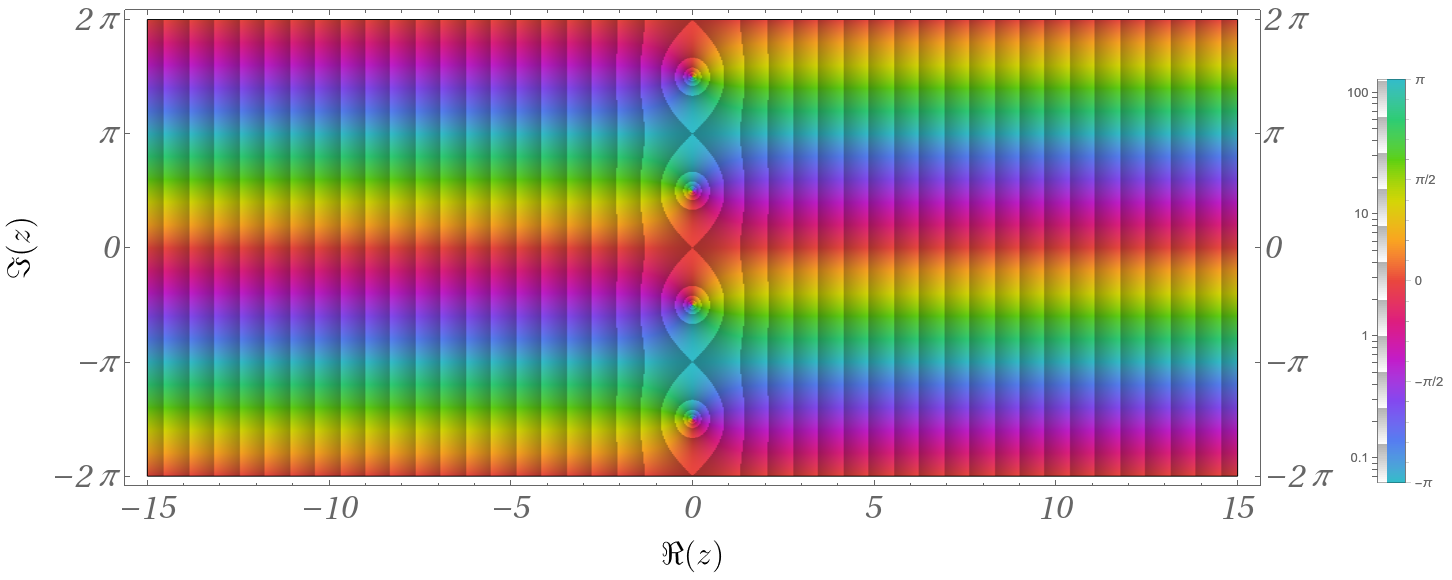}
    \caption{Complex plot \cite{tool:WolframComplexPlot} of $f_0 : \mathbb{C} \mapsto \mathbb{C}, f_0(z)=1/\cosh(z)$.}
    \label{fig:ComplexPlot}
    \end{figure}

    \begin{equation}
        \frac{1}{\cosh(z+i \pi)} = \frac{2}{e^{z+i \pi}+e^{-z-i \pi}} = 
        \frac{2}{e^{i \pi} e^{z}+ e^{-i \pi} e^{-z}} = \frac{-2}{ e^{z}+  e^{-z}} = \frac{-1}{\cosh(z)}
    \end{equation}

    and the absolute value of $|1/\cosh(x+i y)|$ goes to $0$ as $|x|$ tends to infinity:

    \begin{equation}
        \frac{1}{|\cosh(z)|}=\frac{1}{\sqrt{\cosh(z)\overline{\cosh(z)}}} = 
        \frac{1}{\sqrt{\cosh(z) \cosh(\overline{z})}}
    \end{equation}

    \begin{equation}
        \frac{1}{|\cosh(x+i y)|^2} = \frac{4}{(e^{x+i y}+e^{-x-i y})(e^{x-i y}+e^{-x+i y})}
    \end{equation}

    \begin{equation}
    \frac{1}{|\cosh(x+i y)|^2} = \frac{4}{e^{2 x}+e^{-2i y}+ e^{2i y} + e^{-2 x}}
    \end{equation}

    \begin{equation}
    \frac{1}{|\cosh(x+i y)|^2} = \frac{4}{e^{2 x}+2 \cos(2 y) + e^{-2 x}}
    \end{equation}

    meaning that

    \begin{equation}
    \label{eq:coshAsymptotic}
    \frac{1}{|\cosh(x+i y)|} < 2 e^{-|x|}
    \end{equation}

    Now we can define the parametric complex function:

    \begin{equation}
        f_a(z) = \frac{e^{a z}}{\cosh(z)}, \quad a \in \mathbb{R}
    \end{equation}

    This function is shifted by an extra $a$-dependent phase, while $z$ is shifted by $i \pi$ in the imaginary axis:

    \begin{equation}
        f_a(z+i \pi) = - e^{i a \pi} f_a(z)
    \end{equation}

    and due to eq. \eqref{eq:coshAsymptotic}:

    \begin{equation}
        |f_a(z)| < 2 e^{-|x|+a x}
    \end{equation}

    which goes to 0 as $|x| \to \infty$, if $a \in (-1,1)$.

    $f_a(z)$ has poles only at $\{i \pi/2 + i \pi k \}_{k\in \mathbb{Z}}$, and is analytical in every other point.

    Now, we can apply Cauchy's residue theorem on the complex contour integral:

    \begin{equation}
        \begin{split}
            \Gamma_1(T) = [-T,T], &\quad \Gamma_2(T) = [T,T+i \pi] \\
            \Gamma_3(T) = [T+ i \pi, -T + i \pi], &\quad \Gamma_4(T) = [-T + i \pi,-T]
        \end{split}
    \end{equation}

    \begin{equation}
        \Gamma(T) = (\Gamma_1(T), \Gamma_2(T), \Gamma_3(T), \Gamma_4(T)) 
    \end{equation}

    \begin{equation}
        \oint_{\Gamma(T)} f_a(z) dz = 2 \pi i \  \mathrm{Res}(f_a,i \pi/2)
    \end{equation}

    \begin{equation}
        \lim_{T \to \infty} \oint_{\Gamma(T)} f_a(z) dz = 2 \pi i \  \mathrm{Res}(f_a,i \pi/2)
    \end{equation}

    \begin{equation}
        J + 0 + (-1) (- e^{i a \pi} J) + 0 = 2 \pi i \ \mathrm{Res}(f_a,i \pi/2)
    \end{equation}

    To determine the residue of $f_a(z)$, we can approximate the expression in the denominator around $z_0=i \pi/2$:

    \begin{equation}
        \cosh(z) = 0 + \sinh(i \pi/2) (z-i \pi/2) + \mathcal{O}((z-i \pi/2)^2)
    \end{equation}

    \begin{equation}
        \cosh(z) = 0 + i (z-i \pi/2) + \mathcal{O}((z-i \pi/2)^2)
    \end{equation}

    meaning that the residue of $f_a(z)$ at $z_0=i \pi/2$ equals to:

    \begin{equation}
        \mathrm{Res}(f_a,i \pi/2) = \frac{e^{i a \pi/2}}{i}
    \end{equation}

    resulting:

    \begin{equation}
        J + (-1) (- e^{i a \pi} J) = 2 \pi i \frac{e^{i a \pi/2}}{i}
    \end{equation}

    \begin{equation}
        J = 2 \pi \frac{e^{i a \pi/2}}{1+e^{i a \pi}} =
        2 \pi \frac{1}{e^{-i a \pi/2}+e^{i a \pi/2}} =
        \frac{\pi}{\cos(\pi a/2)}
    \end{equation}
    
\end{proof}

Similar integrals can also be found in standard complex analysis textbooks such as \cite{book:SteinComplexAnalysis}.

\begin{lemma}
\label{lemma:C(alpha)}
    \begin{equation}
        C(\alpha) = \int_{-\infty}^\infty e^{-\alpha \tau} (\tau + \log \left ( 1 + e^{-\tau} \right )) d \tau = \frac{\pi}{\alpha} \frac{1}{\sin(\pi \alpha)}, \quad \text{if } \alpha \in (0,1)
    \end{equation}
\end{lemma}

\begin{proof}

    Integration by parts results (because the boundary terms go to 0 if $\alpha \in (0,1)$):

    \begin{equation}
        C = 0 - \int_{-\infty}^\infty \frac{e^{-\alpha \tau}}{-\alpha}
        \left( 1 - \frac{e^{-\tau}}{1+e^{-\tau}}\right ) d \tau
    \end{equation}

    \begin{equation}
        C = \frac{1}{\alpha} \int_{-\infty}^\infty \frac{e^{-\alpha \tau}}{1+e^{-\tau}} d \tau
    \end{equation}

    \begin{equation}
        C = \frac{1}{\alpha} \int_{-\infty}^\infty \frac{e^{(1/2-\alpha) \tau}}{e^{\tau/2}+e^{-\tau/2}} d \tau
    \end{equation}

    Substituting the variables:

    \begin{equation}
        t = \tau/2, \quad \tau = 2 t
    \end{equation}

    \begin{equation}
        a = (1/2 - \alpha) \ 2 = 1- 2 \alpha, \quad a \in (-1,1)
    \end{equation}

    \begin{equation}
        C = \frac{1}{\alpha} \int_{-\infty}^\infty \frac{e^{a t}}{2 \cosh{t}} 2 dt
    \end{equation}

    using lemma \ref{lemma:1/cosh} we get:

    \begin{equation}
        C = \frac{1}{\alpha} \frac{\pi}{\cos(\pi a /2)}
    \end{equation}

    \begin{equation}
         C = \frac{1}{\alpha} \frac{\pi}{\cos(\pi/2 - \pi \alpha)}
    \end{equation}

    \begin{equation}
         C = \frac{1}{\alpha} \frac{\pi}{\sin(\pi \alpha)}
    \end{equation}
    
\end{proof}

with lemma \ref{lemma:C(alpha)} the expressions \eqref{eq:CA_res},\eqref{eq:CB_res} follow from \eqref{eq:CA_o},\eqref{eq:CB_o}.

\begin{remark}
    The function $C(\alpha)$ can be extended to complex $\alpha \in \mathbb{C}$ values, if $0<\Re(\alpha)<1$.
    
\end{remark}

For further use we introduce $C_A(.),C_B(.)$ functions as well:

\begin{equation}
    \label{eq:CA_f}
    C_A(\alpha_A) = - \frac{\pi}{\alpha_A} \frac{1}{\sin(\pi \alpha_A)}, \quad \text{if } \Re(\alpha_A) \in (0,1)
\end{equation}

\begin{equation}
    \label{eq:CB_f}
    C_B(\alpha_B) = - \frac{\pi}{\alpha_B} \frac{1}{\sin(\pi \alpha_B)}, \quad \text{if } \Re(\alpha_B) \in (-1,0)
\end{equation}

\subsection{Asymptotic expansion of Binomial Bayesian prior approximation}
\label{sec:AsymptoticExpansionBinomialBayesianPriorApproximation}

The derivation supports the approximation in remark \ref{remark:AsymptoticExpansionSubleadingTermApproximation}.

\paragraph{Notation:}

Normalized growth factor difference:

\begin{equation}
    \Delta g_\theta^{\clubsuit,\varphi}(\vartheta) = 
    \frac{\Delta G_\theta^{\clubsuit,\varphi}(\vartheta)}{
    \frac{1}{\sqrt{2 \pi x_0^* (1-x_0^*)}} \frac{e^{- N \varepsilon(x_A,x_B)}}{\sqrt{N} \beta}
    }
\end{equation}

\paragraph{First-order asymptotic expansion:}

We need to solve the following equation to obtain a first-order asymptotic expansion:

\begin{equation}
    \Delta g_A^\clubsuit(\vartheta^*_\clubsuit + \sampi/N) - 
    \Delta g_B^\clubsuit(\vartheta^*_\clubsuit + \sampi/N) 
    = 0 + \mathcal{O}(1/N^2)
\end{equation}

\subsubsection{Ingredients}

\paragraph{Stirling series:}
For further coefficients and more context, see \cite{book:OrszagBender, book:NISThandbook,paper:StirlingSeriesMadeEasy}

\begin{equation}
    \log(n!) = n \log(n) - n + \frac{1}{2} \log(n) + \frac{1}{2} \log(2 \pi) + \frac{1}{12} \frac{1}{n} + \mathcal{O}(1/n^2)
\end{equation}

\paragraph{Taylor expansion:}

\begin{equation}
    I_\theta(x) = I_\theta(x_0^*) + I'_\theta(x) (x-x_0^*) +
    \frac{1}{2} I''_\theta(x) (x-x_0^*)^2 
    + \mathcal{O}((x-x_0^*)^3)
\end{equation}

\paragraph{Perturbative approach:}

For appropriate $X(\tau)$ functions:

\begin{equation}
    e^{-\alpha \tau + \frac{1}{N} X(\tau) + \mathcal{O}(1/N^2)} = 
    e^{-\alpha \tau} \left ( 1 + \frac{1}{N} X(\tau) + \mathcal{O}(1/N^2)  \right )
\end{equation}

For $n \in \mathbb{N}$ monomials of $\tau$:

\begin{equation}
    \int_{-\infty}^\infty e^{-\alpha \tau} \tau^n L(\tau) d \tau = 
    (-1)^n \frac{\partial^n}{\partial \alpha^n} 
    \int_{-\infty}^\infty e^{-\alpha \tau} L(\tau) d \tau
\end{equation}

\subsubsection{Derivation}

First-order approximation of density function:

\begin{equation}
    f_\theta(x) = \frac{1}{\sqrt{2 \pi N}}
    e^{-N I_\theta(x) + n(x) + \frac{1}{N} b_1(x)}
\end{equation}

where:

\begin{equation}
    n(x) = - \frac{1}{2} \log(x (1-x)), \quad
    b_1(x) = \frac{1}{12} \left ( 1 - \frac{1}{x} - \frac{1}{1-x} \right )
\end{equation}

Introducing:

\begin{equation}
    J_\theta(x) = -N I_\theta(x) + n(x) + \frac{1}{N} b_1(x)
\end{equation}

and

\begin{equation}
    x = x_0^* + \frac{1}{N} \Delta x + \frac{1}{N^2} \sampi
\end{equation}

we get:

\begin{equation}
    \begin{split}
        J_\theta(x) = & 
        -N I_\theta(x_0^*) + \\ 
        & n(x_0^*) - I'_\theta(x_0^*) \Delta x + \\
        & \frac{1}{N} 
        \left ( b_1(x_0^*) + n'(x_0^*) \Delta x - I'_\theta(x_0^*) \sampi
        - \frac{1}{2} I''_\theta(x_0^*) (\Delta x)^2
        \right ) + \\ 
        & \mathcal{O}(1/N^2)
    \end{split}
\end{equation}

Introducing:

\begin{equation}
    f_0^* = f_A(x_0^*) = f_B(x_0^*)
\end{equation}

we get an expression for the ``normalized'' density function:

\begin{equation}
        \frac{f_\theta(x)}{f_0^*} = 
        e^{- I'_\theta(x_0^*) \Delta x}
        \left ( 
        1 
        + \frac{1}{N}
        \left (
         n'(x_0^*) \Delta x - I'_\theta(x_0^*) \sampi
            - \frac{1}{2} I''_\theta(x_0^*) (\Delta x)^2
        \right )
        + \mathcal{O}(1/N^2)  \right )
\end{equation}

\paragraph{Expressing the growth rates:}

\begin{equation}
    \Delta x = \frac{\tau+\vartheta}{\beta}
\end{equation}

We can introduce a curvature factor euqal to \eqref{eq:BinomialRateFunctionSecondDerivativeI''}:

\begin{equation}
    \kappa(x) = I''_A(x) = I''_B(x) =
    \frac{1}{x} + \frac{1}{1-x} = 
    \frac{1}{x(1-x)}
\end{equation}

After solving the equation perturbatively \cite{book:PerturbationMethods}, up to the first term in $1/N$ expansion:

\begin{equation}
    \int_{-\infty}^\infty f_A(x(\tau)) L_A(\tau) d \tau
    =
    \int_{-\infty}^\infty f_B(x(\tau)) L_B(\tau) d \tau
\end{equation}

and solving for $\sampi$ we get:

\begin{equation}
    \begin{split}
        \sampi = &
        \left (
        \frac{n'(x_0^*)}{\beta}  -
        \frac{\kappa(x_0^*) \vartheta^*_\clubsuit}{\beta^2} 
        \right )
        \left (
        \frac{C_B'(\alpha_B)}{C_B(\alpha_B)} -
        \frac{C_A'(\alpha_A)}{C_A(\alpha_A)}
        \right ) 
        + \\
        & \frac{1}{2} \frac{\kappa(x_0^*)}{\beta^2}
        \left (
        \frac{C_B''(\alpha_B)}{C_B(\alpha_B)} -
        \frac{C_A''(\alpha_A)}{C_A(\alpha_A)}
        \right )
    \end{split}
\end{equation}

After simplifications:

\begin{equation}
    \left (
        \frac{C_B'(\alpha_B)}{C_B(\alpha_B)} -
        \frac{C_A'(\alpha_A)}{C_A(\alpha_A)}
        \right ) =
        \frac{d}{d \alpha} 
        \log \left (
        \frac{C_B(\alpha-1)}{C_A(\alpha)} 
        \right )
        \Biggr|_{\alpha=\alpha_A} =
        \frac{1}{\alpha_A (1-\alpha_A)}
\end{equation}

We can obtain an explicit expression for the first-order asymptotic approximation term:

\begin{equation}
    \label{eq:sampiResult01}
    \sampi = \frac{1}{\alpha_A (1-\alpha_A)}
    \left (
    \frac{n'(x_0^*)}{\beta}  -
    \frac{\kappa(x_0^*) \vartheta^*_\clubsuit}{\beta^2} 
    \right ) +
    \frac{1}{2} \frac{\kappa(x_0^*)}{\beta^2}
    \left (
    \frac{C_B''(\alpha_B)}{C_B(\alpha_B)} -
    \frac{C_A''(\alpha_A)}{C_A(\alpha_A)}
    \right )
\end{equation}

\begin{equation}
    \label{eq:sampiResult02}
    \left (
    \frac{C_B''(\alpha_B)}{C_B(\alpha_B)} -
    \frac{C_A''(\alpha_A)}{C_A(\alpha_A)}
    \right ) = 
    -2 \frac{1-2 \alpha_A}{\alpha_A^2 (1-\alpha_A)^2} -
    \frac{2 \pi}{\alpha_A (1-\alpha_A)} \frac{\cos(\pi \alpha_A)}{\sin(\pi \alpha_A)}
\end{equation}

\subsection{Limiting prior approximation for general Statistical games}

A similar calculation can be performed for general statistical games \ref{def:StatisticalGame} with an isoelastic utility function.
In this case, we can start the derivation by recalling the isoelastic equilibrium splitting ratios \eqref{eq:ppkPgamma}:

\begin{equation}
    p'^*_{\gamma,k} = \frac{(P \ p_k(A))^{1/\gamma}}{(P \ p_k(A))^{1/\gamma} + ((1-P) p_k(B))^{1/\gamma}}
\end{equation}

Introducing variables similar to the ones in Section~\ref{sec:BinomialBayesianLimitingApproximation}:

\begin{equation}
    \tau = - \log \left ( \frac{p'^*_{\gamma,k}}{1-p'^*_{\gamma,k}} \right ) 
\end{equation}

\begin{equation}
    \tau = (- \vartheta + N \beta (x-x_0^*))/\gamma
\end{equation}

\begin{equation}
    x = x_0^* + \frac{\gamma \ \tau  + \vartheta}{N \beta}
\end{equation}

Recalling the isoelastic utility function from equation \eqref{eq:ugamma}:

\begin{equation}
    u_\gamma(c) = \frac{c^{1-\gamma}-1}{1-\gamma}
\end{equation}

We define the expected utilities for each scenario:

\begin{equation}
    U_A(P) = \sum_{k} p_k(A) u_\gamma \left ( p'^*_{\gamma,k} \right )
\end{equation}

\begin{equation}
    U_B(P) = \sum_{k} p_k(B) u_\gamma \left ( 1 - p'^*_{\gamma,k} \right )
\end{equation}

Expressions for the isoelastic loss functions:

\begin{equation}
    L^\gamma_A(\tau) = u_\gamma \left ( p'^*_{\gamma,k} \right ) =
    \frac{\left ( \frac{e^{-\tau}}{1 + e^{-\tau}} \right )^{1-\gamma}-1}{1-\gamma} 
\end{equation}

\begin{equation}
    L^\gamma_B(\tau) = u_\gamma \left ( 1 - p'^*_{\gamma,k} \right ) =
    \frac{\left ( \frac{e^{\tau}}{1 + e^{\tau}} \right )^{1-\gamma}-1}{1-\gamma} 
\end{equation}

The expected utilities expressed by the loss functions:

\begin{equation}
    U_\theta(P) = \sum_{k} p_k(\theta) L^\gamma_\theta(\tau_k), \quad
    \theta \in \{A,B\}
\end{equation}

After making the same approximations as in Section~\ref{sec:SEC_Approx} we get:

\begin{equation}
    U_\theta^\mathrm{S,E,C}(\vartheta) = \int_{-\infty}^\infty 
    f_\theta \left ( x_0^* + \frac{\gamma \ \tau + \vartheta}{N \beta} \right ) L^\gamma_\theta(\tau) \frac{\gamma}{N \beta} d \tau
\end{equation}

which can be expressed as:

\begin{equation}
    U_\theta^\clubsuit(\vartheta) = 
    \frac{e^{-\alpha_\theta \vartheta}}{\sqrt{2 \pi x_0^* (1-x_0^*)}} \frac{e^{- N \varepsilon(x_A,x_B)}}{\sqrt{N} \beta}
    \left ( C^\gamma_\theta + \mathcal{O}(1/N) \right )
\end{equation}

where $C^\gamma_A$, $C^\gamma_B$ are defined by the following integrals:

\begin{equation}
    \label{eq:CA_gamma}
    C^\gamma_A = \int_{-\infty}^\infty e^{-\gamma \alpha_A \tau} 
    \frac{\left ( \frac{e^{-\tau}}{1 + e^{-\tau}} \right )^{1-\gamma}-1}{1-\gamma} 
    \gamma d \tau
\end{equation}

\begin{equation}
    \label{eq:CB_gamma}
    C^\gamma_B = \int_{-\infty}^\infty e^{-\gamma \alpha_B \tau}
    \frac{\left ( \frac{e^{\tau}}{1 + e^{\tau}} \right )^{1-\gamma}-1}{1-\gamma} 
    \gamma d \tau
\end{equation}

The result can be obtained by the symbolic integration \cite{tool:WolframIntegrate} available in \textit{Mathematica 13.0}, and are formally calculated in Section~\ref{sec:gammaGammaIntegrals}.

\begin{equation}
    C^\gamma_A = \gamma \frac{\Gamma(-\alpha_A \gamma) \Gamma(1-\gamma(1-\alpha_A))}{\Gamma(2-\gamma)}
\end{equation}

if $0 < \Re(\gamma \alpha_A) < 1$ and $\Re(\gamma (1-\alpha_A)) < 1 $

\begin{equation}
    C^\gamma_B = \gamma \frac{\Gamma(\alpha_B \gamma) \Gamma(1-\gamma(1+\alpha_B))}{\Gamma(2-\gamma)}
\end{equation}

if $-1 < \Re(\gamma \alpha_B) < 0$ and $\Re(\gamma (1+\alpha_B)) < 1 $

Where $\Gamma(z)$ is the Gamma function \cite{book:SeymourMathematicalHandbook, book:Bronshtein, book:NISThandbook, book:Abramowitz, book:HigherTranscendentalFunctions}.

\paragraph{Gamma function identities:}

Euler's Reflection formula \cite{book:SeymourMathematicalHandbook}:

\begin{equation}
    \Gamma(z) \Gamma(1-z) = \frac{\pi}{\sin(\pi z)}
\end{equation}

Recursion formula \cite{book:SeymourMathematicalHandbook}:

\begin{equation}
    \Gamma(z+1) = z \ \Gamma(z)
\end{equation}

\paragraph{Rewriting the result:} After using the identities for Gamma functions, we can get:

\begin{equation}
    C^\gamma_A = - \frac{\pi}{\alpha_A \sin(\pi \gamma \alpha_A)}
    \frac{\Gamma(1-\gamma(1-\alpha_A))}
    {\Gamma(2-\gamma) \Gamma(\gamma \alpha_A)}
\end{equation}

\begin{equation}
    C^\gamma_B = - \frac{\pi}{\alpha_B \sin(\pi \gamma \alpha_B)}
    \frac{\Gamma(1-\gamma(1+\alpha_B))}
    {\Gamma(2-\gamma) \Gamma(-\gamma \alpha_B)}
\end{equation}

Alternatively, we can express these formulas by the Beta function \cite{book:HigherTranscendentalFunctions, book:SeymourMathematicalHandbook, book:Bronshtein, book:NISThandbook, book:Abramowitz}:

\begin{equation}
    C^\gamma_A = - \frac{1}{\alpha_A}
    B(1-\gamma \alpha_A, 1- \gamma (1-\alpha_A))
\end{equation}

\begin{equation}
    C^\gamma_B = \frac{1}{\alpha_B}
    B(1-\gamma (-\alpha_B), 1- \gamma (1+\alpha_B))
\end{equation}

For further use we introduce $C^\gamma_A(.),C^\gamma_B(.)$ functions as well:

\begin{equation}
    \label{eq:CAgamma_f}
    C^\gamma_A(\alpha_A) = - \frac{1}{\alpha_A}
    B(1-\gamma \alpha_A, 1- \gamma (1-\alpha_A))
\end{equation}

\begin{equation}
    \label{eq:CBgamma_f}
    C^\gamma_B(\alpha_B) = \frac{1}{\alpha_B}
    B(1-\gamma (-\alpha_B), 1- \gamma (1+\alpha_B))
\end{equation}

\paragraph{Limiting prior approximation:}
We need to solve the following equation up to the zeroth-order:

\begin{equation}
    U_A^\clubsuit(\vartheta^*_{\gamma,\clubsuit}) =
    U_B^\clubsuit(\vartheta^*_{\gamma,\clubsuit})
\end{equation}

Resulting in the equation:

\begin{equation}
    e^{-\alpha_A \vartheta^*_{\gamma,\clubsuit}} C^\gamma_A = 
    e^{-\alpha_B \vartheta^*_{\gamma,\clubsuit}} C^\gamma_B
\end{equation}

Which has a simple solution for $\vartheta^*_{\gamma,\clubsuit}$:

\begin{equation}
\label{eq:vartheta_gamma_approx}
    \vartheta^*_{\gamma,\clubsuit} = \frac{\log \left (  \frac{C^\gamma_A}{C^\gamma_B}  \right ) }{\alpha_A - \alpha_B}
\end{equation}

The expression simplifies remarkably, because $\alpha_A - \alpha_B = 1$:

\begin{equation}
    \vartheta^*_{\gamma,\clubsuit} = 
    \log \left (  \frac{-\alpha_B}{\alpha_A}  \right ) =
    \log \left (  \frac{1-\alpha_A}{\alpha_A}  \right ), \quad
    \text{if } \frac{\gamma-1}{\gamma} < \alpha_A < \frac{1}{\gamma}
\end{equation}

\begin{remark}
    $\alpha_A \in (0,1)$, therefore there is a critical value of relative risk aversion parameter $\gamma$, under which a finite equilibrium log-odds approximation can be made.
    \begin{equation}
        \gamma \in (0,\gamma^{\overline{\diamond}}), \quad \gamma^{\overline{\diamond}} = 2
    \end{equation}

    For general $0<x_A<x_B<1$ values, the critical maximal relative risk aversion can be expressed by $\alpha_A(x_A,x_B)$:

    \begin{equation}
        \gamma^\diamond(x_A,x_B) = 
        \min \left (
        \frac{1}{\alpha_A(x_A,x_B)},
        \frac{1}{1-\alpha_A(x_A,x_B)}
        \right )
    \end{equation}

    \begin{equation}
        1 < \gamma^\diamond(x_A,x_B) \le 2
    \end{equation}
    
\end{remark}

\begin{remark}
    The obtained approximation $\vartheta^*_{\gamma,\clubsuit}$ in eq. \eqref{eq:vartheta_gamma_approx} is the same for all $\gamma < \gamma^\diamond(x_A,x_B)$, i.e. independent of the relative risk aversion $\gamma$.
    
\end{remark}

\subsubsection{Visualization of the critical relative risk aversion}

\begin{figure}[H]
    \centering
    \begin{subfigure}[b]{0.45\textwidth}
        \includegraphics[width=\textwidth]{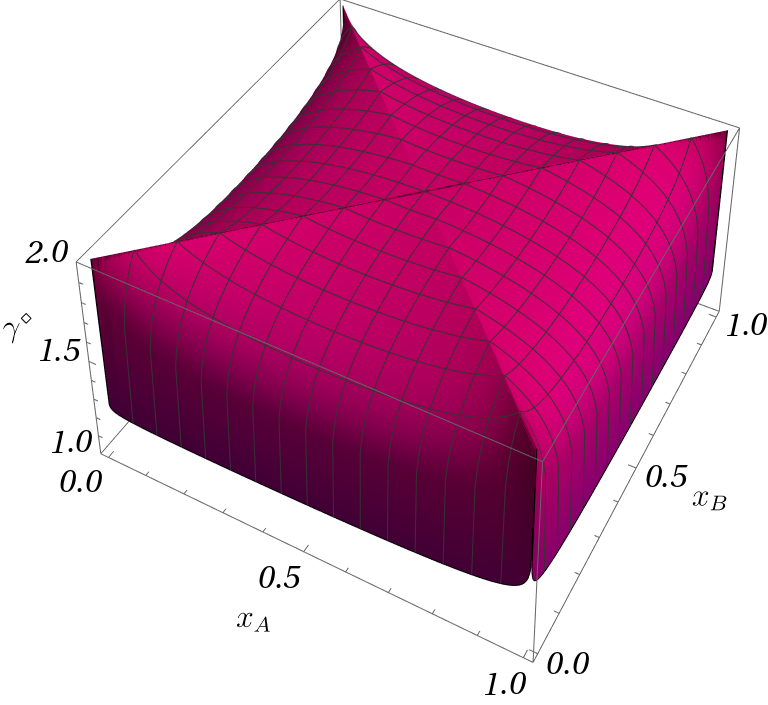}
        \caption{3D plot}
        \label{fig:gamma_diamond_3D}
    \end{subfigure}
    \hspace{0.05\textwidth} 
    \begin{subfigure}[b]{0.45\textwidth}
        \includegraphics[width=\textwidth]{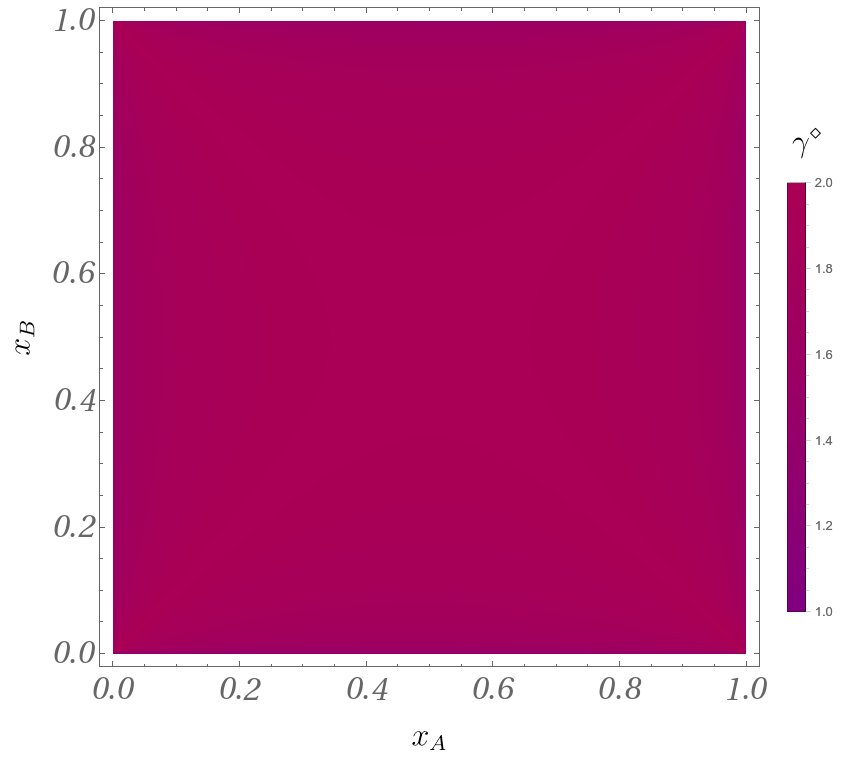}
        \caption{Density plot}
        \label{fig:gamma_diamond_Density}
    \end{subfigure}
    \par
    \begin{subfigure}[b]{0.45\textwidth}
        \includegraphics[width=\textwidth]{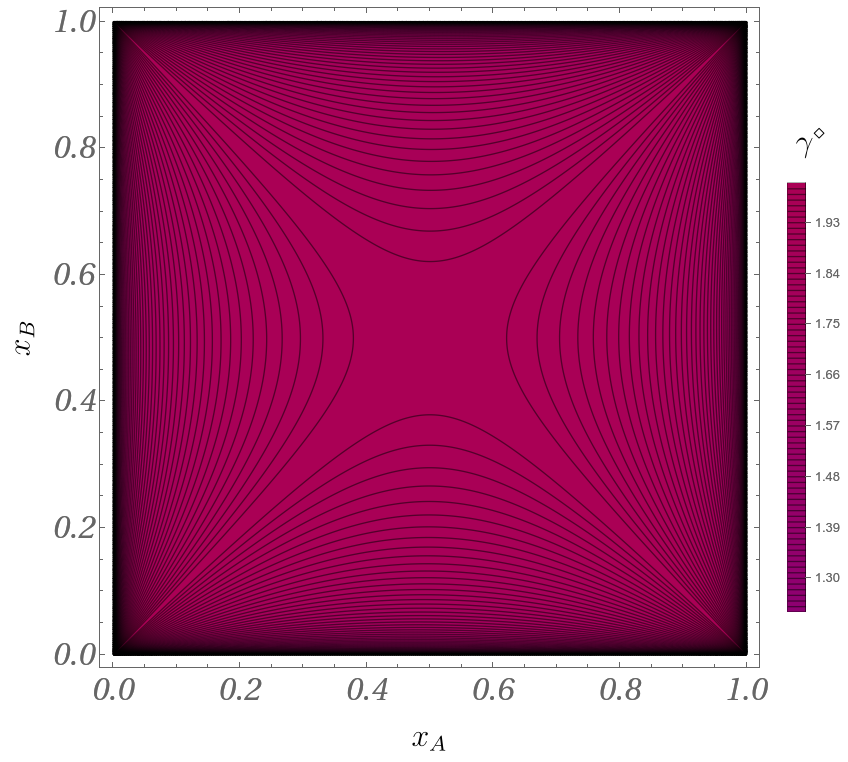}
        \caption{Density plot}
        \label{fig:gamma_diamond_Contour}
    \end{subfigure}
    
    \caption{Critical relative risk aversion $\gamma^\diamond(x_A,x_B)$. 
    The colour coding is the same as in figures \ref{fig:IsoelasticUtility}, \ref{fig:PolicyPlot_0_gamma}.
    (Contour lines show $0.01$ difference, and goes from $1.00$ to $1.99$.)}
    \label{fig:gamma_diamond}
\end{figure}

All the 3D plot\footnote{Resembling a \href{https://www.wippell.co.uk/Online-Shop/Clergy-Outffiting/Hats-\%281\%29/Style-250-Canterbury-Cap.aspx}{Canterbury cap} or Oxford University \href{https://www.museodelestudiante.com/Indumentaria_y_distintivos/SoftCapAA.htm}{women's soft cap}. The plot also features elements that may remind one of a Pagoda roof \cite{book:BuddhistArchitecture,book:JapaneseArchitecture} (such as \href{https://en.wikipedia.org/wiki/File:Chureito_Pagoda_and_Mount_Fuji.jpg}{Chureito Pagoda} in Fujiyoshida, Japan).}, Density plot and Contour plot have been extended to the whole $x_A,x_B \in (0,1)$ domain (in accordance with \eqref{eq:PKAKBextended}).

\subsubsection{Formal calculation of the integrals}
\label{sec:gammaGammaIntegrals}

\paragraph{Using the Beta function:}

\begin{lemma}
\label{lemma:BetaDef}

For all $\Re(z_1),\Re(z_2) > 0$
    \begin{equation}
        B(z_1,z_2) = \int_0^1 p^{z_1-1} (1-p)^{z_2-1} dp
    \end{equation}

    where $B(z_1,z_2)$ is the Beta function \cite{book:HigherTranscendentalFunctions, book:SeymourMathematicalHandbook, book:Bronshtein, book:NISThandbook, book:Abramowitz}, which can be expressed by the Gamma function:

    \begin{equation}
        B(z_1,z_2) = \frac{\Gamma(z_1) \Gamma(z_2)}{\Gamma(z_1+z_2)}
    \end{equation}
    
\end{lemma}

We aim to calculate the following integral:

\begin{equation}
    \label{eq:C(alpha)_gamma}
    C^\gamma(\alpha) = \int_{-\infty}^\infty e^{-\gamma \alpha \tau} 
    \frac{\left ( \frac{e^{-\tau}}{1 + e^{-\tau}} \right )^{1-\gamma}-1}{1-\gamma} 
    \gamma d \tau
\end{equation}

After rearranging the terms, we get:

\begin{equation}
    C^\gamma(\alpha) = \int_{-\infty}^\infty e^{-\gamma \alpha \tau} 
    \frac{\left ( \frac{1}{1 + e^{\tau}} \right )^{1-\gamma}-1}{1-\gamma} 
    \gamma d \tau
\end{equation}

\begin{equation}
    C^\gamma(\alpha) = - \frac{1}{-\gamma \alpha} \int_{-\infty}^\infty e^{-\gamma \alpha \tau}
    \frac{\gamma}{1-\gamma}
    (-(1-\gamma))
    \frac{e^\tau}{
    \left (
    1+e^\tau
    \right )^{2-\gamma}
    }
    d \tau
\end{equation}

\begin{equation}
    C^\gamma(\alpha) = - \frac{1}{\alpha} \int_{-\infty}^\infty e^{-\gamma \alpha \tau}
    \frac{e^\tau}{
    \left (
    1+e^\tau
    \right )^{2-\gamma}
    }
    d \tau
\end{equation}

\begin{equation}
    \label{eq:CgammaExp}
    C^\gamma(\alpha) = - \frac{1}{\alpha} \int_{-\infty}^\infty
    \frac{\left ( e^\tau \right )^{1-\gamma \alpha}}{
    \left (
    1+e^\tau
    \right )^{2-\gamma}
    }
    d \tau
\end{equation}

After this, we can make the following change of variable:

\begin{equation}
    \tau = \log \left ( \frac{p}{1-p} \right ), \quad
    e^\tau = \frac{p}{1-p},
    \quad
    d \tau = \frac{dp}{p (1-p)}
\end{equation}

Rewriting the terms to the $p$ variable instead of $\tau$ results in:

\begin{equation}
    C^\gamma(\alpha) = - \frac{1}{\alpha} \int_{0}^1
    \frac{
     p^{1-\gamma \alpha}
    \left ( 1-p \right )^{\gamma \alpha-1}
    }{
    \left (
    1+\frac{p}{1-p}
    \right )^{2-\gamma}
    }
    \frac{dp}{p (1-p)}
\end{equation}

\begin{equation}
    C^\gamma(\alpha) = - \frac{1}{\alpha} \int_{0}^1
     p^{1-\gamma \alpha}
    \left ( 1-p \right )^{\gamma \alpha-1+2-\gamma}
    \frac{dp}{p (1-p)}
\end{equation}

\begin{equation}
    C^\gamma(\alpha) = - \frac{1}{\alpha} \int_{0}^1
     p^{1-\gamma \alpha}
    \left ( 1-p \right )^{1-\gamma (1-\alpha)}
    \frac{dp}{p (1-p)}
\end{equation}

Which is, by definition (or by recalling lemma \ref{lemma:BetaDef}):

\begin{equation}
\boxed{
    C^\gamma(\alpha) = - \frac{1}{\alpha}
    B(1-\gamma \alpha, 1- \gamma (1-\alpha))
    }
\end{equation}

or expressed by Gamma functions:

\begin{equation}
\boxed{
    C^\gamma(\alpha) = - \frac{1}{\alpha}
    \frac{\Gamma(1-\gamma \alpha)\Gamma(1- \gamma (1-\alpha))}{\Gamma(2-\gamma)}
    }
\end{equation}

\paragraph{Using Ramanujan’s formula:}

\begin{lemma}
    For $a>0$
    \begin{equation}
    \label{eq:RamanujanFormula}
        \int_{-\infty}^\infty \Gamma(a+i t) \Gamma(a-i t) e^{-i \xi t} d t = 
        \sqrt{\pi} \Gamma(a) \Gamma
        \left (
        a + \frac{1}{2}
        \right  )
        \left (
        \cosh \left ( \frac{\xi}{2} \right )
        \right )^{-2 a}
    \end{equation}

    See the proof in \cite{arxiv:Ramanujan}, and Ramanujan’s formula in \cite{paper:RamanujanOriginal,book:HigherTranscendentalFunctions}.
    
\end{lemma}

The right-hand side of equation \eqref{eq:RamanujanFormula} can be rewritten according to the Legendre's Duplication formula for the Gamma function \cite{book:ClassicalTopicsInComplexFunctionTheory,arxiv:Ramanujan}:

\begin{equation}
    \sqrt{\pi} \Gamma(2 a) = 2^{2 a -1} \Gamma(a)  
    \Gamma \left ( a + \frac{1}{2} \right )
\end{equation}

Resulting in:

\begin{equation}
        \int_{-\infty}^\infty \Gamma(a+i t) \Gamma(a-i t) e^{-i \xi t} d t = 
        \frac{2 \pi \Gamma(2 a)}
        {4^a 
        \left (
        \cosh \left ( \frac{\xi}{2} \right )
        \right )^{2 a}}
\end{equation}

Performing an inverse Fourier transformation, we get the following formula:

\begin{equation}
    \frac{1}{2 \pi}
    \int_{-\infty}^\infty
    \frac{2 \pi \Gamma(2 a)}
        {4^a 
        \left (
        \cosh \left ( \frac{\xi}{2} \right )
        \right )^{2 a}}
        e^{i \xi t} d \xi =
        \Gamma(a+i t) \Gamma(a-i t)
\end{equation}

After rearranging the terms and simplification:

\begin{equation}
\label{eq:CoshInvFourierRamanujan}
    \int_{-\infty}^\infty
    \frac{1}
        { 
        \left (
        \cosh \left ( \frac{\xi}{2} \right )
        \right )^{2 a}}
        e^{i \xi t} d \xi =
        4^a \frac{\Gamma(a+i t) \Gamma(a-i t)}
        {{\Gamma(2 a)}}
\end{equation}

The formula can be analytically continued to complex values of $t \in \mathbb{C}$, if $-a < \Im(t) < a$.

Recalling equation \eqref{eq:CgammaExp} for a rearranged version of the original $C^\gamma(\alpha)$ expression in \eqref{eq:C(alpha)_gamma} gives:

\begin{equation}
    C^\gamma(\alpha) = - \frac{1}{\alpha} \int_{-\infty}^\infty
    \frac{\left ( e^\tau \right )^{1-\gamma \alpha}}{
    \left (
    1+e^\tau
    \right )^{2-\gamma}
    }
    d \tau = 
    - \frac{1}{\alpha} \int_{-\infty}^\infty
    \frac{e^{\tau(1-\gamma \alpha)} e^{-\tau (1-\gamma/2)}}{
    \left (
    e^{\tau/2}+e^{-\tau/2}
    \right )^{2-\gamma}
    }
    d \tau
\end{equation}

\begin{equation}
\label{eq:cosh^2a}
    C^\gamma(\alpha)
    = 
    - \frac{1}{\alpha} 
    \frac{1}{2^{2-\gamma}}
    \int_{-\infty}^\infty
    \frac{e^{\tau(\gamma (1/2-\alpha))}}
    {
    \left (
    \cosh(\tau/2)
    \right )^{2-\gamma}
    }
    d \tau
\end{equation}

By choosing the following parameters:

\begin{equation}
    a=1-\frac{\gamma}{2}, \quad
    t=\frac{1}{i} \gamma 
    \left (
    \frac{1}{2} - \alpha
    \right )
\end{equation}

and recalling equation \eqref{eq:cosh^2a} and \eqref{eq:CoshInvFourierRamanujan} we get:

\begin{equation}
    C^\gamma(\alpha)
    = 
    - \frac{1}{\alpha} 
    \frac{1}{2^{2-\gamma}}
    4^{1-\gamma/2}
    \frac{\Gamma(1-\gamma \alpha) \Gamma(1-\gamma (1-\alpha))}
    {\Gamma(2-\gamma)}
\end{equation}

\begin{equation}
\boxed{
    C^\gamma(\alpha)
    = 
    - \frac{1}{\alpha}
    \frac{\Gamma(1-\gamma \alpha) \Gamma(1-\gamma (1-\alpha))}
    {\Gamma(2-\gamma)}
    }
\end{equation}

\subsubsection{Asymptotic expansion of general statistical game prior approximation}

After performing a similar calculation as in Section~\ref{sec:AsymptoticExpansionBinomialBayesianPriorApproximation}, we get the following 
explicit expression for the first-order asymptotic approximation term 
for general isoelastic statistical games $\sampi_\gamma$:

\begin{equation}
    \sampi_\gamma = \frac{1}{\alpha_A (1-\alpha_A)}
    \left (
    \frac{n'(x_0^*)}{\beta}  -
    \frac{\kappa(x_0^*) \vartheta^*_{\gamma,\clubsuit}}{\beta^2} 
    \right ) +
    \frac{1}{2} \frac{\kappa(x_0^*)}{\beta^2}
    \left (
    \frac{{C_B^\gamma}''(\alpha_B)}{C_B^\gamma(\alpha_B)} -
    \frac{{C_A^\gamma}''(\alpha_A)}{C_A^\gamma(\alpha_A)}
    \right )
\end{equation}

\begin{equation}
    \left (
    \frac{{C_B^\gamma}''(\alpha_B)}{C_B^\gamma(\alpha_B)} -
    \frac{{C_A^\gamma}''(\alpha_A)}{C_A^\gamma(\alpha_A)}
    \right ) 
    =
    -2 \frac{1-2 \alpha_A}{\alpha_A^2 (1-\alpha_A)^2}
    - 2 \gamma \frac{\psi(1-\gamma \alpha_A) - 
    \psi(1-\gamma (1-\alpha_A))}{\alpha_A (1-\alpha_A)}
\end{equation}

where $\psi(z)$ is the Digamma function \cite{ book:SeymourMathematicalHandbook, book:NISThandbook, book:Abramowitz,book:HigherTranscendentalFunctions}, the logarithmic derivative of the Gamma function:

\begin{equation}
    \psi(z) = \frac{\Gamma'(z)}{\Gamma(z)}
\end{equation}

\subsection{Binomial Bayesian limiting prior asymptotics}
\label{deriv:BinomialBayesianAsymptotics}

The derivation supports conjecture \ref{conj:BayesianLimitPrior}.

Previously, in Section~\ref{sec:SEC_Approx}, we used the Euler-Maclaurin formula to approximate the summation with integration.
However, we can relax this assumption and introduce a ``Dirac comb'' measure for integration, which mimics the discrete summation.

To approximate the integrals with respect to the ``Dirac comb'' measure, we will use the Poisson summation formula \cite{book:DistributionsFourierTransforms,book:TrigonometricSeries}.

\subsubsection{Approximations}
\label{sec:SPC_Approx}

\paragraph{Stirling's formula:}
After applying Stirling's formula in the same way as in Section~\ref{sec:SEC_Approx}, we get:

\begin{equation}
    \label{eq:SumToDiracComb}
    \Delta G_\theta^\mathrm{S}(\vartheta) = \sum_{k} \frac{1}{N} f_\theta(k/N) L_\theta(\tau_k) =
    \int_0^1 f_\theta(x) L_\theta(\tau(x)) \sh(N x) dx
\end{equation}

Where we substituted the summation over the discrete variable $k$ with an integration with a Dirac comb (tempered-) distribution $\sh(x)$, which is a periodic Dirac $\delta$-function \cite{book:DistributionsFourierTransforms,book:DistributionsOperators} grid:

\begin{equation}
    \sh(y) = \sum_{k=-\infty}^\infty \delta(y-k)
\end{equation}

From the scaling of Dirac $\delta$-function: $\delta(a x) = 1/|a| \delta(x)$ \cite{book:DistributionsOperators} we have:

\begin{equation}
    \sh(N x) = \sum_k \delta(N x - k) = 
    \frac{1}{N} \sum_k \delta(x - k/N)
\end{equation}

which justifies equation \eqref{eq:SumToDiracComb}.

\paragraph{Poisson summation formula:}

Formally, we can substitute the Dirac comb with its Fourier series representation:

\begin{equation}
    \sh(y) = \sum_{k=-\infty}^\infty \delta(y-k) = 
    \sum_{m=-\infty}^\infty e^{2 \pi i m y}
\end{equation}

resulting:

\begin{equation}
    \Delta G_\theta^\mathrm{S,P}(\vartheta) = 
    \int_0^1 \sum_m f_\theta(x) L_\theta(\tau(x)) 
    e^{2 \pi i m N x} dx
\end{equation}

\paragraph{Change of variables:}
We do the same variable changes as in Section~\ref{sec:SEC_Approx}:

\begin{equation}
    \Delta G_\theta^\mathrm{S,P,C}(\vartheta) =  \int_{-\infty}^\infty \sum_m
    f_\theta \left ( x_0^* + \frac{\tau + \vartheta}{N \beta} \right ) L_\theta(\tau)
    e^{2 \pi i m \left ( N x_0^* + (\tau + \vartheta)/\beta \right )}
    \frac{d\tau}{N \beta}
\end{equation}

\paragraph{Interchanging summation and integration:}

\begin{equation}
    \Delta G_\theta^\mathrm{S,P,C,I}(\vartheta) = \sum_m \int_{-\infty}^\infty 
    f_\theta \left ( x_0^* + \frac{\tau + \vartheta}{N \beta} \right ) L_\theta(\tau)
    e^{2 \pi i m \left ( N x_0^* + (\tau + \vartheta)/\beta \right )}
    \frac{d\tau}{N \beta}
\end{equation}

\subsubsection{Phase dependent expression}

By introducing a phase parameter:

\begin{equation}
    \varphi = 2 \pi N x_0^* \mod 2 \pi
\end{equation}

we get the following expression:

\begin{equation}
    \begin{split}
        \Delta G_\theta^{\spadesuit,\varphi}(\vartheta) &= \Delta G_\theta^\mathrm{S,P,C,I}(\vartheta) \\
        &= 
        \sum_m \int_{-\infty}^\infty 
        f_\theta \left ( x_0^* + \frac{\tau + \vartheta}{N \beta} \right ) L_\theta(\tau)
        e^{i m \varphi + 2 \pi i m (\tau + \vartheta)/\beta}
        \frac{d\tau}{N \beta}
    \end{split}
\end{equation}

Introducing normalized growth factors:

\begin{equation}
    \Delta g_\theta^{\spadesuit,\varphi}(\vartheta) = 
    \frac{\Delta G_\theta^{\spadesuit,\varphi}(\vartheta)}{
    \frac{1}{\sqrt{2 \pi x_0^* (1-x_0^*)}} \frac{e^{- N \varepsilon(x_A,x_B)}}{\sqrt{N} \beta}
    }
\end{equation}

we can get the following approximation:

\begin{equation}
    \Delta g_\theta^{\spadesuit,\varphi}(\vartheta) = 
    e^{-\alpha_\theta \vartheta}
    \sum_m C_\theta(\alpha_\theta - 2 \pi i m/\beta) e^{i m \varphi + 2 \pi i m \vartheta / \beta }
    +\mathcal{O}(1/N)
\end{equation}

where we analytically continued the $C_A(.)$, $C_B(.)$ functions in \eqref{eq:CA_f}, \eqref{eq:CB_f}.
We can solve the following equation up to the zeroth-order to get a phase-dependent asymptotic expression for the equilibrium log-odds $\vartheta^{*,\varphi}_\spadesuit$:

\begin{equation}
    \Delta g_A^{\spadesuit,\varphi}(\vartheta^{*,\varphi}_\spadesuit) = \Delta g_B^{\spadesuit,\varphi}(\vartheta^{*,\varphi}_\spadesuit)
\end{equation}

yielding the following implicit equation:

\begin{equation}
\label{eq:ImplicitAsymtoticCACB}
    \vartheta^{*,\varphi}_\spadesuit = 
    \frac{1}{\alpha_A - \alpha_B} \log \left (
    \frac{
    \sum_m C_A(\alpha_A - 2 \pi i m/\beta) e^{i m \varphi + 2 \pi i m \vartheta^{*,\varphi}_\spadesuit / \beta }
    }
    {
    \sum_m C_B(\alpha_B - 2 \pi i m/\beta) e^{i m \varphi + 2 \pi i m \vartheta^{*,\varphi}_\spadesuit / \beta }
    }
    \right )
\end{equation}

Introducing:

\begin{equation}
    \widetilde{C}_\theta(\alpha,\omega) = 
    \sum_m C_\theta(\alpha - 2 \pi i m/\beta) e^{i m \omega}
\end{equation}

we can define a quantity:

\begin{equation}
    \Xi^\varphi(\vartheta) = 
    \frac{1}{\alpha_A - \alpha_B} \log \left (
    \frac{
    \widetilde{C}_A(\alpha_A,\varphi + 2 \pi \vartheta / \beta )
    }
    {
    \widetilde{C}_B(\alpha_B,\varphi + 2 \pi \vartheta / \beta )
    }
    \right )
\end{equation}

The solution of equation \eqref{eq:ImplicitAsymtoticCACB} can be defined as a fixed point equation:

\begin{equation}
\label{eq:AsymptoticBayesianLogOdds01}
\boxed{
    \vartheta^{*,\varphi}_\spadesuit = 
    \Xi^\varphi(\vartheta^{*,\varphi}_\spadesuit)
    }
\end{equation}

where:

\begin{equation}
\label{eq:AsymptoticBayesianLogOdds02}
    \Xi^\varphi(\vartheta) = 
    \log \left (
    \frac{
    \sum_m \frac{\pi}{\alpha - 2 \pi i m /\beta} \frac{
    e^{i m (\varphi + 2 \pi \vartheta / \beta) }}
    {\sin(\pi \alpha - 2 \pi^2 i m /\beta)} 
    }
    {
    \sum_m \frac{\pi}{1-\alpha + 2 \pi i m /\beta} \frac{
    e^{i m (\varphi + 2 \pi \vartheta / \beta) }}
    {\sin(\pi \alpha + 2 \pi^2 i m /\beta)} 
    }
    \right )
\end{equation}

or alternatively:

\begin{equation}
\label{eq:AsymptoticBayesianLogOdds03}
    \Xi^\varphi(\vartheta) = 
    \log \left (
    \frac{
    \frac{\pi}{\alpha} \frac{1}
    {\sin(\pi \alpha)} + 2 \ \Re \left ( 
    \sum_{m=1}^\infty \frac{\pi}{\alpha - 2 \pi i m /\beta} \frac{
    e^{i m (\varphi + 2 \pi \vartheta / \beta) }}
    {\sin(\pi \alpha - 2 \pi^2 i m /\beta)}
    \right )
    }
    {
    \frac{\pi}{1-\alpha} \frac{1}
    {\sin(\pi \alpha)} + 2 \ \Re \left ( 
    \sum_{m=1}^\infty \frac{\pi}{1-\alpha + 2 \pi i m /\beta} \frac{
    e^{i m (\varphi + 2 \pi \vartheta / \beta) }}
    {\sin(\pi \alpha + 2 \pi^2 i m /\beta)}
    \right )
    }
    \right )
\end{equation}

\paragraph{Statement of the conjecture:}
Conjecture \ref{conj:BayesianLimitPrior} claims that:

\begin{equation}
    \vartheta^{*,\varphi}_\loopedsquare = 
    \vartheta^{*,\varphi}_\spadesuit
\end{equation}

More precisely:

\begin{equation}
    \vartheta^*_N(x_A,x_B) = 
    \vartheta^{*,\varphi_N(x_A,x_B)}_\spadesuit +
    \mathcal{O}(1/N)
\end{equation}

\begin{remark}
    The limiting prior approximation $P^\approx_\loopedsquare$ in \eqref{eq:Papprox_club}, \eqref{eq:P*_club} and the related approximative log-odds $\vartheta^\approx_\loopedsquare$ can be interpreted as the fixed point of \eqref{eq:AsymptoticBayesianLogOdds01} if we take the zeroth harmonic-order approximation of the expression \eqref{eq:AsymptoticBayesianLogOdds03}
    
\end{remark}

\subsection{Limiting prior asymptotics for general Statistical games}

A similar calculation can be performed for general statistical games \ref{def:StatisticalGame}, with an isoelastic utility function.

\subsubsection{Approximations}

\paragraph{Stirling's formula:}

\begin{equation}
    \label{eq:USumToDiracComb}
    U_\theta^\mathrm{S}(\vartheta) = \sum_{k} \frac{1}{N} f_\theta(k/N) L^\gamma_\theta(\tau_k) =
    \int_0^1 f_\theta(x) L^\gamma_\theta(\tau(x)) \sh(N x) dx
\end{equation}

\paragraph{Poisson summation formula:}

\begin{equation}
    U_\theta^\mathrm{S,P}(\vartheta) = 
    \int_0^1 \sum_m f_\theta(x) L^\gamma_\theta(\tau(x)) 
    e^{2 \pi i m N x} dx
\end{equation}

\paragraph{Change of variables:}

\begin{equation}
    U_\theta^\mathrm{S,P,C}(\vartheta) = \int_{-\infty}^\infty \sum_m 
    f_\theta \left ( x_0^* + \frac{\gamma \ \tau + \vartheta}{N \beta} \right ) L^\gamma_\theta(\tau)
    e^{2 \pi i m \left ( N x_0^* + (\gamma \tau + \vartheta)/\beta \right )}
    \frac{\gamma}{N \beta} d\tau
\end{equation}

\paragraph{Interchanging summation and integration:}

\begin{equation}
    U_\theta^\mathrm{S,P,C,I}(\vartheta) = \sum_m \int_{-\infty}^\infty 
    f_\theta \left ( x_0^* + \frac{\gamma \ \tau + \vartheta}{N \beta} \right ) L^\gamma_\theta(\tau)
    e^{2 \pi i m \left ( N x_0^* + (\gamma \tau + \vartheta)/\beta \right )}
    \frac{\gamma}{N \beta} d\tau
\end{equation}

\subsubsection{Phase dependent expression}

\begin{equation}
    \begin{split}
        U_\theta^{\spadesuit,\varphi}(\vartheta) &= 
        U_\theta^\mathrm{S,P,C,I}(\vartheta) \\
        &= 
        \sum_m \int_{-\infty}^\infty 
        f_\theta \left ( x_0^* + \frac{\gamma \ \tau + \vartheta}{N \beta} \right ) L^\gamma_\theta(\tau)
        e^{i m \varphi + 2 \pi i m (\gamma \tau + \vartheta)/\beta}
        \frac{\gamma}{N \beta} d\tau
    \end{split}
\end{equation}

Introducing normalized expected utility:

\begin{equation}
    u_\theta^{\spadesuit,\varphi}(\vartheta) = 
    \frac{U_\theta^{\spadesuit,\varphi}(\vartheta)}{
    \frac{1}{\sqrt{2 \pi x_0^* (1-x_0^*)}} \frac{e^{- N \varepsilon(x_A,x_B)}}{\sqrt{N} \beta}
    }
\end{equation}

we can get the following approximation:

\begin{equation}
    u_\theta^{\spadesuit,\varphi}(\vartheta) = 
    e^{-\alpha_\theta \vartheta}
    \sum_m C^\gamma_\theta(\alpha_\theta - 2 \pi i m/\beta) e^{i m \varphi + 2 \pi i m \vartheta / \beta }
    +\mathcal{O}(1/N)
\end{equation}

where we analytically continued the $C^\gamma_A(.)$, $C^\gamma_B(.)$ functions in \eqref{eq:CAgamma_f}, \eqref{eq:CBgamma_f}.
We can solve the following equation up to the zeroth-order to get a phase-dependent asymptotic expression for the equilibrium log-odds $\vartheta^{*,\varphi}_{\gamma,\spadesuit}$:

\begin{equation}
    u_A^{\spadesuit,\varphi}(\vartheta^{*,\varphi}_{\gamma,\spadesuit}) =
    u_B^{\spadesuit,\varphi}(\vartheta^{*,\varphi}_{\gamma,\spadesuit})
\end{equation}

yielding the following implicit equation:

\begin{equation}
\label{eq:ImplicitvarphigammaCACB}
    \vartheta^{*,\varphi}_{\gamma,\spadesuit} = 
    \frac{1}{\alpha_A - \alpha_B} \log \left (
    \frac{
    \sum_m C^\gamma_A(\alpha_A - 2 \pi i m/\beta) e^{i m \varphi + 2 \pi i m \vartheta^{*,\varphi}_{\gamma,\spadesuit} / \beta }
    }
    {
    \sum_m C^\gamma_B(\alpha_B - 2 \pi i m/\beta) e^{i m \varphi + 2 \pi i m \vartheta^{*,\varphi}_{\gamma,\spadesuit} / \beta }
    }
    \right )
\end{equation}

Introducing:

\begin{equation}
\label{eq:CtildeAB}
    \widetilde{C}^\gamma_\theta(\alpha,\omega) = 
    \sum_m C^\gamma_\theta(\alpha - 2 \pi i m/\beta) e^{i m \omega}
\end{equation}

we can define a quantity:

\begin{equation}
    \Xi_\gamma^\varphi(\vartheta) =
    \frac{1}{\alpha_A - \alpha_B} \log \left (
    \frac{
    \widetilde{C}^\gamma_A(\alpha_A,\varphi + 2 \pi \vartheta / \beta )
    }
    {
    \widetilde{C}^\gamma_B(\alpha_B,\varphi + 2 \pi \vartheta / \beta )
    }
    \right )
\end{equation}

The solution of equation \eqref{eq:ImplicitvarphigammaCACB} can be defined as a fixed point equation:

\begin{equation}
\boxed{
    \vartheta^{*,\varphi}_{\gamma,\spadesuit} = 
    \Xi_\gamma^\varphi(\vartheta^{*,\varphi}_{\gamma,\spadesuit})
    }
\end{equation}

\begin{remark}
    Typically, the implicit expression might have multiple solutions for smaller $\gamma$ and large $\beta$.
    However, it may have only one stable solution.
    Further analysis is needed to explore the proposed implicit equation's validity range.

\end{remark}

\begin{remark}
    This phase-dependent asymptotic expression does depend on $\gamma$, signalling that it has the potential to interpolate between Binomial Fisher and Binomial Bayesian limiting priors.
    
\end{remark}

By taking the $\gamma \to 0$ limit, we attempt to recover the Binomial Fisher limiting prior bounds in Section~\ref{deriv:liminfPN}, \ref{deriv:limsupPN}.

\subsubsection{Taking the $\gamma \to 0$ limit}

In this limit, the $C^\gamma_A(.)$, $C^\gamma_B(.)$ functions in \eqref{eq:CAgamma_f}, \eqref{eq:CBgamma_f} simplifies radically:

\begin{equation}
    \lim_{\gamma \to 0} C_A^\gamma(\alpha_A) = -\frac{1}{\alpha_A}, \quad
    \lim_{\gamma \to 0} C_B^\gamma(\alpha_B) = \frac{1}{\alpha_B}
\end{equation}

\begin{lemma}
\label{lemma:FourierSeriesExp}

For any $\omega \in (0,2 \pi)$

    \begin{equation}
        \lim_{M \to \infty} \sum_{m=-M}^M
        \frac{1}{a - i m} e^{i m} =
        \frac{2 \pi}{e^{2 \pi a}-1}  e^{a \omega}
    \end{equation}
    
\end{lemma}

\begin{proof}

For any $m \in \mathbb{Z}$:

    \begin{equation}
        \frac{1}{2 \pi} \int_0^{2 \pi} e^{a \omega} e^{-i \omega m} = 
        \frac{1}{2 \pi} \frac{e^{2 \pi (a - i m)}-1}{a - i m} =
        \frac{1}{2 \pi} \frac{e^{2 \pi a}-1}{a - i m}
    \end{equation}

    Therefore

    \begin{equation}
        c_m = \frac{1}{2 \pi} \int_0^{2 \pi} \frac{2 \pi}{e^{2 \pi a}-1}  e^{a \omega} e^{-i \omega m} = \frac{1}{a - i m}
    \end{equation}

    According to the theory of Fourier series \cite{book:HarmonicAnalysis,book:PartialDifferentialEquationsHilbertSpaceMethods,book:FourierAnalysis,book:TrigonometricSeries}\footnote{in this case, for example Dini's test can guarantee the pointwise convergence in all points, where the function is not discontinuous}, this implies that for $\omega \mod 2 \pi \ne 0$:

    \begin{equation}
        \lim_{M \to \infty} \sum_{m=-M}^M
        c_m e^{i \omega m} =
        \frac{2 \pi}{e^{2 \pi a}-1}  e^{a (\omega \mod 2 \pi))}
    \end{equation}
    
\end{proof}

We can continue by evaluating the expression:

\begin{equation}
    \lim_{\gamma \to 0}
    \frac{
    \sum_m C^\gamma_A(\alpha_A - 2 \pi i m/\beta) e^{i m \omega}
    }
    {
    \sum_m C^\gamma_B(\alpha_B - 2 \pi i m/\beta) e^{i m \omega}
    }=
    \frac{
    \sum_m \frac{-1}{\alpha_A - 2 \pi i m/\beta} e^{i m \omega}
    }
    {
    \sum_m \frac{1}{\alpha_B - 2 \pi i m/\beta} e^{i m \omega}
    }
\end{equation}

\begin{equation}
    \label{eq:LimitXigamma}
    \frac{
    \sum_m \frac{-1}{\alpha_A - 2 \pi i m/\beta} e^{i m \omega}
    }
    {
    \sum_m \frac{1}{\alpha_B - 2 \pi i m/\beta} e^{i m \omega}
    } = 
    \frac{
    \sum_m \frac{-1}{\frac{\alpha_A \beta}{2 \pi} - i m} e^{i m \omega}
    }
    {
    \sum_m \frac{1}{\frac{\alpha_B \beta}{2 \pi} - i m} e^{i m \omega}
    }
\end{equation}

Using lemma \ref{lemma:FourierSeriesExp}, and that $\alpha_A - \alpha_B = 1$, we get:

\begin{equation}
    \frac{
    \sum_m \frac{-1}{\frac{\alpha_A \beta}{2 \pi} - i m} e^{i m \omega}
    }
    {
    \sum_m \frac{1}{\frac{\alpha_B \beta}{2 \pi} - i m} e^{i m \omega}
    } = 
    \frac{1-e^{\alpha_B \beta}}{e^{\alpha_A \beta} - 1} e^{\frac{\beta}{2 \pi} \omega}
\end{equation}

Now we can take the limit of $\Xi^\varphi_\gamma(\vartheta)$:

\begin{equation}
    \lim_{\gamma \to 0} \Xi_\gamma^\varphi(\vartheta) =
    \log \left ( \frac{1-e^{\alpha_B \beta}}{e^{\alpha_A \beta} - 1}
    \right )
    + \frac{\beta}{2 \pi} ((\varphi + 2 \pi \vartheta) \mod 2 \pi)
\end{equation}

Resulting bounds for the equilibrium $\vartheta^{*,\varphi}_{\gamma,\spadesuit}$ for any $\varphi \in [0,2 \pi)$ as $\gamma \to 0$:

\begin{equation}
\label{eq:varthetagammaphiLimitBounds}
\boxed{
    \log \left ( 
    \frac{1-e^{\alpha_B \beta}}{e^{\alpha_A \beta} - 1}
    \right ) \le
    \lim_{\gamma \to 0}
    \vartheta^{*,\varphi}_{\gamma,\spadesuit}
    \le 
    \log \left ( 
    \frac{1-e^{\alpha_B \beta}}{e^{\alpha_A \beta} - 1}
    \right ) + \beta
    }
\end{equation}

\begin{remark}
    The obtained bounds in \eqref{eq:varthetagammaphiLimitBounds} are equivalent to the Binomial Fisher limiting prior bounds in Section~\ref{deriv:liminfPN}, \ref{deriv:limsupPN}.
    
\end{remark}

\begin{remark}
    The expression \eqref{eq:LimitXigamma} can be expressed by the Lerch transcendent \cite{arxiv:JesusLerch,book:NISThandbook,book:HigherTranscendentalFunctions} \footnote{also know as Hurwitz-\href{https://mathshistory.st-andrews.ac.uk/Biographies/Lerch/}{Lerch} transcendent}:

    \begin{equation}
    \Phi_L(z,s,a) = \sum_{m=0}^\infty \frac{z^m}{(m+a)^s}
    \end{equation}

    \begin{equation}
    \frac{
    \sum_m \frac{-1}{\frac{\alpha_A \beta}{2 \pi} - i m} e^{i m \omega}
    }
    {
    \sum_m \frac{1}{\frac{\alpha_B \beta}{2 \pi} - i m} e^{i m \omega}
    } =
    \frac{
    \sum_m \frac{-1}{m + i \frac{\alpha_A \beta}{2 \pi}} e^{i m \omega}
    }
    {
    \sum_m \frac{1}{m + i \frac{ \alpha_B \beta}{2 \pi}} e^{i m \omega}
    }
    \end{equation}

    \begin{equation}
    \frac{
    \sum_m \frac{-1}{m + i \frac{\alpha_A \beta}{2 \pi}} e^{i m \omega}
    }
    {
    \sum_m \frac{1}{m + i \frac{ \alpha_B \beta}{2 \pi}} e^{i m \omega}
    } = 
    - \frac{
    \Phi_L(e^{i \omega},1,i \frac{\alpha_A \beta}{2 \pi}) - 
    \Phi_L(e^{-i \omega},1,-i \frac{\alpha_A \beta}{2 \pi}) +
    i \frac{2 \pi}{\alpha_A \beta}
    }{
    \Phi_L(e^{i \omega},1,i \frac{\alpha_B \beta}{2 \pi}) - 
    \Phi_L(e^{-i \omega},1,-i \frac{\alpha_B \beta}{2 \pi}) +
    i \frac{2 \pi}{\alpha_B \beta}
    }
    \end{equation}

\end{remark}

\subsection{First-order asymptotic expansion for general statistical games}

The indicated derivation supports conjecture \ref{conj:AsymptoticExpansionSubleadingTerm}.

Repeating essentially the same steps as in Section~\ref{sec:AsymptoticExpansionBinomialBayesianPriorApproximation}, but adopting the approximations in Section~\ref{sec:SPC_Approx} one can derive the following phase dependent first-order asymptotic term $\sampi_\gamma^\varphi$:

\begin{equation}
\label{eq:sampiPhiGammaResult}
    \begin{split}
        \sampi_\gamma^\varphi = &
        \left (
        \frac{n'(x_0^*)}{\beta}  -
        \frac{\kappa(x_0^*) \vartheta^{*,\varphi}_{\gamma,\spadesuit}}{\beta^2} 
        \right )
        \left (
        \frac{\widetilde{C}^\gamma_B{}'(\alpha_B,\varphi + 2 \pi \vartheta^{*,\varphi}_{\gamma,\spadesuit} / \beta )}
        {\widetilde{C}^\gamma_B(\alpha_B,\varphi + 2 \pi \vartheta^{*,\varphi}_{\gamma,\spadesuit} / \beta )} -
        \frac{\widetilde{C}^\gamma_A{}'(\alpha_A,\varphi + 2 \pi \vartheta^{*,\varphi}_{\gamma,\spadesuit} / \beta )}
        {\widetilde{C}^\gamma_A(\alpha_A,\varphi + 2 \pi \vartheta^{*,\varphi}_{\gamma,\spadesuit} / \beta )}
        \right ) 
        + \\
        & \frac{1}{2} \frac{\kappa(x_0^*)}{\beta^2}
        \left (
        \frac{\widetilde{C}^\gamma_B{}''(\alpha_B,\varphi + 2 \pi \vartheta^{*,\varphi}_{\gamma,\spadesuit} / \beta )}
        {\widetilde{C}^\gamma_B(\alpha_B,\varphi + 2 \pi \vartheta^{*,\varphi}_{\gamma,\spadesuit} / \beta )} -
        \frac{\widetilde{C}^\gamma_A{}''(\alpha_A,\varphi + 2 \pi \vartheta^{*,\varphi}_{\gamma,\spadesuit} / \beta )}
        {\widetilde{C}^\gamma_A(\alpha_A,\varphi + 2 \pi \vartheta^{*,\varphi}_{\gamma,\spadesuit} / \beta )}
        \right ) 
    \end{split}
\end{equation}

where $\widetilde{C}_\theta^\gamma(\alpha,\omega)$ are defined in equation \eqref{eq:CtildeAB}, and its derivatives are taken with respect to the first variable.

\begin{remark}
    The result can be applied to conjecture \ref{conj:AsymptoticExpansionSubleadingTerm} if we set $\gamma$ to $1$:

    \begin{equation}
        \sampi^\varphi = \sampi^\varphi_{\gamma = 1}
    \end{equation}
    
\end{remark}

\begin{remark}
    In higher-order asymptotic expansions -- and for further simplification of the results -- the following quantity appears to be useful:

    \begin{equation}
    \Xi_\gamma^\varphi(\alpha,\vartheta) =
    \log \left (
    \frac{
    \widetilde{C}^\gamma_A(\alpha,\varphi + 2 \pi \vartheta / \beta )
    }
    {
    \widetilde{C}^\gamma_B(\alpha-1,\varphi + 2 \pi \vartheta / \beta )
    }
    \right )
    \end{equation}

\end{remark}

\section{Natural multiplicative utility functions}
\label{appendix:UtilityFunctions}

The Bayesian game described in Section~\ref{subsection:BettingGame} defines the change of an agent's capital. However, it can not specify how the players translate their amount of capital to \emph{utility}.
Utility plays a central role in game theory because in the framework of expected utility theory \cite{plato:ExpectedUtility}, this is, by definition, the quantity for which the expectation a ``rational'' agent ought to maximize.

In the following sections, we will see that different choices of the utility function can quantitatively change the agent's strategy. Therefore, a careful adoption of the utility function is vital for suggesting optimal strategies for a gambling situation.

\subsection{A case for a nontrivial utility function}

Let us assume that there is an agent who tries to maximize her expected amount of capital after $n$ rounds of multiplicative betting games, starting with an initial capital $c_0$.

Assuming that in all rounds the agent is placing her capitals $p'$ portion to A and $1-p'$ portion to B, after $n$ rounds, her expected capital would be:

\begin{equation}
    C_n(p') = \sum_{k=0}^n p_k \  (2 \ p' \ c_0)^k (2 \ (1-p') \ c_0)^{n-k}
\end{equation}

Assuming that outcomes A and B appear independently with probability $P$ and $1-P$ respectively, this expression can be simplified:

\begin{equation}
    C_n(p') = (2 \ c_0)^n \ \left ( P \ p' + (1-P) \ (1-p') \right )^n
\end{equation}

which can be maximized with $p' = 1$ if $P>1/2$ and $p' = 0$ if $P<1/2$.\footnote{When $P=1/2$, all choices of $p'$ give the same expected capital. Because of continuity arguments, we will adopt the choice $p'=1/2$ when $P=1/2$} Formally:

\begin{equation}
    p'^*_0(P) =
    \begin{cases}
        0 & \text{if } P < 1/2 \\
        1/2 & \text{if } P = 1/2 \\
        1 & \text{if } P > 1/2 
    \end{cases}
\end{equation}

\begin{figure}[H]
    \centering
    \includegraphics[width=12 cm]{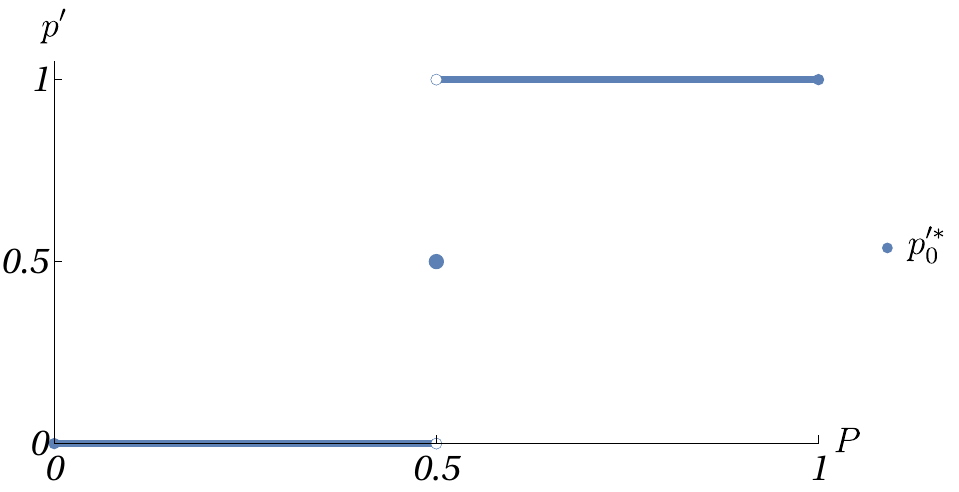}
    \caption{Strategy plot for the capital maximizing agent, $p'^*_0(P)$.}
    \label{fig:PolicyPlot_CapitalMax}
\end{figure}

However, following this strategy, the agent will be bankrupt after $n$ rounds with probability $1-\max(P,1-P)^n$. This means that the agent's capital will go to $0$ almost surely in the long run, and only an almost never-occurring but stellar gain will compensate for her losses.

This extreme strategy might be unacceptable for most human players, signalling that alternative utility functions should be explored and considered if we want to keep the principle of expected utility.

\subsection{Properties of utility functions}

In general a utility function is a mapping from the consequences of a game $\mathcal{C}$ \footnote{or states in the context of Reinforcement Learning \cite{book:RL}} to a utility set $\mathbb{U}$, which will be $\mathbb{R}$, $\mathbb{R}_{\le 0}$ or $\overline{\mathbb{R}}_{\le 0} = \mathbb{R}_{\le 0} \cup \{- \infty\}$ in the followings.

Formally, in general:

\begin{equation}
    u : \mathcal{C} \mapsto \mathbb{U}
\end{equation}

in particular:

\begin{equation}
    u : [0,1] \mapsto \mathbb{R}_{\le 0} \cup \{- \infty\}
\end{equation}

\subsubsection{Equivalent utility functions}

We will call two utility functions equivalent if they result in the same equilibrium strategies.
Positive affine transformations do not influence the strategy profiles in equilibrium, meaning that:

\begin{equation}
    u_1 \sim u_2 \iff \exists C \in \mathbb{R}, D \in \mathbb{R}_{+}, \forall c \in [0,1] \ u_1(c) = C + D \ u_2(c)
\end{equation}

If we denote the set of bounded utility functions by:

\begin{equation}
    B_\downmapsto([0,1]) = \{ u : [0,1] \mapsto \mathbb{R} \cup \{-\infty\} \ | \  \exists M \in \mathbb{R}, \forall c \in [0,1],  u(c) < M \}
\end{equation}

then we can introduce the factorized set of utility functions with respect to this equivalence relation:

\begin{equation}
    \tilde{B}_\downmapsto([0,1]) = \left [ B_\downmapsto([0,1]) \right ]_{\sim}
\end{equation}

\subsubsection{Possible properties of utility functions}

\begin{description}
   \item[Monotonicity] ``more is better'':
    \begin{description}
    \item[Strict Monotonicity] $c_1 < c_2 \implies u(c_1) < u(c_2)$ 
    \item[Week Monotonicity] $c_1 < c_2 \implies u(c_1) \le u(c_2)$
    \begin{description}
        \item[Burnability] there can be a ``burning'' operation, which can reduce a player's capital by any amount. This operation can change any utility function to a weekly monotone utility: $u^{+}(c) = \sup_{c' \le c} u(c')$
    \end{description}
    \end{description}
    
   \item[Strict Concavity] ``risk aversion'' $\mathbb{E}[u(X)] < u(\mathbb{E}[X])$

   \item[Continuity] ``marginal gains results marginal increase of utility'' $u \in C^0([0,1])$

   \item[Differentiability] ``marginal utilities are well defined'' $u \in C^1[0,1]$ (or $u \in C^r[0,1]$)

    \item[Diminishing returns] ``Diminishing Marginal Returns'' $c_1 < c_2 \implies u'(c_1) > u'(c_2)$ 

    \item[Scale free] ``the currency of the capital does not change the behavior'' $\forall 0 < r \le 1 \ \exists C \in \mathbb{R}, D \in \mathbb{R}_{+}, \forall c \in [0,1] \quad u(r \ c) = C + D \ u(c)$ 

    \item[Isoelasticity] \footnote{also known as constant relative risk aversion} for all degree of relative risk aversion \cite{book:Arrow,paper:Pratt} $\gamma \ge 0$, $\gamma \ne 1$
    \begin{equation}
        u_\gamma(c) = \frac{c^{1-\gamma}-1}{1-\gamma}
    \end{equation}

    \item[Logarithmic utility] can also be viewed as a special case of the isoelastic utility, as $\gamma \to 1$
    \begin{equation}
        u_1(c) = \log(c)
    \end{equation}
   
\end{description}

In the following, we will assume as little as possible about the inherently subjective utility functions. However, we will argue that isoelastic and, in particular, logarithmic utilities can emerge naturally as instrumental goals in a repeatable multiplicative gambling situation.

\subsection{Derivation of instrumental goals}

\subsubsection{Repeatable multiplicative gambles}

Let us assume that an agent can play a multiplicative ``double or nothing'' betting game with outcome probabilities $P$, $1-P$, for $n$ rounds.
Initially, she starts with a small $2^{-n}$ capital, meaning that in the final round, the range of her possible capital is the interval $[0,1]$.
We will make as few further assumptions about the agent's final subjective utility values as possible. At this stage, we only assume that it is a function bounded from above:

\begin{equation}
    u_0 \in \tilde{B}_\downmapsto([0,1])
\end{equation}

A natural approach to find the best strategy in this game is to search for a subgame perfect optimal strategy \cite{book:GameTheory}.
In practice, this means that we can recursively solve the
game. 
Assuming that one step before the last the agent has $c_1$ capital, we can search for the optimal splitting ratio $p'_1(c_1)$.

If we can find all optimal strategies as a function of $c_1$, then we can calculate all the optimal expected utilities for the capitals one step before the end.
Formally, we can introduce the following operator:

\begin{equation}
    \mathbfcal{B}[u](c) = \sup_{p' \in [0,1]} \left ( P \ u(2 \ p' \ c) + (1-P) \ u(2 \ (1-p') \ c) \right )
\end{equation}

which is essentially a Bellman operator \cite{book:RL} in the framework of Reinforcement Learning.

By this operator, we can propagate the final subjective utilities back to the previous state, where the possible capital is in the interval $[0,1/2]$.

\begin{equation}
    \mathbfcal{B} : \tilde{B}_\downmapsto([0,2^{-n}]) \mapsto  \tilde{B}_\downmapsto([0,2^{-n-1}])
\end{equation}

If we introduce a scaling (or dilatation) operator:

\begin{equation}
\label{eq:Dilatation}
    \mathbfcal{D}_r[u](c) = u(r \ c)
\end{equation}

then we can reduce our $n$ step gamble to an $n-1$ step gamble with a modified ``instrumentalized'' \cite{book:ValueTheory} utility function $u_1$:

\begin{equation}
    \mathbfcal{I} = \mathbfcal{D}_2 \circ \mathbfcal{B}
\end{equation}

\begin{equation}
    \mathbfcal{I}[u_n](c) = u_{n+1}(c)
\end{equation}

This newly defined instrumentalisation operator $\mathbfcal{I} : \tilde{B}_\downmapsto([0,1]) \mapsto \tilde{B}_\downmapsto([0,1])$ has the following explicit definition:

\begin{equation}
    \mathbfcal{I}[u](c) = \sup_{p' \in [0,1]} \left ( P \ u(p' \ c) + (1-P) \ u((1-p') \ c) \right )
\end{equation}

In the following, we will argue that the instrumentalisation operator will converge to an isoelastic utility function for a wide range of initial subjective utility functions.

\subsubsection{Properties of isoelastic utility functions}

Isoelastic utility functions have several remarkable properties:

    \begin{equation}
        u_\gamma(c) = 
        \begin{cases}
            \frac{c^{1-\gamma}-1}{1-\gamma} & \text{if } \gamma \ne 1 \\
            \log(c) & \text{if } \gamma = 1 
    \end{cases}
    \end{equation}

\begin{figure}[H]
    \centering
    \includegraphics[width=12 cm]{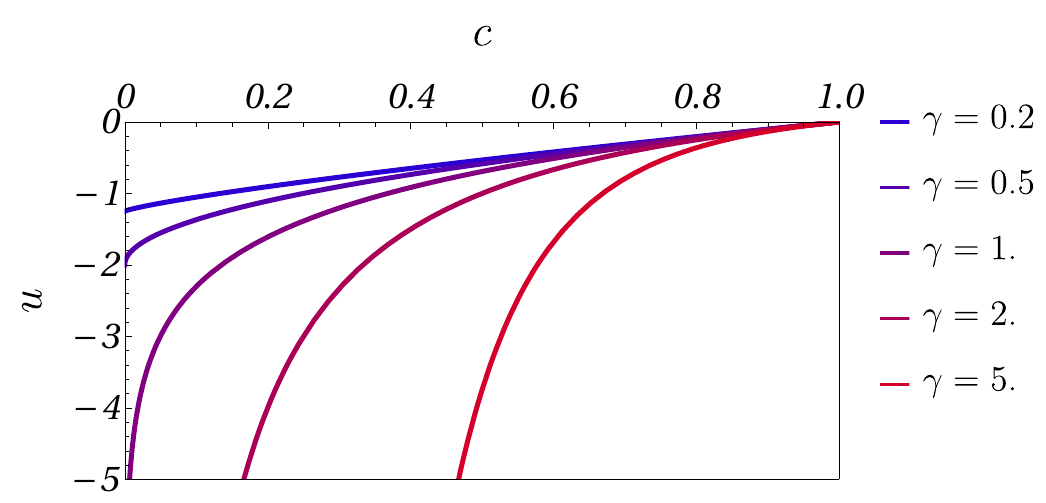}
    \caption{Isoelastic utility functions, for several relative risk aversion ($\gamma$) parameters: $u_\gamma(c)$.}
    \label{fig:IsoelasticUtility}
\end{figure}

It is easy to observe that for all $\gamma > 0$, $u_\gamma$ is a strictly monotone, strictly concave, and infinitely differentiable.

\begin{theorem}
    Isoelastic utility functions are scale-free:

    \begin{equation}
        \mathbfcal{D}_r[u_\gamma] \sim u_\gamma
    \end{equation}

    where $\mathbfcal{D}_r$ is the scaling (or dilatation) operator defined in eq. \eqref{eq:Dilatation}.
    
\end{theorem}

\begin{proof}
    Direct calculation shows that
    
    \begin{equation}
        \mathbfcal{D}_r[u_\gamma] = C + D \ u_\gamma(c)
    \end{equation}

    with parameters:

    \begin{equation}
        C = \frac{r^{1-\gamma}-1}{1-\gamma}, \quad D = r^{\gamma -1}
    \end{equation}
\end{proof}

\begin{theorem}
    All scale-free and twice differentiable utility functions are isoelastic.
    
\end{theorem}

\begin{proof}
    To simplify the proof, it is better to change from capital variable ($c$) to growth-related variable ($g$):
    
    \begin{equation}
        g = \log(c), \quad c = e^g
    \end{equation}

    And to introduce a new marginal utility-related quantity:
    
    \begin{equation}
    \label{eq:w(g)}
        w(g) = \text{\bf \dh}[u](g)  =  \log(u'(e^g))
    \end{equation}

    \begin{equation}
        u(c) = \text{\bf \c{S}}[w](c) =  \int_{1}^{c} e^{w(\log(c'))} d c'
    \end{equation}

    These modified differentiating ($\text{\bf \dh}$) and integrating (\text{\bf \c{S}}) operators satisfy the following identities:

    \begin{equation}
        \left ( \text{\bf \c{S}} \circ \text{\bf \dh} \right ) [u] \sim u, \quad \left ( \text{\bf \dh} \circ \text{\bf \c{S}}  \right ) [w] = w
    \end{equation}

    A positive affine transformation on utility is only an additive transformation for $w$:

    \begin{equation}
        u(c) \to C + D \ u(c), \implies w(g) \to \log(D) + w(g)
    \end{equation}

    in a similar fashion, scaling of the utility can be ``pushforward'' \cite{book:SmoothManifolds} to $w$:

    \begin{equation}
        u(c) \to \mathbfcal{D}_r[u](c) = u(r c) \implies w(g) \to \mathbfcal{D}^{\text{\tiny \FiveStarLines}}_r[w](g) = \log(r) + w(g + \log(r))
    \end{equation}

    The requirement for the scale invariance of the utility can be rewritten to the $w$ variable:

    \begin{equation}
        \mathbfcal{D}^{\text{\tiny \FiveStarLines}}_r[w](g) = \log(D(r)) + w(g)
    \end{equation}

    applying an infinitesimal scale transformation $r=1+\epsilon$ and keeping only the first-order terms results in the following equation:

    \begin{equation}
        \epsilon + w(g) + w'(g) \ \epsilon = \log(D(1)) + \alpha \ \epsilon + w(g) + \mathcal{O}(\epsilon^2)
    \end{equation}

    where $\alpha = D'(1)/D(1)$.
    This can be satisfied if $D(1) = 1$ (which is a natural requirement, meaning that if we do not scale at all, then we do not need to transform the utility function either) and if $w$ fulfils the following simple differential equation:

    \begin{equation}
        w'(g) = \alpha - 1 = - \gamma 
    \end{equation}

    This determines the utility up to two integration constants:

    \begin{equation}
        w(g) = - \gamma \ g + A
    \end{equation}

    \begin{equation}
        u(c) = B + \int_{1}^{c} e^{w(\log(c'))} dc'
    \end{equation}

    \begin{equation}
        u(c) = B + e^A \ \frac{c^{1- \gamma} - 1}{1- \gamma}
    \end{equation}

    Which is, by definition, equivalent to an isoelastic function for any choice of $A$ and $B$.
    
\end{proof}

\begin{theorem}
    Isoelastic functions are invariant under the instrumentalisation (and Bellmann) operator:
    \begin{equation}
        \mathbfcal{I}[u_\gamma] \sim u_\gamma
    \end{equation}

    \begin{equation}
        \mathbfcal{B}[u_\gamma] \sim u_\gamma
    \end{equation}
\end{theorem}

\begin{proof}
    For isoelastic utility functions, there is an optimal $p'_\gamma$ for which the expected utility is maximal:
    \begin{equation}
        \frac{\partial}{\partial p'} \left ( P \ u_\gamma(p' \ c) + (1-P) \ u_\gamma((1-p') \ c) \right ) \Bigr|_{p'=p'_\gamma} = 0
    \end{equation}
    By introducing:
        \begin{equation}
        v_\gamma(c) = \frac{\partial}{\partial c} u_\gamma(c) = c^{-\gamma}
    \end{equation}
    We get:
    \begin{equation}
        c \ \left ( P \ v_\gamma(p'_\gamma c) - (1-P) \ v_\gamma((1-p'_\gamma) c) \right ) = 0
    \end{equation}
    Assuming $c>0$
    \begin{equation}
        \frac{P}{{p'}_\gamma^{\gamma}} - \frac{1-P}{(1-p'_\gamma)^{\gamma}} = 0
    \end{equation}
    The solution for $p'_\gamma$ is:
    \begin{equation}
        p'_\gamma = \frac{P^{1/\gamma}}{P^{1/\gamma}+(1-P)^{1/\gamma}}
    \end{equation}
    To prove that it is indeed a maximum, one can check its second derivative:
    \begin{equation}
        \frac{\partial^2}{\partial {p'}^2} \left ( P \ u_\gamma(p' \ c) + (1-P) \ u_\gamma((1-p') \ c) \right ) \Bigr|_{p'=p'_\gamma} = d_\gamma
    \end{equation}
    Straightforward calculation gives that:
    \begin{equation}
        d_\gamma = - \frac{\gamma}{c^{\gamma + 1}} \frac{\left ( P^{1/\gamma} + (1-P)^{1/\gamma} \right )^{\gamma + 2}}{\left ( P (1-P) \right )^{1/\gamma}}
    \end{equation}
    which is negative for all $\gamma > 0$, signaling a maximum.

    This means that:
    \begin{equation}
    \mathbfcal{I}[u_\gamma](c) = P \ u_\gamma(p'_\gamma \ c) + (1-P) \ u_\gamma((1-p'_\gamma) \ c)
    \end{equation}

    \begin{equation}
    \mathbfcal{I}[u_\gamma](c) = \frac{c^{1-\gamma} \left ( P^{1/\gamma} + (1-P)^{1/\gamma} \right )^\gamma -1}{1-\gamma}
    \end{equation}

    For which one can find a positive affine transformation, making it equivalent to the original isoelastic utility $u_\gamma$:

    \begin{equation}
        \mathbfcal{I}[u_\gamma](c) = C_\gamma + D_\gamma \ u_\gamma(c)
    \end{equation}
    
    \begin{equation}
        C_\gamma = - \frac{1-\left ( P^{1/\gamma} + (1-P)^{1/\gamma} \right )^\gamma}{1-\gamma}
        \quad
        D_\gamma = \left ( P^{1/\gamma} + (1-P)^{1/\gamma} \right )^\gamma
    \end{equation}
\end{proof}

The relative risk aversion parameter, $\gamma$, influences $p'$, the portion of capital placed to outcome A; therefore, different isoelastic utility functions can observably change the agent's behaviour. 

\begin{figure}[H]
    \centering
    \includegraphics[width=10 cm]{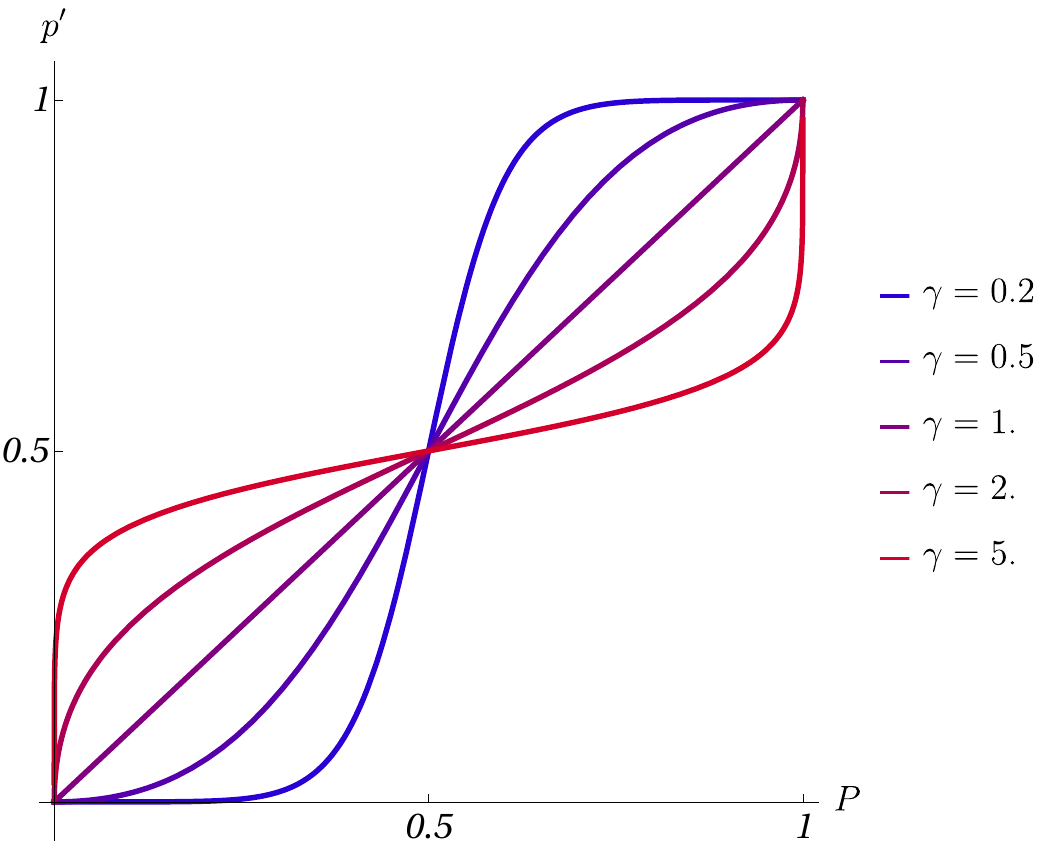}
    \caption{Strategy plot for expected utility maximizing agent, which adopt isoelastic utilities: $p'_\gamma(P)$.}
    \label{fig:PolicyPlot_0_gamma}
\end{figure}

\begin{theorem}
    For any subjective utility function, which is twice continuously differentiable, strictly monotone increasing, strictly concave, its local relative risk aversion parameter is bounded, strictly positive and converges to a finite $\gamma_\downarrow \in \mathbb{R}_{>0}$ as $c \to 0$, will converge to an isoelastic utility function under repeated application of the instrumentalisation operator:

    \begin{equation}
        \lim_{n \to \infty} \mathbfcal{I}^n[u_0] \sim u_{\gamma_\downarrow}
    \end{equation}
\end{theorem}

\begin{proof}
    Introducing a few new functions and operators is useful to prove this theorem.

    We will introduce the following growth-dependent relative risk aversion function $\gamma(g): (-\infty, 0] \mapsto \mathbb{R}$: 

    \begin{equation}
        \gamma(g) = - w'(g) = - \frac{\partial}{\partial g} \log(u'(e^g)) =  - e^g \ \frac{ u''(e^g)}{u'(e^g)}
    \end{equation}

    this can be formulated using differential operators:

    \begin{equation}
        \gamma = - \bm{\partial} [w] = - \bm{\partial} \circ \text{\bf \dh} [u] = \mathbfcal{R} [u]
    \end{equation}

    A nice property of the relative risk aversion function is that from the knowledge of $\gamma(g)$, one can recover a utility function, which is unique up to a linear affine transformation.

    \begin{equation}
        u(c) = \int_{1}^{c} e^{-\int_{1}^{\log(c')} \gamma(g') d g'} d c' 
        \sim  A + \int_{1}^{c} e^{B - \int_{1}^{\log(c')} \gamma(g') d g'} d c' 
    \end{equation}

    This can be formulated using integral operators:

    \begin{equation}
        u = \text{\bf \c{S}} [w] = \text{\bf \c{S}} \circ \left ( - \text{\bf S} \right ) [\gamma] = \mathbfcal{U} [\gamma] 
    \end{equation}

    The defined differential and integral operators satisfy the following properties:

    \begin{equation}
        \left ( \mathbfcal{U} \circ \mathbfcal{R} \right ) [u] \sim u, \quad \left ( \mathbfcal{R} \circ \mathbfcal{U}  \right ) [\gamma] = \gamma
    \end{equation}

    This construction allows us to prove the convergence of utility functions by showing the convergence of $\gamma_n(g)$ functions.
    To do this, we can define a ``second order pushforward'' instrumentalisation operator $\mathbfcal{J}$ by the following commuting diagram:
    
    {
    \begin{center}
    
        \tikzcdset{every label/.append style = {font = \normalsize}}
        \begin{tikzcd}[row sep=32 pt, column sep=62 pt]
            u_n \arrow[r, "\mathbfcal{I}"] \arrow[d, "\mathbfcal{R}", shift left] 
            & u_{n+1} \arrow[d, "\mathbfcal{R}", shift left]  \\
            \gamma_n \arrow[r, "\mathbfcal{J}"] \arrow[u, "\mathbfcal{U}", shift left]
            &  \gamma_{n+1} \arrow[u, "\mathbfcal{U}", shift left]
        \end{tikzcd}

    \end{center}
    }

    \paragraph{Construction of the $\mathbfcal{J}$ operator:}

    As a first step, we can determine $p'(c)$ and $p'(g)$ functions:

    \begin{equation}
        \frac{\partial}{\partial p'} \left ( P \ u(p' \ c) + (1-P) \ u((1-p') \ c) \right ) \Bigr|_{p'=p'(c)} = 0
    \end{equation}

    or in $v(c) = u'(c)$ variables:

    \begin{equation}
        \label{eq:ppv}
         P \ v(p'(c) \ c) = (1-P) \ v((1-p'(c)) \ c)
    \end{equation}

    Taking the logarithm and substituting the $w(g)$ variables defined in eq. \eqref{eq:w(g)} we get:

    \begin{equation}
         \log(P) + w(\log(p'(g)) + g) = \log(1-P) + w(\log(1-p'(g))+g)
    \end{equation}

    which can be rearranged as an expression of $\gamma(g)$:

    \begin{equation}
        \log \left ( \frac{P}{1-P} \right ) = \int_{g+\log(1-p'(g))}^{g+\log(p'(g))} \gamma(g) dg
    \end{equation}

    We assume that $0 < \underline{\gamma} \le \gamma(g) \le \overline{\gamma}$, in which case this implicit equation always has a solution:

    \begin{equation}
        F_g(p') = \int_{g+\log(1-p'(g))}^{g+\log(p'(g))} \gamma(g) dg
    \end{equation}

    It is easy to see that $F_g(1/2) =0$, and that it goes to $+\infty$ as $p' \to 1$ and $-\infty$ as $p' \to 0$. Because of the continuity of $F_g(.)$ there will be always a $p'(g)$ where it intersects the value $\log(P/(1-P))$.

    For every $g$, we can introduce an effective relative risk aversion:

    \begin{equation}
        \gamma_\mathrm{eff}(g) = \frac{1}{\log(p'(g)) - \log(1-p'(g))} \int_{g+\log(1-p'(g))}^{g+\log(p'(g))} \gamma(g') dg'
    \end{equation}

    which is an integral mean of $\gamma(g)$ on the $[g+\log(p'(g)), g+ \log(1-p'(g)]$ interval.

    \begin{equation}
        p'(g) = \frac{P^{1/\gamma_\mathrm{eff}(g)}}{P^{1/\gamma_\mathrm{eff}(g)} + (1-P)^{1/\gamma_\mathrm{eff}(g)}}
    \end{equation}

    because $\gamma_\mathrm{eff}(g)$ is an integral mean, it is surely bounded as $\gamma(g)$:
    $0 < \underline{\gamma} < \gamma_\mathrm{eff}(g) < \overline{\gamma}$, which gives bounds for $p'(g)$ as well. If we assume that $P>1/2$:

    \begin{equation}
        \frac{P^{1/\overline{\gamma}}}{P^{1/\overline{\gamma}} + (1-P)^{1/\overline{\gamma}}}
        \le 
        p'(g) 
        \le 
        \frac{P^{1/\underline{\gamma}}}{P^{1/\underline{\gamma}} + (1-P)^{1/\underline{\gamma}}}
    \end{equation}

    \begin{equation}
        u_1(c) = P \ u(p'(c) \ c) + (1-P) \ u((1-p'(c)) \ c)
    \end{equation}

    Taking the derivative respect to $c$, we get:

    \begin{equation}
        \begin{split}
            v_1(c) =& P \ v(p'(c) \ c) \ \left ( p'(c) + \frac{d p'(c)}{d c} \right ) + \\
                    &(1-P) \ v((1-p'(c)) \ c) \ \left ( (1-p'(c)) - \frac{d p'(c)}{d c} \right )
        \end{split}
    \end{equation} 

    because of equation \eqref{eq:ppv}, this considerably simplifies to:

    \begin{equation}
        \begin{split}
            v_1(c) = & P \ v(p'(c) \ c) \\ 
                    = & (1-P) \ v((1-p'(c)) \ c)
        \end{split}
    \end{equation}

    Taking the logarithm and changing the variables from $c$ to $g$:

    \begin{equation}
        \begin{split}
            w_1(c) = & \log(P) + w(g + \log(p'(g)) ) \\ 
                    = & \log(1-P) + w(g + \log(1-p'(g)) )
        \end{split}
    \end{equation}

    taking the derivative respect to $g$ for both equations results:

    \begin{equation}
        \gamma_1(g) = - \frac{d}{d g} w_1(g) = - w'(g + \log(p'(g)) ) \ \left ( 1 + \frac{1}{p'(g)}\frac{d p'(g)}{d g} \right )
    \end{equation}

    \begin{equation}
        \gamma_1(g) = - \frac{d}{d g} w_1(g) = - w'(g + \log(1-p'(g)) ) \ \left ( 1 - \frac{1}{1-p'(g)}\frac{d p'(g)}{d g} \right )
    \end{equation}

    Combining the two equations, we get:

    \begin{equation}
        \frac{\gamma_1(g)}{\gamma(g + \log(p'(g)))} \ p'(g) +
        \frac{\gamma_1(g)}{\gamma(g + \log(1-p'(g)))} \ (1-p'(g)) = 1  
    \end{equation}
    
    resulting:

    \begin{equation}
        \mathbfcal{J}[\gamma](g) = \gamma_1(g) = \frac{1}{\frac{p'(g)}{\gamma(g+ \log(p'(g)))} + \frac{1-p'(g)}{\gamma(g+ \log(1-p'(g)))}}
    \end{equation}

    which is the weighted harmonic mean of $\gamma(g+ \log(p'(g)))$ and $\gamma(g+ \log(1-p'(g)))$.

    \begin{equation}
        \mathbfcal{J}[\gamma](g) = H^{(p'(g),1-p'(g))} ((\gamma(g+ \log(p'(g))), \gamma(g+ \log(1-p'(g)))))
    \end{equation}

    where we used the notation for weighted harmonic mean:
    \begin{equation}
        H^{\underline{\alpha}}(\underline{x}) = \frac{\sum_i \alpha_i}{\sum_i \frac{\alpha_i}{x_i}} 
    \end{equation}

    Because the instrumentalisation of $\gamma(g)$ can be calculated as a mean, it can not increase the upper and lower bounds of $\gamma(g)$:

    \begin{equation}
        [\underline{\gamma}_{n+1},\overline{\gamma}_{n+1}] \subseteq [\underline{\gamma}_{n},\overline{\gamma}_{n}] \subseteq \dots \subseteq [\underline{\gamma}, \overline{\gamma}]
    \end{equation}

    This means that there will be an upper bound for $\underline{\Delta g} < 0$, by which the instrumentalisation operator shifts the argument of $\gamma(.)$ toward $-\infty$:

    \begin{equation}
        \underline{\Delta g} = \sup_{\gamma \in [\underline{\gamma}, \overline{\gamma}]} \max
        \left ( \log \left ( \frac{P^{1/\gamma}}{P^{1/\gamma}+(1-P)^{1/\gamma}} \right )  , 
        \log \left ( \frac{(1-P)^{1/\gamma}}{P^{1/\gamma}+(1-P)^{1/\gamma}} \right ) \right )
    \end{equation}
    
    as all application of $\mathbfcal{J}$ shifts the argument of $\gamma$ at least by $\underline{\Delta g}$, the following two sided inequality holds:

    \begin{equation}
        \inf_{g \le n \cdot \underline{\Delta g}}{\gamma(g)} \le 
        \mathbfcal{J}^n[\gamma](g)
        \le \sup_{g \le n \cdot \underline{\Delta g}}{\gamma(g)}
    \end{equation}

    if $\lim_{g \to -\infty} \gamma(g) = \gamma_\downarrow$, then

    \begin{equation}
        \lim_{g_M \to -\infty} \inf_{g \le g_M}{\gamma(g)} = \gamma_\downarrow, \quad
        \lim_{g_M \to -\infty} \sup_{g \le g_M}{\gamma(g)} = \gamma_\downarrow
    \end{equation}

    resulting that:

    \begin{equation}
        \forall g \in (-\infty,0], \  \lim_{n \to \infty}  \mathbfcal{J}^n[\gamma](g) = \gamma_\downarrow
    \end{equation}
    
\end{proof}

Besides the mathematical proof, the intuitive understanding of the result is the following:
If the gambler values capital and is risk averse, then she will not commit all her resources to the more probable outcome but reserve some of her capital in case the less probable scenario comes out.
Suppose she associates the value of an intermediate stage with the maximally achievable expected utility from that state. In that case, she can deduce the value of all intermediate stages step by step from her subjective utility judgement at the final stage.

If we imagine gamblers playing more and more rounds, starting from less and less starting capital, their possible maximal gain will decrease (because of their successive careful choices).

This means that if an agent wants to evaluate the value of a state, she needs to consider only very small final capital values.

The degree of relative risk aversion is conserved in a multiplicative gamble; therefore, if the agent's relative risk aversion converges to a finite value for very small final capitals, then the initial values of a state will be dominated by this limit of relative risk aversion.

\begin{remark}
    Because of the averaging property of the instrumentalisation operator, the result can most probably be generalised further. In the proof, we principally used the property that instrumentalisation does not change the bounds of a relative risk aversion function and that it includes a finite ``shift'' if we construct a sequence of gambles with decreasing starting capital and increasing number of rounds.

    However, because of the averaging, the convergence could probably be proved for a more general sequence of gambles (for instance, where the gambler starts with a unit capital, and her gain can grow exponentially) and more general subjective final utility functions.

    Numerical evidence shows that very general utility functions (non-monotone, non-convex) can converge to an isoelastic utility, signalling a wider domain of validity of the fixed point result.
    
\end{remark}

The paper of Hakansson from 1974 \cite{paper:IsoelasticConvergence} might be relevant for investigating the convergence of instrumental utilities. (In his model, the initial capital is finite, so the terminal utilities are dominated by its behaviour around infinity.)

\subsection{Contest view}
\label{sec:ContestView}

A natural way to change the capital collecting game into a contest is to introduce another player with the same opportunities as our original agent. A natural criterion for winning this resource-collecting game is to have more capital after the final round.

\subsubsection{Raffle contest}

First, let us consider a game with two stages.
\begin{description}
    \item[First stage:] two gamblers can decide what amount of their capital are placing to the two possible outcomes, A and B, in $n$ rounds:

    \begin{equation}
        c_1 = (2 \ p'_1)^\kappa (2 \ (1-p'_1))^{n-\kappa}, \quad
        c_2 = (2 \ p'_2)^\kappa (2 \ (1-p'_2))^{n-\kappa}
    \end{equation}

    \begin{equation}
        \kappa \sim \mathrm{Binom}(n,P)
    \end{equation}

    \item[Second stage:] The two players can spend their capital $c_1$ and $c_2$ to buy tickets in a Raffle (or a Lottery).
    Player 1 can win with probability $c_1/(c_1+c_2)$, and Player 2 can win with probability $c_2/(c_1+c_2)$, resulting the following utility function:

    \begin{equation}
        u_1(c_1,c_2) = \frac{c_1}{c_1+c_2} u_1^\mathrm{W} + \frac{c_2}{c_1+c_2} u_1^\mathrm{L}
    \end{equation}

    \begin{equation}
        u_2(c_1,c_2) = \frac{c_2}{c_1+c_2} u_2^\mathrm{W} + \frac{c_1}{c_1+c_2} u_2^\mathrm{L}
    \end{equation}
\end{description}

where $u_i^\mathrm{W}$ represents the utility associated with the $i$-th Player winning, while $u_i^\mathrm{L}$ represents the utility associated with losing.
If all Players prefer winning more than losing ($u_i^\mathrm{W}>u_i^\mathrm{L}$), then by a suitable positive affine transformation, all $u_i^\mathrm{W}$ can be changed to $1$ and all $u_i^\mathrm{L}$ to $0$.

This results in the following explicit final utility function for Player~1 and Player~2:

\begin{equation}
    U_1(p'_1,p'_2) = \mathbb{E} \left [  \frac{(2 \ p'_1)^\kappa (2 \ (1-p'_1))^{n-\kappa}}{(2 \ p'_1)^\kappa (2 \ (1-p'_1))^{n-\kappa} + (2 \ p'_2)^\kappa (2 \ (1-p'_2))^{n-\kappa}}  \right ]
\end{equation}

\begin{equation}
    U_1(p'_1,p'_2) + U_2(p'_1,p'_2) = 1
\end{equation}

These expected utilities can be interpreted as winning rates after these transformations.

\begin{figure}[H]
    \centering
    \includegraphics[width=12 cm]{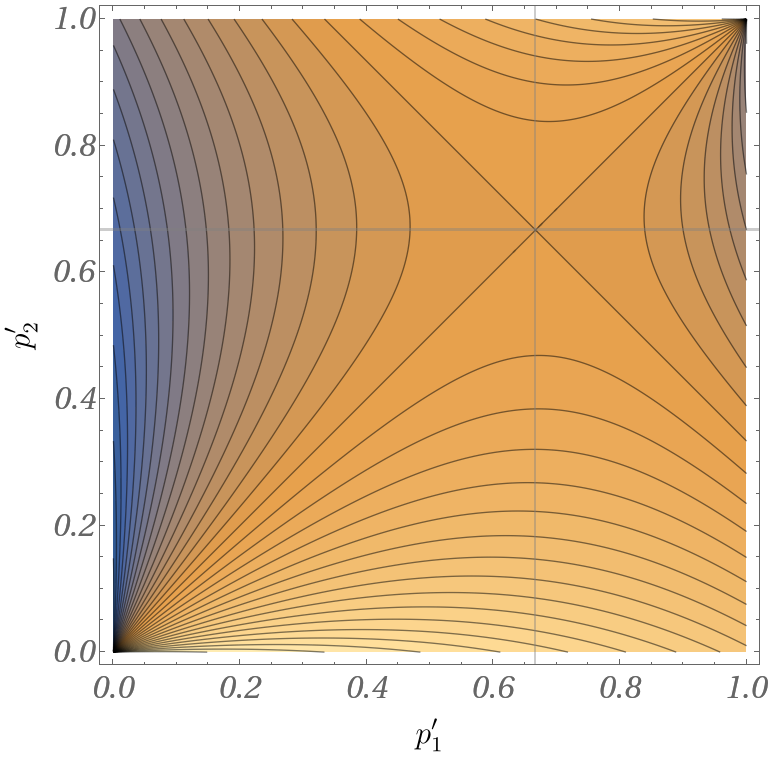}
    \caption{$U_1(p'_1,p'_2)$ function for $P=2/3$, $n=1$.}
    \label{fig:Contest_Up1p2}
\end{figure}

Because of symmetry, we can search the equilibrium, i.e. the saddle point on the $p'_1 = p'_2$ line.
At the saddle point, the function has to fulfil the condition:

\begin{equation}
    \frac{\partial}{\partial p'_1} U_1(p'_1,p'_2) |_{p'_1 = p'^*, p'_2 = p'^*} = \mathbb{E} \left [ \frac{\kappa - n p'^*}{4 p'^* (1-p'^*)} \right ] = n \frac{P-p'^*}{4 p'^* (1-p'^*)} =
    0
\end{equation}

Which has a very simple solution:

\begin{equation}
\boxed{
    p'^* = P
    }
\end{equation}

Direct calculation of the Hessian gives:

\begin{equation}
    H_{U_1}(P,P) = \frac{n}{4 P (1-P)}
    \begin{bmatrix}
-1 & 0 \\
0 & 1
\end{bmatrix}
\end{equation}

Which confirms that we found a saddle point.

\paragraph{Illustration of the equilibrium:}

To illustrate how gradual changes in the splitting ratios might lead to the equilibrium point, we show a stream plot of so-called ``adaptive dynamics'' \cite{paper:AdaptiveDynamics} in figure \ref{fig:AdaptiveDynamics0}:

\begin{equation}
    \overrightarrow{V}(p'_1,p'_2) = 
    \begin{bmatrix}
    \frac{\partial}{\partial p'_1} U_1(p'_1,p'_2) &
    \frac{\partial}{\partial p'_2} U_2(p'_1,p'_2)
    \end{bmatrix}
\end{equation}

\begin{figure}[H]
    \centering
    \includegraphics[width=12 cm]{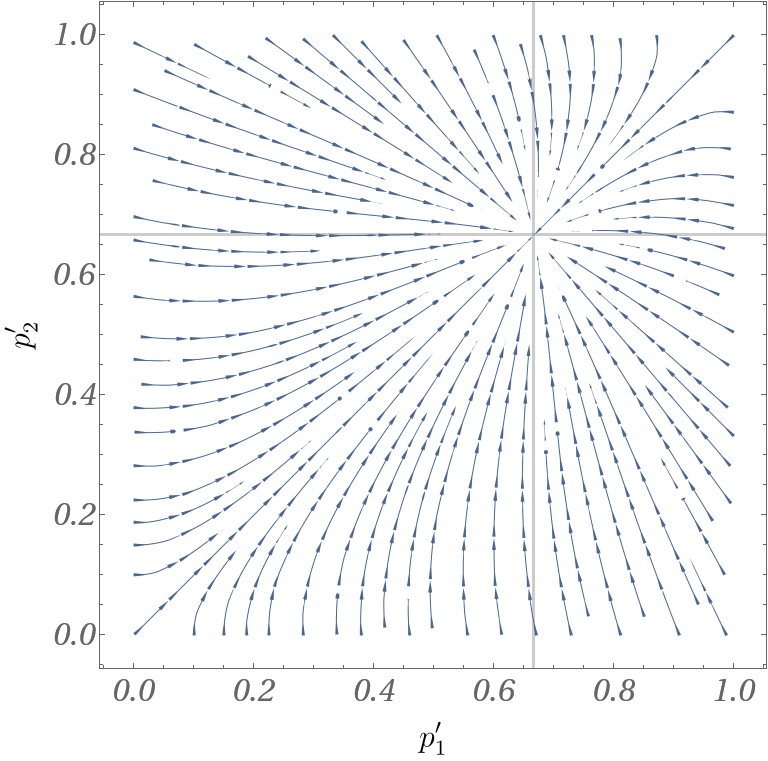}
    \caption{Adaptive dynamics stream plot generated from the vector field $\overrightarrow{V}(p'_1,p'_2)$, for $P=2/3$.}
    \label{fig:AdaptiveDynamics0}
\end{figure}

\subsubsection{Generalizations}

This result can be generalized in multiple ways, signalling the robustness of the concept:

\paragraph{Arbitrary number of players:}

The contest view can be generalized to multiple players. If the number of players is $m$, we get the following equation for equilibrium:

\begin{equation}
    \frac{\partial}{\partial p'_1} U(\underline{p'}) \Bigr|_{p'_i = p'^*}  = n \frac{m-1}{m^2} \frac{P-p'^*}{p'^* (1-p'^*)}
\end{equation}

Which has the same remarkably simple solution:

\begin{equation}
    p'^*_i = P
\end{equation}

\paragraph{Nonlinear ticket prices:}

Remarkably, the equilibrium point remains the same if the raffle ticket prices scale nonlinearly.
The calculation gives the same result if for a collected capital $c$, $c^\alpha$ amount of ticket can be purchased.

This would generalise the utility function of the second stage in the following way:

    \begin{equation}
        u_1^\alpha(c_1,c_2) = \frac{c_1^\alpha}{c_1^\alpha+c_2^\alpha}
    \end{equation}

    \begin{equation}
    \forall \alpha > 0, n \ge 1, \quad p'^*_{\alpha,n}(P) = P
    \end{equation}

\paragraph{Allowing Behavioural strategies:}

In the original construction of the raffle context, we assumed that the players were choosing a splitting ratio for all rounds in the beginning.
This requirement can be relaxed and allow players to choose their strategy according to their current capital and the game history they can recall. These history-dependent strategies are called Behavioural strategies \cite{book:GameTheory}.

In this case, we can try to find the first so-called subgame perfect equilibrium \cite{book:GameTheory}, which means that at every stage, we assume that in the future, everybody will play optimally.
In this framework, it is natural to introduce an instrumentalisation of the utility function, which associates the equilibrium expected utilities to nonterminal stages of the game when the players have the pair of collected capitals $(c_1,c_2)$:

\begin{equation}
    \mathbfcal{I}[u](c_1,c_2) = \mathbb{E} \left [ \frac{c_1 (p'_1)^\kappa (1-p'_1)^{n-\kappa}}{c_1 (p'_1)^\kappa (1-p'_1)^{n-\kappa} + c_2 (p'_2)^\kappa  (1-p'_2)^{n-\kappa}} \Bigr|_{p'_1 = p'^*_1, p'_2 = p'^*_2} \right ]
\end{equation}

This instrumentalisation results in a very simple form of the instrumental utilities in an intermediate stage:

\begin{equation}
    \mathbfcal{I}[u](c_1,c_2) = \frac{c_1}{c_1+c_2}
\end{equation}

This means that all potentially history-dependent subgame perfect equilibria have the same history-independent optimal splitting ratios:

\begin{equation}
    p'^{[1]}_i = P, \quad p'^{[2]}_i(c^{[1]}_1,c^{[1]}_2) = P, \quad p'^{[3]}_i(c^{[2]}_1,c^{[2]}_2,c^{[1]}_1,c^{[1]}_2) = P,
    \dots
\end{equation}

\subsubsection{Capital contest:}

An even simpler contest could be defined if we practically skip the second stage and declare that:

Whoever has more capital wins; if the capitals are equal, all players have the same chance of winning.

\begin{figure}[H]
    \centering
    \includegraphics[width=12 cm]{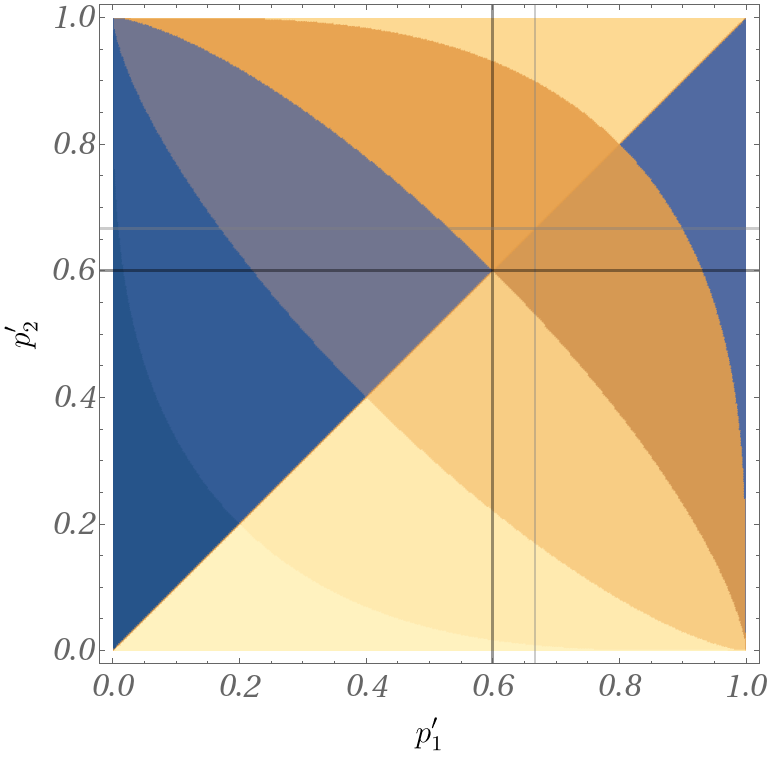}
    \caption{$U_1(p'_1,p'_2)$ function for a capital contest with $P=2/3$, $n=5$.}
    \label{fig:Contest_Up1p2_Disco}
\end{figure}

This simple contest results in a piecewise constant but discontinuous utility function, which depends on the splitting ratios $p'_1$ and $p'_2$ and the number of rounds $n$. 
Figure \ref{fig:Contest_Up1p2_Disco} shows a concrete example of the utility function, which clearly shows that for this contest, the equilibrium solution can differ from the continuous case, where we found $p'^*_i=P$.

The utility function is not differentiable; therefore, we can not use partial derivatives to find an equilibrium, but there are other techniques by which we can find the solution, such as the iterative elimination of weakly dominated strategies \cite{book:GameTheory}:

\begin{equation}
    \vcenter{\hbox{\includegraphics[width=21mm,scale=0.5]{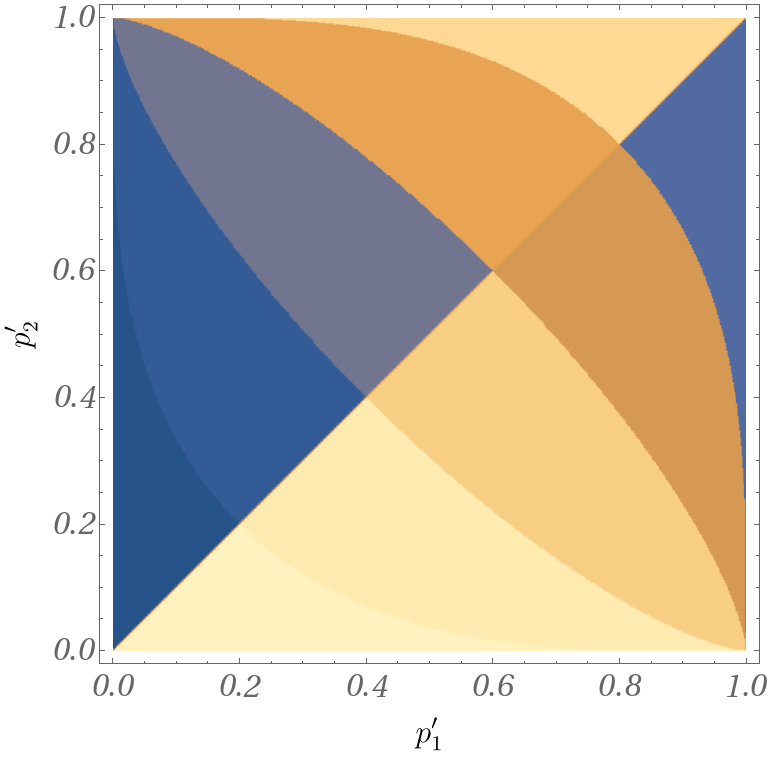}}}
    \vcenter{\hbox{\scalebox{2}{ $\rightarrow$}}}
    \vcenter{\hbox{\includegraphics[width=21mm,scale=0.5]{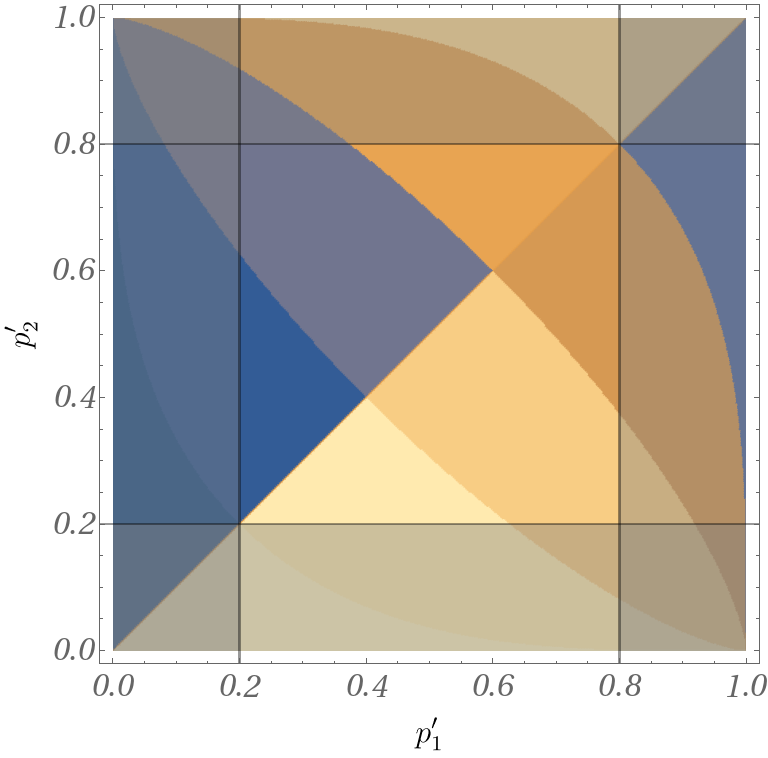}}}
    \vcenter{\hbox{\scalebox{2}{ $\rightarrow$}}}
    \vcenter{\hbox{\includegraphics[width=21mm,scale=0.5]{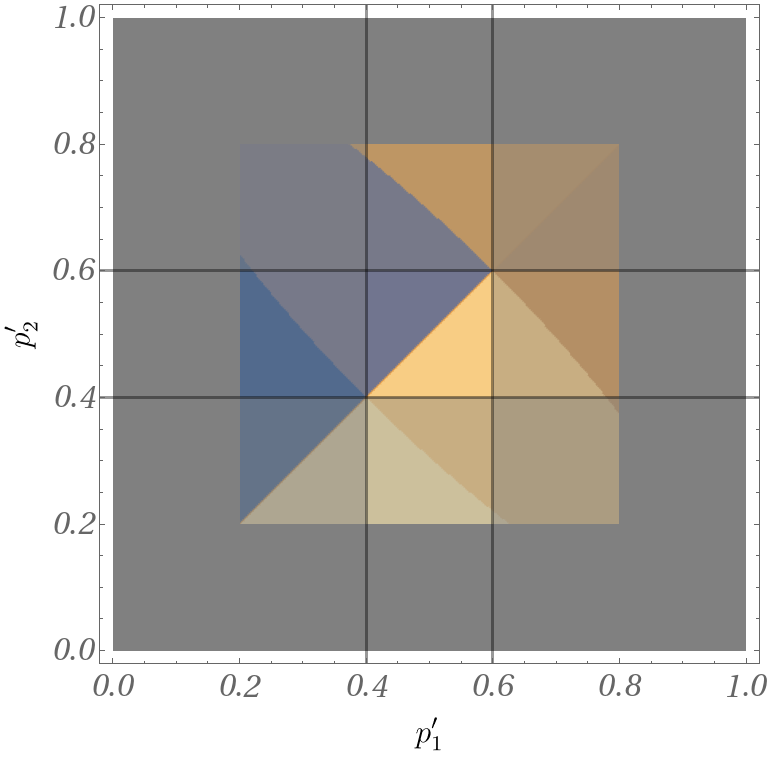}}}
    \vcenter{\hbox{\scalebox{2}{ $\rightarrow$}}}
    \vcenter{\hbox{\includegraphics[width=21mm,scale=0.5]{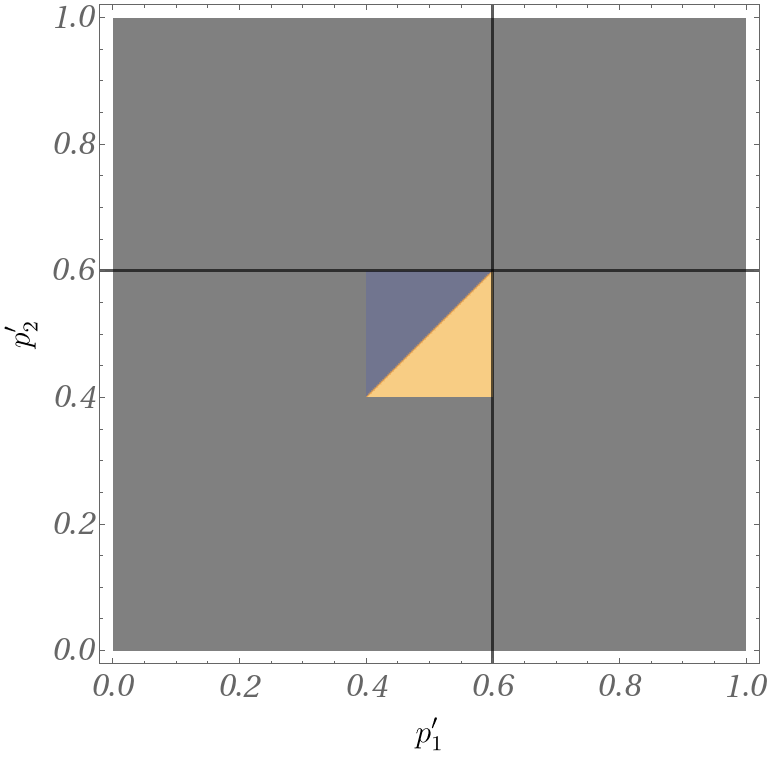}}}
\end{equation}

Arguments based on the iterative elimination of weakly dominated strategies lead to a general formula for the equilibrium splitting ratios:

\begin{equation}
    p'^*_{\text{\textifsymbol{32}} n}(P) = \frac{[n \ P]}{n}
\end{equation}

In some edge cases, there is only an interval in which the splitting ratio has to fall. These cases are also captured with the following general inequalities:

\begin{equation}
    \frac{[n \ P]_{-}}{n} \le p'^*_{\text{\textifsymbol{32}} n}(P) \le \frac{[n \ P]_{+}}{n}
\end{equation}

Where $[x]_{-}$ stands for ``round half down'', in which case half-way values of $x$ are always rounded down, and $[x]_{+}$ stands for ``round half up'' rounding. \footnote{for more background about rounding conventions see \href{https://www.eetimes.com/an-introduction-to-different-rounding-algorithms/}{Rounding Algorithms Compared}, \href{https://www.clivemaxfield.com/diycalculator/sp-round.shtml}{Rounding Algorithms 101} by Clive Maxfield, or \href{https://en.wikipedia.org/wiki/Rounding}{Rounding on Wikipedia}.}

\begin{figure}[H]
    \centering
    \includegraphics[width=12 cm]{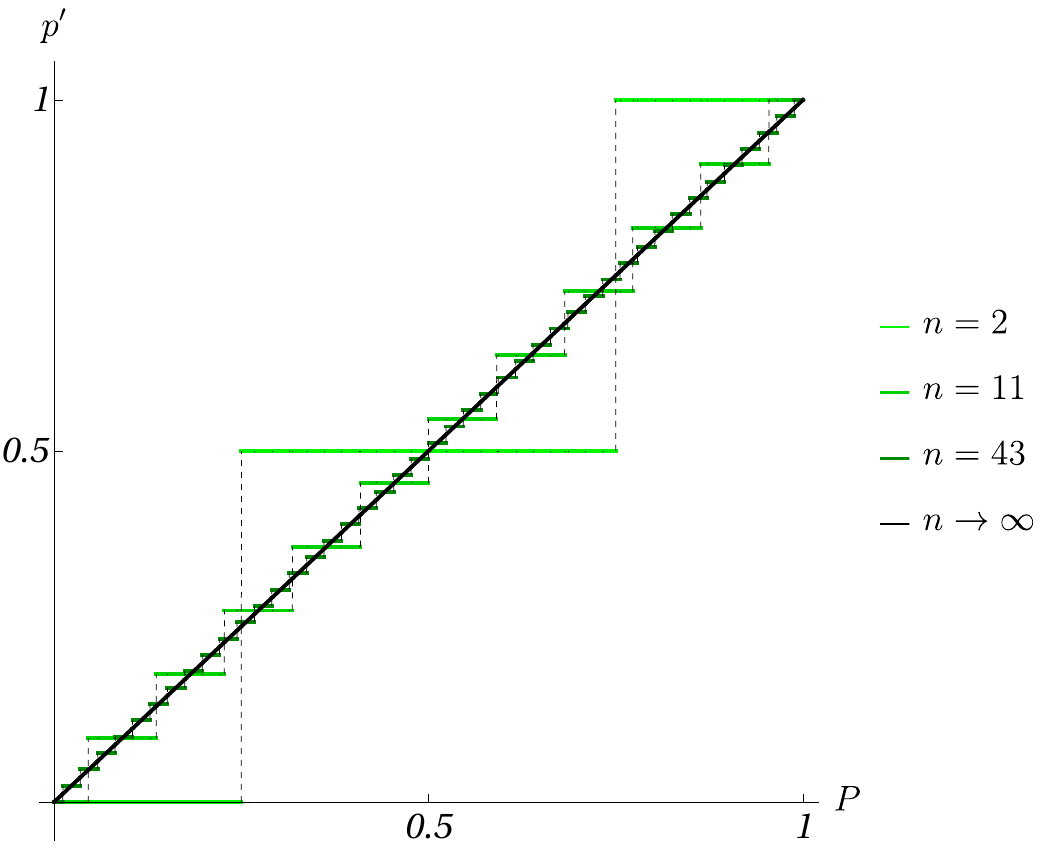}
    \caption{$p'^*_{\text{\textifsymbol{32}} n}(P)$ for several values of $n$, and the limit case as $n \to \infty$.}
    \label{fig:pp_n_P}
\end{figure}

\begin{remark}
    The defined discontinuous utility function can be viewed as the limit case of nonlinear raffle games. However, finding the equilibrium and taking the limit are not interchangeable in general:

    \begin{equation}
    \lim_{\alpha \to \infty} p'^*_{\alpha,n}(.) \nequiv p'^*_{\infty,n}(.) = p'^*_{\text{\textifsymbol{32}} n}(.)
    \end{equation}
    
\end{remark}

\begin{remark}
    However, for all $P \in [0,1]$ the interchangeability is restored as $n \to \infty$:

    \begin{equation}
    \lim_{n \to \infty } \lim_{\alpha \to \infty} p'^*_{\alpha,n}(P) = P =  \lim_{n \to \infty }p'^*_{\text{\textifsymbol{32}} n}(P) 
    \end{equation}
    
\end{remark}

\subsubsection{Emergent isoelastic utility in contests}

\begin{remark}
    Less than $1$ relative risk aversion can also emerge in contests if two or more players can play in independent gambles and have a capital-sharing agreement.
    This can be a natural way for seemingly non-logarithmic utility functions to emerge and be aligned with an evolutionarily stable strategy.
    
\end{remark}

\paragraph{Illustrative calculation:}

A simple back-of-the-envelope calculation can demonstrate the emergence of aligned isoelastic utility functions:

Suppose there are two independent environments or casinos $\aleph$ and $\beth$, in which gamblers can bet on scenarios A and B.
In the first casino, A comes out with probability $P_\aleph$, while in the second casino, with $P_\beth$.
Let us assume that there are two cooperating players gambling in different casinos, both with a logarithmic utility function, who are sharing all their capital but can choose freely and independently their splitting ratios:

\begin{equation}
    p'_\aleph, \quad p'_\beth
\end{equation}

There is an explicit expression for the expected utility of player $\aleph$, depending on both splitting ratios:

\begin{equation}
\begin{split}
    U_\aleph(p'_\aleph, p'_\beth) = 
    &P_\aleph P_\beth \log ((p'_\aleph + p'_\beth)/2) \\
    &(1-P_\aleph) P_\beth \log ((1-p'_\aleph + p'_\beth)/2) \\
    &P_\aleph (1-P_\beth) \log ((p'_\aleph + 1 - p'_\beth)/2) \\
    &(1-P_\aleph) (1-P_\beth) \log ((1-p'_\aleph + 1- p'_\beth)/2)
\end{split}
\end{equation}

The two players share all their capital, and they have the same utility function; therefore, their utility function is always the same:

\begin{equation}
    U_\aleph(p'_\aleph, p'_\beth) =
    U_\beth(p'_\aleph, p'_\beth) 
\end{equation}

To simplify the calculation, let us assume that chances in both casinos are identical, i.e.:

\begin{equation}
    P = P_\aleph = P_\beth
\end{equation}

In this case, the symmetry of the problem suggests that both players will have identical equilibrium splitting ratios:

\begin{equation}
    p'^* = p'^*_\aleph = p'^*_\beth
\end{equation}

Both players try to maximize their expected utility by choosing their own splitting ratio. Equilibrium is reached when:

\begin{equation}
    \frac{\partial}{\partial p'_\aleph} 
    U_\aleph(p'_\aleph, p'_\beth) 
    \Bigr|_{p'_\aleph = p'^*,p'_\beth = p'^*} = 0, \quad
    \frac{\partial}{\partial p'_\beth} 
    U_\beth(p'_\aleph, p'_\beth) 
    \Bigr|_{p'_\aleph = p'^*,p'_\beth = p'^*} = 0
\end{equation}

The expression can be explicitly calculated for player $\aleph$:

\begin{equation}
    \frac{\partial}{\partial p'_\aleph} 
    U_\aleph(p'_\aleph, p'_\beth) 
    \Bigr|_{p'_\aleph = p'^*,p'_\beth = p'^*}  = 
    P^2 \frac{1}{2 p'^*} + (1-P)P (-1) + P(1-P) +
    (1-P)^2 \frac{-1}{2 (1-p'^*)}
\end{equation}

The equilibrium requirement simplifies to:

\begin{equation} 
    P^2 \frac{1}{2 p'^*}  -
    (1-P)^2 \frac{1}{2 (1-p'^*)} =0
\end{equation}

Which results in the equilibrium splitting ratio:

\begin{equation}
    p'^* = \frac{P^2}{P^2+(1-P)^2}
\end{equation}

This shows that the behaviour of the capital sharing players looks \emph{as if} they would maximize an isoelastic utility function with relative risk aversion parameter $\gamma = 1/2$.

If there are more capital-sharing players $\aleph, \beth, \gimel, \dots$ betting in independent casinos, then the emergent relative risk aversion  -- aligned with the optimal equilibrium strategy -- is $\gamma=1/m$, where $m$ is the number of cooperating players.

\begin{remark}
    The construction can have a biological/evolutionary interpretation as well. In a framework where propagation of genes is the primary goal of biological organisms \cite{book:SelfishGene,book:TheOriginsOfLife}, the following observation can be made:

    If a biological actor has a twin or multiple copies having the same set of genes, and they are acting in similar but uncorrelated environments, then (from the gene's point of view) the evolutionary optimal strategy is aligned with an isoelastic utility function with lowered relative risk aversion.
    
\end{remark}

\section{Unification through relative risk aversion}
\label{appendix:Unifiction}

\subsection{Equilibrium of general Statistical games}

The proof belongs to Theorem~\ref{thm:StatisticalGameEquilibrium}.

\begin{proof}

In a general Statistical game \SG{N,K_A,K_B,M,\gamma} \PI/ has an isoelastic utility function, with relative risk aversion $\gamma >0$, $\gamma \ne 1$:

\begin{equation}
    u_\gamma(c) = \frac{c^{1-\gamma}-1}{1-\gamma}
\end{equation}

\PI/'s expected utility is:

\begin{equation}
    U_\gamma(P) = P \ \left ( \sum_k p_k(A) u_\gamma(p'_k(P)) \right ) 
    + (1-P) \left ( \sum_k p_k(B) u_\gamma(1-p'_k(P)) \right )
\end{equation}

This expression is maximized by the choice:

\begin{equation}
    p'^*_{\gamma,k}(P) = \frac{(P \ p_k(A))^{1/\gamma}}{(P \ p_k(A))^{1/\gamma} + ((1-P) \  p_k(B))^{1/\gamma}}
\end{equation}

Introducing:

\begin{equation}
    \begin{split}
        U_A(P) &= \sum_k p_k(A) u_\gamma(p'^*_{\gamma,k}(P)) \\
        U_B(P) &= \sum_k p_k(B) u_\gamma(1-p'^*_{\gamma,k}(P)) 
    \end{split}
\end{equation}

In equilibrium, $P^*_\gamma$ has to satisfy the requirement:

\begin{equation}
    U_A(P^*_\gamma) = U_B(P^*_\gamma)
\end{equation}

or alternatively:

\begin{equation}
    \frac{U_A(P^*_\gamma) + \frac{1}{1-\gamma}}{U_B(P^*_\gamma) + \frac{1}{1-\gamma}} = 1
\end{equation}

Taking the logarithm, we get:

\begin{equation}
    \log \left ( \frac{U_A(P^*_\gamma) + \frac{1}{1-\gamma}}{U_B(P^*_\gamma) + \frac{1}{1-\gamma}} \right ) = 0
\end{equation}

Collecting the terms, we get:

\begin{equation}
    \frac{1-\gamma}{\gamma} \log \left ( \frac {P^*_\gamma}{1-P^*_\gamma}\right ) +
    \log \left ( 
    \frac {\sum_k p_k(A) p_k(A)^{\frac{1-\gamma}{\gamma}} 
    \frac{1}{\left ( (P^*_\gamma p_k(A))^{1/\gamma} + ((1-P^*_\gamma) p_k(B))^{1/\gamma} \right )^{1-\gamma}}}
    {\sum_k p_k(B) p_k(B)^{\frac{1-\gamma}{\gamma}} 
    \frac{1}{\left ( (P^*_\gamma p_k(A))^{1/\gamma} + ((1-P^*_\gamma) p_k(B))^{1/\gamma} \right )^{1-\gamma}}}
    \right ) = 0
\end{equation}

Introducing log-odds:

\begin{equation}
    \log \left ( \frac {P^*_\gamma}{1-P^*_\gamma}\right ) = \vartheta^*_\gamma
\end{equation}

This expression can be rewritten as:

\begin{equation}
    \vartheta^*_\gamma = \Phi(\vartheta^*_\gamma)
\end{equation}

where

\begin{equation}
\label{eq:PhiDef1}
    \Phi(\vartheta) = - \frac{\gamma}{1-\gamma} 
    \log \left ( 
    \frac {\sum_k p_k(A) p_k(A)^{\frac{1-\gamma}{\gamma}} 
    \frac{1}{\left ( e^{\frac{1}{2}\frac{\vartheta}{\gamma}} p_k(A)^{1/\gamma} +  e^{-\frac{1}{2}\frac{\vartheta}{\gamma}} p_k(B)^{1/\gamma} \right )^{1-\gamma}}}
    {\sum_k p_k(B) p_k(B)^{\frac{1-\gamma}{\gamma}} 
    \frac{1}{\left ( e^{\frac{1}{2}\frac{\vartheta}{\gamma}} p_k(A)^{1/\gamma} + e^{-\frac{1}{2}\frac{\vartheta}{\gamma}} p_k(B)^{1/\gamma} \right )^{1-\gamma}}}
    \right )
\end{equation}

or

\begin{equation}
\label{eq:PhiDef2}
    \Phi(\vartheta) = - \frac{\gamma}{1-\gamma} 
    \log \left ( 
    \frac {\sum_k p_k(A)^{1/\gamma} 
    \frac{1}{\left ( e^{\frac{1}{2}\frac{\vartheta}{\gamma}} p_k(A)^{1/\gamma} +  e^{-\frac{1}{2}\frac{\vartheta}{\gamma}} p_k(B)^{1/\gamma} \right )^{1-\gamma}}}
    {\sum_k p_k(B)^{1/\gamma} 
    \frac{1}{\left ( e^{\frac{1}{2}\frac{\vartheta}{\gamma}} p_k(A)^{1/\gamma} + e^{-\frac{1}{2}\frac{\vartheta}{\gamma}} p_k(B)^{1/\gamma} \right )^{1-\gamma}}}
    \right )
\end{equation}

It is straightforward but tedious to take the derivative of $\Phi(\vartheta)$ with respect to $\vartheta$. The following notation will considerably simplify the results:

\begin{equation}
    z = e^{\frac{1}{2} \frac{\vartheta}{\gamma}}
\end{equation}

\begin{equation}
    \begin{split}
        a_k &= z \ p_k(A)^{1/\gamma} + z^{-1} \ p_k(B)^{1/\gamma} \\
        b_k &= z \ p_k(A)^{1/\gamma} - z^{-1} \ p_k(B)^{1/\gamma}
    \end{split}
\end{equation}

\begin{equation}
    \Phi'(\vartheta) = 
    \frac{1}{2} 
    \left (
    \frac {\sum_k p_k(A)^{1/\gamma} 
    \frac{b_k}{a_k^{2-\gamma}}}
    {\sum_k p_k(A)^{1/\gamma} 
    \frac{1}{a_k^{1-\gamma}}} - 
    \frac {\sum_\ell p_\ell(B)^{1/\gamma} 
    \frac{b_\ell}{a_\ell^{2-\gamma}}}
    {\sum_\ell p_\ell(B)^{1/\gamma} 
    \frac{1}{a_\ell^{1-\gamma}}}
    \right )
\end{equation}

Introducing further notation:

\begin{equation}
    \begin{split}
        x_k &= p_k(A)^{1/\gamma} 
        \frac{1}{a_k^{1-\gamma}} \ge 0 \\
        y_\ell &= p_\ell(B)^{1/\gamma} 
        \frac{1}{a_\ell^{1-\gamma}} \ge 0
    \end{split}
\end{equation}

we get the following expression:

\begin{equation}
    \Phi'(\vartheta) = 
    \frac
    {\sum_{k,\ell} x_k y_\ell \ \frac{1}{2} \left ( \frac{b_k}{a_k} - \frac{b_\ell}{a_\ell} \right )}
    {\sum_{k,\ell} x_k y_\ell}
\end{equation}

Using Hölder's inequality \cite{book:Holder} in the numerator gives an upper bound for the absolute value of the derivative:

\begin{equation}
    |\Phi'(\vartheta)| \le 
    \frac
    {\left (\sum_{k,\ell} |x_k y_\ell| \right ) \ 
    \max_{k,l} \left ( \frac{1}{2} \left | \frac{b_k}{a_k} - \frac{b_\ell}{a_\ell} \right | \right )}
    {\sum_{k,\ell} |x_k y_\ell|}
\end{equation}

which simplifies to:

\begin{equation}
    |\Phi'(\vartheta)| \le 
    \max_{k,l} \left ( \frac{1}{2} \left | \frac{b_k}{a_k} - \frac{b_\ell}{a_\ell} \right | \right )
\end{equation}

\begin{remark}
    The expression $b_k/a_k$ is a shifted hyperbolic tangent function of $\vartheta$:
    \begin{equation}
        \frac{b_k}{a_k} = \tanh(\vartheta - \Delta \vartheta_k)
    \end{equation}
\end{remark}

Straightforward calculation shows that $\left | \frac{b_k}{a_k} - \frac{b_\ell}{a_\ell} \right |$ has only one local (and global) maximum, in which point its value is:

\begin{equation}
    \max_z  \frac{1}{2} \left | \frac{b_k(z)}{a_k(z)} - \frac{b_\ell(z)}{a_\ell(z)} \right | =
    \left |
    \frac{\sqrt{p_\ell(A) p_k(B)}^{1/\gamma} - \sqrt{p_k(A) p_\ell(B)}^{1/\gamma}}
    {\sqrt{p_\ell(A) p_k(B)}^{1/\gamma} + \sqrt{p_k(A) p_\ell(B)}^{1/\gamma}}
    \right |
\end{equation}

By introducing the following exponent:

\begin{equation}
\label{eq:OverlineLambdaDefinition}
    \overline{\Lambda} = \max_{k,\ell}
    \left |
    \log \left ( \sqrt{\frac{p_k(A)}{p_\ell(A)} \frac{p_\ell(B)}{p_k(B)}} \right )
    \right |
\end{equation}

a global bound can be given:

\begin{equation}
\label{eq:GlobalBoundPhi}
    |\Phi'(\vartheta)| \le 
    \frac{1-e^{-\overline{\Lambda}/\gamma}}{1+e^{-\overline{\Lambda}/\gamma}}
    ,
\end{equation}

which is strictly smaller than $1$ if all $p_k(A),p_\ell(B)>0$.
This means that $\Phi$ is a contraction for all $\gamma > 0$, $\gamma \ne 1$.

The Banach fixed point theorem \cite{book:BanachFixedPoint} guarantees the existence of a unique fixed point for this contraction:

\begin{equation}
    \forall \gamma > 0, \gamma \ne 1 \quad \exists! \ \vartheta^*_\gamma \in \mathbb{R}, \quad \Phi(\vartheta^*_\gamma) = \vartheta^*_\gamma
\end{equation}

\end{proof}

\begin{remark}
    For the case when $\mathbb{K}_A \ne \mathbb{K}_B$ the exponent in equation \eqref{eq:OverlineLambdaDefinition} can be undefined. 
    Therefore the existence and uniqueness of $P^*_\gamma$ is given by a weaker convexity argument similar to the statement made in Section~\ref{sec:EquilibriumSolutionGeneralBayesianGame}, which maps the equilibrium finding problem to a convex optimization \cite{book:ConvexOpt}.
    
\end{remark}

\subsection{Fisher games as a limit case}

\paragraph{Notation:}

\begin{equation}
    \vartheta_0 = \log \left ( \frac{p_{k^*}(B)}{p_{k^*}(A)} \right ), \quad 
    \tau_0 = \log \left ( \frac{\nu^*}{1-\nu^*} \right )
\end{equation}

To obtain the $\gamma \to 0$ limit of equilibrium parameters, it is useful to introduce the following limit function:

\begin{equation}
    \Psi(\tau) = \lim_{\gamma \to 0} \frac{(\vartheta_0 + \gamma \tau) - \Phi(\vartheta_0 + \gamma \tau)}{\gamma}
\end{equation}

Some suitable notation and expressions to calculate this limit are:

\begin{equation}
    \frac{p_k(A)^{\frac{1-\gamma}{\gamma}}}{\left ( z \ p_k(A)^{1/\gamma} + z^{-1} \ p_k(B)^{1/\gamma} \right )^{1-\gamma}} = 
    z^{-(1-\gamma)} \left ( 1 + z^{-2} \left ( \frac{p_k(B)}{p_k(A)} \right )^{1/\gamma} \right )^{-(1-\gamma)}
\end{equation}

\begin{equation}
    \frac{p_k(B)^{\frac{1-\gamma}{\gamma}}}{\left ( z \ p_k(A)^{1/\gamma} + z^{-1} \ p_k(B)^{1/\gamma} \right )^{1-\gamma}} = 
    z^{(1-\gamma)} \left ( 1 + z^{2} \left ( \frac{p_k(A)}{p_k(B)} \right )^{1/\gamma} \right )^{-(1-\gamma)}
\end{equation}

\begin{equation}
    \vartheta - \Phi(\vartheta) = 
    \frac{\gamma}{1-\gamma} \log \left (
    \frac{
    \sum_k p_k(A) \left (1 + z^{-2} \left ( \frac{p_k(B)}{p_k(A)} \right )^{1/\gamma} \right )^{-(1-\gamma)}
    }
    {
    \sum_k p_k(B) \left ( 1 + z^{2} \left ( \frac{p_k(A)}{p_k(B)} \right )^{1/\gamma} \right )^{-(1-\gamma)}
    }
    \right )
\end{equation}

\paragraph{Taking the limit:}

Now, if we substitute $\vartheta = \vartheta_0 + \gamma \tau$, then all parts of the sums converge to a well-defined limit as $\gamma \to 0$.

Introducing the notation:

\begin{equation}
    \begin{split}
        c_k &= \left (1 + z^{-2} \left ( \frac{p_k(B)}{p_k(A)} \right )^{1/\gamma} \right )^{-(1-\gamma)} \\
        d_k &= \left ( 1 + z^{2} \left ( \frac{p_k(A)}{p_k(B)} \right )^{1/\gamma} \right )^{-(1-\gamma)}
    \end{split}
\end{equation}

\begin{equation}
    c_k =  \left (1 + e^{-\tau} \left ( \frac{p_{k^*}(A)}{p_{k^*}(B)} \frac{p_k(B)}{p_k(A)} \right )^{1/\gamma} \right )^{-(1-\gamma)}
\end{equation}

\begin{equation}
    d_k =  \left (1 + e^{\tau} \left ( \frac{p_{k^*}(B)}{p_{k^*}(A)} \frac{p_k(A)}{p_k(B)} \right )^{1/\gamma} \right )^{-(1-\gamma)}
\end{equation}

\begin{equation}
    c_k = 
    \begin{cases}
        1 + \mathcal{O}(e^{-|\lambda_k|/\gamma}) & \text{if } k < k^* \\
        \left (1 + e^{-\tau} \right )^{-(1-\gamma)} & \text{if } k=k^* \\
        0 + \mathcal{O}(e^{-|\lambda_k|/\gamma}) & \text{if } k > k^* \\
\end{cases}
\end{equation}

\begin{equation}
    d_k = 
    \begin{cases}
        0 + \mathcal{O}(e^{-|\lambda_k|/\gamma}) & \text{if } k < k^* \\
        \left (1 + e^{\tau} \right )^{-(1-\gamma)} & \text{if } k=k^* \\
        1 + \mathcal{O}(e^{-|\lambda_k|/\gamma}) & \text{if } k > k^* \\
\end{cases}
\end{equation}

where

\begin{equation}
    \lambda_k =   \log \left( \frac{p_{k^*}(A)}{p_{k^*}(B)} \frac{p_k(B)}{p_k(A)} \right )  
\end{equation}

Neglecting the $\mathcal{O}(e^{-|\lambda|/\gamma})$ contributions for $\Psi(\tau)$, considerably simplifies the expression. Borrowing the notation from Appendix \ref{Appendix:FisherGameEquilibrium}, equations \eqref{eq:FisherNotation_p*AB}, \eqref{eq:FisherNotation_SigmaAB} results:

\begin{equation}
    \Psi(\tau) = \lim_{\gamma \to 0} \frac{1}{1-\gamma}
    \log \left (
    \frac
    {\Sigma_A + p^*_A \left (1 + e^{-\tau}\right )^{-(1-\gamma)}}
    {\Sigma_B + p^*_B \left (1 + e^{\tau} \right )^{-(1-\gamma)}}
    \right )
\end{equation}

\begin{equation}
    \Sigma_A + p^*_A \nu^* = \Sigma_B + p^*_B (1-\nu^*) = v^*
\end{equation}

this further simplifies the expression:

\begin{equation}
    \Psi(\tau) = 
    \log \left (
    \frac
    {v^* + p^*_A \left ( \left (1 + e^{-\tau}  \right )^{-1} - \nu^* \right)}
    {v^* + p^*_B \left ( \left (1 + e^{\tau}  \right )^{-1} - (1-\nu^*) \right )}
    \right )
\end{equation}

or

\begin{equation}
    \Psi(\tau) = 
    \log \left (
    \frac
    {1 + r_A \left ( \left (1 + e^{-\tau}  \right )^{-1} - \nu^* \right)}
    {1 + r_B \left ( \left (1 + e^{\tau}  \right )^{-1} - (1-\nu^*) \right )}
    \right )
\end{equation}

where

\begin{equation}
    r_A = p^*_A / v^*, \quad r_B = p^*_B / v^*
\end{equation}

\begin{theorem}
    For any $r_A >0$, $r_B >0$, $\nu^* \in (0,1)$
    \begin{equation}
        \Psi(\tau_0 + \Delta \tau) > 0 \text{ if } \Delta \tau >0, \quad \Psi(\tau_0 + \Delta \tau) < 0 \text{ if } \Delta \tau <0
    \end{equation}
\end{theorem}

\begin{proof}
    A slightly stronger set of inequalities hold:

    \begin{equation}
        \left ( 1 + e^{-\Delta \tau} \frac{1-\nu^*}{\nu^*}\right )^{-1} - \nu^* > 0, \quad \text{if } \Delta \tau>0 
    \end{equation}
    
    \begin{equation}
        \left ( 1 + e^{-\Delta \tau} \frac{1-\nu^*}{\nu^*}\right )^{-1} - \nu^* < 0, \quad \text{if } \Delta \tau<0
    \end{equation}

    \begin{equation}
        \left ( 1 + e^{\Delta \tau} \frac{\nu^*}{1-\nu^*}\right )^{-1} - (1-\nu^*) < 0, \quad \text{if } \Delta \tau>0
    \end{equation}

        \begin{equation}
        \left ( 1 + e^{\Delta \tau} \frac{\nu^*}{1-\nu^*}\right )^{-1} - (1-\nu^*) > 0, \quad \text{if } \Delta \tau<0
    \end{equation}

    All these inequalities can be directly checked. As an example, we can take the first one:

    \begin{equation}
        \left ( 1 + e^{-\Delta \tau} \frac{1-\nu^*}{\nu^*}\right )^{-1} - \nu^* > 0
    \end{equation}
    
    \begin{equation}
        1 + e^{-\Delta \tau} \frac{1-\nu^*}{\nu^*} < \frac{1}{\nu^*} 
    \end{equation}

    \begin{equation}
        e^{-\Delta \tau} \frac{1-\nu^*}{\nu^*} < \frac{1-\nu^*}{\nu^*} 
    \end{equation}

    \begin{equation}
        e^{-\Delta \tau}  < 1, \quad \text{if } \Delta \tau > 0
    \end{equation}

    From these four inequalities, the theorem follows.

\end{proof}

\begin{theorem}
    $\forall \varepsilon_{+}, \varepsilon_{-} > 0, \ \exists \underline{\gamma}>0$ that:
    \begin{equation}
        \forall \gamma < \underline{\gamma}, \quad  \vartheta^*(\gamma) -
        \left ( \log \left ( \frac{p^*_B}{p^*_A} \right ) + \gamma \log \left ( \frac{\nu^*}{1-\nu^*} \right ) \right ) \in [-\gamma \epsilon_{-}, \gamma \epsilon_{+}]
    \end{equation}
\end{theorem}

\begin{proof}
    This follows from the uniqueness and existence of $\vartheta^*(\gamma)$ for all $\gamma>0$ and from the definition of $\Psi(\tau)$ as a limit $\gamma \to 0$.

\end{proof}

\begin{remark}
    To generate higher order approximations in $\gamma$ for $\vartheta^*(\gamma)$ one can iterate the expression:
    \begin{equation}
        \vartheta^*_{n+1}(\gamma) = \Phi(\vartheta^*_n(\gamma)), \quad
        \vartheta^*_0 = \vartheta_0 + \gamma \tau_0
    \end{equation}

    For example, the first iteration gives:
    \begin{equation}
        \begin{split}
            \vartheta^*_1(\gamma) =& \log \left ( \frac{p^*_B}{p^*_A} \right ) + \gamma \log \left ( \frac{\nu^*}{1-\nu^*} \right ) + \\ 
            & \gamma^2 \left ( r_A \nu^* \log(\nu^*) - r_B (1-\nu^*) \log(1-\nu^*)\right )+
            \mathcal{O}(\gamma^3)
        \end{split}
    \end{equation}

    However, as we will see, this does not give the exact second-order approximation of $\vartheta^*(\gamma)$.
    
\end{remark}

\begin{remark}
    Remarkably, a similar derivation for the first derivative of $\Phi(\vartheta)$ shows that it remains strictly smaller than 1 at $\vartheta_0$ as $\gamma \to 0$.

    \begin{equation}
        \lim_{\gamma \to 0} \Phi'(\vartheta^*_0) = 1 - (r_A + r_B) \nu^* (1-\nu^*)
    \end{equation}

    which signals a much faster convergence than what is guaranteed by the global upper bound in \eqref{eq:GlobalBoundPhi}.
    
\end{remark}

\begin{remark}
    In fact, all derivatives of $\vartheta^*(\gamma)$ can be determined as $\gamma \to 0$ by using the ``perturbative parts'' in the definition of $\Psi$:

    \begin{equation}
        \Psi_P(\tau,\gamma) = 
        \log \left (
        \frac
        {1 + r_A \left ( \left (1 + e^{-\tau} \right )^{-(1-\gamma)} - \nu^* \right)}
        {1 + r_B \left ( \left (1 + e^{\tau} \right )^{-(1-\gamma)} - (1-\nu^*) \right )}
        \right )
    \end{equation}

    An implicit definition for $\tau^\times(\gamma)$:
    \begin{equation}
        \Psi_P(\tau^\times(\gamma),\gamma) = 0
    \end{equation}

    which is equivalent to finding the solution for \footnote{because for Fisher games $P^*_0 = p^*_B/(p^*_A + p^*_B)=r_B/(r_A+r_B)$}:

    \begin{equation}
        f(\tau,\gamma) = 
        (1-P^*_0) \left ( \left (1 + e^{-\tau} \right )^{-(1-\gamma)} - \nu^* \right) - 
        P^*_0 \left ( \left (1 + e^{\tau} \right )^{-(1-\gamma)} - (1-\nu^*) \right )
    \end{equation}

    and

    \begin{equation}
        f(\tau^\times(\gamma),\gamma) = 0
    \end{equation}

    by which, one can define a perturbative approximation of $\vartheta^*(\gamma)$:

    \begin{equation}
        \vartheta^\times(\gamma) = \vartheta_0 + \gamma \ \tau^\times(\gamma)
    \end{equation}

    $\vartheta^*(\gamma)$ and $\vartheta^\times(\gamma)$ are not the same, but their derivatives match upto every finite order \cite{note:AsymptoticSeries} \footnote{the $\mathcal{O}(e^{-|\lambda_k|/\gamma})$ terms can remind someone to ``Instanton contributions'' in quantum field theories \cite{arxiv:Instantons}}:

    \begin{equation}
        \forall n \in \mathbb{N}, \quad 
        \lim_{\gamma \to 0} \vartheta^*(\gamma)^{(n)} = \vartheta^\times(0)^{(n)}
    \end{equation}

    For the difference, the $\mathcal{O}(e^{-|\lambda|/\gamma})$ terms are responsible, which are not visible for the derivatives at 0.
    
\end{remark}

\begin{remark}
    Using the previous formula, an exact second-order approximation can be obtained:
    \begin{equation}
        \begin{split}
            \vartheta^*(\gamma) =& \log \left ( \frac{p^*_B}{p^*_A} \right ) + \gamma \log \left ( \frac{\nu^*}{1-\nu^*} \right ) + \\ 
            & \gamma^2 \ \frac{P^*_0 \nu^* \log(\nu^*) - (1-P^*_0) (1-\nu^*) \log(1-\nu^*)}{ \nu^* (1-\nu^*)}+
            \mathcal{O}(\gamma^3)
        \end{split}
    \end{equation}
    
\end{remark}

\begin{remark}
    The quantity $\tau^*(\gamma)$ is intimately related to $p'^*_{\gamma,k^*}$, in the following way:

    \begin{equation}
        \frac{p'^*_{\gamma,k^*}}{1-p'^*_{\gamma,k^*}} = 
        \left ( \frac{P^*_\gamma}{1-P^*_\gamma} \right )^{1/\gamma}
        \left ( \frac{p_{k^*}(A)}{p_{k^*}(B)} \right )^{1/\gamma}
    \end{equation}

    \begin{equation}
        \log \left ( \frac{p'^*_{\gamma,k^*}}{1-p'^*_{\gamma,k^*}} \right ) = 
        \frac{1}{\gamma} \left ( \vartheta^*(\gamma) - \vartheta_0 \right ) = 
        \tau^*(\gamma)
    \end{equation}

    \begin{equation}
        p'^*_{\gamma,k^*} = \frac{e^{\tau^*(\gamma)}}{1+e^{\tau^*(\gamma)}}
    \end{equation}
    
\end{remark}

Using the connection between $p'^*_{\gamma,k^*}$ and $\tau^*(\gamma)$ we can define a perturbative splitting ratio near $\gamma \approx 0$:

\begin{equation}
    p'^\times_{\gamma,k^*} = \frac{e^{\tau^\times(\gamma)}}{1+e^{\tau^\times(\gamma)}}
\end{equation}

\subsubsection{Example:}

\begin{figure}[H]
    \centering
    \includegraphics[scale=0.8]{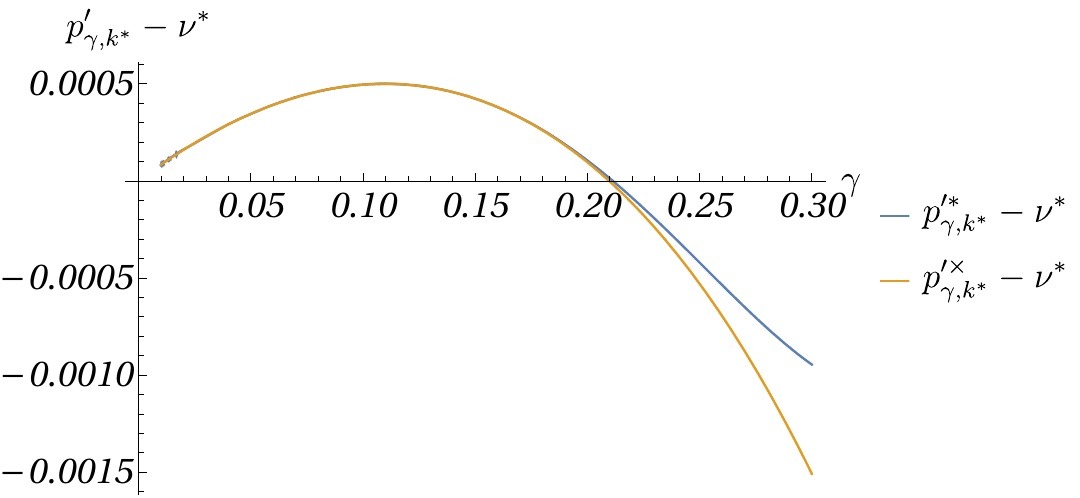}
    \caption{Illustration of $p'^*_{\gamma,k^*}$ and $p'^\times_{\gamma,k^*}$ for \SG{N=2,K_A=2,K_B=7,M=10}.}
    \label{fig:Asymptotic_pp}
\end{figure}

\section{Restricted iteration}
\label{appendix:Restricted}

This proof attempt belongs to Conjecture \ref{conj:Restricted}

\vspace{\baselineskip}

\noindent \emph{Proof attempt.} \ We can take the derivative of $\invbreve{F}(\chi)$ with respect to $\chi$:

    \begin{equation}
        \invbreve{F}'(\chi) = 
        \frac{1}{\sqrt{Z_A Z_B}} \sum_{k \in \mathbb{K}_{AB}} \frac{(p_k(A)-p_k(B))^2}{P(\chi) p_k(A)+(1-P(\chi)p_k(B))} 
        \left ( \frac{d \chi}{d P} \right )^{-1}
    \end{equation}

where

    \begin{equation} 
        \frac{1}{\sqrt{Z_A Z_B}} 
        \left ( \frac{d \chi}{d P} \right )^{-1} = \frac{1}{Z_A/P(\chi)+Z_B/(1-P(\chi))}
    \end{equation}

We can introduce a new normalized set of variables:

    \begin{equation}
        \invbreve{p}_k(\theta) = \frac{1}{Z_\theta} p_k(\theta)
    \end{equation}

For any normalized ``On-shell''\cite{book:Weinberg} set of probabilities ($\sum_k \invbreve{p}_{\theta,k} =1$)

    \begin{equation}
        \underline{\underline{\invbreve{p}}} \in [0,1]^{2 \times |\mathbb{K}_{AB}|}
    \end{equation}

We will denote the set of all normalized probability measures on $\mathbb{K}_{AB}$ by $\mathbb{P}$:

\begin{equation}
    \mathbb{P} = \{ p:\mathbb{K}_{AB} \mapsto [0,1] \ | \  ||p||_1 = 1 \} 
\end{equation}

in this notation:

\begin{equation}
    \underline{\underline{\invbreve{p}}} \in
    \mathbb{P} \times \mathbb{P} = \mathbb{P}_2
\end{equation}
    
\begin{equation}
    \invbreve{F}^\circ{}'(\chi,\underline{\underline{\invbreve{p}}})
    =
    \frac{1}{Z_A/P+Z_B/(1-P)} \sum_{k \in \mathbb{K}_{AB}} \frac{(Z_A \ \invbreve{p}_{A,k} - Z_B \ \invbreve{p}_{B,k})^2}{P \ Z_A \invbreve{p}_{A,k} + (1-P) Z_B \invbreve{p}_{B,k}}
\end{equation}

(Where $P$ here means the function $P(\chi)$.)

It is easy to observe that this expression for the derivative is always positive:

\begin{equation}
    \invbreve{F}^\circ{}'(\chi,\underline{\underline{\invbreve{p}}}) > 0
\end{equation}

To show that this derivative is always less than $1$, we show that its value can take $1$ only for an extreme normalized probability distribution ($\underline{\underline{\invbreve{p}}}^E \in \mathbb{P}_2^E$) and that the value for a non-extreme distribution is always less than this maximal value.

\begin{equation}
    \underline{\underline{\invbreve{p}}}^E \in \mathbb{P}_2^E \iff 
    \forall k \in \mathbb{K}_{AB}, \ 
    \invbreve{p}^E_{A,k} \cdot \invbreve{p}^E_{B,k} = 0
\end{equation}

For such extreme probability distribution pairs it can be seen that:

\begin{equation}
    \invbreve{F}^\circ{}'(\chi,\underline{\underline{\invbreve{p}}}^E) =
    \frac{1}{Z_A/P+Z_B/(1-P)} 
    \left (
    \frac{Z_A}{P} + \frac{Z_B}{1-P}
    \right ) = 1
\end{equation}

To show that making normalized distributions more extreme increases the value of the derivative, we introduce the continuous elementary permutations or continuous swaps. For any $\lambda \in [0,1]$, $k,\ell \in \mathbb{K}_{AB}$, $\theta \in \Theta = \{A,B\}$ we can define $\mathscr{P}^\lambda_{\theta,(k,\ell)} : \mathbb{P}_2 \mapsto \mathbb{P}_2$ in the following way:

\begin{equation}
    \underline{\underline{\invbreve{q}}} = \mathscr{P}^\lambda_{\theta,(k,\ell)} (\underline{\underline{\invbreve{p}}})
    \iff
    \invbreve{q}_{\theta,k} = \lambda (\invbreve{p}_{\theta,k}+\invbreve{p}_{\theta,\ell}) , \quad
    \invbreve{q}_{\theta,\ell} = (1-\lambda) (\invbreve{p}_{\theta,k}+\invbreve{p}_{\theta,\ell})
\end{equation}

It can be shown that all $\underline{\underline{p}}, \underline{\underline{q}} \in \mathbb{P}_2$ pairs can be connected by some finite chain of continuous elementary permutations: $\underline{\underline{q}} = \mathscr{P}_n \circ,\dots,\circ \mathscr{P}_1 (\underline{\underline{p}})$. \footnote{This is similar to the ``Transposition Theorem'', which states that ``any permutation of a finite set containing at least two
elements can be written as the product of transpositions'' \cite{book:AbstractAlgebra}.}

To see, how the continuous permutation can change $\invbreve{F}^\circ{}'$, we can take its second derivative with respect to $\lambda$:

\begin{equation}
    \frac{\partial^2}{\partial \lambda^2}
    \invbreve{F}^\circ{}'(\chi,\mathscr{P}^\lambda_{\theta,(k,\ell)}(\underline{\underline{\invbreve{p}}}))
\end{equation}

The second derivative can be calculated explicitly:

\begin{equation}
    \frac{2 \left(p_{A,k}+p_{A,\ell}\right)^2 p_ {B,k}^2}{\left(P \
\lambda  \left(p_{A,k}+p_{A,\ell}\right)+(1-P) p_{B,k}\right)^3}+\frac{2 \left(p_{A,k}+p_{A,\ell}\right)^2 p_{B,\ell}^2}{\left(P
(1-\lambda ) \left(p_{A,k}+p_{A,\ell}\right)+(1-P) p_{B,\ell}\right)^3}
\end{equation}

which shows that it is always greater or equal to $0$:

\begin{equation}
    \frac{\partial^2}{\partial \lambda^2}
    \invbreve{F}^\circ{}'(\chi,\mathscr{P}^\lambda_{\theta,(k,\ell)}(\underline{\underline{\invbreve{p}}})) \ge 0
\end{equation}

This means that as a function of $\lambda$, the expression will take its maximum on the boundary, i.e. $\lambda \in \{0,1\}$.

Explicit calculation gives the difference between values at $\lambda = 1$ and $\lambda=0$:

\begin{equation}
    \Delta \invbreve{F}' = 
    \invbreve{F}^\circ{}'(\chi,\mathscr{P}^1_{\theta,(k,\ell)}(\underline{\underline{\invbreve{p}}})) -
    \invbreve{F}^\circ{}'(\chi,\mathscr{P}^0_{\theta,(k,\ell)}(\underline{\underline{\invbreve{p}}}))
\end{equation}

\begin{equation}
    \Delta \invbreve{F}' = 
    \frac{\left(p_{A,k}+p_{A,\ell}\right)^2 \left(p_{B,\ell}-p_{B,k}\right)}
    {(1-P) \left(P \left(p_{A,k}+p_{A,\ell}\right)+(1-P) p_{B,k}\right) \left(P
\left(p_{A,k}+p_{A,\ell}\right)+(1-P) p_{B,\ell}\right)}
\end{equation}

From which the following conditional equations follow:

\begin{equation}
    \begin{split}
        \invbreve{F}^\circ{}'(\chi,\mathscr{P}^1_{A,(k,\ell)}(\underline{\underline{\invbreve{p}}})) = &
        \sup_{\lambda \in [0,1]} \invbreve{F}^\circ{}'(\chi,\mathscr{P}^\lambda_{A,(k,\ell)}(\underline{\underline{\invbreve{p}}})), \quad \text{if } \invbreve{p}_{B,k} \le \invbreve{p}_{B,\ell} \\
        \invbreve{F}^\circ{}'(\chi,\mathscr{P}^0_{A,(k,\ell)}(\underline{\underline{\invbreve{p}}})) = & 
        \sup_{\lambda \in [0,1]} \invbreve{F}^\circ{}'(\chi,\mathscr{P}^\lambda_{A,(k,\ell)}(\underline{\underline{\invbreve{p}}})), \quad \text{if } \invbreve{p}_{B,k} \ge \invbreve{p}_{B,\ell} \\
        \invbreve{F}^\circ{}'(\chi,\mathscr{P}^1_{B,(k,\ell)}(\underline{\underline{\invbreve{p}}})) = &
        \sup_{\lambda \in [0,1]} \invbreve{F}^\circ{}'(\chi,\mathscr{P}^\lambda_{B,(k,\ell)}(\underline{\underline{\invbreve{p}}})), \quad \text{if } \invbreve{p}_{A,k} \le \invbreve{p}_{A,\ell} \\
        \invbreve{F}^\circ{}'(\chi,\mathscr{P}^0_{B,(k,\ell)}(\underline{\underline{\invbreve{p}}})) = & 
        \sup_{\lambda \in [0,1]} \invbreve{F}^\circ{}'(\chi,\mathscr{P}^\lambda_{B,(k,\ell)}(\underline{\underline{\invbreve{p}}})), \quad \text{if } \invbreve{p}_{A,k} \ge \invbreve{p}_{A,\ell}
    \end{split}
\end{equation}

We can now introduce the following extremization operators:

\begin{equation}
    \mathscr{E}_{\theta, k \to \ell} = \mathscr{P}_{\theta,(k,\ell)}^0 , \quad
    \mathscr{E}_{\theta, \ell \to k} = \mathscr{P}_{\theta,(k,\ell)}^1
\end{equation}

By successively applying appropriate extremization operators, we can get an extreme state from an initial state:

\begin{equation}
    \underline{\underline{\invbreve{p}}}^E = \mathscr{E}_n \circ \dots \circ \mathscr{E}_1(\underline{\underline{\invbreve{p}}}), \quad
    \underline{\underline{\invbreve{p}}}^E \in \mathbb{P}_2^E, \ 
    \underline{\underline{\invbreve{p}}} \in \mathbb{P}_2
\end{equation}

where every step increases the value of $\invbreve{F}^\circ{}'(\chi,.)$.
This implies that if $\underline{\underline{\invbreve{p}}}$ is not extreme:

\begin{equation}
    \invbreve{F}^\circ{}'(\chi,\underline{\underline{\invbreve{p}}}) <
    \invbreve{F}^\circ{}'(\chi,\underline{\underline{\invbreve{p}}}^E) = 1, \quad \text{if} \ \underline{\underline{\invbreve{p}}} \notin \mathbb{P}_2^E
\end{equation}

From this follows that:

\begin{equation}
    \forall \chi \in \mathbb{R}, \ \invbreve{F}'(\chi) < 1
\end{equation}

To show that the inequality uniformly holds for any range of $\chi$, it is beneficial to observe that:

\begin{equation}
    \lim_{\chi \to -\infty} \invbreve{F}'(\chi) =0 \quad \text{and} \quad  \lim_{\chi \to \infty} \invbreve{F}'(\chi) = 0
\end{equation}

This means that:

\begin{equation}
    \forall \varepsilon >0, \exists \chi_{-}, \chi{+}, \quad 
    \forall \chi < \chi_{-} \vee \chi > \chi_{+}, \quad
    |\invbreve{F}'(\chi)| < \varepsilon
\end{equation}

We are interested in the smallest global upper bound of $|\invbreve{F}'(\chi)|$:

\begin{equation}
    \invbreve{q} = \sup_{\chi \in \mathbb{R}} |\invbreve{F}'(\chi)|
\end{equation}

The domain of $\chi$ can be split to $3$ subsets:

\begin{equation}
    \invbreve{q} = \max \left \{
    \sup_{\chi < \chi_{-}} |\invbreve{F}'(\chi)|,
    \sup_{\chi \in [\chi_{-},\chi_{+}]} |\invbreve{F}'(\chi)|,
    \sup_{\chi > \chi_{+}} |\invbreve{F}'(\chi)|
    \right \}
\end{equation}

\begin{equation}
    \invbreve{q} \le \max \left \{
    \varepsilon,
    \sup_{\chi \in [\chi_{-},\chi_{+}]} |\invbreve{F}'(\chi)|,
    \varepsilon
    \right \}
\end{equation}

A continuous function takes its supremum and infinum on a compact interval, therefore:

\begin{equation}
    \invbreve{F}' \in C^0(\mathbb{R},\mathbb{R}), \ \forall \chi, |\invbreve{F}'(\chi)|<1 \implies 
    \sup_{\chi \in [\chi_{-},\chi_{+}]} |\invbreve{F}'(\chi)| < 1
\end{equation}

which for $\varepsilon<1$ choice guarantees that:

\begin{equation}
    \invbreve{q} = \sup_{\chi \in \mathbb{R}} |\invbreve{F}'(\chi)| < 1
\end{equation}

This attempts to complete the proof.

\begin{flushright}
\rotatebox[origin=c]{180}{$\heartsuit$}
\end{flushright}

\section{Questions leading to the present work}
\label{Appendix:PersonalMotivation}

\paragraph{Personal but genuine motivation:}

The way I followed and the puzzles which interested me might not be everybody's path, but sketching my own story might resonate with some because it has grown from genuine enthusiasm. 

\paragraph{Random mass events:}

For a relatively long time -- more or less until the end of my undergraduate studies -- my default framework for probabilistic reasoning would be a mixture of Classical probability theory (providing the principle of indifference \footnote{or principle of insufficient reason}) and the Frequentist interpretation (by which I was able to ground derived probabilistic statements to the results of many times repeated experiments).
I think this is no wonder, given that all formal educational material backing my theoretical physicist degree were in accordance (or at least mostly compatible) with this blend of frameworks.
Error analysis in experimental physics \cite{book:ErrorAnalysis,book:ExperimentationBaird}, the convenient thermodynamic limit in statistical physics, the counters of ionizing radiation from radioactive decay gave a convincing grounding for the objective reality of the long-term relative frequencies. Axiomatic probability theory \cite{book:Kolmogorov} provided a rich and robust mathematical language, in which initial assumptions could be transformed to probabilistic statements about complicated events.

However, personally, I always felt a little uneasy when I wanted to sincerely interpret the result of error analysis for a concrete laboratory experiment or -- to give a different example -- the law of large numbers, which was sometimes introduced as the link between probability theory and statistics.
With tools from ``frequentist statistics,'' I was equipped with techniques that were asymptotically correct and consistent. However, we only ever had finite samples (and for some more delicate experiments, only a few).
Furthermore, theoretically, the connection between probability theory and real-world events was not entirely convincing because the arguments always felt suspiciously circular \cite{book:UniversalArtificialIntelligence}. \footnote{Probably the ``\textsection 2. The Relation to Experimental Data'' in Kolmogorov's axiomatization \cite{book:Kolmogorov} made a symbolic attempt to make the connection possible by postulating: ``4)
Under certain conditions, which we shall not discuss here, we may assume that to an event $A$ which may or may not occur under conditions $\mathfrak{S}$, is assigned a real number $P(A)$ which has
the following characteristics:
(a) One can be practically certain that if the complex of conditions $\mathfrak{S}$ is repeated a large number of times, $n$, then if to be the number of occurrences of event $A$, the ratio $m /n$ will differ very
slightly from $P(A)$.'', Kolmogorov points out in a straightforward footnote that his axiomatization aims to establish a purely mathematical framework, and writes: ``In establishing the premises
necessary for the applicability of the theory of probability to the world of
actual events, the author has used, in large measure, the work of R. v. Mises \cite{book:VonMises}''}

Against all my philosophical doubts, I adopted a pragmatic attitude, neatly summarized by Mark Kac in the following quote:
``If a probability of a certain event was calculated in accordance
with certain assumptions and rules, then the probability (again
calculated according to the same assumptions and rules) that the
frequency with which the event will occur in a large assembly of
trials will differ significantly from the calculated probability is
small.
Modest as the statement is, it is essentially all one can expect
from a purely mathematical theory.
The applicability of such a theory to natural sciences must
ultimately be tested by an experiment. But this is true of all
mathematical theories when applied outside the realm of mathematics, and the vague feeling of discomfort one encounters (mostly
among philosophers!) when first subjected to statistical reasoning
must be attributed to the relative novelty of the ideas.
To me there is no methodological distinction between the
applicability of differential equations to astronomy and of probability theory to thermodynamics or quantum mechanics.
It works! And brutally pragmatic as this point of view is, no
better substitute has been found.'' \cite{book:Kac}

\paragraph{A Bayesian turn:}

My hesitation to accept the previously described views as the only valid framework started to grow when I began to successfully use statistical techniques -- routinely applied in experimental physics -- to deterministic (but truncated) numerical results related to purely theoretical models \cite{paper:BajnokKonczeretall}.
In these cases, seemingly none of the usual requirements were present, which might justify a probabilistic or statistical treatment, but these concepts and interpretations somehow worked remarkably well.

This was the point when I started to entertain the thought that maybe it is more beneficial if we liberate the phenomena from the requirement of being objectively stochastic and we reserve the probabilistic language to describe only our partial and subjective knowledge about physical (or even theoretical) ``objects'' or state of affairs.

My entry point to Bayesian statistics was the famous textbook of E. T. Jaynes (published posthumously in 2003) \cite{book:Jaynes}. The possibility of a new, broader and more coherent framework captivated my imagination, and I soon started to explore the applicability of its methods to various problems.

However, one issue remained: the unbearable subjectivity of the prior probabilities.

\paragraph{Quest for Objective prior(s):}

I think in most situations, the appropriate place where probabilities could be identified is not the outside world but our internal partial description of the world.
However, in my view, this does not automatically mean that one has to give up all principles or criteria to filter out priors, which can potentially ruin any inference based on a finite amount of data. \footnote{A simple example for such principle might be Cromwell's rule \cite{book:Lindley,book:BayesianSocialScience}, which simply suggests not to use a prior, which associates $0$ or $1$ to specific events (the statement can be made without further complication if the set of possibilities is discrete).}
There is wisdom in the bold view of Bruno de Finetti \cite{book:deFinetti}, emphasizing that an agent's model about the world is already a subjective mental construction; therefore, constructing a prior might not be a seriously demanding step.
However, putting the subjective Bayesian view aside, I started searching for ``Objective'' (also known as uninformative) priors, which might serve at least as robust default choices. 

Once one starts to search for objective priors, many different suggestions appear, such as the classical principle of indifference, maximum entropy priors, transformation group invariant priors \cite{book:Jaynes}, Jeffreys prior (which is invariant under diffeomorphisms on the parameter space) \cite{book:Jaynes,paper:InformationGeometry,paper:JeffreysPriorOriginal}, Reference prior \cite{paper:ReferencePrior,book:Bernardo}, etc.

A seemingly plentiful menu of ``objective'' prior choices gave the motivation to search for arguments and principles that might prefer one to another or construct more ``natural'' ones.

\paragraph{Little less thinking, little more action:}

During the search for convincing principles for priors in Bayesian inference, it appeared that to rate different prior choices (with finite sample sizes), one has to link the result of inference to actions and then compare utilities (or loss functions or rewards) associated with possible consequences.
However, the Bayesian framework is traditionally introduced mainly as an epistemological discipline concerned about knowledge, belief and ``rationality'', but makes no meaningful statements about different prior choices. I felt it has to be embedded into a more general decision-making framework, where the main concepts are action, consequences, utilities and strategy.

For instance, consider a scenario where an agent receives a reward in a parametric model. This reward is based on the relative entropy -- also known as Kullback-Leibler divergence -- between her inferred distribution, which is based on her data, and the real distribution. Furthermore, assume that the parameters of this model have a known $\pi_0$ prior distribution.
In such a case, the agent can maximize her expected reward (or minimize her expected loss). It can be proven that to achieve this, she must strictly follow the Bayesian updating rule to determine her inferred distribution.

This observation gave motivation to search for decision-making problems, where Bayesian update is not ``only'' a result of a consistency argument but an optimal solution to an optimization or equilibrium finding problem.

\paragraph{Game theory as foundation:}

A natural framework in which randomization and mixed strategies appear is Game Theory \cite{book:EssentialGameTheory,book:GameTheory,book:EvolutionaryGames,review:NeumannMorgensternGameThoery,book:GameTheoryOriginal}.
Knowing that they will play a rock-paper-scissors, players can deduce in advance that their best strategy is to choose from all possible moves as unpredictably as possible.
In this case, the concepts of randomization and mixed strategy are results of a careful consideration, solving a paradoxical dilemma where no deterministic decision function could work well.

Traditionally, randomization in mixed strategies was defined and understood with the help of probability theory, but in a way, this might be reversed.
One can ask: if an agent can plan and execute randomized mixed strategies in simple games, then could we use such games to ground concepts in probability theory?
Is it possible, or even desirable, to use other analogies to statistical problems, which do not model real-world events based on degrees of belief or fair coins and dice? Instead, could we map such problems to games and interpret the statistical decision functions as equilibrium strategies?
What could or should we assume about our ``opponent''? Are there good default assumptions?

In particular, what is a game in which Bayesian update is the optimal strategy?
Can this framework provide an ``objective'' prior as an ingredient of our optimal strategy in some kind of inference game?

These questions motivated me to explore the topic, resulting in the present work.
This quest brought me to an exciting journey, leading to various scientific disciplines and providing a range of mathematical puzzles and the possibility of a rich abstract structure.
To illustrate the general concept, I tried to find the simplest statistical inference problem, which can serve as a toy model, and where exact answers can be given.
I hope that I will be able to partly share with the reader the fruits of these mathematical explorations and open (or reopen) the door to structures that emerge from a game theoretic interpretation of probability.

\bibliographystyle{plaindin}

\bibliography{ref}

\end{document}